\documentclass[10pt]{article}
\usepackage{hyperref}
\usepackage{amsmath}
\usepackage{amsfonts}
\usepackage{amssymb}
\usepackage{amsthm}

\usepackage[shortlabels]{enumitem}
\usepackage{verbatim}
\usepackage{graphicx}
\usepackage{microtype}
\usepackage{fullpage}
\usepackage{xparse}
\usepackage{mathrsfs}
\usepackage[capitalise]{cleveref}
\usepackage{mleftright}

\crefname{thm}{Theorem}{Theorems}
\Crefname{thm}{Theorem}{Theorems}
\crefname{conj}{Conjecture}{Conjectures}
\Crefname{conj}{Conjecture}{Conjectures}
\crefname{prop}{Proposition}{Propositions}
\Crefname{prop}{Proposition}{Propositions}
\crefname{cor}{Corollary}{Corollaries}
\Crefname{cor}{Corollary}{Corollaries}
\crefname{defn}{Definition}{Definitions}
\Crefname{defn}{Definition}{Definitions}
\crefname{rmk}{Remark}{Remarks}
\Crefname{rmk}{Remark}{Remarks}
\crefname{prob}{Problem}{Problems}
\Crefname{prob}{Problem}{Problems}

\crefname{enumi}{}{}
\Crefname{enumi}{}{}

\crefname{figure}{Figure}{Figures}
\Crefname{figure}{Figure}{Figures}

\crefformat{equation}{(#2#1#3)}
\crefrangeformat{equation}{(#3#1#4) to~(#5#2#6)}
\crefmultiformat{equation}{(#2#1#3)}%
{ and~(#2#1#3)}{, (#2#1#3)}{, and~(#2#1#3)}

\Crefformat{equation}{(#2#1#3)}
\Crefrangeformat{equation}{(#3#1#4) to~(#5#2#6)}
\Crefmultiformat{equation}{(#2#1#3)}%
{ and~(#2#1#3)}{, (#2#1#3)}{, and~(#2#1#3)}

\begin{document}

\NewDocumentCommand{\C}{}{{\mathbb{C}}}
\NewDocumentCommand{\R}{}{{\mathbb{R}}}
\NewDocumentCommand{\Q}{}{{\mathbb{Q}}}
\NewDocumentCommand{\Z}{}{{\mathbb{Z}}}
\NewDocumentCommand{\N}{}{{\mathbb{N}}}
\NewDocumentCommand{\M}{}{{\mathbb{M}}}
\NewDocumentCommand{\grad}{}{\nabla}
\NewDocumentCommand{\sA}{}{\mathcal{A}}
\NewDocumentCommand{\sF}{}{\mathcal{F}}
\NewDocumentCommand{\sH}{}{\mathcal{H}}
\NewDocumentCommand{\sD}{}{\mathcal{D}}
\NewDocumentCommand{\sB}{}{\mathcal{B}}
\NewDocumentCommand{\sC}{}{\mathcal{C}}
\NewDocumentCommand{\sE}{}{\mathcal{E}}
\NewDocumentCommand{\sL}{}{\mathcal{L}}
\NewDocumentCommand{\sT}{}{\mathcal{T}}
\NewDocumentCommand{\sO}{}{\mathcal{O}}
\NewDocumentCommand{\sP}{}{\mathcal{P}}
\NewDocumentCommand{\sQ}{}{\mathcal{Q}}
\NewDocumentCommand{\sR}{}{\mathcal{R}}
\NewDocumentCommand{\sM}{}{\mathcal{M}}
\NewDocumentCommand{\sI}{}{\mathcal{I}}
\NewDocumentCommand{\sK}{}{\mathcal{K}}
\NewDocumentCommand{\sZ}{}{\mathcal{Z}}
\NewDocumentCommand{\Span}{}{\mathrm{span}}
\NewDocumentCommand{\fM}{}{\mathfrak{M}}
\NewDocumentCommand{\fN}{}{\mathfrak{N}}
\NewDocumentCommand{\fX}{}{\mathfrak{X}}
\NewDocumentCommand{\fY}{}{\mathfrak{Y}}
\NewDocumentCommand{\gammat}{}{\tilde{\gamma}}
\NewDocumentCommand{\ct}{}{\tilde{c}}
\NewDocumentCommand{\bt}{}{\tilde{b}}
\NewDocumentCommand{\ch}{}{\hat{c}}
\NewDocumentCommand{\hd}{}{\hat{d}}
\NewDocumentCommand{\eh}{}{\hat{e}}
\NewDocumentCommand{\Ut}{}{\tilde{U}}
\NewDocumentCommand{\Gt}{}{\widetilde{G}}
\NewDocumentCommand{\Vt}{}{\tilde{V}}
\NewDocumentCommand{\ah}{}{\hat{a}}
\NewDocumentCommand{\at}{}{\tilde{a}}
\NewDocumentCommand{\Yh}{}{\widehat{Y}}
\NewDocumentCommand{\Yt}{}{\widetilde{Y}}
\NewDocumentCommand{\Ah}{}{\widehat{A}}
\NewDocumentCommand{\Ch}{}{\widehat{C}}
\NewDocumentCommand{\At}{}{\widetilde{A}}
\NewDocumentCommand{\Bt}{}{\widetilde{B}}
\NewDocumentCommand{\Bh}{}{\widehat{B}}
\NewDocumentCommand{\Mh}{}{\widehat{M}}
\NewDocumentCommand{\sLh}{}{\widehat{\sL}}
\NewDocumentCommand{\Sh}{}{\widehat{S}}
\NewDocumentCommand{\Xt}{}{\widetilde{X}}
\NewDocumentCommand{\Lt}{}{\widetilde{L}}
\NewDocumentCommand{\Xh}{}{\widehat{X}}
\NewDocumentCommand{\Lh}{}{\widehat{L}}
\NewDocumentCommand{\Eh}{}{\widehat{E}}
\NewDocumentCommand{\Fh}{}{\widehat{F}}
\NewDocumentCommand{\wh}{}{\hat{w}}
\NewDocumentCommand{\hh}{}{\hat{h}}
\NewDocumentCommand{\Phih}{}{\widehat{\Phi}}
\NewDocumentCommand{\Leaf}{o}{\IfNoValueTF{#1}{ \mathrm{Leaf}}{\mathrm{Leaf}_{#1}} }

\NewDocumentCommand{\Ho}{}{\mathbb{H}^1}

\NewDocumentCommand{\AMatrix}{}{\mathcal{A}}
\NewDocumentCommand{\AMatrixh}{}{\widehat{\AMatrix}}

\NewDocumentCommand{\hm}{}{g}

\NewDocumentCommand{\Zh}{}{\widehat{Z}}

\NewDocumentCommand{\Compact}{}{\mathcal{K}}


\NewDocumentCommand{\fg}{}{\mathfrak{g}}
\NewDocumentCommand{\etai}{}{\hat{\eta}}
\NewDocumentCommand{\etat}{}{\tilde{\eta}}

\NewDocumentCommand{\Deriv}{}{\mathscr{D}}
\NewDocumentCommand{\BofA}{}{\mathscr{B}}
\NewDocumentCommand{\ADeriv}{}{\mathscr{A}}

\NewDocumentCommand{\transpose}{}{\top}

\NewDocumentCommand{\ICond}{}{\sC}

\NewDocumentCommand{\LebDensity}{}{\sigma_{\mathrm{Leb}}}


\NewDocumentCommand{\Lie}{m}{\sL_{#1}}

\NewDocumentCommand{\ZygSymb}{}{\mathscr{C}}

\NewDocumentCommand{\Zyg}{m o}{\IfNoValueTF{#2}{\ZygSymb^{#1}}{\ZygSymb^{#1}(#2) }}
\NewDocumentCommand{\ZygX}{m m o}{\IfNoValueTF{#3}{\ZygSymb^{#2}_{#1}}{\ZygSymb^{#2}_{#1}(#3) }}

\NewDocumentCommand{\CSpace}{m o}{\IfNoValueTF{#2}{C(#1)}{C(#1;#2)}}

\NewDocumentCommand{\CjSpace}{m o o}{\IfNoValueTF{#2}{C^{#1}}{ \IfNoValueTF{#3}{ C^{#1}(#2)}{C^{#1}(#2;#3) } }  }

\NewDocumentCommand{\CXjSpace}{m m o}{\IfNoValueTF{#3}{C^{#2}_{#1}}{ C^{#2}_{#1}(#3) } }

\NewDocumentCommand{\ASpace}{m m o}{\mathscr{A}^{#1,#2}\IfNoValueTF{#3}{}{(#3)}}

\NewDocumentCommand{\AXSpace}{m m m o}{\mathscr{A}_{#1}^{#2,#3}\IfNoValueTF{#4}{}{(#4)}}

\NewDocumentCommand{\OSpace}{m o}{\mathscr{O}^{#1}_{b}\IfNoValueTF{#2}{}{(#2)}}
\NewDocumentCommand{\ONorm}{m m o}{\Norm{#1}[\IfNoValueTF{#3}{\OSpace{#2}}{\OSpace{#2}[#3]}]}

\NewDocumentCommand{\sBSpace}{m m m}{\mathscr{B}^{#1,#2}_{#3}}
\NewDocumentCommand{\sBNorm}{m m m m}{\Norm{#1}[\sBSpace{#2}{#3}{#4}]}

\NewDocumentCommand{\DSpace}{m m m m m}{\mathscr{D}^{#1,#2}_{#3,#4,#5}}
\NewDocumentCommand{\DNorm}{m m m m m m}{\Norm{#1}[\DSpace{#2}{#3}{#4}{#5}{#6} ]}

\NewDocumentCommand{\ComegaSpace}{m o o}{\IfNoValueTF{#2}{\CjSpace{\omega,#1}}{
\IfNoValueTF{#3}
{\CjSpace{\omega,#1}[#2]}
{\CjSpace{\omega,#1}[#2][#3]}
}
}

\NewDocumentCommand{\CXomegaSpace}{m m o o}{\IfNoValueTF{#3}{\CXjSpace{#1}{\omega,#2}}{
\IfNoValueTF{#4}
{\CXjSpace{#1}{\omega,#2}[#3]}
{\CXjSpace{#1}{\omega,#2}[#3][#4]}
}
}

\NewDocumentCommand{\ANorm}{m m m o}{\IfNoValueTF{#4}{\Norm{#1}[ \ASpace{#2}{#3} ]}{ \Norm{#1}[ \ASpace{#2}{#3}[#4] ] }}
\NewDocumentCommand{\BANorm}{m m m o}{\IfNoValueTF{#4}{\BNorm{#1}[ \ASpace{#2}{#3} ]}{ \BNorm{#1}[ \ASpace{#2}{#3}[#4] ] }}

\NewDocumentCommand{\AXNorm}{m m m m o}{\IfNoValueTF{#5}{\Norm{#1}[ \AXSpace{#2}{#3}{#4} ]}{ \Norm{#1}[ \AXSpace{#2}{#3}{#4}[#5] ] }}
\NewDocumentCommand{\BAXNorm}{m m m m o}{\IfNoValueTF{#5}{\BNorm{#1}[ \AXSpace{#2}{#3}{#4} ]}{ \BNorm{#1}[ \AXSpace{#2}{#3}{#4}[#5] ] }}

\NewDocumentCommand{\ComegaNorm}{m m o o}{\IfNoValueTF{#3}{ \Norm{#1}[\ComegaSpace{#2}] }
{
\IfNoValueTF{#4}
{\Norm{#1}[\ComegaSpace{#2}[#3]] }
{\Norm{#1}[\ComegaSpace{#2}[#3][#4]] }
}
}

\NewDocumentCommand{\CXomegaNorm}{m m m o o}{\IfNoValueTF{#4}{ \Norm{#1}[\CXomegaSpace{#2}{#3}] }
{
\IfNoValueTF{#5}
{\Norm{#1}[\CXomegaSpace{#2}{#3}[#4]] }
{\Norm{#1}[\CXomegaSpace{#2}{#3}[#4][#5]] }
}
}

\NewDocumentCommand{\HSpace}{m m o o}{\IfNoValueTF{#3}{C^{#1,#2}}{ \IfNoValueTF{#4} {C^{#1,#2}(#3)} {C^{#1,#2}(#3;#4)} }}

\NewDocumentCommand{\HXSpace}{m m m o}{\IfNoValueTF{#4}{C_{#1}^{#2,#3}}{  {C_{#1}^{#2,#3}(#4)}  }}

\NewDocumentCommand{\ZygSpace}{m o o}{\IfNoValueTF{#2}{\ZygSymb^{#1}}{ \IfNoValueTF{#3} { \ZygSymb^{#1}(#2) }{\ZygSymb^{#1}(#2;#3) } } }
\NewDocumentCommand{\ZygSpacemap}{m o o}{\IfNoValueTF{#2}{\ZygSymb^{#1}_{\mathrm{loc}}}{ \IfNoValueTF{#3} { \ZygSymb^{#1}_{\mathrm{loc}}(#2) }{\ZygSymb^{#1}_{\mathrm{loc}}(#2;#3) } } }

\NewDocumentCommand{\ZygXSpace}{m m o}{\IfNoValueTF{#3}{\ZygSymb^{#2}_{#1}}{\ZygSymb^{#2}_{#1}(#3) }}

\NewDocumentCommand{\Norm}{m o}{\IfNoValueTF{#2}{\| #1\|}{\|#1\|_{#2} }}
\NewDocumentCommand{\BNorm}{m o}{\IfNoValueTF{#2}{\mleft\| #1\mright\|}{\mleft\|#1\mright\|_{#2} }}

\NewDocumentCommand{\CjNorm}{m m o o}{ \IfNoValueTF{#3}{ \Norm{#1}[\CjSpace{#2}]} { \IfNoValueTF{#4}{\Norm{#1}[\CjSpace{#2}[#3]]} {\Norm{#1}[\CjSpace{#2}[#3][#4]]}  }  }

\NewDocumentCommand{\CNorm}{m m o}{\IfNoValueTF{#3}{\Norm{#1}[\CSpace{#2}]}{\Norm{#1}[\CSpace{#2}[#3]]}}

\NewDocumentCommand{\BCNorm}{m m}{\BNorm{#1}[\CSpace{#2}]}

\NewDocumentCommand{\CXjNorm}{m m m o}{\Norm{#1}[
\IfNoValueTF{#4}
{\CXjSpace{#2}{#3}}
{\CXjSpace{#2}{#3}[#4]}
]}

\NewDocumentCommand{\BCXjNorm}{m m m o}{\BNorm{#1}[
\IfNoValueTF{#4}
{\CXjSpace{#2}{#3}}
{\CXjSpace{#2}{#3}[#4]}
]}

\NewDocumentCommand{\LpNorm}{m m o o}{
\Norm{#2}[L^{#1}
\IfNoValueTF{#3}{}{
(#3
\IfNoValueTF{#4}{}{;#4}
)
}
]
}

\NewDocumentCommand{\BCjNorm}{m m o}{ \IfNoValueTF{#3}{ \BNorm{#1}[C^{#2}]} { \BNorm{#1}[C^{#2}(#3)]  }  }

\NewDocumentCommand{\HNorm}{m m m o o}{ \IfNoValueTF{#4}{ \Norm{#1}[\HSpace{#2}{#3}]} {
\IfNoValueTF{#5}
{\Norm{#1}[\HSpace{#2}{#3}[#4]]}
{\Norm{#1}[\HSpace{#2}{#3}[#4][#5]] }
}  }

\NewDocumentCommand{\HXNorm}{m m m m o}{ \IfNoValueTF{#5}{ \Norm{#1}[\HXSpace{#2}{#3}{#4}]} {
{\Norm{#1}[\HXSpace{#2}{#3}{#4}[#5]]}
}  }

\NewDocumentCommand{\BHXNorm}{m m m m o}{ \IfNoValueTF{#5}{ \BNorm{#1}[\HXSpace{#2}{#3}{#4}]} {
{\BNorm{#1}[\HXSpace{#2}{#3}{#4}[#5]]}
}  }

\NewDocumentCommand{\ZygNorm}{m m o o}{ \IfNoValueTF{#3}{ \Norm{#1}[\ZygSpace{#2}]} {
\IfNoValueTF{#4}
{\Norm{#1}[\ZygSpace{#2}[#3]]}
{\Norm{#1}[\ZygSpace{#2}[#3][#4]]}
}  }

\NewDocumentCommand{\BZygNorm}{m m o o}{ \IfNoValueTF{#3}{ \Norm{#1}[\ZygSpace{#2}]} {
\IfNoValueTF{#4}
{\BNorm{#1}[\ZygSpace{#2}[#3]]}
{\BNorm{#1}[\ZygSpace{#2}[#3][#4]]}
}  }

\NewDocumentCommand{\ZygXNorm}{m m m o}{\Norm{#1}[
\IfNoValueTF{#4}
{\ZygXSpace{#2}{#3}}
{\ZygXSpace{#2}{#3}[#4]}
]}

\NewDocumentCommand{\BZygXNorm}{m m m o}{\BNorm{#1}[
\IfNoValueTF{#4}
{\ZygXSpace{#2}{#3}}
{\ZygXSpace{#2}{#3}[#4]}
]}

\NewDocumentCommand{\diff}{o m}{\IfNoValueTF{#1}{\frac{\partial}{\partial #2}}{\frac{\partial^{#1}}{\partial #2^{#1}} }}

\NewDocumentCommand{\dt}{o}{\IfNoValueTF{#1}{\diff{t}}{\diff[#1]{t} }}

\NewDocumentCommand{\Zygad}{m}{\{ #1\}}

\NewDocumentCommand{\Zygsonu}{}{[s_0;\nu]}
\NewDocumentCommand{\Zygomeganu}{}{[\omega;\nu]}

\NewDocumentCommand{\Had}{m}{\langle #1\rangle}

\NewDocumentCommand{\BanachSpace}{}{\mathscr{X}}
\NewDocumentCommand{\BanachAlgebra}{}{\mathscr{Y}}

\NewDocumentCommand{\Field}{}{\mathbb{F}}

\NewDocumentCommand{\Real}{}{\mathrm{Re}}
\NewDocumentCommand{\Imag}{}{\mathrm{Im}}

\NewDocumentCommand{\VectorSpace}{m}{\mathscr{#1}}
\NewDocumentCommand{\VVS}{}{\VectorSpace{V}}
\NewDocumentCommand{\XVS}{}{\VectorSpace{X}}
\NewDocumentCommand{\LVS}{}{\VectorSpace{L}}
\NewDocumentCommand{\ZVS}{}{\VectorSpace{Z}}
\NewDocumentCommand{\WVS}{}{\VectorSpace{W}}
\NewDocumentCommand{\LVSh}{}{\widehat{\VectorSpace{L}}}
\NewDocumentCommand{\LVSb}{o}{\IfNoValueTF{#1}{\overline{\VectorSpace{L}}}{\overline{\VectorSpace{L}_{#1}}}}
\NewDocumentCommand{\WVSb}{o}{\IfNoValueTF{#1}{\overline{\VectorSpace{W}}}{\overline{\VectorSpace{W}_{#1}}}}
\NewDocumentCommand{\TVS}{}{\VectorSpace{T}}

\NewDocumentCommand{\Lb}{o}{\IfNoValueTF{#1}{\overline{L}}{\overline{L_{#1}}}}
\NewDocumentCommand{\wb}{o}{\IfNoValueTF{#1}{\overline{w}}{\overline{w}_{#1}}}
\NewDocumentCommand{\zb}{o}{\overline{z}\IfNoValueTF{#1}{}{_{#1}}}
\NewDocumentCommand{\Hb}{o}{\overline{H\IfNoValueTF{#1}{}{_{#1}}}}
\NewDocumentCommand{\Rb}{o}{\overline{R\IfNoValueTF{#1}{}{_{#1}}}}
\NewDocumentCommand{\Zb}{o}{\IfNoValueTF{#1}{\overline{Z}}{\overline{Z_{#1}}}}
\NewDocumentCommand{\Zbdelta}{m o}{\IfNoValueTF{#1}{\overline{Z^{#1}}}{\overline{Z^{#1}_{#2}}}}
\NewDocumentCommand{\Zhb}{o}{\IfNoValueTF{#1}{\overline{\Zh}}{\overline{\Zh_{#1}}}}

\NewDocumentCommand{\dbar}{}{\overline{\partial}}
\NewDocumentCommand{\etah}{}{\hat{\eta}}


\NewDocumentCommand{\SSFunctionSpacesSection}{}{Section 2}
\NewDocumentCommand{\SSStrangeZygSpace}{}{Remark 2.1}
\NewDocumentCommand{\SSBeyondManifold}{}{Section 2.2.1}
\NewDocumentCommand{\SSNormsAreInv}{}{Proposition 2.3}
\NewDocumentCommand{\SSDefineVectDeriv}{}{Remark 2.4}

\NewDocumentCommand{\SSSectionMoreOnAssumptions}{}{Section 4.1}
\NewDocumentCommand{\SSMainResult}{}{Theorem 4.7}
\NewDocumentCommand{\SSLemmaMoreOnAssump}{}{Proposition 4.14}

\NewDocumentCommand{\SSDivideWedge}{}{Section 5}
\NewDocumentCommand{\SSDerivWedge}{}{Lemma 5.1}

\NewDocumentCommand{\SSDensities}{}{Section 6}
\NewDocumentCommand{\SSDensitiesTheorem}{}{Theorem 6.5}
\NewDocumentCommand{\SSDensityCor}{}{Corollary 6.6}

\NewDocumentCommand{\SSScaling}{}{Section 7}
\NewDocumentCommand{\SSNSW}{}{Section 7.1}
\NewDocumentCommand{\SSHormandersCondition}{}{Section 7.1.1}
\NewDocumentCommand{\SSGenSubR}{}{Section 7.3}
\NewDocumentCommand{\SSGenSubResult}{}{Theorem 7.6}

\NewDocumentCommand{\SSCompareFunctionSpaces}{}{Lemma 8.1}
\NewDocumentCommand{\SSZygIsAlgebra}{}{Proposition 8.3}
\NewDocumentCommand{\SSBiggerNormMap}{}{Proposition 8.6}
\NewDocumentCommand{\SSCompareEuclidNorms}{}{Proposition 8.12}

\NewDocumentCommand{\SSDeriveODE}{}{Proposition 9.1}
\NewDocumentCommand{\SSExistODE}{}{Proposition 9.4}
\NewDocumentCommand{\SSExistXiOne}{}{Lemma 9.23}
\NewDocumentCommand{\SSDifferentOneAdmis}{}{Proposition 9.26}
\NewDocumentCommand{\SSCXjNormWedgeQuotient}{}{Lemma 9.32}
\NewDocumentCommand{\SSExistXiTwo}{}{Lemma 9.35}
\NewDocumentCommand{\SSComputefjzero}{}{Lemma 9.38}
\NewDocumentCommand{\SSSectionDensities}{}{Section 9.4}

\NewDocumentCommand{\SSProofInjectiveImmersion}{}{Appendix A}
\NewDocumentCommand{\SSFinerTopology}{}{Lemma A.1}

\newtheorem{thm}{Theorem}[section]
\newtheorem{cor}[thm]{Corollary}
\newtheorem{prop}[thm]{Proposition}
\newtheorem{lemma}[thm]{Lemma}
\newtheorem{conj}[thm]{Conjecture}
\newtheorem{prob}[thm]{Problem}

\theoremstyle{remark}
\newtheorem{rmk}[thm]{Remark}

\theoremstyle{definition}
\newtheorem{defn}[thm]{Definition}

\theoremstyle{definition}
\newtheorem{assumption}[thm]{Assumption}

\theoremstyle{remark}
\newtheorem{example}[thm]{Example}

\theoremstyle{remark}
\newtheorem{goal}[thm]{Goal}

\numberwithin{equation}{section}

\title{Sub-Hermitian Geometry and the Quantitative Newlander-Nirenberg Theorem}
\author{Brian Street\footnote{The author was partially supported by National Science Foundation Grant Nos.\ 1401671 and 1764265.}}
\date{}

\maketitle

\begin{abstract}
Given a finite collection of $C^1$ complex vector fields on a $C^2$ manifold $M$ such that they and their complex conjugates span the complexified tangent space at every point, the classical Newlander-Nirenberg theorem gives conditions on the vector fields so that there is a complex structure on $M$ with respect to which the vector fields are $T^{0,1}$.  In this paper, we give intrinsic, diffeomorphic invariant, necessary and sufficient conditions on the vector fields so that they have a desired level of regularity with respect to this complex structure (i.e., smooth, real analytic, or have Zygmund regularity of some finite order).  By addressing this in a quantitative way we obtain a holomorphic analog of the quantitative theory of sub-Riemannian geometry initiated by Nagel, Stein, and Wainger.  We call this sub-Hermitian geometry.  Moreover, we proceed more generally and obtain similar results for manifolds which have an associated formally integrable elliptic structure.  This allows us to introduce a setting which generalizes both the real and complex theories. 
\end{abstract}


\section{Introduction}
Let $M$ be a $C^2$ manifold and let $L_1,\ldots, L_m$ be $C^1$ complex vector fields on $M$.  Suppose, $\forall \zeta\in M$,
\begin{itemize}
\item $L_1(\zeta),\ldots, L_m(\zeta),\Lb[1](\zeta),\ldots, \Lb[m](\zeta)$ span $\C T_\zeta M$.
\item $[L_j,L_k](\zeta)\in \Span_{\C}\left\{L_1(\zeta),\ldots, L_m(\zeta)\right\}$, $\forall 1\leq j,k\leq m$.
\item $\Span_{\C}\left\{L_1(\zeta),\ldots, L_m(\zeta)\right\} \bigcap \Span_{\C}\left\{\Lb[1](\zeta),\ldots, \Lb[m](\zeta)\right\}=\{0\}$.
\end{itemize}
Under these conditions, the classical Newlander-Nirenberg Theorem (see \cite{HillTaylorIntegrabilityOfRoughAlmostComplexStructures}) states that $M$ can be given the structure of a complex manifold
such that $L_1(\zeta),\ldots, L_m(\zeta)$ form a spanning set of $T^{0,1}_\zeta(M)$, $\forall \zeta\in M$; and this is the unique such complex structure on $M$.
For $s>0$ we let $\ZygSpace{s}$ denote the Zygmund\footnote{For non-integer exponents, the Zygmund space agrees with the H\"older space.  More precisely, for $m\in \N$ and $a\in (0,1)$,
the Zygmund space $\ZygSpace{m+a}$ is locally the same as the H\"older space $\HSpace{m}{a}$ (see \cite[Theorem 1.118 (i)]{TriebelTheoryOfFunctionSpacesIII}).  For $a\in \{0,1\}$, these spaces differ: $\HSpace{m+1}{0}\subsetneq \HSpace{m}{1}\subsetneq \ZygSpace{m+1}$.} space of order $s$ (see \cref{Section::FuncSpaces::Euclid}), $\ZygSpace{\infty}$ denote the space of smooth functions, and $\ZygSpace{\omega}$
the space of real analytic functions.
For $s\in (0,\infty]\cup\{\omega\}$ if $M$ is known to be a $\ZygSpace{s+2}$ manifold\footnote{We use the convention $\infty+1=\infty+2=\infty$ and $\omega+1=\omega+2=\omega$.} and $L_1,\ldots, L_m$ are known to be $\ZygSpace{s+1}$ vector fields on $M$, then it is a result of
Malgrange \cite{MalgrangeSurLIntegbrabilite} that the complex structure on $M$ is compatible with the original $\ZygSpace{s+2}$ manifold structure, and therefore $L_1,\ldots, L_m$ are also $\ZygSpace{s+1}$ with respect to the complex structure on $M$--and this is the best one can say in general regarding the regularity of the vector fields $L_1,\ldots, L_m$ with respect to the complex structure.\footnote{\cite{MalgrangeSurLIntegbrabilite} used H\"older spaces with non-integer exponents instead of Zygmund spaces, though the proof extends to Zygmund spaces.  See  \cite{StreetNirenberg} for a further discussion in the setting of Zygmund spaces.}

In this paper, we proceed in a different direction and only assume $M$ is a $C^2$ manifold and $L_1,\ldots, L_m$ are $C^1$ vector fields on $M$ as above, and investigate the following two closely related questions for $s\in (1,\infty]\cup \{\omega\}$:
\begin{enumerate}[(i)]
\item\label{Item::Intro::WhenZygs} When are the vector fields, $L_1,\ldots, L_m$, $\ZygSpace{s+1}$ with respect to the above complex structure on $M$?  We present necessary and sufficient conditions for this to hold, which are intrinsic to the $C^2$
structure on $M$ (and can be checked locally in any $C^2$ coordinate system on $M$).
\item Under the conditions we give for \cref{Item::Intro::WhenZygs}, how can we pick a holomorphic coordinate system near each point so that the vector fields $L_1,\ldots, L_m$ are normalized
in this coordinate system in a way which is useful for applying techniques from analysis? See \cref{Section::Intro::Normalize} for an example of what we mean by ``normalized''.
\end{enumerate}

The real analogs of the above two questions were answered in a work of Stovall and the author \cite{StovallStreetI,StovallStreetII,StovallStreetIII}.
The coordinate charts in those papers were seen as scaling maps in sub-Riemannian geometry.  The quantitative study of scaling maps
in sub-Riemannian geometry began with the foundational work of Nagel, Stein, and Wainger \cite{NagelSteinWaingerBallsAndMetrics}
and the closely related work of C.\ Fefferman and S\'anchez-Calle \cite{FeffermanSanchezCalleFundamentalSoltuions},
and was furthered by Tao and Wright \cite[Section 4]{TaoWrightLpImproving} and the author \cite{StreetMultiparameterCCBalls}, and most recently in
the above mentioned series of papers \cite{StovallStreetI,StovallStreetII,StovallStreetIII}.  Since Nagel, Stein, and Wainger's original work,
these ideas have had many applications.  They have been particularly useful in the study of partial differential equations defined by vector fields;
see the notes at the end of Chapter 2 of \cite{StreetMultiParamSingInt} for some comments on this history.

When applying these ideas to questions in several complex variables (when working on, for example, a complex manifold) a problem immediately arises.
The scaling maps  studied by Nagel, Stein, and Wainger (and in the subsequent works described above) are not holomorphic.  Thus, if one tries to rescale questions
using these maps, one destroys any holomorphic aspects of the questions under consideration.  Nevertheless, scaling techniques are one of the main tools needed
to prove the quantitative estimates required to apply the theory of singular integrals to partial differential operators.
Thus, when working in the complex category, one needs a different approach than the one given by Nagel, Stein, and Wainger
to be able to scale with holomorphic maps.
Some authors use ad hoc methods to create these scaling maps for the particular problem they wish to study (e.g., by using non-isotropic dilations determined by the Taylor series of some ingredients in the problem)--see, e.g.,
\cite[Section 3]{NagelRosaySteinWaingerEstimatesForTheBergmanAndSzegoKernelsInCt}, \cite[Section 3.3.2]{CharpentierDupainExtermalBases}, and
\cite[Section 2.1]{CharpentierDupainEstimatesForBergmanAndSezgoLocallyDiag}.

A main goal of this paper is to adapt the results of Nagel, Stein, and Wainger \cite{NagelSteinWaingerBallsAndMetrics} (and more generally, the results of \cite{StovallStreetI,StovallStreetII,StovallStreetIII})
to the complex category.  Thus, in an appropriate setting, one obtains \textit{holomorphic} scaling maps adapted to a collection of complex vector fields.
Much as the theory of Nagel, Stein, and Wainger allows one to quantitatively study sub-Riemannian geometry on a real manifold,
the theory in this paper allows one to quantitatively study certain sub-Riemannian geometries on a complex manifold which are well adapted to the complex structure,
using only holomorphic maps.  We call such geometries sub-Hermitian.

While the complex setting is easier to understand, we proceed more generally than above.  Instead of working with the category of complex manifolds, we work more generally in the category of real manifolds endowed with an elliptic
structure; we call these manifolds E-manifolds (see \cref{Section::Emfld}).  This allows us to state a general theorem which implies both the results in the complex setting, as well as generalizes the results from the real setting in \cite{StovallStreetI,StovallStreetII,StovallStreetIII}.  The more general results apply, in some cases, to CR manifolds (see \cref{Section::CRMfld} for the relationship
between E-manifolds and CR manifolds and \cref{Section::ExtremalBasis::CR} for a discussion of our results in a setting on CR manifolds).

Our main result in the complex setting can be seen as a diffeomorphic invariant,\footnote{Here, by diffeomorphic invariant, we mean that all of the quantitative estimates are invariant under arbitrary $C^2$ diffeomorphisms.  See \cref{Section::DiffeoInv}.}
quantitative version of the classical Newlander-Nirenberg theorem \cite{NewlanderNirenbergComplexAnalyticCoordiantesInAlmostComplexManifolds},
while the more general main result in the elliptic setting can be seen as a diffeomorphic invariant, quantitative version of Nirenberg's theorem on the integrability of elliptic structures \cite{NirenbergAComplexFrobeniusTheorem}.



	\subsection{Comparison with previous results}
The results in this paper can be compared to previous work in two ways:
\begin{itemize}
\item We provide a quantitatively diffeomorphic invariant approach to the classical Newlander-Nirenberg theorem, and more generally Nirenberg's theorem on the integrability of elliptic
structures.
\item We provide a holomorphic analog of the quantitative theory of sub-Riemannian geometry due to Nagel, Stein, and Wainger \cite{NagelSteinWaingerBallsAndMetrics};
and more generally results on ``E-manifolds.''  See \cref{Section::Emfld} for the definition of E-manifolds.
\end{itemize}
We have already described the second point, so we focus on the first.

In previous results on the Newlander-Nirenberg theorem, one is given complex vector fields $L_1,\ldots, L_m$, as described at the start of the introduction, with some fixed regularity (e.g., in $\ZygSpace{s+1}$ for some $s>0$).
Given a fixed point $\zeta_0\in M$, the goal is to find a $\ZygSpace{s+2}$ coordinate chart $\Phi:B_{\C^n}(1)\rightarrow W$ (where $W$ is a neighborhood of $\zeta_0$),
such that $\Phi^{*}L_1,\ldots, \Phi^{*}L_m$ are $T^{0,1}$ (i.e., are spanned by $\diff{\zb[1]},\ldots, \diff{\zb[n]}$);
 in this case $\Phi^{*} L_1,\ldots, \Phi^{*}L_m$ are $\ZygSpace{s+1}$.
$\ZygSpace{s+2}$ is the optimal possible regularity for $\Phi$ (in general), and was established by Malgrange \cite{MalgrangeSurLIntegbrabilite}.

Our results take a different perspective.  In this paper, the vector fields are only assumed to be $C^1$, and we ask the question as to when it is possible to choose a $C^2$ coordinate chart $\Phi$
so that the vector fields are $\ZygSpace{s+1}$ and $T^{0,1}$.  Our results imply the above classical results on the Newlander-Nirenberg theorem\footnote{At least for $s>1$.} but are more general:   our results are invariant under arbitrary
$C^2$ diffeomorphisms (whereas previous results are only invariant under $\ZygSpace{s+2}$ diffeomorphisms).


\begin{rmk}
The main results of this paper are in \cref{Section::MainResult}.  There are many aspects of the main results which are important for applications.  Some of these are:
\begin{itemize}
\item They are invariant under arbitrary $C^2$ diffeomorphisms (see \cref{Section::DiffeoInv}).  For example, this allows us to understand the regularity of a given collection of $C^1$ complex vector fields, satisfying the conditions of the Newlander-Nirenberg theorem, with respect to the induced complex structure.  See, e.g., \cref{Section::CorRes::OptSmooth} and more generally \cref{Section::Res::OptE}.
\item They are quantitative.  This allows us to view the induced coordinate charts as scaling maps in ``sub-Hermitian geometry'' (see \cref{Section::ResGeom::SubHerm}) and more generally  ``sub-E geometry'' (see \cref{Section::Res::SubE}).  The quantitative nature of our results also has some applications to singular foliations; see \cref{Section::SingFoliation}.
\item Instead of dealing with complex structures, we state our results in the context of elliptic structures (see \cref{Section::Emfld}).  This allows us to state a general theorem which includes both the complex setting
and the real setting of \cite{StovallStreetI,StovallStreetII,StovallStreetIII} as special cases.  This more general setting applies, in some instances, to CR-manifolds.
\end{itemize}
Because we include all these considerations into our main results, the statements of these results are quite technical.
In \cref{Section::InformalCor} we state some simple corollaries of the main results of this paper which are less technical, to help give the reader an idea of the types of results we are interested in.
Furthermore, we describe several more significant consequences of the main results in \cref{Section::CorRes}.  We
hope that if the reader reads these results before the main results, it will make the main results easier to digest.
\end{rmk} 
	
	\subsection{Some simple corollaries}\label{Section::InformalCor}
Before we introduce all the relevant function spaces and notation, in this section we present some easy to understand corollaries of our main result to help give the reader an idea of the direction of this paper.  Here, we only consider the smooth setting;
precise statements of more general results appear later in the paper.
We also only consider the complex setting in this section; the more general setting of E-manifolds is described in \cref{Section::Emfld,Section::Res::CorE}.
There are two, related, ways in which the main result of this paper (\cref{Thm::Results::MainThm}) can be understood.  Below we give examples
of these two perspectives.  The main result addresses both of these perspectives simultaneously,
and we will see that it also applies in several other situations.
	
		\subsubsection{Smoothness in the Newlander-Nirenberg Theorem}
Let $L_1,\ldots, L_m$ be $C^1$ complex vector fields defined on an open set $W\subseteq \C^n$.
Fix a point $\zeta_0\in U$.  We wish to understand when the following goal is possible:

\begin{goal}\label{Goal::Smoothness}
Find a $C^2$ diffeomorphism $\Phi:U\rightarrow W'$, where $U\subseteq \C^n$ is open and $W'\subseteq W$ is an open set containing $\zeta_0$
such that:
\begin{itemize}
	\item The vector fields $\Phi^{*}L_1,\ldots, \Phi^{*}L_m$ are $C^\infty$ vector fields on $U$.
	\item $\forall \zeta\in U$,
		\begin{equation*}
			\Span_{\C} \{ \Phi^{*}L_1(\zeta),\ldots, \Phi^{*}L_m(\zeta)\} = \Span_{\C} \mleft\{\diff{\zb[1]} ,\ldots \diff{\zb[n]} \mright\}.
		\end{equation*}
\end{itemize}
\end{goal}

There are some obvious necessary conditions for Goal \ref{Goal::Smoothness} to be possible.  Namely, that there be an open neighborhood $W''\subseteq W$ containing $\zeta_0$, such that the following holds:
\begin{enumerate}[(i)]
	\item $\Span_{\C} \{ L_1(\zeta),\ldots, L_m(\zeta)\} \bigcap \Span_{\C} \{ \Lb[1](\zeta),\ldots, \Lb[m](\zeta)\}=\{0\}$, $\forall \zeta\in W''$.
	\item $\Span_{\C} \{  L_1(\zeta),\ldots, L_m(\zeta),  \Lb[1](\zeta),\ldots, \Lb[m](\zeta)\} = \C T_{\zeta}W''$, $\forall \zeta\in W''$.
	\item $[L_j, L_k] = \sum_{l=1}^m c_{j,k}^{1,l} L_l$ and $[L_j, \Lb[k]] = \sum_{l=1}^m c_{j,k}^{2,l} L_l + \sum_{l=1}^m c_{j,k}^{3,l} \Lb[l]$,
	where $c_{j,k}^{1,l}, c_{j,k}^{2,l}, c_{j,k}^{3,l}:W''\rightarrow \C$ and satisfy the following: for any sequence
	$V_1,\ldots, V_K \in  \{ L_1,\ldots, L_m, \Lb[1],\ldots, \Lb[m]\}$, of any length $K\in \N$, we have
	\begin{equation*}
		V_1V_2\cdots V_K c_{j,k}^{p,l}
	\end{equation*}
	defines a continuous function $W''\rightarrow \C$, $1\leq p\leq 3$, $1\leq j,k,l\leq m$.
\end{enumerate}

That these conditions are necessary to achieve Goal \ref{Goal::Smoothness} is clear:  if Goal \ref{Goal::Smoothness} holds, the
above conditions all clearly hold for the vector fields 
$\Phi^{*}L_1,\ldots, \Phi^{*}L_m$.  
Indeed, for the vector fields $\Phi^{*}L_1,\ldots, \Phi^{*}L_m$ one may take $c_{j,k}^{p,l}$ to be $C^\infty$ functions on $U$.
The above conditions are all invariant under $C^2$ diffeomorphisms, and therefore if they hold for $\Phi^{*}L_1,\ldots, \Phi^{*}L_m$, they must also hold for
the original vector fields $L_1,\ldots, L_m$.  Our first corollary of \cref{Thm::Results::MainThm} says the following:

\begin{cor}\label{Cor::IntroOptimal::Smoothness}
The above necessary conditions are also sufficient to obtain Goal \ref{Goal::Smoothness}.
\end{cor}

See \cref{Section::ResSmooth::Complex} for a more general version of \cref{Cor::IntroOptimal::Smoothness}.
		
		\subsubsection{Normalizing vector fields}\label{Section::Intro::Normalize}
Suppose one is given complex vector fields $L_1,\ldots, L_m$ on an open set $W\subseteq \C^n$ of the form:
\begin{equation}\label{Eqn::IntroNormal::VFs}
	L_j=\sum_{k=1}^n b_j^k \diff{\zb[k]},\quad b_{j}^k \in \CjSpace{\infty}[W],
\end{equation}
and such that $\forall \zeta\in W$,
		\begin{equation}\label{Eqn::IntroNormal::Span}
			\Span_{\C} \{ L_1(\zeta),\ldots, L_m(\zeta)\} = \Span_{\C} \mleft\{\diff{\zb[1]} ,\ldots \diff{\zb[n]} \mright\}.
		\end{equation}
Assign to each $L_j$ a formal degree $d_j\in [1,\infty)$.
For $\delta\in (0,1]$ (we think of $\delta$ as small), one tends to think of the vector fields
$\delta^{d_1}L_1,\ldots, \delta^{d_m} L_m$ as being small.  Fix a point $\zeta_0\in W$.
Our next goal is to find a holomorphic coordinate system, near $\zeta_0$, in which the vector fields are not small.
More precisely, we wish to understand when the following goal is possible:

\begin{goal}\label{Goal::Rescale}
For each $\delta\in (0,1]$ find a biholomorphism $\Phi_{\delta}:B_{\C^n}(1)\rightarrow W_\delta$ with $\Phi_{\delta}(0)=\zeta_0$, where
$B_{\C^n}(1)$ is the unit ball in $\C^n$ and $W_\delta\subseteq W$ is an open neighborhood of $\zeta_0$, such that:
\begin{itemize}
	\item $\Phi_{\delta}^{*}\delta^{d_1}L_1,\ldots, \Phi_\delta^{*}\delta^{d_m}L_m$ are $C^\infty$ vector fields, uniformly in $\delta\in (0,1]$, in the sense
	that
	\begin{equation*}
		\max_{1\leq j\leq m} \sup_{\delta\in (0,1]} \CjNorm{\Phi_{\delta}^{*} \delta^{d_j} L_j}{k}[B_{\C^n}(1)][\C^n]<\infty, \quad \forall k\in \N.
	\end{equation*}
	
	\item Because $\Phi_{\delta}$ is a biholomorphism, we have
	\begin{equation*}
			\Span_{\C} \{ \Phi_{\delta}^{*} \delta^{d_1} L_1(z),\ldots, \Phi_{\delta}^{*} \delta^{d_m} L_m(z)\} = \Span_{\C} \mleft\{\diff{\zb[1]} ,\ldots \diff{\zb[n]} \mright\}, \quad \forall z\in B_{\C^n}(1).
	\end{equation*}
		We ask that this be true uniformly in $\delta$ in the sense
		\begin{equation*}
			\inf_{\delta\in (0,1]}\max_{j_1,\ldots, j_n\in  \{1,\ldots, m\}} \inf_{z\in B_{\C^n}(1) } \mleft| \det \mleft( \Phi_{\delta}^{*} \delta^{d_{j_1}} L_{j_1}(z) | \cdots | \Phi_{\delta}^{*} \delta^{d_{j_n}} L_{j_n}(z) \mright) \mright|>0;
		\end{equation*}
		where the matrix $\mleft( \Phi_{\delta}^{*} \delta^{d_{j_1}} L_{j_1}(z) | \cdots | \Phi_{\delta}^{*} \delta^{d_{j_n}} L_{j_n}(z) \mright)$ is the $n\times n$ matrix whose columns are given
		by the coefficients of the vector fields $ \Phi_{\delta}^{*} \delta^{d_{j_k}} L_{j_k}(z)$, written as linear combinations of $\diff{\zb[1]},\ldots, \diff{\zb[n]}$.
\end{itemize}
\end{goal}

Goal \ref{Goal::Rescale} can be thought of as rescaling the vector fields so that they are ``normalized''.  Indeed, the vector fields $\Phi_{\delta}^{*} \delta^{d_1}L_m,\ldots, \Phi_{\delta}^{*}\delta^{d_m} L_m$ are $C^\infty$ uniformly in $\delta$ and span $T^{0,1} B_{\C^n}(1)$ uniformly in $\delta$.  In short, we have changed coordinates near $\zeta_0$ to turn the case
of $\delta$ small back into a situation similar to $\delta=1$.  Notice that Goal \ref{Goal::Smoothness} is trivial in the situation we are considering; nevertheless
we will see that the necessary and sufficient condition for when to Goal \ref{Goal::Rescale} is possible looks very similar to the necessary and sufficient conditions
for when Goal \ref{Goal::Smoothness} is possible.

There is an obvious necessary condition for Goal \ref{Goal::Rescale} to be possible.  Namely, that for every $\delta\in (0,1]$, there is an open neighborhood $W_{\delta}'\subseteq W$ of $\zeta_0$ such that
the follow holds.  For every $\delta\in (0,1]$,
$[\delta^{d_j} L_j, \delta^{d_k} L_k] = \sum_{l=1}^m c_{j,k}^{1,l,\delta} \delta^{d_l} L_l$ and
$[\delta^{d_j} L_j, \delta^{d_k} \Lb[k]] = \sum_{l=1}^m c_{j,k}^{2,l,\delta} \delta^{d_l} L_l + \sum_{l=1}^m c_{j,k}^{3,l,\delta} \delta^{d_l} \Lb[l]$,
where $c_{j,k}^{1,l,\delta}, c_{j,k}^{2,l,\delta}, c_{j,k}^{3,l,\delta}:W_\delta'\rightarrow \C$ and satisfy the following:
for any sequence $V_1^{\delta}, \ldots, V_K^{\delta}\in \{ \delta^{d_1} L_1,\ldots, \delta^{d_m} L_m, \delta^{d_1} \Lb[1],\ldots, \delta^{d_m} \Lb[m]\}$,
of any length $K\in \N$, we have
\begin{equation*}
	\sup_{\delta\in (0,1]} \CNorm{ V_1^{\delta}\cdots V_K^{\delta} c_{j,k}^{p,l,\delta}}{W_\delta'}<\infty,
\end{equation*}
$\forall 1\leq p\leq 3$, $1\leq j,k,l\leq m$.  That this condition is necessary is clear:  if $\Phi_{\delta}$ exists as in Goal \ref{Goal::Rescale},
then one may write
\begin{equation*}
	[\Phi_{\delta}^{*} \delta^{d_j} L_j, \Phi_{\delta}^{*} \delta^{d_k} L_k] = \sum_{l=1}^m\ch_{j,k}^{1,l,\delta} \Phi_{\delta}^{*} \delta^{d_l} L_l, \quad
	[\Phi_{\delta}^{*} \delta^{d_j} L_j, \Phi_{\delta}^{*} \delta^{d_k} \Lb[k]] = \sum_{l=1}^m \ch_{j,k}^{2,l,\delta} \Phi_{\delta}^{*}\delta^{d_l} L_l + \sum_{l=1}^m \ch_{j,k}^{3,l,\delta} \Phi_{\delta}^{*}\delta^{d_l} \Lb[l],
\end{equation*}
where $\ch_{j,k}^{p,l,\delta}\in \CjSpace{\infty}[B_{\C^n}(1)]$, \textit{uniformly in }$\delta\in (0,1]$.
Setting $c_{j,k}^{p,l,\delta}:=\ch_{j,k}^{p,l,\delta}\circ \Phi_{\delta}^{-1}$ and $W_\delta':=\Phi_{\delta}(B^n(1))$, we see that the above condition is necessary.
It is also necessary that the set $W_\delta'$ must not be too small:  it must essentially contain a sub-Riemannian ball adapted to the vector fields $\delta^{d_1}L_1,\ldots, \delta^{d_m} L_m$.
This is somewhat technical to make precise (see \cref{Rmk::IntroNormal::NecessityBall} for a precise statement), and the reader may wish to skip this on a first reading.
Our next corollary is that the above necessary condition is also sufficient.

\begin{cor}\label{Cor::Intro::Scaling}
Once the requirement on the size of $W_\delta'$ described above is made precise (see \cref{Rmk::IntroNormal::NecessityBall}), the above necessary condition is also sufficient for Goal \ref{Goal::Rescale} to be possible.
\end{cor}
\begin{proof}
This follows from \cref{Thm::Results::MainThm}, using \cref{Lemma::MoreAssume::ExistEtaDelta0}.
\end{proof}

At first glance, it may be hard to see \cref{Cor::Intro::Scaling} as a consequence of \cref{Thm::Results::MainThm}.  Indeed, the thrust of \cref{Cor::Intro::Scaling}
is that we have a result which is ``uniform in $\delta$''.  In \cref{Thm::Results::MainThm}, there is no parameter similar to $\delta$ for the results to be uniform in:
there is just one finite list of vector fields, which does not depend on any variable like $\delta$.  The key is that we keep careful track of what all the estimates
in \cref{Thm::Results::MainThm} depend on.  Because of this, we may apply \cref{Thm::Results::MainThm} to each of the lists $\delta^{d_1} L_1,\ldots, \delta^{d_m} L_m$,
for $\delta\in (0,1]$, and obtain results which are uniform in $\delta$--this is because we can see from the dependencies  of the estimates in \cref{Thm::Results::MainThm}
that they do not depend on $\delta\in (0,1]$, when applied to $\delta^{d_1} L_1,\ldots, \delta^{d_m} L_m$.

Thus, to proceed in this way, it is essential to keep careful track of what each constant depends on in \cref{Thm::Results::MainThm}.  This is  notationally
cumbersome, but is justified because it applies not only to results like \cref{Cor::Intro::Scaling}, but also to much more complicated situations.
For example, one might consider vector fields that depend on $\delta$ in a more complicated way than above, or consider the multi-parameter
case $\delta\in (0,1]^\nu$, or look for results which are uniform in the base point $\zeta_0$.  All of these are possible, and follow from \cref{Thm::Results::MainThm}
in the same way \cref{Cor::Intro::Scaling} does.  See, for example, \cref{Section::ExtremalBasis}.

For a setting which generalizes \cref{Cor::Intro::Scaling} and which appears in several complex variables, see \cref{Section::ExtremalBasis}.
For some more significant results similar to, but slightly different than \cref{Cor::Intro::Scaling}, see \cref{Section::ResGeom::SubHerm}--there we will see similar ideas as providing
holomorphic scaling maps adapted sub-Riemannian geometries on a complex manifold.

\begin{rmk}
In light of the above discussion, one way to think about one aspect of \cref{Thm::Results::MainThm} is the following.  Suppose you are given smooth
vector fields $L_1,\ldots, L_m$ of the form described in \cref{Eqn::IntroNormal::VFs} satisfying \cref{Eqn::IntroNormal::Span}.  But suppose the vector fields
have very large coefficients, or very small coefficients (for example, in the above setting the coefficients were very small when $\delta$ was small).
\Cref{Thm::Results::MainThm} provides necessary and sufficient conditions on when one can apply a holomorphic change of variables to normalize the coefficients in the way
described above.
\end{rmk}

\begin{rmk}\label{Rmk::IntroNormal::NecessityBall}
The size of $W_\delta'$ can be described as follows.  There exists $\xi>0$ (independent of $\delta\in (0,1]$)
such that
\begin{equation}\label{Eqn::IntroNormal::WdeltaSubRiemannian}
	B_{\delta^{d_1}L_1,\ldots, \delta^{d_m}L_m}(\zeta_0,\xi)\subseteq W_\delta'.
\end{equation}
See \cref{Eqn::FuncMan::DefnSRBall,Eqn::FuncComplex::Ball} for the definition of this ball.  In the above description of necessity of our condition for Goal \ref{Goal::Rescale},
we chose $W_\delta'=\Phi_{\delta}(B^n(1))$.  Thus, to prove the necessity of \cref{Eqn::IntroNormal::WdeltaSubRiemannian}, under the conclusions of Goal \ref{Goal::Rescale}, we wish to show
\begin{equation}\label{Eqn::IntroNorma::PhideltaSub}
	B_{\delta^{d_1}L_1,\ldots, \delta^{d_m}L_m}(\zeta_0,\xi)\subseteq \Phi_{\delta}(B^n(1)),
\end{equation}
for some $\xi>0$, independent of $\delta\in (0,1]$.  Once we prove \cref{Eqn::IntroNorma::PhideltaSub}, it will show \cref{Eqn::IntroNormal::WdeltaSubRiemannian} is necessary
for Goal \ref{Goal::Rescale} to hold.  To see \cref{Eqn::IntroNorma::PhideltaSub}, note that the Picard-Lindel\"of Theorem shows that there exists $\xi>0$,
independent of $\delta\in (0,1]$ such that
\begin{equation}\label{Eqn::IntroNormal::PullBackSub}
	B_{\Phi_\delta^{*} \delta^{d_1} L_1,\ldots, \Phi_{\delta}^{*} \delta^{d_m}L_m}(0,\xi)\subseteq B_{\C^n}(1/2).
\end{equation}
Applying $\Phi_{\delta}$ to both sides of \cref{Eqn::IntroNormal::PullBackSub} implies \cref{Eqn::IntroNorma::PhideltaSub}, which completes the proof of necessity.
\end{rmk}

\section{Function Spaces}\label{Section::FuncSpaces}
In this section, we introduce the function spaces which are used in this paper.
We make a distinction between function spaces on open subsets of $\R^n$ and function spaces on a $C^2$ manifold $M$.
$\R^n$ is endowed with its usual real analytic structure, and it makes sense to consider all the usual function
spaces on an open subset of $\R^n$.  Since $M$ is merely a $C^2$ manifold, it does not make sense to consider,
for example, $C^\infty$ functions on $M$.  However, if we are given a finite collection of $C^1$ vector fields on $M$,
it makes sense to consider functions which are $C^\infty$ with respect to these vector fields, and that is how we will proceed.
The following function spaces were defined in \cite{StovallStreetI}, and we refer the reader there for a more detailed discussion.
Throughout the paper, given a Banach space $\BanachSpace$, we denote by $B_{\BanachSpace}(r)$ the ball of radius $r>0$
centered at $0\in \BanachSpace$.

	\subsection{Function Spaces on Euclidean Space}\label{Section::FuncSpaces::Euclid}
Let $\Omega\subset \R^n$ be a bounded, connected, open set (we will almost always be considering the case when $\Omega$ is a ball in $\R^n$).
We have the following classical spaces of functions on $\Omega$:
\begin{equation*}
\CSpace{\Omega}=\CjSpace{0}[\Omega]:=\{f:\Omega\rightarrow \C \:\big|\: f\text{ is continuous and bounded}\},\quad \CNorm{f}{\Omega}=\CjNorm{f}{0}[\Omega]:=\sup_{x\in \Omega}|f(x)|.
\end{equation*}
For $m\in \N$, (we use the convention $0\in \N$)
\begin{equation*}
\CjSpace{m}[\Omega]:=\{f\in \CSpace{\Omega}\: \big|\: \partial_x^{\alpha}f \in \CSpace{\Omega}, \forall |\alpha|\leq m\}, \quad \CjNorm{f}{m}[\Omega]:=\sum_{|\alpha|\leq m} \CNorm{\partial_x^{\alpha} f}{\Omega}.
\end{equation*}
Next we define the classical H\"older spaces.  For $s\in [0,1]$,
\begin{equation}\label{Eqn::FSEuclid::DefnHolder}
\HNorm{f}{0}{s}[\Omega]:=\CNorm{f}{\Omega} + \sup_{\substack{x,y\in \Omega \\ x\ne y}} |x-y|^{-s} |f(x)-f(y)|, \quad \HSpace{0}{s}[\Omega]:=\{f\in \CSpace{\Omega} : \HNorm{f}{0}{s}[\Omega]<\infty\}.
\end{equation}
For $m\in \N$, $s\in [0,1]$,
\begin{equation*}
\HNorm{f}{m}{s}[\Omega]:=\sum_{|\alpha|\leq m} \HNorm{\partial_x^{\alpha} f}{0}{s}[\Omega], \quad \HSpace{m}{s}[\Omega]:=\{f\in \CjSpace{m}[\Omega] : \HNorm{f}{m}{s}[\Omega]<\infty\}.
\end{equation*}
Next, we turn to the classical Zygmund spaces.  Given $h\in \R^n$ define $\Omega_h:=\{x\in \R^n : x,x+h,x+2h\in \Omega\}$.
For $s\in (0,1]$ set
\begin{equation*}
\ZygNorm{f}{s}[\Omega]:=\HNorm{f}{0}{s/2}[\Omega]+ \sup_{\substack{0\ne h\in \R^n \\ x\in \Omega_h}} |h|^{-s} |f(x+2h)-2f(x+h)+f(x)|,
\end{equation*}
\begin{equation*}
\ZygSpace{s}[\Omega]:=\{f\in \CSpace{\Omega} : \ZygNorm{f}{s}[\Omega]<\infty\}.
\end{equation*}
For $m\in \N$, $s\in (0,1]$, set
\begin{equation*}
\ZygNorm{f}{m+s}[\Omega]:=\sum_{|\alpha|\leq m} \ZygNorm{\partial_x^{\alpha} f}{s}[\Omega], \quad \ZygSpace{m+s}[\Omega]:=\{ f\in \CjSpace{m}[\Omega] : \ZygNorm{f}{m+s}[\Omega]<\infty\}.
\end{equation*}
We set
\begin{equation*}
\ZygSpace{\infty}[\Omega]:=\bigcap_{s>0} \ZygSpace{s}[\Omega],\quad \CjSpace{\infty}[\Omega] := \bigcap_{m\in \N} \CjSpace{m}[\Omega].
\end{equation*}
If $\Omega$ is a ball, $\ZygSpace{\infty}[\Omega]=\CjSpace{\infty}[\Omega]$.

Finally, we turn to spaces of real analytic functions.  Given $r>0$, we define
\begin{equation*}
\ComegaNorm{f}{r}[\Omega]:=\sum_{\alpha\in \N^n} \frac{\CNorm{\partial_x^{\alpha} f}{\Omega} }{\alpha!} r^{|\alpha|}, \quad \ComegaSpace{r}[\Omega]:=\{f\in \CjSpace{\infty}[\Omega] : \ComegaNorm{f}{r}[\Omega]<\infty\}.
\end{equation*}
We set
\begin{equation*}
\CjSpace{\omega}[\Omega]:=\bigcup_{r>0} \ComegaSpace{r}[\Omega], \quad \ZygSpace{\omega}[\Omega]:=\CjSpace{\omega}[\Omega].
\end{equation*}
We also define another space of real analytic functions.  
We define
$\ASpace{n}{r}$ to be the space of those $f\in \CSpace{B_{\R^n}(r)}$ such that $f(t)=\sum_{\alpha\in \N^n} \frac{c_\alpha}{\alpha!} t^{\alpha}$, where
\begin{equation*}
\ANorm{f}{n}{r}:= \sum_{\alpha\in \N^n} \frac{|c_{\alpha}|}{\alpha!} r^{|\alpha|}<\infty.
\end{equation*}
See \cref{Lemma::FuncSpaceRev::Properties} \cref{Item::FuncSpaceRev::ComegainsA} and \cref{Item::FuncSpaceRev::sAinComega} for the relationship
between $\ASpace{n}{r}$ and $\CjSpace{\omega}$.

For $s\in (0,\infty]\cup\{\omega\}$, we say $f\in \ZygSpacemap{s}[\Omega]$ if $\forall x\in \Omega$, there exists an open ball $B\subseteq \Omega$, centered at $x$, such that $f\big|_{B} \in \ZygSpace{s}[B]$.
It is immediate to verify that $\ZygSpacemap{\infty}[\Omega]$ is the usual space of smooth functions on $\Omega$ and $\ZygSpacemap{\omega}[\Omega]$ is the usual space of real analytic functions on $\Omega$.

If $\BanachSpace$ is a Banach Space, we define the same spaces taking values in $\BanachSpace$ in the obvious way, and denote these spaces
by $\CSpace{\Omega}[\BanachSpace]$, $\CjSpace{m}[\Omega][\BanachSpace]$, $\HSpace{m}{s}[\Omega][\BanachSpace]$, $\ZygSpace{s}[\Omega][\BanachSpace]$, $\ComegaSpace{r}[\Omega][\BanachSpace]$,
 $\CjSpace{\omega}[\Omega][\BanachSpace]$, and $\ASpace{n}{r}[\BanachSpace]$.  Given a complex vector field $X$ on $\Omega$, we identify
 $X=\sum_{j=1}^n a_j(x) \frac{\partial}{\partial x_j}$ with the function $(a_1,\ldots, a_n):\Omega\rightarrow \C^n$.  It therefore makes sense to consider quantities
 like $\ZygNorm{X}{s}[\Omega][\C^n]$.
 When $\BanachSpace$ is clear from context, we sometimes suppress it and write, e.g., $\ZygNorm{f}{s}[\Omega]$ instead of $\ZygNorm{f}{s}[\Omega][\BanachSpace]$ for readability considerations. 
	
	\subsection{Function Spaces on Manifolds}
Let $W_1,\ldots, W_N$ be $C^1$ real vector fields on a connected $C^2$ manifold $M$.  Define the Carnot-Carath\'eodory ball associated to $W_1,\ldots, W_N$,
centered at $x\in M$, of radius $\delta>0$ by
\begin{equation}\label{Eqn::FuncMan::DefnSRBall}
\begin{split}
B_W(x,\delta) := \Bigg\{ y\in M\: \bigg| \: &\exists \gamma:[0,1]\rightarrow M, \gamma(0)=x,\gamma(1)=y, \gamma'(t)=\sum_{j=1}^N a_j(t) \delta W_j(\gamma(t)),
\\& a_j\in L^\infty([0,1]), \BNorm{\sum_{j=1}^N |a_j|^2}[L^\infty]<1
\Bigg\},
\end{split}
\end{equation}
and for $y\in M$, set
\begin{equation}\label{Eqn::FuncMan::Defnrho}
\rho(x,y):=\inf\{\delta>0 : y\in B_W(x,\delta)\}.
\end{equation}
$\rho$ is an extended metric:  it is possible that $\rho(x,y)=\infty$ for some $x,y\in M$.  When $\rho(x,y)=\infty$, we define $\rho(x,y)^{-s}=0$ for $s>0$ and $\rho(x,y)^{0}=1$.
See \cref{Rmk::FuncMfld::DefineDeriv} for the precise definition of $\gamma'(t)$ used in \cref{Eqn::FuncMan::DefnSRBall}.

We use ordered multi-index notation $W^\alpha$.  Here, $\alpha$ denotes a list of elements of $\{1,\ldots, N\}$ and $|\alpha|$ denotes the length of the list.
For example, $W^{(2,1,3,1)}=W_2W_1W_3W_1$ and $|(2,1,3,1)|=4$.

Associated to the vector fields $W_1,\ldots, W_N$, we have the following function spaces on $M$.
\begin{equation*}
\CSpace{M}=\CXjSpace{W}{0}[M]:=\{f:M\rightarrow \C\: \big|\: f\text{ is bounded and continuous}\},\quad \CNorm{f}{M}=\CXjNorm{f}{W}{0}[M]:=\sup_{x\in M} |f(x)|.
\end{equation*}
For $m\in \N$, we define
\begin{equation*}
\CXjSpace{W}{m}[M]:=\{f\in \CSpace{M} : W^{\alpha} f\text{ exists and }W^{\alpha} f\in \CSpace{M}, \forall|\alpha|\leq m\}, \quad \CXjNorm{f}{W}{m}[M]:=\sum_{|\alpha|\leq m} \CNorm{W^{\alpha} f}{M}.
\end{equation*}
For $s\in [0,1]$ we define the H\"older spaces associated to $W_1,\ldots, W_N$ by
\begin{equation*}
\HXNorm{f}{W}{0}{s}[M]:=\CNorm{f}{M}+\sup_{\substack{x,y\in M\\x\ne y}} \rho(x,y)^{-s} |f(x)-f(y)|, \quad \HXSpace{W}{0}{s}[M]:=\{f\in \CSpace{M} : \HXNorm{f}{W}{0}{s}[M]<\infty\}.
\end{equation*}
For $m\in \N$ and $s\in [0,1]$, set
\begin{equation*}
\HXNorm{f}{W}{m}{s}[M]:=\sum_{|\alpha|\leq m} \HXNorm{W^{\alpha}f}{W}{0}{s}[M], \quad \HXSpace{W}{m}{s}[M]:=\{f\in \CXjSpace{W}{m}[M] : \HXNorm{f}{W}{m}{s}[M]<\infty\}.
\end{equation*}

Next, we turn to the Zygmund spaces associated to $W_1,\ldots, W_N$.  For this, we use the H\"older spaces $\HSpace{0}{s}[[a,b]]$ for a closed interval $[a,b]\subset \R$;
$\HNorm{\cdot}{0}{s}[[a,b]]$ is defined via the formula \cref{Eqn::FSEuclid::DefnHolder}.  Given $h>0$, $s\in (0,1)$, define
\begin{equation*}
\sP_{W,s}^h :=\left\{ \gamma:[0,2h]\rightarrow M\: \bigg|\: \gamma'(t)=\sum_{j=1}^N d_j(t) W_j(\gamma(t)), d_j\in \HSpace{0}{s}[[0,2h]], \sum_{j=1}^q \HNorm{d_j}{0}{s}[[0,2h]]^2<1  \right\}.
\end{equation*}
For $s\in (0,1]$ set
\begin{equation*}
\ZygXNorm{f}{W}{s}[M]:=\HXNorm{f}{W}{0}{s/2}[M] + \sup_{\substack{h>0 \\ \gamma\in \sP_{W,s/2}^h}} h^{-s} |f(\gamma(2h))-2f(\gamma(h))+f(\gamma(0))|,
\end{equation*}
and for $m\in \N$,
\begin{equation*}
\ZygXNorm{f}{W}{m+s}[M]:=\sum_{|\alpha|\leq m} \ZygXNorm{W^{\alpha}f}{W}{s}[M],
\end{equation*}
and we set
\begin{equation*}
\ZygXSpace{W}{m+s}[M]:=\{f\in \CXjSpace{W}{m}[M] : \ZygXNorm{f}{W}{m+s}[M]<\infty\}.
\end{equation*}
Set
\begin{equation*}
\ZygXSpace{W}{\infty}[M]:=\bigcap_{s>0} \ZygXSpace{W}{s}[M]\text{ and } \CXjSpace{W}{\infty}[M]:= \bigcap_{m\in \N} \CXjSpace{W}{m}[M].
\end{equation*}
We have $\ZygXSpace{W}{\infty}[M]=\CXjSpace{W}{\infty}[M]$;
indeed, $\ZygXSpace{W}{\infty}[M]\subseteq \CXjSpace{W}{\infty}[M]$ is obvious, while the reverse containment follows from \cref{Lemma::FuncSpaceRev::Properties}.

Finally, we turn to functions which are real analytic with respect to $W_1,\ldots, W_N$.  Given $r>0$, we set
\begin{equation*}
\CXomegaNorm{f}{W}{r}[M]:=\sum_{m=0}^\infty \frac{r^m}{m!} \sum_{|\alpha|=m} \CNorm{W^{\alpha} f}{M}, \quad \CXomegaSpace{W}{r}[M]:=\{f\in \CXjSpace{W}{\infty}[M] : \CXomegaNorm{f}{W}{r}[M]<\infty\};
\end{equation*}
this definition was introduced in greater generality by Nelson \cite{NelsonAnalyticVectors}.
We set $\CXjSpace{W}{\omega}[M]:=\bigcup_{r>0} \CXomegaSpace{W}{r}[M]$, and $\ZygXSpace{W}{\omega}[M]:=\CXjSpace{W}{\omega}[M]$.

Given $x_0\in M$ and $r>0$ we define $\AXSpace{W}{x_0}{r}$ to be the space of those $f\in \CSpace{M}$ such that
$h(t_1,\ldots, t_N):=f(e^{t_1W_1+\cdots+t_NW_N}x_0)\in \ASpace{N}{r}$ (here, we are assuming $e^{t_1W_1+\cdots+t_NW_N}x_0$ exists for $(t_1,\ldots, t_N)\in B_{\R^N}(r)$--see \cref{Defn::MainRes::sC}).
We set $\AXNorm{f}{W}{x_0}{r}:=\ANorm{h}{N}{r}$.  Note that $\AXNorm{f}{W}{x_0}{r}$ depends only on the values of $f(y)$ where $y=e^{t_1 W_1+\cdots +t_N W_N}x_0$ and
$(t_1,\ldots, t_N)\in B^N(r)$; thus this is merely a semi-norm.

An important property of the above spaces and norms is that they are invariant under diffeomorphisms.
\begin{prop}\label{Prop::FuncMan::DiffeoInv}
	Let $L$ be another $C^2$ manifold, let $\Phi:M\rightarrow L$ be a $C^2$ diffeomorphism, and let $\Phi_{*}W$ denote the list of vector fields $\Phi_{*}W_1,\ldots, \Phi_{*} W_N$.
	Then, the map $f\mapsto f\circ \Phi$ is an isometric isomorphism between the following spaces:
	$\CXjSpace{\Phi_{*}W}{m}[L]\rightarrow \CXjSpace{W}{m}[M]$, $\HXSpace{\Phi_{*}W}{m}{s}[L]\rightarrow \HXSpace{W}{m}{s}[M]$, $\ZygXSpace{\Phi_{*}W}{s}[L]\rightarrow \ZygXSpace{W}{s}[M]$,
	$\CXomegaSpace{\Phi_{*}W}{r}[L]\rightarrow \CXomegaSpace{W}{r}[M]$, and $\AXSpace{\Phi_{*}W}{\Phi(x_0)}{r}\rightarrow \AXSpace{W}{x_0}{r}$.
\end{prop}
\begin{proof}This is immediate from the definitions.\end{proof}

\begin{rmk}\label{Rmk::FuncMan::CoordFree}
Informally, \cref{Prop::FuncMan::DiffeoInv} says that the spaces described in this
section are ``coordinate-free''.  One can locally compute the norms in any $C^2$ coordinate system, and one gets the same result no matter what coordinate system is used.
\end{rmk}

\begin{rmk}
When we write $Vf$ for a $C^1$ vector field $V$ and $f:M\rightarrow \R$, we define this as $Vf(x):=\frac{d}{dt}\big|_{t=0} f(e^{tV} x)$.  When we say $Vf$ exists, it means that this derivative exists
in the classical sense, $\forall x$.  If we have several $C^1$ vector fields $V_1,\ldots, V_K$, we define $V_1V_2\cdots V_K f := V_1(V_2(\cdots V_K(f)))$ and to say that this
exists means that at each stage the derivative exists.
\end{rmk}

\begin{rmk}
All of the above function spaces can be defined, with the same formulas, with $M$ replaced by $B_W(x,\delta)$, whether or not $B_W(x,\delta)$ is a manifold.
Indeed, for a function $f:B_W(x,\delta)\rightarrow \C$, one may define $W_j f(x):= \frac{d}{dt}\big|_{t=0} f(e^{t W_j}x)$.  Using this one may define all the above norms, with the same formulas,
for $M$ replaced by $B_W(x,\delta)$.
See \cite[\SSBeyondManifold]{StovallStreetI} for a further discussion of this.
\end{rmk}

\begin{rmk}
Let $\Omega\subseteq \R^n$ be a bounded, open set.  Let $\grad$
denote the list of vector fields $\grad=\mleft(\diff{x_1},\ldots,\diff{x_n}\mright)$.
We have $\AXSpace{\grad}{0}{r}=\ASpace{n}{r}$ and $\CXomegaSpace{\grad}{r}[\Omega]=\ComegaSpace{r}[\Omega]$, with equality of norms.
\end{rmk}

\begin{rmk}\label{Rmk::FuncMfld::DefineDeriv}
In \cref{Eqn::FuncMan::DefnSRBall} (and in the rest of the paper), $\gamma'(t)$ is defined as follows.  In the case that $M$ is an open subset $\Omega\subseteq \R^n$
and $\gamma:[a,b]\rightarrow \Omega$, $\gamma'(t)=\sum_{j=1}^q a_j(t) X_j(\gamma(t))$ is defined to mean
$\gamma(t)= \gamma(a)+\int_a^t \sum_j a_j(s) X_j(\gamma(s))\: ds$; note that this definition is local in $t$ (equivalently, we are requiring that $\gamma$ be absolutely continuous
and have the desired derivative almost everywhere).  For an abstract $C^2$ manifold, this is interpreted locally.
I.e., if $\gamma:[a,b]\rightarrow M$, we say $\gamma'(t)= \sum_{j=1}^q a_j(t) X_j(\gamma(t))$ if $\forall t_0\in [a,b]$, there is an open neighborhood $N$ of $\gamma(t_0)$
and a $C^2$ diffeomorphism $\Psi:N\rightarrow \Omega$, where $\Omega\subseteq \R^n$ is open, such that $(\Psi\circ \gamma)'(t) = \sum_{j=1}^q a_j(t) (\Psi_{*}X_j)(\Psi\circ \gamma(t))$
for $t$ near $t_0$ ($t\in [a,b]$).
\end{rmk}

		\subsubsection{Complex Vector Fields}
Let $M$ be a $C^2$ manifold, let $L_1,\ldots, L_m$ be complex $C^1$ vector fields on $M$ (i.e., $L_1,\ldots, L_m$ take values in the complexified
tangent space), and let $X_1,\ldots, X_q$ be real $C^1$ vector fields on $M$.
We denote by $X,L$ the list $X_1,\ldots, X_q,L_1,\ldots, L_m$.
Associated to $X,L$ we define the list of real vector fields $W_1,\ldots, W_{q+2m} = X_1,\ldots, X_q,2\Real(L_1),\ldots, 2\Real(L_m),2\Imag(L_1),\ldots, 2\Imag(L_m)$.
Set
\begin{equation}\label{Eqn::FuncComplex::Ball}
B_{X,L}(x,\delta):=B_W(x,\delta).
\end{equation}

We define
$\CXjSpace{X,L}{m}[M]:=\CXjSpace{W}{m}[M]$, with equality of norms.  We similarly define
$\HXSpace{X,L}{m}{s}[M]$, $\ZygXSpace{X,L}{s}[M]$, $\CXomegaSpace{X,L}{r}[M]$, $\AXSpace{X,L}{x_0}{r}$, $\CXjSpace{X,L}{\infty}[M]$, and $\CXjSpace{X,L}{\omega}[M]$.
We will often consider the case when $q=0$, and in that case we just write $\CXjSpace{L}{m}[M]$ instead of $\CXjSpace{X,L}{m}[M]$, and similarly for 
$\HXSpace{L}{m}{s}[M]$, $\ZygXSpace{L}{s}[M]$, $\CXomegaSpace{L}{r}[M]$, $\AXSpace{L}{x_0}{r}$, $\CXjSpace{L}{\infty}[M]$, and $\CXjSpace{L}{\omega}[M]$.

\begin{rmk}
The factor $2$ in $2\Real(L_j)$ and $2\Imag(L_j)$ in the definition of $W$ is not an essential point.
It is chosen so that if $M=\R^q\times \C^m$, with coordinates $(t_1,\ldots, t_q, z_1,\ldots, z_m)$, and if
$X_k=\diff{t_k}$ and $L_j=\diff{\zb_j}$, then $W=\grad$, where $\grad$ denotes the gradient on $\R^{q+2m}\cong \R^{q}\times \C^{m}$.
\end{rmk}
		
\section{Corollaries of the Main Result}\label{Section::CorRes}
Our main result (\cref{Thm::Results::MainThm}) concerns the existence of a certain coordinate chart which satisfies good quantitative properties.
This coordinate chart is useful in two, related, ways:  
\begin{itemize}
\item It is a coordinate system in which given vector fields
have the optimal level of regularity.
\item It normalizes vector fields in a way which is useful for applying techniques from analysis.  When viewed in this light, it can be seen
as a scaling map for sub-Riemannian, or sub-Hermitian, geometries.
\end{itemize}

In this section, we present two corollaries of our main result, which separate the above two uses.  In each of these corollaries, we present the real setting (which is known)
and the complex setting (which is new).  In \cref{Section::Res::CorE}, we will revisit these corollaries and present a setting which unifies both the real and complex settings.

	\subsection{Optimal Smoothness}\label{Section::CorRes::OptSmooth}
		
		\subsubsection{The Real Case}
Let $W_1,\ldots, W_N$ be $C^1$ real vector fields on a $C^2$ manifold $M$ of dimension $n$, which span the tangent space at every point.
In this section, we describe when there is a smoother structure on $M$ with respect to which $W_1,\ldots, W_N$ have a desired level of regularity.
These results were proved in \cite{StovallStreetI,StovallStreetII,StovallStreetIII} (though in \cref{Section::CorProof::Smoothness}, we will see them as corollaries of the main result of this paper), and they set the stage for the results in the complex setting in
\cref{Section::ResSmooth::Complex}.

\begin{thm}[The Local Theorem]\label{Thm::ResSmooth::Real::Local}
For $x_0\in M$, $s\in (1,\infty]\cup\{\omega\}$, the following three conditions are equivalent:
\begin{enumerate}[(i)]
\item\label{Item::ResSmooth::Real::Local::1} There is an open neighborhood $V\subseteq M$ of $x_0$ and a $C^2$ diffeomorphism $\Phi:U\rightarrow V$ where $U\subseteq \R^n$ is open,
such that $\Phi^{*}W_1,\ldots, \Phi^{*}W_N\in \ZygSpace{s+1}[U][\R^n]$.
\item\label{Item::ResSmooth::Real::Local::2} Re-order the vector fields so that $W_1(x_0),\ldots, W_n(x_0)$ are linearly independent.  There is an open neighborhood $V\subseteq M$ of $x_0$
such that:
\begin{itemize}
	\item $[W_i, W_j]=\sum_{k=1}^n \ch_{i,j}^k W_k$, $1\leq i,j\leq n$, where $\ch_{i,j}^k\in \ZygXSpace{W}{s}[V]$.
	\item For $n+1\leq j\leq N$, $W_j=\sum_{k=1}^n b_j^k W_k$, where $b_j^k\in \ZygXSpace{W}{s+1}[V]$.
\end{itemize}
\item\label{Item::ResSmooth::Real::Local::3} There exists an open neighborhood $V\subseteq M$ of $x_0$ such that $[W_i,W_j]=\sum_{k=1}^N c_{i,j}^k W_k$, $1\leq i,j\leq N$, where
$c_{i,j}^k\in \ZygXSpace{W}{s}[V]$.
\end{enumerate}
\end{thm}

\begin{rmk}
Note that \cref{Thm::ResSmooth::Real::Local} \cref{Item::ResSmooth::Real::Local::2} and \cref{Item::ResSmooth::Real::Local::3} can be checked in any $C^2$ coordinate system (see \cref{Prop::FuncMan::DiffeoInv,Rmk::FuncMan::CoordFree}),
while \cref{Thm::ResSmooth::Real::Local} \cref{Item::ResSmooth::Real::Local::1} gives the existence of a ``nice'' coordinate system.
\end{rmk}

\begin{thm}[The Global Theorem]\label{Thm::ResSmooth::Real::Global}
For $s\in (1,\infty]\cup \{\omega\}$, the following two conditions are equivalent:
\begin{enumerate}[label=(\roman*),series=qualrealglobaltheoremenumeration]
	\item\label{Item::ResSmooth::Real::Atlas} There exists a $\ZygSpace{s+2}$ atlas on $M$, compatible with its $C^2$ structure, such that $W_1,\ldots, W_N$ are  $\ZygSpace{s+1}$ vector fields with respect
	to this atlas.
	\item For each $x_0\in M$, any of the three equivalent conditions from \cref{Thm::ResSmooth::Real::Local} hold for this choice of $x_0$.
\end{enumerate}
Furthermore, under these conditions, the $\ZygSpace{s+2}$ manifold structure induced by the atlas in \cref{Item::ResSmooth::Real::Atlas} is unique,
in the sense that if there is another $\ZygSpace{s+2}$ atlas on $M$, compatible with its $C^2$ structure, and such that $W_1,\ldots, W_N$
are locally $\ZygSpace{s+1}$ with respect to this second atlas, then the identity map $M\rightarrow M$  is a $\ZygSpace{s+2}$ diffeomorphism
between these two $\ZygSpace{s+2}$ manifold structures on $M$.
Finally, when $s\in (1,\infty]$,
there is a third equivalent condition
\begin{enumerate}[resume*=qualrealglobaltheoremenumeration]
\item $[W_i,W_j]=\sum_{k=1}^N c_{i,j}^k W_k$, $1\leq i,j\leq N$, where $\forall x_0\in M$, $\exists V\subseteq M$ open with $x_0\in V$
	such that $c_{i,j}^k\big|_V\in \ZygXSpace{W}{s}[V]$, $1\leq i,j,k\leq N$.
\end{enumerate}
%
\end{thm}

\begin{rmk}
\Cref{Thm::ResSmooth::Real::Local,Thm::ResSmooth::Real::Global} are stated for $s>1$.  It would be desirable to have the same results for $s>0$, but our proof runs into technical
difficulties for $s\in (0,1]$.  See \cite{StovallStreetII} for details.  Similar remarks hold for many of the main results in this paper; in particular, the same remark holds for the main result
of the paper:  \cref{Thm::Results::MainThm}.
\end{rmk} 
		
		\subsubsection{The Complex Case}\label{Section::ResSmooth::Complex}

Let $M$ be a $C^2$ manifold and
let $L_1,\ldots, L_m$ be complex $C^1$ vector fields on $M$.   We assume:
\begin{itemize}
\item $\forall \zeta \in M$, $\Span_{\C} \mleft\{ L_1(\zeta),\ldots, L_m(\zeta), \Lb[1](\zeta),\ldots, \Lb[m](\zeta)\mright\} = \C T_\zeta M$.
\item $\forall \zeta\in M$, $\Span_{\C} \mleft\{ L_1(\zeta),\ldots, L_m(\zeta)\mright\} \cap \Span_{\C}\mleft\{ \Lb[1](\zeta),\ldots, \Lb[m](\zeta)\mright\} =\{0\}$.
\end{itemize}
By \cref{Lemma::AppendCR::dimFormula} and the above assumptions we have, $\forall \zeta\in M$,
$$\dim M = \dim \Span_{\C}  \mleft\{ L_1(\zeta),\ldots, L_m(\zeta), \Lb[1](\zeta),\ldots, \Lb[m](\zeta)\mright\}=2 \dim \Span_{\C}\mleft\{ L_1(\zeta),\ldots, L_m(\zeta)\mright\}.$$
In particular, let $n:=\dim \Span_{\C}\mleft\{ L_1(\zeta),\ldots, L_m(\zeta)\mright\}$, then $n$ does not depend on $\zeta$ and $\dim M=2n$.

\begin{thm}[The Local Theorem]\label{Thm::QualComplex::LocalThm}
Fix $\zeta_0\in M$ and $s\in (1,\infty]\cup\{\omega\}$.  The following three conditions are equivalent:
\begin{enumerate}[(i)]
\item There exists an open neighborhood $V\subseteq M$ of $\zeta_0$ and a $C^2$ diffeomorphism $\Phi:U\rightarrow V$, where
$U\subseteq \C^n$ is open, such that $\forall z\in U$, $1\leq j\leq m$,
\begin{equation*}
	\Phi^{*} L_j(z) \in \Span_{\C} \mleft\{ \diff{\zb[1]},\ldots, \diff{\zb[n]} \mright\},
\end{equation*}
and $\Phi^{*} L_j \in \ZygSpace{s+1}[U][\C^{n}]$.

\item Reorder $L_1,\ldots, L_m$ so that $L_1(\zeta_0),\ldots, L_n(\zeta_0)$ are linearly independent.  There exists a neighborhood $V\subseteq M$ of $\zeta_0$ such that:
\begin{itemize}
	\item $[L_j, L_k]=\sum_{l=1}^n \ch_{j,k}^{1,l} L_l$ and  $[L_j, \Lb[k]] =\sum_{l=1}^n \ch_{j,k}^{2,l} L_l + \sum_{l=1}^n \ch_{j,k}^{3,l} \Lb[l]$, where $\ch_{j,k}^{a,l}\in \ZygXSpace{L}{s}[V]$, $1\leq j,k,l\leq n$, $1\leq a\leq 3$.
	\item $L_j= \sum_{l=1}^n b_j^l L_l$, where $b_j^l\in \ZygXSpace{L}{s+1}[V]$, $n+1\leq j\leq m$, $1\leq l\leq n$.
\end{itemize}

\item There exists a neighborhood $V\subseteq M$ of $\zeta_0$ such that $[L_j, L_k] = \sum_{l=1}^m c_{j,k}^{1,l} L_l$ and  $[L_j, \Lb[k]] =\sum_{l=1}^m c_{j,k}^{2,l} L_l + \sum_{l=1}^m c_{j,k}^{3,l} \Lb[l]$,
where $c_{j,k}^{a,l}\in \ZygXSpace{L}{s}[V]$, $1\leq a\leq 3$, $1\leq j,k,l\leq m$.
\end{enumerate}
\end{thm}

\begin{thm}[The Global Theorem]\label{Thm::QualComplex::GlobalThm}
For $s\in (1,\infty]\cup\{\omega\}$ the following two conditions are equivalent:
\begin{enumerate}[label=(\roman*),series=qualComplexglobaltheoremenumeration]
\item\label{Item::QualComplex::Global::Manifold} There exists a complex manifold structure on $M$, compatible with its $C^2$ structure, such that $L_1,\ldots, L_m$ are  $\ZygSpace{s+1}$ vector fields on $M$ (with respect to this complex structure),
and $\forall \zeta\in M$,
$$\Span_{\C} \mleft\{L_1(\zeta),\ldots, L_m(\zeta)\mright\}= T^{0,1}_{\zeta} M.$$

\item For each $\zeta_0\in M$, any of the three equivalent conditions from \cref{Thm::QualComplex::LocalThm} hold for this choice of $\zeta_0$.
\end{enumerate}
Furthermore, under these conditions, the complex manifold structure in \cref{Item::QualComplex::Global::Manifold} is unique, in the sense that if $M$ has another complex manifold structure
satisfying the conditions of \cref{Item::QualComplex::Global::Manifold}, then the identity map $M\rightarrow M$ is a biholomorphism between these two complex structures.
Finally, when $s\in (1,\infty]$, there is a third equivalent condition:
\begin{enumerate}[resume*=qualComplexglobaltheoremenumeration]
\item $[L_j, L_k] = \sum_{l=1}^m c_{j,k}^{1,l} L_l$ and  $[L_j, \Lb[k]] =\sum_{l=1}^m c_{j,k}^{2,l} L_l + \sum_{l=1}^m c_{j,k}^{3,l} \Lb[l]$,
where $\forall \zeta\in M$, there exists an open neighborhood $V\subseteq M$ of $\zeta$ such that
$c_{j,k}^{a,l}\big|_V\in \ZygXSpace{L}{s}[V]$, $1\leq a\leq 3$, $1\leq j,k,l\leq m$.
\end{enumerate}
\end{thm}

\begin{rmk}
\Cref{Thm::QualComplex::GlobalThm} can be seen as a version of the Newlander-Nirenberg theorem (with sharp regularity in terms of Zygmund spaces), which is invariant under arbitrary $C^2$ diffeomorphisms.
\end{rmk}

\begin{rmk}
Because the Zygmund space
$\ZygSpace{m+\alpha}$ is (locally) the same as the H\"older space
$C^{m,\alpha}$ for $m\in \N$, $\alpha\in (0,1)$,
one can obtain analogs of \cref{Thm::QualComplex::LocalThm,Thm::QualComplex::GlobalThm} using the easier to understand H\"older spaces, as long as one avoids
integer exponents.  This is carried out in \cref{Section::Holder}.
For integer exponents, the use of Zygmund spaces is essential, as \cref{Thm::QualComplex::LocalThm} does not hold if we replace the Zygmund spaces
$\ZygSpace{s+1}$ (for $s\in \N$) with $\CjSpace{s+1}$ or $\HSpace{s}{1}$; this is described in \cref{Lemma::Holder::ZygmundRequired}.  As a consequence,
Zygmund spaces are also essential in the main theorem of this paper (\cref{Thm::Results::MainThm}).
The reason our proof requires Zygmund spaces when considering integer exponents is because
it relies on nonlinear elliptic PDEs (via the results from \cite{StovallStreetII,StreetNirenberg}).  As is well-known, the regularity theory
of elliptic PDEs works best when using Zygmund spaces instead of $C^m$ spaces or Lipschitz spaces.
\end{rmk} 
		
%
	\subsection{Geometries defined by vector fields}\label{Section::CorRes::Geometries}
We present the basic results concerning sub-Riemannian and sub-Hermitian geometry in this section.  The results on sub-Riemannian geometry
are just a reprise (in a slightly different language) of the main results of Nagel, Stein, and Wainger's work \cite{NagelSteinWaingerBallsAndMetrics}.\footnote{We present
results on sub-Riemannian geometry which are essentially those of Nagel, Stein, and Wainger, however the main results of this paper (even in this real setting) imply many results
which are beyond those that are implied by Nagel, Stein, and Wainger's methods.  In the real setting, this is described in the series \cite{StovallStreetI,StovallStreetII,StovallStreetIII}.
We present the corollaries in this section in the simplest possible setting (as opposed to a very general setting) to help the reader understand the thrust of our main theorem,
\cref{Thm::Results::MainThm}, which is stated in some generality.
For example, even if one only considers real vector fields, the main results
of this paper imply (and are stronger than) the results in the multi-parameter
setting of \cite{StreetMultiparameterCCBalls}, which could not be achieved by
the methods of \cite{NagelSteinWaingerBallsAndMetrics}.  We also present a more complicated example in the complex setting in \cref{Section::ExtremalBasis}.}
The results on sub-Hermitian geometry can be seen as holomorphic analogs of these results.  In this section, we present these ideas in these two simple settings.
In \cref{Section::Res::SubE} we generalize these results to a single unified result on ``E-manifolds''. 
	
		\subsubsection{Sub-Riemannian Geometry:  the results of Nagel, Stein, and Wainger}\label{Section::Results::NSW}
In this section, we describe the main results of the foundational paper of Nagel, Stein, and Wainger \cite{NagelSteinWaingerBallsAndMetrics}.
This describes how the existence of certain coordinate charts (like the ones developed in our main theorem) can be viewed as scaling maps
in sub-Riemannian geometry.  The results in this section set the stage for the results in the complex setting in \cref{Section::ResGeom::SubHerm}.

Let $W_1,\ldots, W_N$ be $C^\infty$ real vector fields on a connected, $C^\infty$ manifold $M$ of dimension $n$ which span the tangent space at every point.
To each $W_j$ we assign a formal degree $d_j\in [1,\infty)$.  We assume
\begin{equation*}
	[W_j, W_k]=\sum_{d_l\leq d_j+d_k} c_{j,k}^l W_l, \quad c_{j,k}^l\in \CjSpace{\infty}[M].
\end{equation*}
We write $(W,d)$ for the list $(W_1,d_1),\ldots, (W_N,d_N)$ and for $\delta>0$ write $\delta^d W$ for the list $\delta^{d_1}W_1,\ldots, \delta^{d_N} W_N$.
The sub-Riemannian ball associated to $(W,d)$ centered at $x_0\in M$ of radius $\delta>0$ is defined by
\begin{equation*}
B_{S}(x_0,\delta):=B_{\delta^{d} W}(x_0,1),
\end{equation*}
where the later ball is defined by \cref{Eqn::FuncMan::DefnSRBall}.
$B_{S}(x_0,\delta)$ is an open subset of $M$.  We define $\rho_S(x,y):=\inf\{\delta>0 : y\in B_{S}(x,\delta)\}$; $\rho$ is a metric on $M$
and is called a \textit{sub-Riemannian metric}.  For the relationship between this definition
of a sub-Riemannian metric and some of the other common definitions, see \cite{NagelSteinWaingerBallsAndMetrics}.

We define another metric on $M$, which will turn out to be equal to $\rho_S$, as follows.
We say $\rho_F(x,y)<\delta$ if and only if there exists $K\in \N$, smooth functions
$f_1,\ldots, f_K:B_{\R}(1/2)\rightarrow M$, and $\delta_1,\ldots, \delta_K>0$ with $\sum \delta_l\leq \delta$ such that:
\begin{itemize}
\item $f_j'(t)=\sum_{l=1}^N s_j^l(t) \delta_j^{d_l} W_l(f_j(t))$, with $\BNorm{\sum_{l} |s_j^l|^2}[L^\infty(B_{\R}(1/2))]<1$.
\item $f_j(B_\R(1/2))\bigcap f_{j+1}(B_\R(1/2))\ne \emptyset$, $1\leq j\leq K-1$.
\item $x\in f_1(B_\R(1/2))$, $y\in f_K(B_\R(1/2))$.
\end{itemize}
$\rho_F$ is clearly an extended metric.  Once we prove $\rho_F$ and $\rho_S$ are equal, it will then follow that $\rho_F$ is a metric.

Fix a strictly positive, $C^\infty$ density $\nu$ on $M$.\footnote{The results that follow are local and do not depend on the choice of $\nu$, so long as it is strictly positive and smooth.}
For $x\in M$, $\delta>0$, set
$$\Lambda(x,\delta):=\max_{j_1,\ldots, j_n\in \{1,\ldots, N\}} \nu(x)(\delta^{d_{j_1}} X_{j_1}(x),\ldots, \delta^{d_{j_N}}X_{j_N}(x)).$$
The next result follows from the methods of \cite{NagelSteinWaingerBallsAndMetrics} (though we prove it directly
by seeing is as a special case of the result in \cref{Section::Res::SubE}).

\begin{thm}[\cite{NagelSteinWaingerBallsAndMetrics}]\label{Thm::Results::NSW}
\begin{enumerate}[label=(\alph*),series=nswtheoremenumeration]
\item $\forall x,y\in M$, $\rho_S(x,y)= \rho_F(x,y)$.
\end{enumerate}
Fix a compact set $\Compact\subseteq M$.  There exists $\delta_0=\delta_0(\Compact)\in (0,1]$ such that the following holds.
We write $A\lesssim B$ for $A\leq CB$, where $C$ can be chosen independent of $x,y\in \Compact$ and $\delta>0$.  We write $A\approx B$ for $A\lesssim B$ and $B\lesssim A$.
\begin{enumerate}[resume*=nswtheoremenumeration]
\item $\nu(B_S(x,\delta))\approx \Lambda(x,\delta)$, $\forall x\in \Compact,\delta\in(0,\delta_0]$.
\item\label{Item::Results::NSW::HomogType} $\nu(B_S(x,2\delta))\lesssim \nu(B_S(x,\delta))$, $\forall x\in \Compact, \delta\in (0,\delta_0/2]$.
\end{enumerate}
For each $x\in \Compact$, $\delta\in (0,1]$, there exists $\Phi_{x,\delta}:B_{\R^n}(1)\rightarrow B_S(x,\delta)$ such that:
\begin{enumerate}[resume*=nswtheoremenumeration]
\item $\Phi_{x,\delta}(B_{\R^n}(1))\subseteq M$ is open and $\Phi_{x,\delta}:B_{\R^n}(1)\rightarrow \Phi_{x,\delta}(B_{\R^n}(1))$ is a $C^\infty$ diffeomorphism.
\item $\Phi_{x,\delta}^{*} \nu=h_{x,\delta}\LebDensity$, where $h_{x,\delta}\in \CjSpace{\infty}(B_{\R^n}(1))$, $h_{x,\delta}(t)\approx \Lambda(x,\delta)$ $\forall t$,
and $\CjNorm{h_{x,\delta}}{m}[B_{\R^n}(1)]\lesssim \Lambda(x,\delta)$, $\forall m$ (where the implicit constant depends on $m$, but not on $x\in \Compact$ or $\delta\in (0,1]$).
Here, and in the rest of the paper, $\LebDensity$ denotes the usual Lebesgue density on $\R^n$.
\end{enumerate}
Let $Y_j^{x,\delta}:=\Phi_{x,\delta}^{*}\delta^{d_j} W_j$, so that $Y_{j}^{x,\delta}$ is a $C^\infty$ vector field on $B_{\R^n}(1)$.
\begin{enumerate}[resume*=nswtheoremenumeration]
\item\label{Item::Results::NSW::YjSmooth} $\CjNorm{Y_j^{x,\delta}}{m}[B_{\R^n}(1)][\R^n]\lesssim 1$, $\forall x\in \Compact, \delta\in (0,1],m\in \N$, where the implicit constant depends on $m$, but not on $x$ or $\delta$.
\item\label{Item::Results::NSW::YjSpan} $Y_1^{x,\delta}(u),\ldots,Y_N^{x,\delta}(u)$ span the tangent space uniformly in $u,x,\delta$ in the sense that
\begin{equation*}
	\max_{j_1,\ldots, j_n\in \{1,\ldots, N\}} \inf_{u\in B_{\R^N}(1)} \left|\det \left(Y_{j_1}^{x,\delta}(u) | \cdots| Y_{j_n}^{x,\delta}(u)\right)\right|\approx 1, \quad x\in \Compact, \delta\in (0,1].
\end{equation*}
\item $\exists \epsilon\approx 1$ such that $B_S(x,\epsilon \delta)\subseteq \Phi_{x,\delta}(B_{\R^n}(1))\subseteq B_S(x,\delta)$, $\forall x\in \Compact, \delta\in (0,1]$.
\end{enumerate}
\end{thm}

\begin{rmk}
The most important aspects of \cref{Thm::Results::NSW} are \cref{Item::Results::NSW::YjSmooth} and \cref{Item::Results::NSW::YjSpan}; and these allow us to see
the maps $\Phi_{x,\delta}$ as ``scaling maps''.  Indeed, for $\delta$ small, one tends to think of $\delta^{d_j} W_j$ as a ``small'' vector field.  However,
$\Phi_{x,\delta}$ gives a coordinate system in which $\delta^{d_j} W_j$ is of ``unit size'':  not only are $\Phi_{x,\delta}^{*}\delta^{d_1} W_1,\ldots, \Phi_{x,\delta}^{*} \delta^{d_N} W_N$
smooth uniformly in $x$ and $\delta$ (i.e., \cref{Item::Results::NSW::YjSmooth}), but they also span the tangent space uniformly in $x$ and $\delta$ (i.e., \cref{Item::Results::NSW::YjSpan}).
See \cite[\SSHormandersCondition]{StovallStreetI} for some more comments in this direction.
\end{rmk}

\begin{rmk}
\Cref{Item::Results::NSW::HomogType} is the main estimate needed to show that the balls $B_S(x,\delta)$ when paired with the density $\nu$ locally
give a space of homogeneous type.  Because of this, one has access to the Calder\'on-Zygmund theory of singular integrals with respect to these balls.
This has had many uses:  see the remarks at the end of Chapter 2 of \cite{StreetMultiParamSingInt} for a history of these ideas.
\end{rmk} 
		
		\subsubsection{Sub-Hermitian Geometry}\label{Section::ResGeom::SubHerm}
Let $M$ be a connected complex manifold of complex dimension $n$.  Let $L_1,\ldots, L_m$ be $C^\infty$, $T^{0,1}$ vector fields on $M$
such that $\forall \zeta\in M$, $\Span_{\C}\{L_1(\zeta),\ldots, L_m(\zeta)\}=T^{0,1}_\zeta M$.
Our goal in this section is to describe a complex analog of the results in \cref{Section::Results::NSW} with respect to the vector fields $L_1,\ldots, L_m$.
The main point is to achieve as much as possible using only \textit{holomorphic} maps, so that these results can be applied to questions in several complex variables.

To each $L_j$ we assign a formal degree $\beta_j\in [1,\infty)$.
We assume
\begin{equation*}
[L_j,L_k]=\sum_{\beta_l\leq \beta_j+\beta_k} c_{j,k}^{1,l} L_l, \quad [L_j, \Lb[k]] = \sum_{\beta_l\leq \beta_j+\beta_k} c_{j,k}^{2,l} L_l + \sum_{\beta_l\leq \beta_j+\beta_k} c_{j,k}^{3,l} \Lb[l], \quad c_{j,k}^{a,l}\in \CjSpace{\infty}[M].
\end{equation*}
Let $(W_1,d_1),\ldots, (W_{2m},d_{2m}) = (2\Real(L_1), \beta_1),\ldots, (2\Real(L_m), \beta_m),(2\Imag(L_1),\beta_1),\ldots, (2\Imag(L_m), \beta_m)$.
Fix a strictly positive, smooth density $\nu$ on $M$.
It is immediate to verify that the list $(W_1,d_1),\ldots, (W_{2m},d_{2m})$ satisfies all the hypotheses of \cref{Section::Results::NSW}.
Thus we obtain balls $B_S(\zeta,\delta)$ and an associated metric $\rho_S=\rho_F$, and \cref{Thm::Results::NSW} applies.
The main problem is that the definitions of $\rho_S$ and $\rho_F$   use the underlying smooth structure on $M$ and not the complex structure,
and the scaling maps $\Phi_{x,\delta}$ from \cref{Thm::Results::NSW} are only guaranteed to be smooth, not holomorphic.
In particular, when rescaling $\delta^{\beta_j}L_j$ by computing $\Phi_{x,\delta}^{*}\delta^{\beta_j} L_j$ we do not know that
$\Phi_{x,\delta}^{*}\delta^{\beta_j} L_j$ continues to be a $T^{0,1}$ vector field; i.e., we do not know
$\Phi_{x,\delta}^{*}\delta^{\beta_j} L_j$ is spanned by $\diff{\zb[1]},\ldots, \diff{\zb[n]}$.
The results in this section fix these problems.

First, we define a metric using the complex structure on $M$, which we will see is locally equivalent to $\rho_S=\rho_F$.
This metric is obtained by taking the definition for $\rho_F$, and rewriting it with holomorphic maps in place of smooth maps.
We say $\rho_H(\zeta_1,\zeta_2)<\delta$ if and only if there exists $K\in \N$, holomorphic functions $f_1,\ldots, f_K:B_{\C}(1/2)\rightarrow M$, and
$\delta_1,\ldots, \delta_K>0$ with $\sum_{l=1}^K \delta_l\leq \delta$ such that:
\begin{itemize}
\item $df_j(z)\diff{\zb} = \sum_{l=1}^m s_j^l(z,\zb) \delta_j^{\beta_l} L_l(f_j(z))$, with $\Norm{\sum_{l} |s_j^l|^2}[L^\infty(B_{\C}(1/2))]<1$.
\item $f_j(B_{\C}(1/2))\bigcap f_{j+1}(B_{\C}(1/2)) \ne \emptyset$, $1\leq j\leq K-1$.
\item $\zeta_1\in f_1(B_{\C}(1/2))$, $\zeta_2\in f_K(B_{\C}(1/2))$.
\end{itemize}
$\rho_H$ is clearly an extended metric; once we show it is locally equivalent to $\rho_S$, it will follow that $\rho_H$ is a metric.

\begin{thm}\label{Thm::ResSubH}
\begin{enumerate}[label=(\alph*),series=shtheoremenumeration]
\item $\forall \zeta_1,\zeta_2\in M$, $\rho_S(\zeta_1,\zeta_2)= \rho_F(\zeta_1,\zeta_2)\leq  \rho_H(\zeta_1,\zeta_2)$.
\end{enumerate}
Fix a compact set $\Compact\subseteq M$.  We write $A\lesssim B$ for $A\leq CB$ where $C$ can be chosen independent
of $\zeta,\zeta_1,\zeta_2\in \Compact$ and $\delta\in (0,1]$.  We write $A\approx B$ for $A\lesssim B$ and $B\lesssim A$.
\begin{enumerate}[resume*=shtheoremenumeration]
\item $\rho_H(\zeta_1,\zeta_2)\lesssim \rho_S(\zeta_1,\zeta_2)$, $\forall \zeta_1,\zeta_2\in \Compact$, and therefore
$\rho_H$  and $\rho_S$ are equivalent on compact sets.
\item All of the conclusions of \cref{Thm::Results::NSW} hold (when applied to $(W_1,d_1),\ldots, (W_{2m},d_{2m})$) and (by identifying $\R^{2n}\cong \C^n$)
the maps $\Phi_{\zeta,\delta}:B_{\C^n}(1)\rightarrow B_S(\zeta,\delta)\subseteq M$ can be taken to be holomorphic.
\end{enumerate}
Because $\Phi_{\zeta,\delta}$ is holomorphic, $\Phi_{\zeta,\delta}^{*}\delta^{\beta_j} L_j$ is a $T^{0,1}$ vector field; in other words, $\Phi_{\zeta,\delta}^{*}\delta^{\beta_j}L_j(z)\in \Span_{\C}\left\{\diff{\zb_1},\ldots, \diff{\zb_n}\right\}$, $\forall z\in B_{\C^n}(1)$.
We can thus think of $\Phi_{\zeta,\delta}^{*}\delta^{\beta_j} L_j$ as a map $B_{\C^n}(1)\rightarrow \C^n$.
\begin{enumerate}[resume*=shtheoremenumeration]
	\item $\CjNorm{\Phi_{\zeta,\delta}^{*}\delta^{\beta_j} L_j}{k}[B_{\C^n}(1)][\C^n]\lesssim 1$, $\forall \zeta\in \Compact, \delta\in (0,1], k\in \N$, where the implicit constant depends on $k$, but not on $\zeta\in \Compact$ or $\delta\in (0,1]$.
	\item $\Phi_{x,\delta}^{*}\delta^{\beta_1}L_1(z),\ldots, \Phi_{x,\delta}^{*} \delta^{\beta_m}{L_m}(z)$ span $T_z^{0,1}\C^n$ uniformly in $z,\zeta,\delta$ in the sense that
	$$
	\max_{j_1,\ldots, j_n\in \{1,\ldots m\}} \inf_{z\in B_{\C^n}(1)} \left|\det\left( \Phi_{\zeta,\delta}^{*} \delta^{\beta_{j_1}}L_{j_1} (z) | \ldots | \Phi_{\zeta,\delta}^{*} \delta^{\beta_{j_n}}L_{j_n}(z) \right)\right|\approx 1, \quad \zeta\in \Compact, \delta\in (0,1].
	$$
\end{enumerate}
\end{thm}

\begin{rmk}
In \cref{Thm::ResSubH} we described a $C^\infty$ version of sub-Hermitian geometry.  With a very similar proof one can obtain
a similar real analytic version;
see \cref{Rmk::ResSubE::RA}.  One can also obtain results for vector fields with only a finite level of smoothness; see \cref{Rmk::ResSubE::FiniteSmoothness}.
\end{rmk}

\begin{rmk}
In the above discussion, we studied the vector fields $\delta^{\beta_1}L_1,\ldots, \delta^{\beta_m} L_m$.  In many applications, the vector fields depend on $\delta$
in a more complicated way (such an example is given in \cref{Section::ExtremalBasis}).  Furthermore, in some applications, $\delta$ ranges
  over $(0,1]^\mu$ instead of $(0,1]$ (as studied in the real setting in \cite{StreetMultiparameterCCBalls}).  Our proof methods allow us to study such settings in the same way; see \cref{Rmk::ResSubE::MoreComplicated}.  We stated results
in this  setting for simplicity of presentation, so that the reader can easily see the main ideas.
\end{rmk}


\section{The Main Results}\label{Section::MainResult}
Let $X_1,\ldots, X_q$ be real $C^1$ vector fields on a $C^2$ manifold $\fM$ and let $L_1,\ldots, L_m$ be complex $C^1$ vector fields on $\fM$.
For each $x\in \fM$, set $\LVS_x:=\Span_\C\{L_1(x),\ldots, L_m(x), X_1(x),\ldots, X_q(x)\}$, $\XVS_x:=\Span_{\C}\{ X_1(x),\ldots, X_q(x)\}$.

Fix $x_0\in \fM$, $\xi>0$.  Set $r:=\dim \XVS_{x_0}$ and $n+r:=\dim \LVS_{x_0}$.
Our goal in this section is to choose a ``coordinate system'' $\Phi:B_{\R^r\times \C^n}(1)\rightarrow B_{X,L}(x_0,\xi)$
so that
$\Phi^{*}X_1,\ldots, \Phi^{*}X_q, \Phi^{*}L_1,\ldots, \Phi^{*}L_m$ have a desired level of regularity and $\forall (t,z)\in B_{\R^r\times \C^n}(1)$,
\begin{equation*}
\Span_{\C}\{\Phi^{*}X_1(t,z),\ldots, \Phi^{*}X_q(t,z), \Phi^{*}L_1(t,z),\ldots, \Phi^{*}L_m(t,z)\}=\Span_{\C}\left\{\diff{t_1},\ldots,\diff{t_r},\diff{\zb[1]},\ldots, \diff{\zb[n]}\right\},
\end{equation*}
where we have given $\R^r\times \C^n$ coordinates $(t_1,\ldots, t_r, z_1,\ldots, z_n)$.  Finally, we wish to pick this coordinate system so that
$\Phi^{*}X_1,\ldots, \Phi^{*}X_q, \Phi^{*}L_1,\ldots, \Phi^{*}L_m$ are normalized in a way which is useful for applying techniques from analysis.

Let $Z_1,\ldots, Z_{q+m}:=X_1,\ldots, X_q,L_1,\ldots, L_m$.  Our three main \textit{algebraic} assumptions are as follows:
\begin{enumerate}[(i)]
\item $\forall x\in B_{X,L}(x_0,\xi)$, $\LVS_x\bigcap \LVSb[x]=\XVS_x$.
\item $[Z_j,Z_k]=\sum_{l=1}^{m+q} c_{j,k}^{1,l} Z_l$ and $[Z_j,\Zb[k]]=\sum_{l=1}^{m+q} c_{j,k}^{2,l} Z_l+ \sum_{l=1}^{m+q}c_{j,k}^{3,l} \Zb[l]$, where $c_{j,k}^{a,l}\in \CSpace{B_{X,L}(x_0,\xi)}$, $1\leq a\leq 3$, $1\leq j,k,l\leq q+m$ (here we are giving
$B_{X,L}(x_0,\xi)$ the topology induced by the associated metric \cref{Eqn::FuncMan::Defnrho}).
\item $x\mapsto \dim \LVS_x$, $B_{X,L}(x_0,\xi)\rightarrow \N$, is constant in $x$ (it follows from the other assumptions that this is equivalent to the map $x\mapsto \dim \XVS_x$ being constant in $x$; see \cref{Section::CommentsAssump}).
\end{enumerate}

Under the above hypotheses, $B_{X,L}(x_0,\xi)$ is a $C^2$, injectively immersed submanifold of $\fM$ 
(see \cref{Prop::AppendImmerse}), and $\C T_xB_{X,L}(x_0,\xi)=\LVS_x+\LVSb[x]$, $\forall x\in B_{X,L}(x_0,\xi)$.
In particular, using \cref{Lemma::AppendCR::dimFormula},
$$\dim B_{X,L}(x_0,\xi) =\dim T_{x_0} B_{X,L}(x_0,\xi)= \dim (\LVS_{x_0}+\LVSb[x_0]) = 2\dim \LVS_{x_0} - \dim \XVS_{x_0}=2n+r.$$
Henceforth we view $X_1,\ldots, X_q, L_1,\ldots, L_m$ as $C^1$ vector fields on $B_{X,L}(x_0,\xi)$.

For $a,b\in \N$, we set
\begin{equation}\label{Eqn::MainRes::DefnsI}
\sI(a,b):=\{ (i_1,i_2,\ldots, i_a) : i_1,\ldots, i_a\in \{1,\ldots, b\}\}=\{1,\ldots,b\}^a.
\end{equation}
For $K=(k_1,\ldots, k_{r_1})\in \sI(r_1,q)$, we write $X_K$ for the list $X_{k_1},\ldots, X_{k_{r_1}}$ and for $J=(j_1,\ldots, j_{n_1})\in \sI(n_1,m)$ we write
$L_J$ for the list $L_{j_1},\ldots, L_{j_{n_1}}$.  We write $\bigwedge X_K:=X_{k_1}\wedge X_{k_2}\wedge \cdots\wedge X_{k_{r_1}}$ and
$\bigwedge L_{J}:=L_{j_1}\wedge L_{j_2}\wedge \cdots\wedge L_{j_{n_1}}$.

Fix $\zeta\in (0,1]$, $K_0\in \sI(r,q)$, $J_0\in \sI(n,m)$ such that
\begin{equation}\label{Eqn::MainRes::ChooseZeta}
	\max_{\substack{K\in \sI(r_1,q), J\in \sI(n_1,m) \\ r_1+n_1=r+n }}
	\left|
	\frac{\left(\bigwedge X_K(x_0)\right)\bigwedge \left(\bigwedge L_J(x_0)\right)}{\left(\bigwedge X_{K_0}(x_0)\right)\bigwedge\left(\bigwedge L_{J_0}(x_0)\right) }
	\right|  \leq \zeta^{-1}.
\end{equation}
See \cref{Append::LA::Wedge} for the definition of this quotient.  Such a choice of $J_0$, $K_0$, and $\zeta$ always exist; see \cref{Rmk::AppendWedge::AboutJ0K0}.  One cannot necessarily
choose $K_0$, $J_0$ so that \cref{Eqn::MainRes::ChooseZeta} holds with $\zeta=1$, however if $n=0$ or $r=0$ (the two most important special cases) one always can--see
\cref{Rmk::AppendWedge::AboutJ0K0}.
Without loss of generality, reorder $X_1,\ldots, X_q$ and $L_1,\ldots, L_m$ so that $K_0=(1,2,\ldots, r)$, $J_0=(1,2,\ldots, n)$.

Let $W_1,\ldots, W_{2m+q}$ denote the list of vector fields $X_1,\ldots, X_q, 2\Real(L_1),\ldots, 2\Real(L_m),2\Imag(L_1),\ldots, 2\Imag(L_m)$; and order $W_1,\ldots, W_{2m+q}$
so that
\begin{equation}\label{Eqn::MainRes::OrderWs}
W_1,\ldots, W_{2n+r}=X_1,\ldots, X_r,2 \Real(L_1),\ldots, 2\Real(L_n), 2\Imag(L_1),\ldots, 2\Imag(L_n).
\end{equation}
Define $\WVS_x:=\Span_{\R}\{W_1(x),\ldots, W_{2m+q}(x)\} = (\LVS_x+\LVSb[x])\cap T_xB_{X,L}(x_0,\xi)$.
Set $P_0:=(1,\ldots, 2n+r)\in \sI(2n+r,2m+q)$ and for any $P=(p_1,\ldots, p_{2n+r})\in \sI(2n+r,2m+q)$ we
write $W_P$ for the list $W_{p_1},\ldots, W_{p_{2n+r}}$ and set
$\bigwedge W_P= W_{p_1}\wedge W_{p_2}\wedge \cdots \wedge W_{p_{2n+r}}$.  In particular,
\begin{equation*}
\begin{split}
\bigwedge W_{P_0} &= X_1\wedge X_2\wedge \cdots \wedge X_r\wedge 2\Real(L_1)\wedge 2\Real(L_2)\wedge \cdots \wedge 2\Real(L_n)\wedge 2\Imag(L_1)\wedge 2\Imag(L_2)\wedge \cdots \wedge 2\Imag(L_n)
\\&=\left(\bigwedge X_{K_0}\right)\bigwedge\left(\bigwedge 2\Real(L)_{J_0}\right)\bigwedge \left(\bigwedge 2\Imag(L)_{J_0}\right),
\end{split}
\end{equation*}
where $\bigwedge 2\Real(L)_{J_0}$ and $\bigwedge 2\Imag(L)_{J_0}$ are defined in the obvious way; see \cref{Eqn::AppedWedge::DefineWedges}.
Note that $B_{W_{P_0}}(x_0,\xi)$ and $B_{X_{K_0},L_{J_0}}(x_0,\xi)$ are (by definition)  equal; see \cref{Eqn::FuncComplex::Ball}.

\begin{defn}\label{Defn::MainRes::sC}
For $x\in \fM$, $U\subseteq \fM$, and $\eta>0$, we say $W_{P_0}$ satisfies $\sC(x,\eta,U)$ if for every $a\in B^{2n+r}(\eta)$ the expression
\begin{equation*}
	e^{a_1W_1+a_2W_2+\cdots +a_{2n+r} W_{2n+r}} x
\end{equation*}
exists in $U$.  More precisely, consider the differential equation
\begin{equation*}
\diff{r} E(r) = a_1 W_1(E(r))+\cdots + a_{2n+r} W_{2n+r}(E(r)), \quad E(0)=x.
\end{equation*}
We assume that a solution $E:[0,1]\rightarrow U$ exists for this differential equation.  We have $E(r) = e^{ra_1W_1+\cdots + ra_{2n+r} W_{2n+r}} x$.
\end{defn}

We fix the following two quantities:
\begin{itemize}
\item Fix $\eta>0$ so that $W_{P_0}$ satisfies $\sC(x_0,\eta,\fM)$.
\item Fix $\delta_0>0$ such that $\forall \delta\in (0,\delta_0]$, the following holds.  If $z\in B_{X_{K_0},L_{J_0}}(x_0,\xi)$
is such that $W_{P_0}$ satisfies $\sC(z, \delta, B_{X_{K_0},L_{J_0}}(x_0,\xi))$ and if
$t\in B_{\R^{2n+r}}(\delta)$ is such that $e^{t_1 W_1+\cdots+t_{2n+r} W_{2n+r}}z=z$
and if $W_1(z),\ldots, W_{2n+r}(z)$ are linearly independent, then $t=0$.
\end{itemize}
Such a choice of $\eta,\delta_0$ always exist (see \cref{Lemma::MoreAssume::ExistEtaDelta0}).  These constants are invariant under $C^2$ diffeomorphisms, and our quantitative results will be in terms of these constants;
see \cite[\SSSectionMoreOnAssumptions]{StovallStreetI} for a detailed discussion of $\eta$ and $\delta_0$.

In our main result, we keep track of what parameters each estimate depends on\footnote{Keeping track of constants in our main theorem is essential for applications.
For example, to prove the results in \cref{Section::Results::NSW,Section::ResGeom::SubHerm,Section::Res::SubE} we will apply \cref{Thm::Results::MainThm} infinitely many times, and the constants must be uniform over all these applications.}.  To ease notation, we introduce various notions of ``admissible constants''.  These will be constants
which only depend on certain parameters.\footnote{The various notions of admissible constants may vary from section to section, but we are explicit about how they are defined whenever
they are used.  See \cref{Rmk::Nirenberg::AdmissibleOkay} for how this varying notation is exploited in the proofs.}

\begin{defn}\label{Defn::Results::0Admiss}
We say $C$ is a $0$-admissible constant if $C$ can be chosen to depend only on upper bounds for $m$, $q$, $\zeta^{-1}$, $\xi^{-1}$, and
$\CNorm{c_{j,k}^{a,l}}{B_{X_{K_0},L_{J_0}}(x_0,\xi)}$, $1\leq j,k,l\leq m+q$, $1\leq a\leq 3$.
\end{defn}

Fix $s_0\in (1,\infty)\cup \{\omega\}$; when $s_0\in (1,\infty)$ the following result concerns the setting of $\ZygSpace{s}$ for $s\in [s_0,\infty]$ (and the results
are stronger the closer $s_0$ is to $1$, but the constants depend on the choice of $s_0$).  When $s_0=\omega$ the following result concerns
the real analytic setting.   Thus, there are two cases in what follows:  when $s_0\in (1,\infty)$ and when $s_0=\omega$.

\begin{defn}\label{Defn::Results::ZygAdmiss}
If $s_0\in (1,\infty)$, for $s\in [s_0,\infty)$, if we say $C$ is a $\Zygad{s}$-admissible constant, it means that we assume $c_{j,k}^{a,l}\in \ZygXSpace{X_{K_0},L_{J_0}}{s}[B_{X_{K_0},L_{J_0}}(x_0,\xi)]$,
for $1\leq j,k,l\leq m+q$, $1\leq a\leq 3$.  $C$ can then be chosen to depend only on $s$, $s_0$, and upper bounds for $m$, $q$, $\zeta^{-1}$, $\xi^{-1}$, $\eta^{-1}$, $\delta_0^{-1}$,
and $\ZygXNorm{c_{j,k}^{a,l}}{X_{K_0},L_{J_0}}{s}[B_{X_{K_0},L_{J_0}}(x_0,\xi)]$, $1\leq j,k,l\leq m+q$, $1\leq a\leq 3$.  For $s\in (0,s_0)$, we define $\Zygad{s}$-admissible
constants to be $\Zygad{s_0}$-admissible constants.
\end{defn}

\begin{defn}\label{Defn::Results::omegaAdmiss}
If $s_0=\omega$, and if we say $C$ is an $\Zygad{\omega}$-admissible constant, it means that we assume
$c_{j,k}^{a,l}\in \AXSpace{X_{K_0},L_{J_0}}{x_0}{\eta}$, $1\leq j,k,l\leq m+q$, $1\leq a\leq 3$.  $C$ can be chosen to depend only on upper bounds for
$m$, $q$, $\zeta^{-1}$, $\xi^{-1}$, $\eta^{-1}$, $\delta_0^{-1}$,
and $\AXNorm{c_{j,k}^{a,l}}{X_{K_0},L_{J_0}}{x_0}{\eta}$, $1\leq j,k,l\leq m+q$, $1\leq a\leq 3$.
\end{defn}

Whenever we define a notion of $*$-admissible constant (where $*$ can be any symbol), we write $A\lesssim_{*}B$ for $A\leq CB$, where
$C$ is a positive $*$-admissible constant.  We write $A\approx_{*}B$ for $A\lesssim_{*}B$ and $B\lesssim_{*} A$.

In what follows, we give $\R^r\times \C^n$ coordinates $(t,z)$, where $t=(t_1,\ldots, t_r)\in \R^r$ and $z=(z_1,\ldots, z_n)\in \C^n$.
We write $\diff{t}$ for the column vector $[\diff{t_1},\ldots, \diff{t_r}]^{\transpose}$ and $\diff{\zb}$ for the column vector $[\diff{\zb[1]},\ldots, \diff{\zb[n]}]^{\transpose}$.

\begin{thm}\label{Thm::Results::MainThm}
There exists a $0$-admissible constant $\chi\in (0,\xi]$ such that:
\begin{enumerate}[label=(\roman*),series=maintheoremenumeration]
\item\label{Item::Results::MainThm::1} $\forall y\in B_{X_{K_0},L_{J_0}}(x_0,\chi)$,
\begin{equation*}
	\left(\bigwedge X_{K_0}(y)\right)\bigwedge \left(\bigwedge L_{J_0}(y)\right)\ne 0, \quad \bigwedge W_{P_0}(y)\ne 0.
\end{equation*}
In particular, $X_{K_0}(y),L_{J_0}(y)$ is a basis for $\LVS_y$ and $W_{P_0}(y)$ is a basis for $\WVS_y$.  Recall, $$\WVS_y=\Span_\R\{W_1(y),\ldots, W_{2m+q}(y)\}.$$
\item\label{Item::Results::MainThm::2} $\forall y\in B_{X_{K_0},L_{J_0}}(x_0,\chi)$,
\begin{equation*}
	\max_{\substack{J\in \sI(n_1,m), K\in \sI(r_1,q) \\ n_1+r_1=n+r}} \left| \frac{ \left(\bigwedge X_{K}(y)\right)\bigwedge \left(\bigwedge L_{J}(y)\right)}{\left(\bigwedge X_{K_0}(y)\right)\bigwedge \left(\bigwedge L_{J_0}(y)\right)} \right|\approx_0 1, \quad
	\max_{P\in \sI(2n+r,2m+q)} \left| \frac{ \bigwedge W_P(y)}{\bigwedge W_{P_0}(y) } \right|\approx_0 1.
\end{equation*}
\item\label{Item::Results::MainThm::3} $\forall \chi'\in (0,\chi]$, $B_{X_{K_0},L_{J_0}}(x_0,\chi')$ is an open subset of $B_{X,L}(x_0,\xi)$, and is therefore a submanifold.
\end{enumerate}
For the rest of the theorem, we assume:
\begin{itemize}
\item If $s_0\in (1,\infty)$, we assume $c_{j,k}^{a,l}\in \ZygXSpace{X_{K_0},L_{J_0}}{s_0}[B_{X_{K_0},L_{J_0}}(x_0,\xi)]$, $\forall 1\leq j,k,l\leq m+q$, $1\leq a\leq 3$.
\item If $s_0=\omega$, we assume $c_{j,k}^{a,l}\in \AXSpace{X_{K_0},L_{J_0}}{x_0}{\eta}$, $\forall 1\leq j,k,l\leq m+q$, $1\leq a\leq 3$.
\end{itemize}
There exists a $C^2$ map $\Phi:B_{\R^r\times \C^n}(1)\rightarrow B_{X_{K_0},L_{J_0}}(x_0,\chi)$ and $\Zygad{s_0}$-admissible constants
$\xi_1,\xi_2>0$ such that:
\begin{enumerate}[resume*=maintheoremenumeration]
\item\label{Item::ResultsMainThm::PhiOpen} $\Phi(B_{\R^r\times \C^n}(1))$ is an open subset of $B_{X_{K_0},L_{J_0}}(x_0,\chi)$ and is therefore a submanifold of $B_{X,L}(x_0,\xi)$.
\item\label{Item::ResultsMainThem::PhiDiffeo} $\Phi:B_{\R^r\times \C^n}(1)\rightarrow \Phi(B_{\R^r\times \C^n}(1))$ is a $C^2$-diffeomorphism.
\item\label{Item::Results::MainThm::xi1xi2} $B_{X,L}(x_0,\xi_2)\subseteq B_{X_{K_0},L_{J_0}}(x_0,\xi_1)\subseteq \Phi(B_{\R^r\times \C^n}(1))\subseteq B_{X_{K_0},L_{J_0}}(x_0,\chi)\subseteq B_{X,L}(x_0,\xi)$.
\item\label{Item::Results::MainThm::Phiof0} $\Phi(0)=x_0$.
\end{enumerate}
There exists an $\Zygad{s_0}$-admissible constant 
$K\geq 1$
and a matrix
$\AMatrix:B_{\R^r\times \C^n}(1)\rightarrow \M^{(n+r)\times (n+r)}(\C)$ such that:
\begin{enumerate}[resume*=maintheoremenumeration]
\item\label{Item::Results::MainThm::AMatrix0} $\AMatrix(0)=0$.
\item\label{Item::ResultsMainThm::AFormula} \begin{equation*}
\begin{bmatrix}
\diff{t} \\
\diff{\zb}
\end{bmatrix}
=K^{-1} (I+\AMatrix)
\begin{bmatrix}
\Phi^{*} X_{K_0}\\
\Phi^{*} L_{J_0}
\end{bmatrix},
\end{equation*}
where we have written $\Phi^{*} X_{K_0}$ for the column vector of vector fields $[\Phi^{*}X_1,\ldots, \Phi^{*} X_r]^{\transpose}$ and similarly for $\Phi^{*} L_{J_0}$.
\item\label{Item::ResultsMainThm::ABound}
\begin{itemize}
	\item If $s_0\in (1,\infty)$, $\ZygNorm{\AMatrix}{s+1}[B_{\R^r\times \C^n}(1)][\M^{(n+r)\times (n+r)}]\lesssim_{\Zygad{s}} 1$, $\forall s\in (0,\infty)$, and $\ZygNorm{\AMatrix}{s_0+1}[B_{\R^r\times \C^n}(1)][\M^{(n+r)\times (n+r)}]\leq \frac{1}{4}$.
	\item If $s_0=\omega$, $\ANorm{\AMatrix}{r+2n}{1}[\M^{(n+r)\times (n+r)}]\leq \frac{1}{4}$, where we have identified $\R^r\times \C^n\cong \R^{r+2n}$.
\end{itemize}
Note that in either case, this implies the matrix $(I+\AMatrix(\zeta))$ is invertible, $\forall \zeta\in B_{\R^r\times \C^n}(1)$.
\item\label{Item::Results::MainThem::IsEMap} $\forall \zeta\in B_{\R^r\times \C^n}(1)$, $1\leq k\leq q$, $1\leq j\leq m$,
\begin{equation*}
\Phi^{*} X_k(\zeta) \in \Span_{\R}\left\{\diff{t_1},\ldots, \diff{t_r}\right\}, \quad \Phi^{*} L_j(\zeta) \in \Span_{\C} \left\{ \diff{t_1},\ldots, \diff{t_r},\diff{\zb[1]},\ldots, \diff{\zb[n]}\right\}.
\end{equation*}
\item\label{Item::ResultsMainThm::PullbacksSmooth}
\begin{itemize}
	\item If $s_0\in (1,\infty)$, we have $\forall s\in (0,\infty)$, $1\leq k\leq q$, $1\leq j\leq m$,
	\begin{equation*}
		\ZygNorm{\Phi^{*}X_k}{s+1}[B_{\R^r\times \C^n}(1)][\R^r]\lesssim_{\Zygad{s}} 1, \quad \ZygNorm{\Phi^{*}L_j}{s+1}[B_{\R^r\times \C^n}(1)][\C^{r+n}]\lesssim_{\Zygad{s}} 1.
	\end{equation*}
	\item If $s_0=\omega$, we have for $1\leq k\leq q$, $1\leq j\leq m$,
	\begin{equation*}
		\ANorm{\Phi^{*} X_k}{2n+r}{1}[\R^r]\lesssim_{\Zygad{\omega}} 1,\quad \ANorm{\Phi^{*} L_j}{2n+r}{1}[\C^{r+n}]\lesssim_{\Zygad{\omega}} 1.
	\end{equation*}
\end{itemize}
\end{enumerate}
\end{thm}

\begin{rmk}\label{Rmk::MainRes::EMap}
In the language of \cref{Section::Emfld}, the map $\Phi:B_{\R^r\times \C^n}(1)\rightarrow B_{X,L}(x_0,\xi)$ from \cref{Thm::Results::MainThm} is an E-map;
where $B_{X,L}(x_0,\xi)$ is given the E-manifold structure with the associated
elliptic structure $\LVS$.
In particular, when $r=0$, $\LVS$ is a complex structure and
the E-manifold structure on $B_{X,L}(x_0,\xi)$ is the complex manifold structure
associated to $\LVS$ (via the Newlander-Nirenberg theorem).
In this case, $\Phi:B_{\C^n}(1)\rightarrow B_{X,L}(x_0,\xi)$ is a holomorphic map (see \cref{Rmk::EMfld::FullSubcategory}).
This is particularly important for applications to several complex variables.  For example this is used in \cref{Section::ResSmooth::Complex,Section::ResGeom::SubHerm,Section::ExtremalBasis::Biholo}
to guarantee the desired coordinate charts are holomorphic.
\end{rmk}

	\subsection{Densities}\label{Section::MainResult::Densities}
In many applications, one wishes to change variables in an integral using the coordinate chart given in \cref{Thm::Results::MainThm} (see, e.g., the settings in \cref{Section::Results::NSW,Section::ResGeom::SubHerm,Section::Res::SubE}\footnote{For example, such changes of variables were important in the study of multi-parameter
singular integrals and singular radon transforms in \cite{StreetMultiParamSingInt,SteinStreetA,SteinStreetI,SteinStreetII,SteinStreetII}.}).
Thus, it is important to understand pullbacks of certain densities via the map $\Phi$.  We present such results in this section.
We refer the reader to \cite{GuilleminNotes} for a quick introduction to densities (see also \cite{NicolaescuLecturesOnGeometryOfManifolds} where
densities are called $1$-densities).  
In this section, we take all the assumptions as in \cref{Thm::Results::MainThm}
and let $\Phi$ be as in that theorem.

Let $\chi\in (0,\xi]$ be as in \cref{Thm::Results::MainThm} and let $\nu$ be a real $C^1$ density on $B_{X_{K_0},L_{J_0}}(x_0,\chi)$.
Suppose, for $1\leq k\leq r$, $1\leq j\leq n$,
\begin{equation*}
\Lie{X_k} \nu = f_k^1\nu, \quad \Lie{L_{j}} \nu = f_j^2\nu,\quad f_k^1,f_j^2\in \CSpace{B_{X_{K_0},L_{J_0}}(x_0,\chi)},
\end{equation*}
where $\Lie{V}$ denotes the Lie derivative with respect to $V$, and $\Lie{L_j}$ is defined as $\Lie{\Real{L_j}}+i\Lie{\Imag{L_j}}$.

\begin{defn}
If we say $C$ is a $\Zygsonu$-admissible constant, it means $C$ is a $\Zygad{s_0}$-admissible constant which is also allowed to
depend on upper bounds for
$\CNorm{f_k^1}{B_{X_{K_0},L_{J_0}}(x_0,\chi)}$, $1\leq k\leq r$,
and $\CNorm{f_j^2}{B_{X_{K_0},L_{J_0}}(x_0,\chi)}$, $1\leq j\leq n$.  This definition applies in either case: $s_0\in (1,\infty)$ or $s_0=\omega$.
\end{defn}



\begin{defn}
If $s_0\in (1,\infty)$, for $s>0$, if we say $C$ is a $\Zygad{s;\nu}$-admissible constant it means that we assume
$f_k^1,f_j^2\in \ZygXSpace{X_{K_0},L_{J_0}}{s}[B_{X_{K_0},L_{J_0}}(x_0,\chi)]$, $1\leq k\leq r$, $1\leq j\leq n$.
$C$ is allowed to depend on anything an $\Zygad{s}$-admissible constant is allowed to depend on, and is also allowed to
depend on upper bounds for $\ZygXNorm{f_k^1}{X_{K_0},L_{J_0}}{s}[B_{X_{K_0},L_{J_0}}(x_0,\chi)]$, $1\leq k\leq r$,
and $\ZygXNorm{f_j^2}{X_{K_0},L_{J_0}}{s}[B_{X_{K_0},L_{J_0}}(x_0,\chi)]$, $1\leq j\leq n$.
For $s\leq 0$, we define $\Zygad{s;\nu}$-admissible constants to be $\Zygsonu$-admissible constants.
\end{defn}

If $s_0=\omega$ we fix some $r_0>0$; the results which follow depend on the choice of $r_0$.

\begin{defn}
If $s_0=\omega$, if we say $C$ is an $\Zygad{\omega;\nu}$-admissible constant, it means that we assume
$f_k^1,f_j^2\in \AXSpace{X_{K_0},L_{J_0}}{x_0}{r_0}$.  $C$ is allowed to depend on anything an
$\Zygad{\omega}$-admissible constant may depend on, and is allowed to depend on upper bounds for $r_0^{-1}$,
$\AXNorm{f_k^1}{X_{K_0},L_{J_0}}{x_0}{r_0}$, $1\leq k\leq r$,
and $\AXNorm{f_j^2}{X_{K_0},L_{J_0}}{x_0}{r_0}$, $1\leq j\leq n$.
\end{defn}


\begin{thm}\label{Thm::Results::Desnity::MainResult}
Define $h\in \CjSpace{1}[B_{\R^r\times\C^n}(1)]$ by $\Phi^{*}\nu = h\LebDensity$, where $\LebDensity$ denotes the usual
Lebesgue density on $\R^r\times \C^n$.
\begin{enumerate}[(i)]
\item\label{Item::Results::Density::hConst} $\forall \zeta\in B_{\R^r\times \C^n}(1)$,
$$h(\zeta)\approx_{\Zygsonu} \nu(X_1,\ldots, X_r, 2\Real(L_1),\ldots, 2\Real(L_n), 2\Imag(L_1),\ldots, 2\Imag(L_n))(x_0). $$
In particular, $h(\zeta)$ always has the same sign, and is either never zero or always zero.
\item\label{Item::Results::Density::hSmooth}
\begin{itemize}
\item If $s_0\in (1,\infty)$, for $s>0$,
\begin{equation*}
	\ZygNorm{h}{s}[B_{\R^r\times \C^n}(1)]\lesssim_{\Zygad{s-1;\nu}}  \left| \nu(X_1,\ldots, X_r, 2\Real(L_1),\ldots, 2\Real(L_n), 2\Imag(L_1),\ldots, 2\Imag(L_n))(x_0)\right|.
\end{equation*}
\item If $s_0=\omega$,
\begin{equation*}
	\ANorm{h}{2n+r}{\min\{1,r_0\}}\lesssim_{\Zygad{\omega;\nu}}   \left| \nu(X_1,\ldots, X_r, 2\Real(L_1),\ldots, 2\Real(L_n), 2\Imag(L_1),\ldots, 2\Imag(L_n))(x_0)\right|.
\end{equation*}
\end{itemize}
\end{enumerate}
\end{thm}

\begin{cor}\label{Cor::Results::Desnity::MainCor}
Let $\xi_2>0$ be as in \cref{Thm::Results::MainThm}.  Then,
\begin{equation}\label{Eqn::Results::Density::MainEqual}
\begin{split}
	&\nu(B_{X_{K_0},L_{J_0}}(x_0,\xi_2)) \approx_{\Zygsonu} \nu(B_{X,L}(x_0,\xi_2))
	\\&\approx_{\Zygsonu}
	\nu(X_1,\ldots, X_r, 2\Real(L_1),\ldots, 2\Real(L_n), 2\Imag(L_1),\ldots, 2\Imag(L_n))(x_0),
\end{split}
\end{equation}
and therefore
\begin{equation}\label{Eqn::Results::Density::ToShow1}
\begin{split}
	&\left|\nu(B_{X_{K_0},L_{J_0}}(x_0,\xi_2))\right| \approx_{\Zygsonu}\left| \nu(B_{X,L}(x_0,\xi_2) \right|
	\\&\approx_{\Zygsonu}
	\left|\nu(X_1,\ldots, X_r, 2\Real(L_1),\ldots, 2\Real(L_n), 2\Imag(L_1),\ldots, 2\Imag(L_n))(x_0)\right|
	\\&\approx_0
	\max_{\substack{K\in \sI(r,q), J\in \sI(n,m)}} \left|\nu( X_K, 2\Real(L)_J, 2\Imag(L)_J)(x_0)\right|
	\\&\approx_0
	\max_{P\in \sI(2n+r,2m+q)} \left|\nu(W_P)(x_0)\right|.
\end{split}
\end{equation}
\end{cor} 
	
	\subsection{Some Comments on the Assumptions}\label{Section::CommentsAssump}
Because $W_1,\ldots, W_{2m+q}$ span the tangent space of $B_{X,L}(x_0,\xi)$ at every point (see \cref{Prop::AppendImmerse}) and $B_{X,L}(x_0,\xi)$ is a $2n+r$ dimensional
manifold, it follows that $x\mapsto \dim \WVS_x$, taking $B_{X,L}(x_0,\xi)\rightarrow \N$, is constant.
However, the hypothesis that $x\mapsto \dim \LVS_x$ is constant does not follow from the other assumptions.  The next example elucidates this:

\begin{example}
On $\C$, consider the vector fields $L_1=\diff{z}$, $L_2=\zb\diff{\zb}$, $X_1=z\diff{z}+\zb\diff{\zb}$, and $X_2=\frac{1}{i}\left( z\diff{z}-\zb\diff{\zb}\right)$.
We then have
\begin{equation*}
	[L_1,L_2]=[X_1,X_2]=[L_2,X_1]=[L_2,X_2]=0,\quad [L_1,X_1]=L_1,\quad [L_1,X_2]=\frac{1}{i}L_1,
\end{equation*}
and the vector fields $L_1,L_2,X_1,X_2, \overline{L_1},\overline{L_2}$ span the complexified tangent space at every point (in fact $L_1$ and $\overline{L_1}$ do).
However,
$$\dim\Span_{\C}\{L_1(z),L_2(z),X_1(z),X_2(z)\} =
\begin{cases} 2, &z\ne 0,\\
1, &z=0.
\end{cases}$$
\end{example}

The assumption that $x\mapsto \dim \LVS_x$ is constant is equivalent to the assumption that $x\mapsto \dim \XVS_x$ is constant.
Indeed, by \cref{Lemma::AppendCR::dimFormula},  and the fact that $\dim \WVS_x= 2n+r=\dim B_{X,L}(x_0,\xi)$, we have
\begin{equation*}
2n+r=\dim \WVS_x = \dim (\LVS_x+\LVSb[x]) = 2\dim(\LVS_x) - \dim(\LVS_x\bigcap \LVSb[x]) = 2\dim(\LVS_x)-\dim(\XVS_x).
\end{equation*}
In particular, in the two most important special cases $\LVS_x=\XVS_x$ $\forall x$, or $\XVS_x=\{0\}$ $\forall x$, the hypothesis that $x\mapsto \dim \LVS_x$ is constant does follow
from the other assumptions. 

A choice of $\eta,\delta_0>0$, as in the hypotheses of \cref{Thm::Results::MainThm}, always exist.  In fact, they can be chosen uniformly on compact sets, as the next lemma shows.
\begin{lemma}\label{Lemma::MoreAssume::ExistEtaDelta0}
Let $W=W_1,\ldots, W_N$ be a list of $C^1$ vector fields on a $C^2$ manifold $M$ and let $\Compact\Subset M$ be a compact set.
\begin{enumerate}[(i)]
\item\label{Item::MoreAssump::ExistEta} $\exists \eta>0$ such that $\forall x_0\in \Compact$, $W$ satisfies $\sC(x_0,\eta,M)$.
\item\label{Item::MoreAssump::ExistDelta0} $\exists \delta_0>0$ such that $\forall \theta\in S^{N-1}$ if $x\in \Compact$ is such that $\theta_1 W_1(x)+\cdots+\theta_N W_N(x)\ne 0$,
then $\forall r\in (0,\delta_0]$,
\begin{equation*}
e^{r \theta_1 W_1+\cdots + r\theta_N W_N} x\ne x.
\end{equation*}
\end{enumerate}
\end{lemma}
\begin{proof}\Cref{Item::MoreAssump::ExistEta} is a simple consequence of the Phragm\'en--Lindel\"of Principle.  \Cref{Item::MoreAssump::ExistDelta0}
is proved in \cite[\SSLemmaMoreOnAssump]{StovallStreetI}.\end{proof}

Despite the fact that a choice of $\eta,\delta_0>0$ always exist (as described in \cref{Lemma::MoreAssume::ExistEtaDelta0}),
$\eta$ and $\delta_0$ are diffeomorphic invariant quantities\footnote{I.e., $\eta$ and $\delta_0$ remain unchanged when the entire setting is pushed forward
under a $C^2$ diffeomorphism.},
and the proof of existence of these constants in \cref{Lemma::MoreAssume::ExistEtaDelta0} depends on the $C^1$ norms of the vector fields $W_1,\ldots, W_{N}$
 in some fixed coordinate system (which is not a diffeomorphic invariant quantity).  Thus, we state all of our results in terms of $\eta$ and $\delta_0$
 to preserve the quantitative diffeomorphism invariance.  See \cref{Section::DiffeoInv}.

	
	\subsection{Diffeomorphism invariance}\label{Section::DiffeoInv}
The main results of this paper are invariant under arbitrary $C^2$ diffeomorphisms.  This is true quantitatively.  
For example, consider \cref{Thm::Results::MainThm}.  Let $X_1,\ldots, X_q,L_1,\ldots, L_m$ be the vector fields on $\fM$ from \cref{Thm::Results::MainThm} and
let $\Psi:\fM\rightarrow \fN$ be a $C^2$ diffeomorphism. 
Then $X_1,\ldots, X_q,L_1,\ldots, L_m$ satisfy the conditions of \cref{Thm::Results::MainThm} at the point $x_0\in \fM$ if and only if
$\Psi_{*} X_1,\ldots, \Psi_*X_q, \Psi_*L_1,\ldots, \Psi_* L_m$ satisfy the conditions at $\Psi(x_0)$.  Moreover, any constant which is $*$-admissible (where $*$ is any symbol)
with respect to $X_1,\ldots, X_q,L_1,\ldots, L_m$ is $*$-admissible with respect to $\Psi_{*} X_1,\ldots, \Psi_*X_q, \Psi_*L_1,\ldots, \Psi_* L_m$.
Finally, if $\Phi$ is the map guaranteed by \cref{Thm::Results::MainThm} when applied to $X_1,\ldots, X_q,L_1,\ldots, L_m$, then
$\Psi\circ \Phi$ is the map given by \cref{Thm::Results::MainThm} when applied to  $\Psi_{*} X_1,\ldots, \Psi_*X_q, \Psi_*L_1,\ldots, \Psi_* L_m$ (as can be seen by tracing through the proof).
Thus the main results (and, indeed, the entire proofs) are invariant under arbitrary $C^2$ diffeomorphisms.
	
	\subsection{The Frobenius Theorem and Singular Foliations}\label{Section::SingFoliation}
We now describe a consequence of \cref{Thm::Results::MainThm} which is not used in the rest of the paper:  it provides coordinate charts on leaves of singular foliations, which
behave well in a quantitative way near singular points.

Let $\fM$ be a smooth manifold and let $X_1,\ldots, X_q$ be real $C^\infty$ vector fields on $\fM$.
Suppose
\begin{equation*}
[X_j, X_k] = \sum_{l=1}^{q} c_{j,k}^{l} X_l, \quad c_{j,k}^{l}\in \CjSpace{\infty}[\fM].
\end{equation*}
For $x\in \fM$, let $\XVS_x:=\Span_{\R} \{ X_1(x),\ldots, X_q(x)\}$.  Under these hypotheses, the real Frobenius theorem applies to the distribution $\XVS$
to foliate $\fM$ into leaves.  The tangent bundle to each leaf is given by $\XVS$ restricted to the leaf.
Note that this may be a singular foliation:  different leaves may have different dimensions, since $x\mapsto \dim \XVS_x$ might not be constant in $x$.
For $x\in \fM$, let $\Leaf[x]$ denote the leaf passing through $x$; thus $\Leaf[x]$ is an injectively immersed $C^\infty$ submanifold of $\fM$.

\begin{defn}
We say $x\in \fM$ is a singular point of the foliation if $x\mapsto \dim \LVS_x$ is not constant on any neighborhood of $x$ (equivalently, if $x\mapsto \dim \Leaf[x]$ is not constant
on any neighborhood of $x$).
\end{defn}

$\Leaf[x]$ is a manifold, and is therefore defined by an atlas.  For applications in analysis, it is sometimes important to have quantitative control
of the charts which define the atlas.  An interesting aspect of \cref{Thm::Results::MainThm} is that it yields coordinate charts
which behave well whether or not one is near a singular point.

Indeed, let $\Compact \Subset \fM$ be a compact set.  \Cref{Lemma::MoreAssume::ExistEtaDelta0} and some straightforward estimates show that
\cref{Thm::Results::MainThm} (in the case $m=0$) applies
to the vector fields $X_1,\ldots, X_q$ (with, e.g., $s_0=3/2$), \textit{uniformly} for $x_0\in \Compact$.  Thus, any constant
which is $\Zygad{s}$-admissible (for any $s\in (0,\infty)$) in the sense of \cref{Thm::Results::MainThm} can be taken independent of $x_0\in \Compact$.
The map $\Phi$ provided by \cref{Thm::Results::MainThm} can be seen as a coordinate chart on $\Leaf[x_0]$, centered at $x_0$, which has good estimates
which are uniform in $x_0$.  In particular, as $x_0\in \Compact$ approaches a singular point in $\Compact$, the estimates remain uniform.

The above holds in the complex setting as well.  Again let $\fM$ be a smooth manifold, and let $L_1,\ldots, L_m$ be $C^\infty$ complex vector fields on $\fM$.
Suppose
\begin{equation*}
[L_j, L_k] = \sum_{l=1}^m c_{j,k}^{1,l} L_l, \quad [L_j, \Lb[k]] = \sum_{l=1}^m c_{j,k}^{2,l} L_l + \sum_{l=1}^{m} c_{j,k}^{3,l} \Lb[l], \quad c_{j,k}^{a,l}\in \CjSpace{\infty}[\fM].
\end{equation*}
For $x\in \fM$, set $\LVS_x:= \Span_{\C} \{ L_1(x),\ldots, L_m(x)\}$; we assume
$\LVS_x\bigcap \LVSb[x]=\{0\}$, $\forall x\in \fM$.
Under the above assumptions, the real Frobenius theorem applies to the distribution $\LVS+\LVSb$ to foliate $\fM$ into leaves;
as before this may be a singular foliation.
Let $\Leaf[x]$ denote the leaf passing through $x$.
 For each $x\in \fM$,  $\LVS$ (restricted to $\Leaf[x]$) defines a complex structure on $\Leaf[x]$, and the classical Newlander-Nirenberg theorem therefore
gives $\Leaf[x]$ the structure of a complex manifold.
As in the real case, \cref{Thm::Results::MainThm} applies uniformly as the base point $x_0$ ranges over compact sets (in this case, we take $q=0$),
in the sense that $\Zygad{s}$-admissible constants (for any $s\in (0,\infty)$) may be taken independent of $x_0$ as $x_0$ ranges over a compact set.
The map $\Phi$ provided by \cref{Thm::Results::MainThm} can be seen as a holomorphic coordinate chart near $x_0$, which has estimates which are uniform
on compact sets; whether or not that compact set contains a singular point.

\begin{rmk}
A previous (and weaker) version of the above ideas (originally described in \cite{StreetMultiparameterCCBalls}) was an essential point in the work of the author and Stein
on singular Radon transforms \cite{SteinStreetA, SteinStreetI, SteinStreetII, SteinStreetIII}.
For example, a corollary of one of the main results of \cite{SteinStreetIII} is the following.  Suppose $\gamma_t(x)$ is real analytic function defined on a neighborhood of
the origin of $(t,x)\in \R^N\times \R^n$, mapping to $\R^n$, and satisfying $\gamma_0(x)\equiv x$.  Define an operator acting on functions $f(x)$ defined near the origin in $x\in \R^n$ by
\begin{equation*}
Tf(x) = \psi(x) \int f(\gamma_t(x)) K(t) \: dt,
\end{equation*}
where $K(t)$ is a Calder\'on-Zygmund kernel supported near $t=0$, and $\psi\in C_0^\infty(\R^n)$ is supported near $x=0$.  Then, $T:L^p\rightarrow L^p$, for $1<p<\infty$;
see \cite{SteinStreetIII} for a more precise statement and further details.
This result does not follow from the foundational work of Christ, Nagel, Stein, and Wainger on singular Radon transforms \cite{ChristNagelSteinWaingerSingularAndMaximal};
however, the only additional ingredient necessary to conclude this result (beyond the theory in \cite{ChristNagelSteinWaingerSingularAndMaximal})
is the above described uniformity of coordinate charts near singular points (though the theory in \cite{SteinStreetIII} proceeds by proving a more general result, and
concluding the above result as a corollary).
\end{rmk}

\begin{rmk}
One way to view the above discussion is that \cref{Thm::Results::MainThm} is quantitatively invariant under $C^2$ diffeomorphisms (see \cref{Section::DiffeoInv}),
and being ``nearly'' a singular point is not a diffeomorphically invariant concept.  Indeed, consider the real case described above.
Fix $x_0\in \fM$ and let $k:=\dim \XVS_{x_0}=\dim \Leaf[x_0]$ and $N:=\dim \fM$.
Pick a coordinate system on $\fM$ near $x_0$. 
In this coordinate system, we may think of $X_1(x_0),\ldots, X_q(x_0)$ as
vectors in $\R^N$ which span a $k$ dimensional subspace of $\R^N$.  Let $\sigma:=\min |\det B|$, where $B$ ranges over all $k\times k$
submatrices of the $N\times q$ matrix $(X_1(x_0)|\cdots |X_q(x_0))$.  Then, $\sigma>0$.
One might say $x_0$ is ``nearly'' a singular point if $\sigma$ is small.  However, $\sigma$ is not invariant under diffeomorphisms:  the above procedure depended on the
choice of coordinate system.
This is one way of intuitively understanding why the estimates in  \cref{Thm::Results::MainThm} do not depend on a lower bound for $\sigma>0$.
\end{rmk}

\begin{rmk}
While we described the above for smooth vector fields, similar remarks hold for vector fields with a finite level of smoothness
using the same ideas.
\end{rmk}



	\subsection{Proof Outline}
\Cref{Thm::Results::MainThm} is the central result of this paper.  If all we wanted was a coordinate system like $\Phi$ in which the vector fields were normalized and had the desired regularity,
but did not have the key property given in  \cref{Thm::Results::MainThm} \cref{Item::Results::MainThem::IsEMap},\footnote{And replacing $\diff{\zb}$ with $\diff{x}$ in \cref{Thm::Results::MainThm} \cref{Item::ResultsMainThm::AFormula}, where $x\in \R^{2n}$.} then \cref{Thm::Results::MainThm} would be an easy consequence
of the main results in \cite{StovallStreetI,StovallStreetII,StovallStreetII} applied to $W_1,\ldots, W_{2m+q}$ (see \cref{Section::RealCase} for a detailed statement of this).
In particular, in the case when $q=0$ (and $M$ is given the complex structure induced by $\LVS$ via the Newlander-Nirenberg theorem--see \cref{Rmk::MainRes::EMap}),
then if we did not require that $\Phi$ be holomorphic, \cref{Thm::Results::MainThm} would be a simple consequence of the results in \cite{StovallStreetI,StovallStreetII,StovallStreetII}.

The proof proceeds as follows.  We apply the results from \cite{StovallStreetI,StovallStreetII,StovallStreetII} (see \cref{Section::RealCase}) to yield a candidate chart $\Phi_0$ satisfying
all the conclusions of \cref{Thm::Results::MainThm} without the key property discussed above.  
Then, we apply the main technical result of \cite{StreetNirenberg} to obtain another map $\Phi_1$ such that
 if we set $\Phi=\Phi_0\circ \Phi_1$, $\Phi$ satisfies all the conclusions of \cref{Thm::Results::MainThm}.
 
 As described above, in this paper we construct the map $\Phi$ as a composition of two maps $\Phi=\Phi_0\circ \Phi_1$.  
 When $s_0\in (1,\infty)$, $\Phi_0$ is constructed in \cite{StovallStreetII} as a composition of three maps (one of which was a simple dilation map).
 When $s_0\in (1,\infty)$, $\Phi_1$ was constructed in \cite{StreetNirenberg} as a composition of four maps (two of which were simple dilation maps).
 Thus, if $s_0\in (1,\infty)$, when all the proofs are unraveled, $\Phi$ is a composition of seven maps, three of which are simple dilation maps.
 When $s_0=\omega$, $\Phi$ is considerably simpler.
	
\section{Notation}\label{Section::Notation}
If $f:M\rightarrow N$ is a $C^1$ map between $C^1$ manifolds, we write $df(x):T_xM\rightarrow T_xN$ for the usual differential.  We extend this to be a complex linear map
$df(x):\C T_xM\rightarrow \C T_x N$, where $\C T_x M = T_x M\otimes_{\R} \C$ denotes the complexified tangent space.
Even if the manifold $M$ has additional structure (e.g., in the case of a complex manifold), $df(x)$ is defined in terms of the underlying real manifold structure.

When working on $\R^r\times \C^n$ we will often use coordinates $(t,z)$ where $t=(t_1,\ldots, t_r)\in \R^r$ and $z=(z_1,\ldots, z_n)\in \C^n$.
We write
\begin{equation*}
\diff{t}=
\begin{bmatrix}
\diff{t_1}\\
\diff{t_2}\\
\vdots\\
\diff{t_r}
\end{bmatrix},
\quad
\diff{\zb}
=\begin{bmatrix}
\diff{\zb[1]}\\
\diff{\zb[2]}\\
\vdots\\
\diff{\zb[n]}
\end{bmatrix}.
\end{equation*}
At times we will instead use coordinates $(u,w)$ where $u\in \R^r$ and $w\in \C^n$ and define $\diff{u}$ and $\diff{\wb}$ similarly.

We identify $\R^{r}\times \R^{2n}\cong \R^{r}\times \C^n$ via the map $(t_1,\ldots, t_r, x_1,\ldots,x_{2n})\mapsto (t_1,\ldots, t_r, x_1+ix_{n+1}, \ldots, x_n+ix_{2n})$.
Thus, given a function $G(t,z):\R^r\times \C^n\rightarrow \R^s\times \C^m$, we may also think of $G$ as function $G(t,x)=(G_1(t,x),\ldots, G_{s+2m}(t,x)):\R^{r}\times \R^{2n}\rightarrow \R^{s}\times \R^{2m}$.
For such a function, we write
\begin{equation*}
d_{(t,x)} G =
\begin{bmatrix}
\frac{\partial G_1}{\partial t_1} & \cdots & \frac{\partial G_1}{\partial t_r} & \frac{\partial G_1}{\partial x_1} & \cdots & \frac{\partial G_1}{\partial x_{2n}}\\
\vdots &\ddots &\vdots&\vdots&\ddots& \vdots\\
\vdots &\ddots &\vdots&\vdots&\ddots& \vdots\\
\frac{\partial G_{s+2m}}{\partial t_1} & \cdots & \frac{\partial G_{s+2m}}{\partial t_r} & \frac{\partial G_{s+2m}}{\partial x_1} & \cdots & \frac{\partial G_{s+2m}}{\partial x_{2n}}
\end{bmatrix}.
\end{equation*} 
We write $I_{N\times N}\in \M^{N\times N}$ to denote the $N\times N$ identity matrix, and $0_{a\times b}\in \M^{a\times b}$ to denote the $a\times b$ zero matrix.

	
\section{E-manifolds}\label{Section::Emfld}
The results in this paper simultaneously deal with the setting of real vector fields (on a real manifold) and the setting of complex vector fields (on a complex manifold).
It is more convenient to work in a category of manifolds which contains both real manifolds and complex manifolds as full subcategories.
We define these manifolds here, and call them E-manifolds.\footnote{The manifold structure we discuss here is well-known to experts, but we could not find a name for the category of such manifolds,
and decided to call them E-manifolds for lack of a better name.}
This category of manifolds was also used in \cite{StreetNirenberg}, and we refer the reader to that reference for a more detailed description.

\begin{rmk}
``E'' in the name E-manifolds stands for ``elliptic''.  Indeed, using the terminology of \cite[Definition I.2.3]{TrevesHypoanalyticStructures},
a complex manifold is a manifold endowed with a complex structure, a CR-manifold is a manifold endowed with a CR structure, and an
E-manifold is a manifold endowed with an elliptic structure; see \cref{Thm::EMfld::StructureEquiv}  and \cite{StreetNirenberg} for a more detailed discussion.  Unfortunately, the name ``elliptic manifold'' is already taken by an unrelated concept.
\end{rmk}

\begin{defn}
Let $U_1\subseteq \R^{r_1}\times \C^{n_1}$ and $U_2\subseteq \R^{r_2}\times \C^{n_2}$ be open sets.  We give $\R^{r_1}\times \C^{n_1}$ coordinates $(t,z)$
and $\R^{r_2}\times \C^{n_2}$ coordinates $(u,w)$.  We say a $C^1$ map $f:U_1\rightarrow U_2$ is an E-map if
\begin{equation*}
	df(t,z) \diff{t_k}, df(t,z)\diff{\zb[j]}\in \Span_{\C}\left\{\diff{u_1},\ldots, \diff{u_{r_2}},\diff{\wb[1]},\ldots, \diff{\wb[n_2]}\right\},\quad \forall (t,z)\in U_1, 1\leq k\leq r_1, 1\leq j\leq n_1.
\end{equation*}
For $s\in (1,\infty]\cup \{\omega\}$, we say $f$ is a $\ZygSpacemap{s}$ E-map 
if $f$ is an E-map and $f\in \ZygSpacemap{s}[U_1][\R^{r_2}\times \C^{n_2}]$.
\end{defn}

\begin{rmk}
Suppose $U_1,U_2\subseteq \R^{r}\times \C^{n}$ and $f:U_1\rightarrow U_2$ is an E-map which is also a $C^1$-diffeomorphism.  Then, $f^{-1}:U_2\rightarrow U_1$
is an E-map.
\end{rmk}

\begin{rmk}\label{Rmk::Emfld::Holomorphic}
Note that when $r_1=r_2=0$, if $U_1\subseteq \R^{0}\times \C^{n_1}\cong \C^{n_1}$, $U_2\subseteq \R^{0}\times \C^{n_2}\cong \C^{n_2}$, then
$f:U_1\rightarrow U_2$ is an E-map if and only if it is holomorphic.
\end{rmk}

\begin{defn}
Let $M$ be a Hausdorff, paracompact topological space and fix $n, r\in \N$, $s\in (1,\infty]\cup\{\omega\}$.  We say $\{(\phi_\alpha, V_\alpha) : \alpha\in \sI\}$ (where $\sI$ is some index set)
is a $\ZygSpace{s}$ E-atlas of dimension $(r,n)$ if $\{V_\alpha: \alpha\in \sI\}$ is an open cover for $M$,
$\phi_{\alpha}:V_\alpha\rightarrow U_\alpha$ is a homeomorphism where $U_\alpha\subseteq \R^r\times \C^n$ is open, and
$\phi_\beta \circ \phi_{\alpha}^{-1}: \phi_{\alpha}(V_\beta\cap V_\alpha)\rightarrow U_\beta$ is a $\ZygSpacemap{s}$ E-map, $\forall \alpha,\beta$.
\end{defn}

\begin{defn}\label{Defn::EMfld::EMfld}
A $\ZygSpace{s}$ E-manifold $M$ of dimension $(r,n)$ is a Hausdorff, paracompact topological space $M$ endowed with a $\ZygSpace{s}$ E-atlas
of dimension $(r,n)$.
\end{defn}

\begin{rmk}
On may analogously define $C^m$ E-manifolds in the obvious way.  $C^\infty$ E-manifolds and $\ZygSpace{\infty}$ E-manifolds are the same (because $\ZygSpacemap{\infty}$ is the usual space of smooth functions).
\end{rmk}

\begin{defn}
For $s\in (0,\infty]\cup\{\omega\}$,
let $M$ and $N$ be $\ZygSpace{s+1}$ E-manifolds with $\ZygSpace{s+1}$ E-atlases $\{(\phi_{\alpha},V_{\alpha})\}$ and $\{(\psi_\beta, W_\beta)\}$, respectively.
We say $f:M\rightarrow N$ is a $\ZygSpacemap{s+1}$ E-map if $\psi_\beta\circ f\circ \phi_{\alpha}^{-1}$ is a $\ZygSpacemap{s+1}$ E-map, $\forall \alpha,\beta$.
\end{defn}

\begin{lemma}
For $s\in (0,\infty]\cup\{\omega\}$, let $M_1$, $M_2$, and $M_3$ be $\ZygSpace{s+1}$ E-manifolds and $f_1:M_1\rightarrow M_2$ and $f_2:M_2\rightarrow M_3$
be $\ZygSpacemap{s+1}$ E-maps.  Then, $f_2\circ f_1:M_1\rightarrow M_3$ is a $\ZygSpacemap{s+1}$ E-map.
\end{lemma}
\begin{proof}See \cite[Lemma 4.10]{StreetNirenberg} for a proof of this standard result.\end{proof}

\begin{lemma}
For $s\in (0,\infty]\cup\{\omega\}$, let $M_1$ and $M_2$ be $\ZygSpace{s+1}$ E-manifolds and let $f:M_1\rightarrow M_2$ be a $\ZygSpacemap{s+1}$ E-map which is also a $C^1$ diffeomorphism.
Then, $f^{-1}:M_2\rightarrow M_1$ is a $\ZygSpacemap{s+1}$ E-map.
\end{lemma}
\begin{proof}See \cite[Lemma 4.11]{StreetNirenberg} for a proof of this standard result.\end{proof}

\begin{defn}
Suppose $s\in (0,\infty]\cup\{\omega\}$, and $M_1$ and $M_2$ are $\ZygSpace{s+1}$ E-manifolds.  We say $f:M_1\rightarrow M_2$ is a $\ZygSpace{s+1}$ E-diffeomorphism
if $f:M_1\rightarrow M_2$ is invertible and $f:M_1\rightarrow M_2$ and $f^{-1}:M_2\rightarrow M_1$ are $\ZygSpacemap{s+1}$ E-maps.
\end{defn}

\begin{rmk}\label{Rmk::EMfld::FullSubcategory}
For $s\in (1,\infty]\cup\{\omega\}$ the category of $\ZygSpace{s}$ E-manifolds, whose objects are $\ZygSpace{s}$ E-manifolds and morphisms are $\ZygSpacemap{s}$ E-maps,
contains both $\ZygSpace{s}$ real manifolds and complex manifolds as full subcategories.  The real manifolds of dimension $r$ are those
with E-dimension $(r,0)$, while the complex manifolds of complex dimension $n$ are those with E-dimension $(0,n)$.  That complex manifolds (with morphisms given by holomorphic maps)
embed as a \textit{full} subcategory follows from \cref{Rmk::Emfld::Holomorphic}.
The isomorphisms in the category of $\ZygSpace{s}$ E-manifolds are the $\ZygSpace{s}$ E-diffeomorphisms.
\end{rmk}

\begin{rmk}
Note that open subsets of $\R^r\times \C^n$ are $\ZygSpace{\omega}$
E-manifolds of dimension $(r,n)$, by using the atlas consisting
of one coordinate chart (the identity map).  Henceforth,
we give such sets this E-manifold structure.
\end{rmk}



As mentioned above, $\ZygSpace{s}$ E-manifolds of dimension $(r,0)$ are exactly the $\ZygSpace{s}$ manifolds of dimension $r$, in the usual sense
(in particular, one may take \cref{Defn::EMfld::EMfld} in the case $n=0$ as the definition of a $\ZygSpace{s}$ manifolds of dimension $r$).
There is a natural forgetful functor taking $\ZygSpace{s}$ E-manifolds of dimension $(r,n)$ to $\ZygSpace{s}$ manifolds of dimension $2n+r$.
Thus, one may define any of the usual objects from manifolds on E-manifolds.  For example, we have the following standard definitions on $\ZygSpace{s}$ manifolds,
and therefore on $\ZygSpace{s}$ E-manifolds.\footnote{The following standard definitions can all be found in \cite{TrevesHypoanalyticStructures} in the case $s=\infty$,
and in \cite{StreetNirenberg} for finite levels of smoothness.}


\begin{defn}
For $s\in (0,\infty]\cup\{\omega\}$ let $M$ be a $\ZygSpace{s+1}$ manifold of dimension $r$, with $\ZygSpace{s+1}$ atlas $\{(\phi_\alpha, V_\alpha)\}$;
here $\phi_{\alpha}:V_{\alpha}\rightarrow U_\alpha$ is a $\ZygSpace{s+1}$ diffeomorphism and $U_\alpha\subseteq \R^{r}$ is open.
We say a complex vector field $X$ on $M$ is a $\ZygSpace{s}$ vector field if $(\phi_\alpha)_{*} X \in \ZygSpacemap{s}[U_\alpha][\C^{r}]$, $\forall \alpha$.
\end{defn}

\begin{defn}
For $s\in (0,\infty]\cup \{\omega\}$, a $\ZygSpace{s}$ sub-bundle $\LVS$ of $\C TM$ of rank $m\in \N$ is a disjoint union
\begin{equation*}
	\LVS = \bigcup_{\zeta\in M} \LVS_{\zeta}\subseteq \C TM
\end{equation*}
such that:
\begin{itemize}
	\item $\forall \zeta\in M$, $\LVS_{\zeta}$ is an $m$-dimensional vector subspace of $\C T_{\zeta} M$.
	\item $\forall \zeta_0\in M$, there exists an open neighborhood $U\subseteq M$ of $\zeta_0$ and a finite collection of complex
	$\ZygSpace{s}$ vector fields $L_1,\ldots, L_K$ on $U$, such that $\forall \zeta\in U$,
	\begin{equation*}
		\Span_{\C} \{ L_1(\zeta),\ldots, L_K(\zeta)\} = \LVS_{\zeta}.
	\end{equation*}
\end{itemize}
\end{defn}

\begin{defn}
For a $\ZygSpace{s}$ sub-bundle $\LVS$ of $\C TM$, we define $\LVSb$ by $\LVSb[\zeta]=\{ \zb : z\in \LVS_{\zeta}\}$.
It is easy to see that $\LVSb$ is a $\ZygSpace{s}$ sub-bundle of $\C TM$.
\end{defn}

\begin{defn}
Let $W\subseteq M$ be open, $L$ a complex vector field on $W$, and $\LVS$ a $\ZygSpace{s}$ sub-bundle of $\C TM$.  We say $L$ is a section of $\LVS$
over $W$ if $\forall \zeta\in W$, $L(\zeta)\in \LVS_{\zeta}$.  We say $L$ is a $\ZygSpace{s}$ section of $\LVS$ over $W$ if $L$ is a section of $\LVS$
over $W$ and $L$ is a $\ZygSpace{s}$ complex vector field on $W$.
\end{defn}

\begin{defn}
Let $\LVS$ be a $\ZygSpace{s+1}$ sub-bundle of $\C TM$.  We say $\LVS$ is a $\ZygSpace{s+1}$ formally integrable structure if the following holds.
For all $W\subseteq M$ open, and all $\ZygSpace{s+1}$ sections $L_1$ and $L_2$ of $\LVS$ over $W$, we have $[L_1,L_2]$ is a section of $\LVS$ over $W$.
\end{defn}

\begin{defn}
Let $\LVS$ be a $\ZygSpace{s+1}$ formally integrable structure on $M$.  We say $\LVS$ is a $\ZygSpace{s+1}$ elliptic structure if
$\LVS_{\zeta}+\LVSb[\zeta]=\C T_\zeta M$, $\forall \zeta\in M$.
\end{defn}


For $s\in (0,\infty]\cup\{\omega\}$, on a $\ZygSpace{s+2}$ E-manifold of dimension $(r,n)$, there is a naturally associated $\ZygSpace{s+1}$ elliptic structure on $M$
defined as follows.
Let $(\phi_\alpha,V_\alpha)$ be an E-atlas for $M$.  For $\zeta\in M$ let $\zeta\in V_\alpha$ for some $\alpha$.
We set:
\begin{equation*}
\LVS_{\zeta}:= \Span_{\C}\left\{ d\Phi_{\alpha}^{-1}(\Phi_\alpha(\zeta)) \diff{t_1},\ldots, d\Phi_{\alpha}^{-1}(\Phi_\alpha(\zeta)) \diff{t_r},d\Phi_{\alpha}^{-1}(\Phi_\alpha(\zeta)) \diff{\zb[1]},\ldots, d\Phi_{\alpha}^{-1}(\Phi_\alpha(\zeta)) \diff{\zb[n]} \right\}.
\end{equation*}
It is straightforward to check that $\LVS_{\zeta}\subseteq \C T_{x_0}M$ is well-defined\footnote{I.e., $\LVS_{\zeta}$ does not depend on which $\alpha$ we pick with $\zeta\in V_\alpha$.}
and $\LVS= \bigcup_{\zeta\in M} \LVS_{\zeta}$ is a $\ZygSpace{s+1}$ elliptic structure on $M$.
As remarked above, an E-manifold of dimension $(0,n)$ is a complex manifold; in this case $\LVS$ equals $T^{0,1}M$.


\begin{defn}\label{Defn::EMfld::sL}
We call $\LVS$ the elliptic structure associated to the E-manifold $M$.
\end{defn}

\begin{lemma}\label{Lemma::EMfld::RecongnizeEMap}
Suppose $M$ and $\Mh$ are $\ZygSpace{s}$ E-manifolds with associated elliptic structures $\LVS$ and $\LVSh$.
Then a $\ZygSpacemap{s}$ map $f:M\rightarrow \Mh$ is a $\ZygSpacemap{s}$ E-map if and only if
$df(x) \LVS_x \subseteq \LVSh_{f(x)}$, $\forall x\in M$.
\end{lemma}
\begin{proof}This follows immediately from the definitions.\end{proof}

It turns out that the elliptic structure $\LVS$ associated to the E-manifold $M$ uniquely determines the E-manifold structure as the following theorem shows.

\begin{thm}\label{Thm::EMfld::StructureEquiv}
Let $s\in (0,\infty]\cup\{\omega\}$ and let $M$ be a $\ZygSpace{s+2}$ manifold.  For each $\zeta\in M$, let $\LVS_{\zeta}$ be a vector subspace
of $\C T_{\zeta} M$, and let $\LVS=\bigcup_{\zeta\in M} \LVS_{\zeta}$.  The following are equivalent:
\begin{enumerate}[(i)]
	\item\label{Item::EMfld::ExistsStructure} There is a $\ZygSpace{s+2}$ E-manifold structure $M$, compatible with its $\ZygSpace{s+2}$ structure, such that $\LVS$ is the $\ZygSpace{s+1}$ elliptic structure associated to $M$.
	\item $\LVS$ is a $\ZygSpace{s+1}$ elliptic structure.
\end{enumerate}
Moreover, under these conditions, the E-manifold structure given in \cref{Item::EMfld::ExistsStructure} is unique in the sense that of $M$ is given another $\ZygSpace{s+2}$
E-manifold structure, compatible with its $\ZygSpace{s+2}$ structure, with respect to which $\LVS$ is the associated elliptic sub-bundle, then the identity map $M\rightarrow M$
is a $\ZygSpace{s+2}$ E-diffeomorphism between these two $\ZygSpace{s+2}$ E-manifold structures on $M$.
\end{thm}
\begin{proof}[Comments on the proof]
The case $s=\omega$ of this result is classical.  The case $s=\infty$ is due to Nirenberg \cite{NirenbergAComplexFrobeniusTheorem}.
In the special case of complex manifolds (i.e., E-manifolds of dimension $(0,n)$) this is the Newlander-Nirenberg Theorem \cite{NewlanderNirenbergComplexAnalyticCoordiantesInAlmostComplexManifolds}
with sharp regularity, as proved by Malgrange \cite{MalgrangeSurLIntegbrabilite}.
The full result can be found in \cite[Theorem 4.18]{StreetNirenberg}.
\end{proof}

	\subsection{CR manifolds}\label{Section::CRMfld}
There is another, related, category of manifolds of substantial interest, where the original scaling maps
of Nagel, Stein, and Wainger \cite{NagelSteinWaingerBallsAndMetrics} have been widely used:  CR manifolds.  \Cref{Thm::EMfld::StructureEquiv} characterizes E-manifolds as manifolds
endowed with an elliptic structure.  CR manifolds are defined in a similar way.

\begin{defn}
Let $\WVS$ be a $\ZygSpace{s+1}$ formally integrable structure on $M$.  We say $\WVS$ is a $\ZygSpace{s+1}$ CR structure if $\WVS_{\zeta}\cap \WVSb[\zeta]=\{0\}$, $\forall \zeta\in M$.
\end{defn}

\begin{defn}
A $\ZygSpace{s+2}$ CR manifold $M$ is a $\ZygSpace{s+2}$ manifold $M$ endowed with at $\ZygSpace{s+1}$ CR structure on $M$.
\end{defn}

Notice that a CR structure, $\WVS$, is an elliptic structure if and only if $\WVS_{\zeta}\oplus \WVSb[\zeta]=\C T_\zeta M$, $\forall \zeta\in M$.
This is precisely the definition of a complex structure.
There does not seem to be a natural way, given an arbitrary E-manifold, to see it as a CR manifold.  Nor is there a natural way, given an arbitrary CR manifold, to see it as an E-manifold.
Nevertheless, many of the classical examples of CR manifolds can be naturally given the structure of an E-manifold.
Indeed, given a CR structure $\WVS$, it is often the case that there is another sub-bundle, $\TVS$, of $\C TM$ such that
$\WVS\oplus \TVS$ is an elliptic structure.

The simplest example of this is the three dimensional Heisenberg group $\Ho$.  As a manifold, $\Ho$ is diffeomorphic to $\C\times \R$ and we give it coordinates
$(z,t)\in \C\times \R$.  We give $\Ho$ a CR structure by setting $\WVS_{(z,t)}:=\Span_{\C} \mleft\{ \diff{\zb}-iz \diff{t}\mright\}$.
By setting $\TVS_{(z,t)}:= \Span_{\C}\mleft\{ \diff{t}\mright\}$, we have $\WVS\oplus \TVS$ is an elliptic structure on $\Ho$.
In many examples of CR manifolds one has a similar setting: there are local coordinates $(z_1,\ldots, z_n, t_1,\ldots, t_r)\in \C^n\times \R^r$
such that the CR structure, $\WVS$, is contained in the span of $\diff{\zb[1]},\ldots, \diff{\zb[n]}, \diff{t_1},\ldots, \diff{t_r}$
in such at way that if one takes $\TVS_{(z,t)}:= \Span_{\C}\mleft\{ \diff{t_1},\ldots, \diff{t_r}\mright\}$, then
$\WVS\oplus \TVS$ is an elliptic structure.
See  \cref{Section::ExtremalBasis::CR} for a discussion of one way the results of this paper can be applied to CR manifolds.




\begin{rmk}
A major distinction between CR structures and elliptic structures is that elliptic structures of dimension $(r,n)$ have a single canonical example.
Indeed, $\R^r\times \C^n$ is naturally an E-manifold with associated elliptic structure $\LVSh$ given by
$$\LVSh_{\zeta}=\Span_{\C}\mleft\{ \diff{\zb[1]},\ldots, \diff{\zb[n]},\diff{t_1},\ldots, \diff{t_r}\mright\},\quad \forall \zeta\in \R^r\times \C^n.$$
\Cref{Thm::EMfld::StructureEquiv} shows that given any elliptic structure $\LVS$, there is a local coordinate system in which $\LVS$
is given by $\LVSh$, where $n+r=\dim \LVS_{\zeta}$ and $r=\dim \mleft(\LVS_\zeta\cap \LVSb[\zeta]\mright)$ (here, $n$ and $r$ are constant in $\zeta$--see \cite[Section 3]{StreetNirenberg}).
  \Cref{Thm::Results::MainThm} can be thought of as a quantitative, diffeomorphic invariant version of a coordinate system
which sees an elliptic structure as this canonical example.  Since there is no similar canonical example of a CR structure, it is not immediately clear what an analog of \cref{Thm::Results::MainThm}
would be for general CR structures.
\end{rmk}

\section{Corollaries Revisited}\label{Section::Res::CorE}
In this section, we generalize the results from \cref{Section::CorRes} using the language of E-manifolds.  This unifies the complex
and real settings.

	\subsection{Optimal Smoothness}\label{Section::Res::OptE}
Let $X_1,\ldots, X_q$ be real $C^1$ vector fields on a connected $C^2$ manifold $M$ and let $L_1,\ldots, L_m$ be complex $C^1$ vector fields on $M$.
For $x\in M$ set
\begin{equation}\label{Eqn::QualE::DefineLVS}
\LVS_x:=\Span_{\C}\{X_1(x),\ldots, X_q(x), L_1(x),\ldots, L_m(x)\},\quad \XVS_x:=\Span_{\C}\{X_1(x),\ldots, X_q(x)\}.
\end{equation}
We assume:
\begin{itemize}
\item $\LVS_x+\LVSb[x]=\C T_xM$, $\forall x\in M$.
\item $\XVS_x=\LVS_x\cap \LVSb[x]$, $\forall x\in M$.
\end{itemize}

\begin{thm}[The Local Theorem]\label{Thm::QualE::LocalThm}
Fix $x_0\in M$, $s\in (1,\infty]\cup \{\omega\}$, and set $r:=\dim \XVS_{x_0}$ and $n+r:=\dim \LVS_{x_0}$.  The following three conditions are equivalent:
\begin{enumerate}[(i)]
\item\label{Item::QualE::Local::Diffeo} There exists an open neighborhood $V\subseteq M$ of $x_0$ and a $C^2$ diffeomorphism $\Phi:U\rightarrow V$, where $U\subseteq \R^r\times \C^n$ is open,
such that $\forall(t,z)\in U$, $1\leq k\leq q$,
$1\leq j\leq m$,
$$\Phi^{*}X_k(t,z)\in \Span_{\R}\left\{\diff{t_1},\ldots, \diff{t_r}\right\},\quad \Phi^{*}L_j(t,z)\in \Span_{\C}\left\{\diff{t_1},\ldots, \diff{t_r},\diff{\zb[1]},\ldots, \diff{\zb[n]}\right\},$$
and $\Phi^{*}X_k\in \ZygSpace{s+1}[U][\R^r]$, $\Phi^{*}L_j\in \ZygSpace{s+1}[U][\C^{r+n}]$.

\item\label{Item::QualE::Local::Basis} Reorder $X_1,\ldots, X_q$ so that $X_1(x_0),\ldots, X_r(x_0)$ are linearly independent, and reorder $L_1,\ldots, L_m$ so that $L_1(x_0),\ldots, L_n(x_0),X_1(x_0),\ldots, X_r(x_0)$
are linearly independent.  Let $\Zh_1,\ldots, \Zh_{n+r}$ denote the list $X_1,\ldots, X_r,L_1,\ldots, L_n$, and let $Y_1,\ldots, Y_{m+q-(r+n)}$ denote the list $X_{r+1},\ldots, X_q,L_{n+1},\ldots, L_m$.
There exists an open neighborhood $V\subseteq M$ of $x_0$ such that:
\begin{itemize}
\item $[\Zh_j,\Zh_k]=\sum_{l=1}^{n+r} \ch_{j,k}^{1,l} \Zh_l$, and $[\Zh_j,\Zhb[k]]=\sum_{l=1}^{n+r} \ch_{j,k}^{2,l} \Zh_l + \sum_{l=1}^{n+r} \ch_{j,k}^{3,l} \Zhb[l]$, where
$\ch_{j,k}^{a,l}\in \ZygXSpace{X,L}{s}[V]$, $1\leq j,k,l\leq n+r$, $1\leq a\leq 3$.
\item $Y_j=\sum_{l=1}^{n+r} b_j^l \Zh_l$, where $b_j^l\in \ZygXSpace{X,L}{s+1}[V]$, $1\leq j\leq m+q-(r+n)$, $1\leq l\leq n+r$.
\end{itemize}
Furthermore, the map $x\mapsto \dim \LVS_x$, $V\rightarrow \N$ is constant in $x$.

\item\label{Item::QualE::Local::Commute} Let $Z_1,\ldots, Z_{m+q}$ denote the list $X_1,\ldots, X_q, L_1,\ldots, L_m$.  There exists a neighborhood $V\subseteq M$ of $x_0$ such that
$[Z_j,Z_k]=\sum_{l=1}^{m+q} c_{j,k}^{1,l} Z_l$ and $[Z_j, \Zb[k]]=\sum_{l=1}^{m+q} c_{j,k}^{2,l} Z_l  + \sum_{l=1}^{m+q} c_{j,k}^{3,l} \Zb[l]$, where
$c_{j,k}^{a,l}\in \ZygXSpace{X,L}{s}[V]$, $1\leq a\leq 3$, $1\leq j,k,l\leq m+q$.
Furthermore, the map $x\mapsto \dim \LVS_x$, $V\rightarrow \N$ is constant in $x$.

\end{enumerate}
\end{thm}

\begin{thm}[The Global Theorem]\label{Thm::QualE::GlobalThm}
For $s\in (1,\infty]\cup \{\omega\}$ the following two conditions are equivalent:
\begin{enumerate}[label=(\roman*),series=qualEglobaltheoremenumeration]
\item\label{Item::QualE::EMfld} There exists a $\ZygSpace{s+2}$ E-manifold structure on $M$, compatible with its $C^2$ structure, such that $X_1,\ldots, X_q,L_1,\ldots, L_m$ are  $\ZygSpace{s+1}$
vector fields on $M$ and $\LVS$ (as defined in \cref{Eqn::QualE::DefineLVS})  is the
associated elliptic structure (see \cref{Defn::EMfld::sL}).

\item\label{Item::QualE::ThreeEquiv} For each $x_0\in M$, any of the three equivalent conditions from \cref{Thm::QualE::LocalThm} hold for this choice of $x_0$.
\end{enumerate}
Furthermore, under these conditions, the $\ZygSpace{s+2}$ E-manifold structure in \cref{Item::QualE::EMfld} is unique, in the sense that if $M$ has another
$\ZygSpace{s+2}$ E-manifold structure satisfying the conclusions of \cref{Item::QualE::EMfld}, then the identity map $M\rightarrow M$ is
a $\ZygSpace{s+2}$ E-diffeomorphism
between these two E-manifold structures.  Finally, when $s\in (1,\infty]$,
 there is a third equivalent condition:
\begin{enumerate}[resume*=qualEglobaltheoremenumeration]
\item\label{Item::QualE::Commute} Let $Z_1,\ldots, Z_{m+q}$ denote the list $X_1,\ldots, X_q,L_1,\ldots, L_m$.  Then,
$[Z_j,Z_k]=\sum_{l=1}^{m+q} c_{j,k}^{1,l} Z_l$ and $[Z_j,\Zb[k]]=\sum_{l=1}^{m+q} c_{j,k}^{2,l} Z_l + \sum_{l=1}^{m+q} c_{j,k}^{3,l} \Zb[l]$,
where $\forall x\in M$, there exists an open neighborhood $V\subseteq M$ of $x$ such that
$c_{j,k}^{a,l}\big|_V\in \ZygXSpace{X,L}{s}[V]$, $1\leq a \leq 3$, $1\leq j,k,l\leq m+q$.
Furthermore, the map $x\mapsto \dim \LVS_x$, $M\rightarrow \N$ is constant.
\end{enumerate}
\end{thm}

\begin{rmk}
For a discussion of results like \cref{Thm::QualE::LocalThm,Thm::QualE::GlobalThm} using the easier to understand H\"older spaces, see
\cref{Section::Holder}.
\end{rmk} 
	
	\subsection{Sub-E geometry}\label{Section::Res::SubE}
Let $M$ be a connected $C^\infty$ E-manifold of dimension $(r,n)$ and let $\LVS$ be the associated elliptic structure.  For $x\in M$,
set $\XVS_x:=\LVS_x\cap \LVSb[x]$, so that $r=\dim \XVS_x$ and $n+r=\dim \LVS_x$, $\forall x\in M$.
Fix a strictly positive $C^\infty$ density $\nu$ on $M$.\footnote{The results that follow are local and do not depend on the choice of density.}
Suppose $X_1,\ldots, X_q$ are $C^\infty$ real vector fields on $M$ and $L_1,\ldots, L_m$ are $C^\infty$ complex vector fields on $M$
such that
$\XVS_x = \Span_{\C}\{X_1(x),\ldots, X_q(x)\}$ and $\LVS_x=\Span_{\C} \{X_1(x),\ldots, X_q(x), L_1(x),\ldots, L_m(x)\}$, $\forall x\in M$.

To each $X_k$, we assign a formal degree $\beta_k\in [1,\infty)$, and to each $L_j$ we assign a formal degree $\beta_{j+q}\in [1,\infty)$.
We let $Z_1,\ldots, Z_{m+q}$ denote the list $X_1,\ldots, X_q,L_1,\ldots, L_m$, so that $Z_j$ has assigned formal degree $\beta_j$.

We assume:
\begin{equation}\label{Eqn::SubE::AssumeNSW}
[Z_j,Z_k]=\sum_{\beta_l\leq \beta_j+\beta_k} c_{j,k}^{1,l}Z_l, \quad [Z_j, \Zb[k]]=\sum_{\beta_l\leq \beta_j+\beta_k} c_{j,k}^{2,l} Z_l + \sum_{\beta_l\leq \beta_j+\beta_k} c_{j,k}^{3,l} \Zb[l],\quad c_{j,k}^{a,l}\in \CjSpace{\infty}[M].
\end{equation}
For $\delta\in (0,1]$ write $\delta^{\beta} X$ for the list $\delta^{\beta_1} X_1,\ldots, \delta^{\beta_q} X_q$ and write $\delta^{\beta} L= \delta^{\beta_{q+1}} L_1,\ldots, \delta^{\beta_{q+m}} L_m$.
Using the notation from \cref{Section::MainResult} it makes sense to write, for $K\in \sI(r_1,q)$, $J\in \sI(n_1,m)$, $\left(\bigwedge (\delta^{\beta} X)_K\right)\bigwedge\left(\bigwedge (\delta^\beta L)_J\right)$.
We assume:  $\forall \Compact\Subset M$ compact, $\exists \zeta\in(0,1]$ such that $\forall x\in \Compact$, $\delta\in (0,1]$,
$\exists K_0(x,\delta)\in \sI(r,q), J_0(x,\delta)\in \sI(n,m)$ such that
\begin{equation}\label{Eqn::ResSubE::ExistsMaximal}
\sup_{\substack{x\in \Compact \\ \delta\in (0,1]}}\max_{\substack{K\in \sI(r_1,q), J\in \sI(n_1,m)\\ r_1+n_1=r+n}} \left|
\frac{ \left(\bigwedge (\delta^{\beta} X(x))_K \right)\bigwedge \left( (\delta^{\beta}L(x))_J \right)}
{\left(\bigwedge (\delta^{\beta} X(x))_{K_0(x,\delta)} \right)\bigwedge \left( (\delta^{\beta}L(x))_{J_0(x,\delta)} \right)}
\right|\leq \zeta^{-1}.
\end{equation}

\begin{rmk}\label{Rmk::SubE::GeneralsizesSubHR}
The existence of $K_0(x,\delta)$, $J_0(x,\delta)$, and $\zeta$ as in \cref{Eqn::ResSubE::ExistsMaximal} does not follow from the other hypotheses.  However, it is immediate to see that if $r=0$ or $n=0$,
one may always find $J_0(x,\delta)$ and $K_0(x,\delta)$ so that \cref{Eqn::ResSubE::ExistsMaximal} holds with $\zeta=1$.
This accounts for the two most important special cases:  the ones in \cref{Section::Results::NSW,Section::ResGeom::SubHerm}.
\end{rmk}

Under these hypotheses, we will study two metrics on $M$ (and
show these two metrics are equivalent on compact sets).
The first metric is a standard sub-Riemannian metric and we will define it in two different ways, denoted by $\rho_S$ and $\rho_F$.  We will
show that $\rho_S=\rho_F$.  Both of the definitions $\rho_S$
and $\rho_F$ are defined extrinsically:  they are defined
by using the underlying manifold structure on $M$ using
maps which are not necessarily E-maps.  The second metric, $\rho_H$,
has a definition which is similar to that of $\rho_F$, but it is defined
intrinsically on $M$:  it is defined entirely within the category of
E-manifolds.

For $x\in M$, $\delta>0$ set $B_S(x,\delta):=B_{\delta^{\beta}X, \delta^{\beta}L}(x,1)$ (where the later ball is defined in \cref{Eqn::FuncComplex::Ball}) and set $\rho_S(x,y):=\inf\{\delta>0 : y\in B_S(x,\delta)\}$.

Let $(W_1,d_1),\ldots, (W_{2m+q}, d_{2m+q})$ denote the list of vector fields with formal degrees
$$(X_1, \beta_1),\ldots, (X_q,\beta_q), (2\Real(L_1),\beta_{q+1}),\ldots, (2\Real(L_m),\beta_{q+m}), (2\Imag(L_1),\beta_{q+1}),\ldots, (2\Imag(L_m),\beta_{q+m}).$$
We say $\rho_F(x,y)<\delta$ if and only if $\exists K\in \N$, $C^\infty$ functions $f_1,\ldots, f_K:B_{\R}(1/2)\rightarrow M$, and
$\delta_1,\ldots, \delta_K>0$ with $\sum_{j=1}^K \delta_j \leq \delta$, such that:
\begin{itemize}
\item $f_j'(t) = \sum_{l=1}^{2m+q} s_j^l(t) \delta_j^{d_l} W_l(f_j(t))$, with $\Norm{\sum_l |s_j^l|^2}[L^\infty(B_{\R}(1/2))]<1$.
\item $f_j(B_{\R}(1/2))\cap f_{j+1}(B_{\R}(1/2))\ne \emptyset$, $1\leq j\leq K-1$.
\item $x\in f_1(B_{\R}(1/2))$, $y\in f_{K}(B_{\R}(1/2))$.
\end{itemize}
Set $B_F(x,\delta):=\{y\in M:\rho_F(x,y)<\delta\}$.

Finally, we define $\rho_H$.  We say $\rho_H(x,y)<\delta$ if and only if
$\exists K\in \N$, $C^\infty$ E-maps $f_1,\ldots, f_K:B_{\R\times \C}(1/2)\rightarrow M$,
and $\delta_1,\ldots, \delta_K$ with $\sum_{j=1}^K \delta_j\leq \delta$, such that:
\begin{enumerate}[(1)]
\item\label{Item::SubE::DefnSj} Because $f_j$ is an E-map, we may write
\begin{equation*}
df_j(t,z) \diff{t} = \sum_{k=1}^q s_{j,1}^k(t,z) \delta_j^{\beta_k} X_k(f_j(t,z)) + \sum_{l=1}^m s_{j,1}^{l+q}(t,z) \delta_j^{\beta_{l+q}} \frac{2}{\sqrt{2}} L_l(f_j(t,z)),
\end{equation*}
\begin{equation*}
df_j(t,z)\frac{2}{\sqrt{2}} \diff{\zb} = \sum_{k=1}^q s_{j,2}^k(t,z) \delta_j^{\beta_k} X_k(f_j(t,z)) + \sum_{l=1}^m s_{j,2}^{l+q}(t,z) \delta_j^{\beta_{l+q}} \frac{2}{\sqrt{2}} L_l(f_j(t,z)).
\end{equation*}
The choice of $s_j$'s is not necessarily unique.  Let $S_j(t,z)$ denote the $(q+2m)\times 3$ matrix such that the $(l,a)$ component of $S_j(t,z)$ is given by
\begin{equation*}
\begin{cases}
s_{j,a}^l(t,z), & 0\leq l\leq m+q, a=1,2\\
0, & m+q+1\leq l\leq 2m+q, a=1,2\\
\overline{s_{j,a}^l(t,z)}, & 0\leq l\leq q\text{ or }m+q+1\leq l\leq 2m+q, a=3\\
0, & q+1\leq l\leq m+q, a=3.
\end{cases}
\end{equation*}
In particular, $S_j(t,z)$ is a matrix representation of $df_j(t,z)$ thought of as taking the basis $\diff{t}, \frac{2}{\sqrt{2}} \diff{\zb}, \frac{2}{\sqrt{2}}\diff{z}$ to the spanning
set
$$\delta_j^{\beta_1} X_1,\ldots, \delta_j^{\beta_q} X_q, \delta_j^{\beta_{q+1}}\frac{2}{\sqrt{2}} L_1,\ldots, \delta_j^{\beta_{q+m}} \frac{2}{\sqrt{2}} L_m,\delta_j^{\beta_{q+1}}\frac{2}{\sqrt{2}} \Lb[1],\ldots, \delta_j^{\beta_{q+m}} \frac{2}{\sqrt{2}} \Lb[m].$$
We assume
\begin{equation*}
\Norm{S_j}[L^\infty (B_{\R\times \C}(1/2); \M^{(q+2m)\times 3})]<1.
\end{equation*}
The choice of $S_j$ may not be unique\footnote{The choice of $S_j$ is not unique if $m+q>n+r$.}, and we only ask for the existence of such an $S_j$.
\item $f_j(B_{\R\times \C}(1/2))\cap f_{j+1}(B_{\R\times \C}(1/2))\ne \emptyset$, $1\leq j\leq K-1$.
\item $x\in f_1(B_{\R\times \C}(1/2))$, $y\in f_K(B_{\R\times \C}(1/2))$.
\end{enumerate}
Set $B_H(x,\delta):=\{y\in M:\rho_H(x,y)<\delta\}$.

\begin{rmk}\label{Rmk::SubH::Shj}
A consequence of \cref{Item::SubE::DefnSj} is the following.  We identify $\R\times \C$ with $\R^3$ in the usual way.  Let $\Sh_j(t,x_1,x_2)$ be a $(2m+q)\times 3$ matrix representation of $df_j(t,x_1,x_2)$ thought
as taking the basis $\diff{t},\diff{x_1},\diff{x_2}$ to the spanning set $\delta_j^{d_1}W_1(f_j(t,x_1,x_2),\ldots, \delta_j^{d_{2m+q}} W_{2m+q}(f_j(t,x_1,x_2))$.
Then if \cref{Item::SubE::DefnSj} holds we may choose $\Sh_j$ so that
\begin{equation}\label{Eqn::SubH::NormShj}
\Norm{\Sh_j}[L^\infty (B_{\R^3}(1/2); \M^{(q+2m)\times 3})]<1.
\end{equation}
\end{rmk}

Define, for $x\in M$, $\delta>0$,
\begin{equation*}
\Lambda(x,\delta):=\max_{j_1,\ldots, j_{2n+r}\in \{1,\ldots, 2m+q\}} \nu(x)(\delta^{d_{j_1}}W_{j_1}(x),\ldots, \delta^{d_{j_{2n+r}}} W_{j_{2n+r}}(x)).
\end{equation*}

\begin{thm}\label{Thm::SubEGeom}
\begin{enumerate}[label=(\alph*),series=subEtheoremenumeration]
\item\label{Item::SubE::EasyMetrics} $\forall x,y\in M$, $\rho_S(x,y)= \rho_F(x,y)\leq \rho_H(x,y)$.
\end{enumerate}
Fix a compact set $\Compact\Subset M$.  We write $A\lesssim B$ for $A\leq C B$, where $C$ is a positive constant which can be chosen independent of $x,y\in \Compact$, $\delta>0$.
We write $A\approx B$ for $A\lesssim B$ and $B\lesssim A$.  There exists $\delta_1\approx 1$ such that:
\begin{enumerate}[resume*=subEtheoremenumeration]
\item\label{Item::SubEThm::HIsSmallerThanS} $\rho_H(x,y)\lesssim \rho_S(x,y)$, and therefore $\rho_S$ and $\rho_H$ are equivalent on compact sets.
\item\label{Item::SubEThm::EstimateVols} $\nu(B_S(x,\delta))
\approx \nu(B_H(x,\delta))\approx \Lambda(x,\delta)$, $x\in \Compact$, $\delta\in (0,\delta_1]$.
\item\label{Item::SubEThm::Doubling} $\nu(B_S(x,2\delta))\lesssim \nu(B_S(x,\delta))$, $\forall x\in \Compact$, $\delta\in (0,\delta_1/2]$; the same holds with $B_S$ replaced by $B_H$.\footnote{This is the key estimate that shows that the
balls $B_S(x,\delta)$, when paired with the density $\nu$, locally give a space of homogeneous type.}
\end{enumerate}
For each $x\in \Compact$, $\delta\in (0,1]$, there exists a $C^\infty$ E-map $\Phi_{x,\delta}:B_{\R^r\times \C^n}(1)\rightarrow B_S(x,\delta)$ such that
\begin{enumerate}[resume*=subEtheoremenumeration]
\item\label{Item::SubEThm::PhiDiffeo} $\Phi_{x,\delta}(B_{\R^r\times \C^n}(1))\subseteq M$ is open and $\Phi_{x,\delta}:B_{\R^r\times \C^n}(1)\rightarrow \Phi(B_{\R^r\times \C^n}(1))$ is a $C^\infty$ diffeomorphism.
\item\label{Item::SubEThm::Esth} $\Phi_{x,\delta}^{*} \nu = h_{x,\delta} \LebDensity$, where $\LebDensity$ denotes the usual Lebesgue density on $\R^r\times \C^n$, $h_{x,\delta}\in \CjSpace{\infty}[B_{\R^r\times \C^n}(1)]$,
and $\CjNorm{h_{x,\delta}}{m}[B_{\R^r\times \C^n}(1)]\lesssim \Lambda(x,\delta)$, $\forall m$ (where the implicit constant may depend on $m$).  Also, $h_{x,\delta}(t,z)\approx \Lambda(x,\delta)$, $\forall (t,z)\in B_{\R^r\times \C^n}(1)$,
where the implicit constant does not depend on $x\in \Compact$, $\delta\in (0,1]$, or $(t,z)\in B_{\R^r\times \C^n}(1)$.
\end{enumerate}
Let $\Zh_j^{x,\delta}:=\Phi_{x,\delta}^{*} \delta^{\beta_j} Z_j$, so that $\Zh_j^{x,\delta}$ is a $C^\infty$ vector field on $B_{\R^r\times \C^n}(1)$.
\begin{enumerate}[resume*=subEtheoremenumeration]
\item\label{Item::SubEThm::ZhIsE} $\Zh_j^{x,\delta}(t,z)\in \Span_{\C} \left\{\diff{t_1},\ldots, \diff{t_r},\diff{\zb[1]},\ldots, \diff{\zb[n]}\right\}$, $\forall (t,z)\in B_{\R^r\times \C^n}(1)$.
\end{enumerate}
In light of \cref{Item::SubEThm::ZhIsE}, we may think of $\Zh_j^{x,\delta}$ as a map $B_{\R^r\times \C^n}(1)\rightarrow \C^{r+n}$, and we henceforth do this.
\begin{enumerate}[resume*=subEtheoremenumeration]
\item\label{Item::SubEThm::ZsSpan} $\Zh_1^{x,\delta}(t,z),\ldots, \Zh_{m+q}^{x,\delta}(t,z)$ span  $\Span_{\C} \left\{\diff{t_1},\ldots, \diff{t_r},\diff{\zb[1]},\ldots, \diff{\zb[n]}\right\}$ uniformly in $t,z,x,\delta$ in the sense that
\begin{equation*}
\max_{j_1,\ldots, j_{n+r}\in \{1,\ldots, m+q\}} \inf_{(t,z)\in B_{\R^r\times \C^n}(1)} \left|  \det\left( \Zh_{j_1}^{x,\delta}(t,z) | \cdots| \Zh_{j_{n+r}}^{x,\delta}(t,z) \right) \right|\approx 1, \quad x\in \Compact, \delta\in (0,1].
\end{equation*}
In fact, for $x\in \Compact$, $\delta\in (0,1]$,
\begin{equation*}
\begin{split}
&\max_{\substack{k_1,\ldots, k_r\in \{1,\ldots, q\} \\ j_1,\ldots, j_n\in \{1,\ldots, m\}}} \inf_{(t,z)\in B_{\R^r\times \C^n}(1)}
\\&\left|
\det\left( \Phi_{x,\delta}^{*} \delta^{\beta_{k_1}} X_{k_1}(t,z) | \cdots | \Phi_{x,\delta}^{*} \delta^{\beta_{k_r}} X_{k_r}(t,z) |  \Phi_{x,\delta}^{*} \delta^{\beta_{j_1+q}} L_{j_1}(t,z) | \cdots | \Phi_{x,\delta}^{*} \delta^{\beta_{j_n+q}} L_{j_n}(t,z)   \right)
\right|
\approx 1.
\end{split}
\end{equation*}
\item\label{Item::SubEThm::ZsSmooth} $\CjNorm{\Zh_j^{x,\delta}}{k}[B_{\R^r\times \C^n}(1)][\C^{r+n}]\lesssim 1$, $\forall x\in \Compact$, $\delta\in (0,1]$ (where the implicit constant may depend on $k\in \N$).
\item\label{Item::SubEThm::ImageInH} $\exists R\approx 1$ such that $\Phi_{x,\delta}(B_{\R^r\times \C^n}(1))\subseteq B_H(x,R\delta)$, $x\in \Compact$, $\delta\in (0,1]$.
\item\label{Item::SubEGeom::ExistEpsilon} $\exists \epsilon \approx 1$ such that $B_S(x,\epsilon\delta)\subseteq \Phi_{x,\delta}(B_{\R^r\times \C^n}(1)) \subseteq B_S(x,\delta)$, $x\in \Compact$, $\delta\in (0,1]$.
\end{enumerate}
\end{thm}

\begin{rmk}\label{Rmk::ResSubE::RA}
In \cref{Thm::SubEGeom} we stated a result for $C^\infty$ vector fields.  A similar result, with a similar proof, can be stated for real analytic vector fields, where one can ensure the map $\Phi_{x,\delta}$ is
real analytic and the vector fields $\Zh_j^{x,\delta}$ are real analytic in a quantitative way.  This proceeds by using the case $s_0=\omega$ in \cref{Thm::Results::MainThm} (instead of $s_0\in (1,\infty)$).
In the setting of real vector fields, this was done in \cite{StovallStreetIII}.  We leave the details to the interested reader.
\end{rmk}

\begin{rmk}\label{Rmk::ResSubE::MoreComplicated}
In this section, we described geometries where the vector fields at scale $\delta$ where given by
$\delta^{\beta_1} X_1,\ldots, \delta^{\beta_q} X_q, \delta^{\beta_{q+1}} L_1,\ldots, \delta^{\beta_{q+m}} L_m$,
for some fixed vector fields $X_1,\ldots, X_q,L_1,\ldots, L_m$.  It is straightforward to generalize  \cref{Thm::SubEGeom}
to work in a setting where the vector fields have a more complicated dependance on $\delta$.  In this setting, one would
take, for each $\delta\in(0,1]$, a collection of vector fields $X_1^{\delta},\ldots, X_q^{\delta},L_1^{\delta},\ldots, L_m^{\delta}$
and place appropriate axioms on these vector fields so that the  proof of \cref{Thm::SubEGeom} works uniformly for $\delta\in(0,1]$.
This approach was described in the real setting in \cite{StovallStreetI,StovallStreetIII}.  An example in the complex setting is described in \cref{Section::ExtremalBasis}.
Using the same ideas, the results in this paper generalize the result in the multi-parameter
stetting of \cite{StreetMultiparameterCCBalls}.  Here, we fix some $\mu\in \N$, $\mu\geq 1$ and for each $\delta\in (0,1]^\mu$ we are given vector fields $X_1^{\delta},\ldots, X_q^{\delta},L_1^{\delta},\ldots, L_m^{\delta}$ and proceed in the same way.
We leave further details to the interested reader.
\end{rmk}

\begin{rmk}\label{Rmk::ResSubE::FiniteSmoothness}
The assumption that the vector fields are $\CjSpace{\infty}$ is not essential.  In fact, because \cref{Thm::Results::MainThm} is stated for $C^1$ vector fields, one need only assume
the given vector fields are $C^1$.  Then, as in \cref{Rmk::ResSubE::MoreComplicated}, one assumes that the hypotheses of \cref{Thm::Results::MainThm} hold uniformly in the relevant
parameters.  See \cite[\SSGenSubR]{StovallStreetI} for a description of this in the real setting.
\end{rmk}

\section{An example from several complex variables}\label{Section::ExtremalBasis}
In \cref{Section::Intro::Normalize,Section::CorRes::Geometries,Section::Res::SubE} we described how to use the coordinate system $\Phi$ from \cref{Thm::Results::MainThm} as
a generalized scaling map.  In these settings, we applied \cref{Thm::Results::MainThm} to a family of vector fields which depended on $\delta\in (0,1]$.
For example, in \cref{Section::ResGeom::SubHerm}, the vector fields were $\delta^{\beta_1}L_1,\ldots, \delta^{\beta_m} L_m$, where $L_1,\ldots, L_m$ were
sections of $T^{0,1} M$ satisfying certain properties (and $M$ was a complex manifold).  In these settings, the vector fields depend on $\delta$ in a very simple way; and we presented results
in these settings for simplicity.  However, \cref{Thm::Results::MainThm} allows one to consider vector fields which depend on $\delta$ (and on the base point) in much more
complicated ways.  This can be important in applications, and to describe these ideas we present an important setting which arises in several complex variables:  extremal bases.

Extremal bases were first used by McNeal \cite{McNealEstimatesOnTheBergmanKernelsOfConvexDomains} to study Bergman kernels and invariant metrics associated to convex domains of finite type; see also \cite{HeferExtremalBasesAndHolderEstimates}.    More generally, extremal bases can be used to study lineally convex domains \cite{ConradAnisotrope}
(see also \cite[Section 7.1]{CharpentierDupainExtermalBases}).  They can also be used
to study Bergman and Szeg\"o kernels and invariant metrics on pseudoconvex domains of finite type with comparable eigenvalues \cite{KoenigOnMaximalSobolevAndHolder,ChoEstimatesOfInvariantMetricsOnPseudoConvexDomains,ChoEstimatesOfTheBergmanKernel,ChoBoundaryBehaviorOfTheBergmanKernel}.
Finally, they have been used to study pseudoconvex domains of finite type with locally diagonalizable Levi forms \cite{CharpentierDupainGeometryOfPseudoConvexDomainsOfFiniteType,CharpentierDupainEstimatesForBergmanAndSezgoLocallyDiag,CharpentierDupainExtermalBases}.
All of these settings have been generalized to one abstract setting by Charpentier and Dupain \cite{CharpentierDupainExtermalBases}.
The presentation below is closely related to the ideas of \cite{CharpentierDupainExtermalBases}, though expressed in a different way.

As can be seen from the above mentioned works, extremal bases are closely related to a notion of distance in many complex domains; and scaling techniques are central
in using extremal bases to study objects like Bergman and Szeg\"o kernels (many of the above papers use some kind of scaling).
See \cite{McNealSubellipticEstimatesAndScaling} for a particularly straightforward explanation of the form scaling takes in some of these examples.
In this section, we show how to use \cref{Thm::Results::MainThm} to understand this scaling
in a more abstract way.
The idea is to rephrase the notion of an extremal basis in a way which is \textit{quantitatively} invariant under arbitrary biholomorphisms.
We hope that this will give the reader some idea of how to apply the results of this paper
to questions in several complex variables, perhaps even beyond the setting of extremal bases.

Following the philosophy of this paper, we describe the scaling associated to extremal bases in three steps:
\begin{itemize}
\item Extremal bases at the unit scale:  because we wish to scale a small scale into the unit scale, first we must introduce what we mean by the unit scale.  This is the setting
where classical techniques from several complex variables can be used to prove estimates.
\item Extremal bases in a biholomorphic invariant setting:  using \cref{Thm::Results::MainThm}, we rephrase the unit scale from the previous point in a way which is quantitatively
invariant under arbitrary biholomorphisms.  Because of this, we completely remove the notion of ``scale,'' because that notion depends on a choice of coordinate system.
\item Extremal bases at small scales:  here we introduce the notion of extremal bases at small scales, by seeing it as a special case of the biholomorphically invariant version of the previous point.
Because of this, it will immediately follow that the setting of small scales is biholomorphically equivalent to the unit scale.
\end{itemize}
After introducing these three steps, we describe the similar setting of CR manifolds.

Though it will not play a role in our discussion, the setting to keep in mind is the following.  $\fM$ is a complex manifold, and $\Omega=\{ \zeta\in \fM: \rho(\zeta)<0\}$ is
relatively compact domain, where $\rho\in \CjSpace{\infty}[\fM][\R]$ is a defining function of $\Omega$ such that $d\rho(\zeta)\ne 0$, $\forall \zeta\in \partial \Omega=\{\rho=0\}$.
All of the above mentioned papers concern pseudoconvex domains near points of finite type.

While the scaling maps of Nagel, Stein, and Wainger \cite{NagelSteinWaingerBallsAndMetrics} have long been used in such problems (see, e.g., \cite{KoenigOnMaximalSobolevAndHolder}),
we will see that the results of this paper allow us to have similar scaling maps which are \textit{holomorphic}, as opposed to the smooth maps given in \cite{NagelSteinWaingerBallsAndMetrics};
thus they do not destroy the complex nature of the problem.

\begin{rmk}
Other than a new way of viewing extremal bases, the perspective here may not bring much new to this well-studied concept.  However, we hope the general outline may be useful for other problems in several complex variables.
Indeed, as we will explain, the idea is to take a known result at the unit scale, rewrite it in a way which is quantitatively invariant under biholomorphisms (using \cref{Thm::Results::MainThm}).
This then automatically gives a quantitative result at small scales, since the notion of scale is not invariant under biholomorphisms.
\end{rmk} 

	\subsection{Extremal bases at the unit scale}\label{Section::ExtremalBasis::Unit}
Fix $n\in \N$ and let $\rho\in \CjSpace{\infty}[B_{\C^n}(1)][\R]$ satisfy $\rho(0)=0$ and $d\rho(\zeta)\ne 0$, $\forall \zeta\in B_{\C^n}(1)$.
Let $L_n$ be a smooth section of $T^{0,1}B_{\C^n}(1)$ (i.e., $L_n$ is a complex vector field spanned by $\diff{\zb[1]},\ldots, \diff{\zb[n]}$)
be such that $L_n\rho(\zeta)\ne 0$, $\forall \zeta\in B_{\C^n}(1)$.
For example, one often takes
\begin{equation*}
	L_n = \sum_{j=1}^{n} \frac{\partial \rho}{\partial z_j} \diff{\zb[j]},
\end{equation*}
so that $L_n \rho =\sum \mleft| \frac{\partial \rho}{\partial z_j} \mright|^2$.

Let $L_1, \ldots, L_{n-1}$ be smooth sections of $T^{0,1}B_{\C^n}(1)$ such that $L_j\rho =0$ on $B_{\C^n}(1)$,
and such that $L_1(\zeta),\ldots, L_{n}(\zeta)$ span $T^{0,1}_\zeta B_{\C^n}$, $\forall \zeta\in B_{\C^n}(1)$.
Given $\theta=(\theta_1,\ldots, \theta_{n-1})\in \C^{n-1}$ with $|\theta|=1$, set
$L_{\theta}=\sum_{j=1}^{n-1} \theta_{j} L_j$.
Set $\sZ_1^{\theta}=\{ [L_{\theta}, \Lb[\theta] ]\}$, and recursively set $\sZ_j^{\theta} = \{ [ L_\theta, Z], [\Lb[\theta],Z] : Z\in \sZ_{j-1}^{\theta}\}$ for $j\geq 2$.

\begin{defn}
We say $L_1,\ldots, L_n,\rho$ is an extremal system if there exists $K\in \N$ such that
\begin{equation*}
	\{ L_1,\ldots, L_{n-1}, \Lb[1],\ldots, \Lb[n-1], L_n\} \bigcup \mleft(\bigcup_{j=1}^K \sZ_j^{\theta}\mright)
\end{equation*}
spans $\C T_\zeta B_{\C^{n}}(1)$, $\forall \zeta\in B_{\C^n}(1)$, $|\theta|=1$.
\end{defn}

Along with an extremal system, strictly plurisubharmonic functions are often used.  Thus, we assume we are given a function $H\in \CjSpace{3}[B_{\C^n}(1)][\R]$ such
that $\partial \overline{\partial} H$ is strictly positive definite on $B_{\C^n}(1)$.

\begin{rmk}\label{Rmk::ExtremalUnit::QuantAssump}
Given an extremal system and plurisubharmonic function, as above, there are many estimates one can prove using now standard techniques (usually, this occurs under the additional
qualitative assumption that the domain is weakly pseudoconvex, see the above mentioned works for details).  These estimates often depend on the following quantities (or something similar):
\begin{enumerate}[(i)]
	\item\label{Item::ExtremalUnit::nk} Upper bounds for $n$ and $K$.
	\item\label{Item::ExtremalUnit::BoundCN} Upper bounds for $\CjNorm{\rho}{N}[B_{\C^n}(1)]$ and $\max_{1\leq j\leq n} \CjNorm{L_j}{N}[B_{\C^n}(1)][\C^n]$, where $N$ can be chosen to depend only on upper bounds
	for $K$, $n$, and the particular estimate being shown.
	\item A lower bound, $>0$, for $\inf_{\zeta\in B_{\C^n}(1)} \mleft| L_n \rho (\zeta)\mright|$.
	\item\label{Item::ExtrmealUnit::SpanT01} A lower bound, $>0$, for $\inf_{\zeta\in B_{\C^{n}}(1)} \mleft| \det (L_1(\zeta) | \cdots | L_n(\zeta))\mright|$, where this matrix has columns $L_1,\ldots, L_n$, written
	in terms of $\diff{\zb[1]},\ldots, \diff{\zb[n]}$.
	\item\label{Item::ExtremalUnit::SpanAllTangent} A lower bound, $>0$, for
	\begin{equation*}
		\inf_{\substack{\zeta\in B_{\C^n}(1) \\ |\theta|=1}} \max_{Z\in \bigcup_{j=1}^K \sZ_{j}^{\theta}} \mleft| \det \bigg( L_1(\zeta) | \cdots | L_n(\zeta)  | \Lb[1](\zeta)| \cdots | \Lb[n-1](\zeta)| Z(\zeta)  \bigg) \mright|,
	\end{equation*}
	where in the above matrix, the vector fields are written in terms of $\diff{z_1},\ldots, \diff{z_n}, \diff{\zb[1]},\ldots, \diff{\zb[n]}$.
	\item An upper bound for $\CjNorm{H}{3}[B_{\C^n}(1)]$.
	\item\label{Item::ExtremalUnit::HessianLower}  A lower bound, $>0$, for the quadratic form $\partial \overline{\partial} H$ on $B_{\C^n}(1)$.  Equivalently, using
	\cref{Item::ExtremalUnit::BoundCN,Item::ExtrmealUnit::SpanT01}, for $\omega=(\omega_1,\ldots, \omega_n)\in \C^n$ with $|\omega|=1$,
	set $L_\omega=\sum \omega_j L_j$.  The estimates may depend on a lower bound, $>0$,  for:
	\begin{equation}\label{Eqn::ExtremalUnit::Hessian}
		\inf_{ \substack{\zeta\in B_{\C^{n}}(1) \\ |\omega|=1}} \mleft< \partial \overline{\partial} H(\zeta); L_\omega(\zeta), \Lb[\omega](\zeta) \mright>.
	\end{equation}
\end{enumerate}
Thus, if one has an infinite collection of extremal systems and plurisubharmonic functions,  such that the above quantities can be chosen uniformly over this infinite collection, then one can prove the above mentioned
estimates, uniformly over the infinite collection.  See \cite{McNealSubellipticEstimatesAndScaling} for some easy to understand examples of such estimates.
Note that all of the above quantities except for \cref{Item::ExtremalUnit::nk} depend on the choice of coordinate system:  if one applies a biholomorphism to this setting,
it destroys all of the above constants.  The next section fixes this problem.
\end{rmk}

\begin{rmk}
In \cref{Item::ExtremalUnit::HessianLower}, we assumed that the quadratic form $\partial \overline{\partial} H$ was bounded away from $0$.  In the famous work of Catlin \cite{CatlinSubellipticEstimates},
subelliptic estimates are shown using plurisubharmonic functions where this bound can be chosen very large.
While the form being positive definite does not depend on the choice of holomorphic coordinate system, the lower bound for the form does depend on the choice of coordinate system.
In \cref{Section::ExtremalBasis::Small}, we will assume the existence of a plurisubharmonic function adapted to each scale;
those adapted to a small scale will have a large lower bound when viewed in a fixed coordinate system independent of the scale (see \cref{Rmk::ExtremalSmall::LargeHessian}).
\end{rmk}

	\subsection{Extremal bases invariant under biholomorphisms}\label{Section::ExtremalBasis::Biholo}
In this section, we present extremal bases again.
Qualitatively, this is exactly the same as what is written in \cref{Section::ExtremalBasis::Unit}; the difference here is that our quantitative assumptions will be written in
a way which is invariant under biholomorphisms (as opposed to the quantitative assumptions in \cref{Rmk::ExtremalUnit::QuantAssump} which depended on
the choice of coordinate system).

Let $\fM$ be a complex manifold of complex dimension $n$.  Fix a point $\zeta_0\in \fM$ and $\rho\in \CjSpace{\infty}[\fM][\R]$ with $\rho(\zeta_0)=0$.
Let $L_1,\ldots, L_n$ be smooth sections of $T^{0,1}\fM$ and fix $\xi>0$.  We take the following assumptions and definitions:

\begin{enumerate}[(i)]
	\item\label{Item::ExtremalBiholo::Span} $\forall \zeta\in B_{L}(\zeta_0,\xi)$, $\Span_{\C}\{ L_1(\zeta),\ldots, L_n(\zeta)\} = T^{0,1}\fM$.
	\item $c_1:=\inf_{\zeta\in B_L(\zeta_0,\xi)} |L_n \rho(\zeta)|>0$.
	\item\label{Item::ExtremalBiholo::VansihDeriv} For $1\leq j\leq n-1$, $L_j\rho(\zeta)=0$ for $\zeta\in B_L(\zeta_0,\xi)$.
	\item\label{Item::ExtremalBiholo::Invol} Due to \cref{Item::ExtremalBiholo::Span}, we may write $[L_j, L_k]=\sum_{l=1}^n c_{j,k}^{1,l} L_l$ and $[L_j, \Lb[k]] =\sum_{l=1}^n c_{j,k}^{2, l} L_l + \sum_{j=1}^n c_{j,k}^{3,l} \Lb[l]$.
	For each $N\in \N$, take $C_N$ (which we assume to be finite\footnote{$C_N$ can always be chosen to be finite, so long as $B_L(\zeta_0,\xi)$ is relatively compact in $\fM$, which can
	be guaranteed by taking $\xi$ small enough; though it is the particular value of $C_N$ which is important, not just that it is finite.}) so that
	\begin{equation*}
		\CXjNorm{c_{j,k}^{a,l}}{L}{N}[B_{L}(\zeta_0,\xi)], \sum_{m=1}^N \CNorm{L_n^m\rho}{B_{L}(\zeta_0,\xi)}\leq C_N, \quad 1\leq j,k,l\leq n, 1\leq a\leq 3.
	\end{equation*}
	\item\label{Item::ExtremalBiholo::Span2} For each $\theta\in \C^{n-1}$ with $|\theta|=1$, define $\sZ_j^{\theta}$ in terms of $L_1,\ldots, L_{n-1}$ as in \cref{Section::ExtremalBasis::Unit}.  We assume
	that there exists $K\in \N$ such that $\forall \zeta\in B_L(\zeta_0,\xi)$,
	\begin{equation*}
		\Lb[n](\zeta)=\sum_{j=1}^{n} a_j^{1,\theta}(\zeta) L_j(\zeta) + \sum_{j=1}^{n-1} a_j^{2,\theta}(\zeta) \Lb[j](\zeta) + \sum_{Z\in \bigcup_{j=1}^K \sZ_j^{\theta}} b_Z^{\theta}(\zeta) Z(\zeta),
	\end{equation*}
	with
	\begin{equation*}
		D:=\sup_{\substack{\zeta\in B_L(\zeta_0,\xi) \\ |\theta|=0, Z\in \bigcup_{j=1}^K \sZ_j^{\theta}}} |a_j^{1,\theta}(\zeta)|+|a_j^{2,\theta}(\zeta)| + |b_Z^{\theta}(\zeta)|<\infty.
	\end{equation*}
	\item\label{Item::ExtremalBiholo::etadelta} Let $\eta, \delta_0>0$ be as in \cref{Thm::Results::MainThm}.  In applications, usually $L_1,\ldots, L_n$ are often given in a coordinate system in which their $C^1$ norms are very small,
	and then $\eta$ and $\delta_0$ can be bounded below in terms of the $C^1$ norms in this coordinate system (see, e.g., \cref{Lemma::MoreAssume::ExistEtaDelta0} and the discussion following it).
	\item We suppose we are given a function $H\in \CXjSpace{L}{3}[B_L(\zeta_0,\xi);\R]$.  
	For $\omega\in \C^n$ with $|\omega|=1$, define $L_\omega=\sum_{j=1}^n \omega_j L_j$.
	We assume,
	\begin{equation*}
		c_2:=\inf_{ \substack{\zeta\in B_L(\zeta_0,\xi) \\ |\omega|=1}} \mleft< \partial \overline{\partial} H(\zeta); L_\omega(\zeta), \Lb[\omega](\zeta) \mright>>0.
	\end{equation*}
\end{enumerate}

\begin{prop}\label{Prop::ExtremalBiholo::Exists}
In the above setting, there is a biholomorphism $\Phi: B_{\C^n}(1)\rightarrow \Phi(B_{\C^n}(1))\subseteq B_L(\zeta_0,\xi)$, with $\Phi(0)=\zeta_0$, such that
$\Phi^{*}L_1,\ldots, \Phi^{*}L_n, \Phi^{*}\rho$ is an extremal system with plurisubharmonic function $\Phi^{*}H$.  Moreover, all of the estimates described in
\cref{Rmk::ExtremalUnit::QuantAssump} can be bounded in terms of upper bounds for $n$, $K$, $c_1^{-1}$, $C_N$ (where $N$ can be chosen to depend only on $n$, $K$,
and the particular estimate being shown), $D$, $\eta^{-1}$, $\delta_0^{-1}$, $\xi^{-1}$, $c_2^{-1}$, and $\CXjNorm{H}{L}{3}[B_L(\zeta_0,\xi)]$.
\end{prop}
\begin{proof}
This follows immediately from \cref{Thm::Results::MainThm} (and that theorem includes more properties of $\Phi$), by applying the theorem
to $L_1,\ldots, L_n$ (we are talking $m=n$ and $q=0$)--that $\Phi$ is holomorphic is the essence of \cref{Rmk::MainRes::EMap}.  

There are two parts which do not follow directly from the statement of \cref{Thm::Results::MainThm}.  To estimate $\CjNorm{\Phi^{*}\rho}{N}[B_{\C^N}(1)]$ we would like to have
estimates on $\CXjNorm{\rho}{L}{N}[B_{L}(\zeta_0,\xi)]$.  To obtain this, note that it follows easily from \cref{Item::ExtremalBiholo::VansihDeriv,Item::ExtremalBiholo::Invol}
that $\CXjNorm{\rho}{L}{N}[B_{L}(\zeta_0,\xi)]\approx \sum_{m=1}^N \CNorm{L_n^m\rho}{B_{L}(\zeta_0,\xi)}\leq C_N$.

The other part
that does not follow directly from the statement of \cref{Thm::Results::MainThm} is
a lower bound for \eqref{Eqn::ExtremalUnit::Hessian}.  The key here is that this quantity is invariant under biholomorphisms.  Indeed, since $\Phi$ is a biholomorphism,
\begin{equation*}
	 \mleft< \partial \overline{\partial} (\Phi^{*}H)(z); (\Phi^{*}L_\omega)(z), (\Phi^{*}\Lb[\omega])(z) \mright>
	 =\mleft< \partial \overline{\partial} H(\Phi(z)); L_\omega(\Phi(z)), \Lb[\omega](\Phi(z)) \mright>.
\end{equation*}
Thus, \eqref{Eqn::ExtremalUnit::Hessian} is bounded below by $c_2$.
\end{proof}

\begin{defn}\label{Defn::ExtremalBiholo::Uniform}
Let $\sI$ be an index set, and suppose for each $\iota\in \sI$, we are given $\fM^{\iota}$, $\zeta_0^{\iota}$, $L_1^{\iota},\ldots, L_n^{\iota}$, $\rho^{\iota}$, and $H^{\iota}$
as above, satisfying the above estimates uniformly (i.e., there are upper bounds for the quantities discussed in \cref{Prop::ExtremalBiholo::Exists} which can be chosen independent of $\iota$).
We say that the collection $L_1^{\iota},\ldots, L_n^{\iota}$, $\zeta_0^{\iota}$, $\rho^{\iota}$, $H^{\iota}$ is an extremal system with adapted plurisubharmonic function, uniformly in $\iota$.
\end{defn}

\begin{rmk}\label{Rmk::ExtremalBiholo::Uniform}
Combining \cref{Defn::ExtremalBiholo::Uniform} and \cref{Prop::ExtremalBiholo::Exists}, we see that if $L_1^{\iota},\ldots, L_n^{\iota}$, $\zeta_0^{\iota}$, $\rho^{\iota}$, $H^{\iota}$ is an extremal system with adapted plurisubharmonic function, uniformly in $\iota$, then for each $\iota\in \sI$, there exists a biholomorphism $\Phi_\iota:B_{\C^n}(1)\rightarrow \Phi_{\iota}(B_{\C^n}(1))\subseteq \fM^{\iota}$,
with $\Phi_\iota(0)=\zeta_0^{\iota}$ such that
$\Phi_\iota^{*} L_1^{\iota},\ldots, \Phi_{\iota}^{*}L_n^{\iota}$, $\Phi_\iota^{*} \rho^{\iota}$, $\Phi_\iota^{*} H^{\iota}$ is an extremal system with plurisubharmonic function on $B_{\C^n}(1)$,
satisfying all of the estimates outlined in \cref{Rmk::ExtremalUnit::QuantAssump}, uniformly in $\iota$.
\end{rmk}

	\subsection{Extremal bases at small scales}\label{Section::ExtremalBasis::Small}
Let $\fM$ be a complex manifold of complex dimension $n$ and fix $\zeta_0\in \fM$, and let $\rho\in \CjSpace{\infty}[\fM][\R]$.
Let $\Compact\subseteq \{ \zeta\in \fM:\rho(\zeta)=0\}$, and suppose $L_n$ is a smooth section of $T^{0,1}\fM$ such that
$\inf_{\zeta\in \Compact} |L_n\rho(\zeta)|>0$.  To work at scale $\delta\in(0,1]$, we wish to replace $\rho$ with $\delta^{-2} \rho$ and $L_n$ with $\delta^2 L_n$.

To do this, we assume that for each $\delta\in (0,1]$, $\zeta_0\in \Compact$, we are given smooth sections of $T^{0,1}\fM$ defined near $\zeta_0$, $L_1^{\delta,\zeta},\ldots, L_{n-1}^{\delta,\zeta}$,
and a real valued smooth function $H^{\zeta_0,\delta}$ defined near $\zeta_0$, such that
\begin{equation*}
	L_1^{\delta,\zeta_0},\ldots, L_{n-1}^{\delta,\zeta_0}, \delta^{2} L_n, \delta^{-2} \rho,H^{\zeta_0,\delta}, \zeta_0
\end{equation*}
is an extremal system with adapted plurisubharmonic
function, uniformly in $\delta\in (0,1]$, $\zeta_0\in \Compact$.
Thus, using \cref{Rmk::ExtremalBiholo::Uniform}, for each $\delta\in (0,1]$, $\zeta_0\in \Compact$, there is a biholomorphism $\Phi_{\zeta_0,\delta}:B_{\C^n}(1)\rightarrow \Phi_{\zeta_0,\delta}(B_{\C^n}(1))$,
with $\Phi_{\zeta_0,\delta}(0)=\zeta_0$ and
such that $\Phi_{\zeta_0,\delta}^{*} L_1^{\delta,\zeta_0}, \ldots, \Phi_{\zeta_0,\delta}^{*}L_{n-1}^{\delta,\zeta_0}, \Phi_{\zeta_0,\delta}^{*} \delta^{2}L_n$, $\Phi_{\zeta_0,\delta}^{*}\delta^{-2}\rho$, $\Phi_{\zeta_0,\delta}^{*} H^{\zeta_0,\delta}$ is an extremal system at the unit scale, uniformly in $\zeta_0\in \Compact$ and $\delta\in (0,1]$ (in the sense that the constants described in \cref{Rmk::ExtremalUnit::QuantAssump} can be chosen
uniformly in $\zeta_0$ and $\delta$).

If one imagines $\fM$ has having some fixed coordinate system, independent of $\delta$, then $\Phi_{\zeta,\delta}$ takes points which look to be of distance $\approx \delta^2$ from $\{\rho=0\}$ and
``rescales'' them to have distance $\approx 1$ from $\{\Phi_{\zeta,\delta}^{*}\rho =0\}$.

\begin{rmk}\label{Rmk::ExtremalSmall::LargeHessian}
In the coordinate system $\Phi_{\zeta,\delta}$, the quadratic form $\partial\overline{\partial} \Phi_{\zeta,\delta}^{*} H^{\zeta,\delta}$ is bounded above and below, uniformly in $\zeta\in \Compact$, $\delta\in (0,1]$.
However, if one starts with a fixed coordinate system on $\fM$ (independent of $\delta\in (0,1]$), then in terms of this coordinate system, $\partial \overline{\partial} H^{\zeta,\delta}$ has a large
lower bound
 (as $\delta\rightarrow 0$), since the vector fields $L_1^{\delta,\zeta},\ldots, L_{n-1}^{\delta,\zeta}, \delta^{2} L_n$ are small in this fixed coordinate system.
\end{rmk}

In applications, a major difficulty is showing such an extremal basis and adapted plurisubharmonic function exists at each scale.  As mentioned before, this has been done in many settings, and was abstracted
in \cite{CharpentierDupainExtermalBases}.  Indeed, if $L_1,\ldots, L_n$ is a $(M,K,\zeta,\delta)$ extremal basis as in \cite[Definition 3.1]{CharpentierDupainExtermalBases}, then one can obtain
an extremal basis at scale $\delta$ (in the sense discussed here) by considering
$F(L_1, \zeta,\delta)^{-1/2}L_1,\ldots, F(L_{n-1}, \zeta,\delta)^{-1/2}L_{n-1}, \delta^{2} L_n$ and $\delta^{-2}\rho$, where these terms are all defined in \cite{CharpentierDupainExtermalBases} (adapted plurisubharmonic functions are described in \cite[Section 5]{CharpentierDupainExtermalBases}).
While the results in this paper do not help find such an extremal basis, perhaps they will help make clear what to look for in other similar situations.  Indeed, once one has a result on the unit scale,
to translate the result to small scales it often suffices to write the unit scale result in a way which is invariant under biholomorphisms using \cref{Thm::Results::MainThm}, as we did in \cref{Section::ExtremalBasis::Biholo}.  Then one can immediately translate the setting to small scales as we did in \cref{Section::ExtremalBasis::Small}. 
	
	\subsection{CR Manifolds}\label{Section::ExtremalBasis::CR}
Instead of studying extremal bases on a neighborhood of a point on the boundary of a complex domain, one could try to work directly on the boundary by working with abstract CR manifolds;
this is the approach taken, for example, in \cite{KoenigOnMaximalSobolevAndHolder} (which addressed the ``comparable eigenvalue'' setting).

Let $\fM$ be a smooth manifold of dimension $2n-1$ endowed with a smooth CR structure $\WVS$ of rank $n-1$ (so that $\WVS\oplus \WVSb$ has codimension $1$ in $\C T\fM$).
Fix a point $\zeta_0\in \fM$ and pick a smooth real vector field $T$, defined near $\zeta_0$, such that
\begin{equation*}
	\WVS_{\zeta_0}+\WVSb[\zeta_0]+\Span_{\C} \{ T(\zeta_0)\} = \C T\fM.
\end{equation*}
If we pick smooth sections $L_1,\ldots, L_{n-1}$ of $\WVS$ near $\zeta_0$, these can play the role that the vector fields of the same name did in
\cref{Section::ExtremalBasis::Biholo}.  We then take the same assumptions as \cref{Item::ExtremalBiholo::Invol}, \cref{Item::ExtremalBiholo::Span}, and \cref{Item::ExtremalBiholo::etadelta} from
\cref{Section::ExtremalBasis::Biholo}, where we replace $\Lb[n]$ with $T$, $L_n$ with $0$, and remove $\rho$ from \cref{Item::ExtremalBiholo::Invol};
thus we obtain constants $\xi$, $C_N$, $D$, $K$, $\eta$, and $\delta_0$ as in those assumptions.

In this setting, \cref{Thm::Results::MainThm}  applies with $X_1,\ldots, X_q = \Real(L_1),\Imag(L_1),\ldots, \Real(L_{n-1}), \Imag(L_{n-1}), T$,
to obtain a $C^\infty$ diffeomorphism\footnote{Throughout we identify $\R^{2n-1}$ with $\R\times \C^{n-1}$.} $\Phi:B_{\R\times \C^{n-1}}(1)\rightarrow \Phi(B_{\R \times \C^{n-1}}(1))$ with $\Phi(0)=\zeta_0$ and such that
$\Phi^{*}L_1,\ldots, \Phi^{*}L_{n-1}, \Phi^{*} T$ satisfy good estimates at the unit scale.  Namely, we obtain
\cref{Rmk::ExtremalUnit::QuantAssump} \cref{Item::ExtremalUnit::BoundCN} and \cref{Item::ExtremalUnit::SpanAllTangent},
where in \cref{Item::ExtremalUnit::BoundCN} we replace $L_n$ with $T$ and $\rho$ with $0$, and in \cref{Item::ExtremalUnit::SpanAllTangent} we replace $L_n$ with $0$
and the vector fields are written in terms of the standard basis for vector fields on $\R \times \C^{n-1}\cong \R^{2n-1}$--and these quantities can be estimated in terms of
upper bounds for $\xi^{-1}$, $n$, $K$, $C_N$ (where $N$ can be chosen to depend only on $n$, $K$, and the estimate at hand),
$D$, $\eta^{-1}$, and $\delta_0^{-1}$.

Thus, \cref{Thm::Results::MainThm} allows us to rewrite a setting at the ``unit scale'' in a way which is invariant under arbitrary diffeomorphisms; which, as in \cref{Section::ExtremalBasis::Small},
allows us to view these maps as scaling maps.
In \cref{Section::ExtremalBasis::Biholo} we asked that the map be bioholomorphic.  It does not make a priori sense to insist that the map $\Phi$ given here is a CR map
because $\R\times \C^{n-1}$ does not have a  canonical CR structure.  We could give $B_{\R \times \C^{n-1}}(1)$ the CR structure $\Phi^{*} \WVS$,
and then $\Phi$ is automatically a CR map, but this does not add much useful information.

However, in the special case that one can choose $T$ so that the bundle $\WVS_{\zeta}+\C T(\zeta)$ is formally integrable (and therefore an elliptic structure), then we can do more.  In this case, we can choose
$\Phi$ so that $\Phi^{*}L_1,\ldots, \Phi^{*} L_{n-1}, \Phi^{*} T$ are all spanned by $\diff{\zb[1]},\ldots, \diff{\zb[n]}, \diff{t}$ (where $\R\times \C^{n-1}$ is given coordinates
$(t, z_1,\ldots, z_n)$).  One can see this by applying \cref{Thm::Results::MainThm} to the vector fields $L_1,\ldots, L_{n-1}$ and $X_1=T$ (with $m=n-1$ and $q=1$).
In many of the standard examples of CR manifolds, one can find such a $T$--see \cref{Section::CRMfld}.

	
\section{Function Spaces Revisited}
In this section we present the basic properties of the function spaces introduced in \cref{Section::FuncSpaces}; most of these properties were proved
in \cite{StovallStreetI,StovallStreetII,StovallStreetIII}, and we refer the reader to those references for proofs and a further discussion of the results not proved here.
We take $W_1,\ldots, W_N$ to be real $C^1$ vector fields on a $C^2$ manifold $M$ as in \cref{Section::FuncSpaces}.

\begin{lemma}\label{Lemma::FuncSpaceRev::Properties}
\begin{enumerate}[(i)]
\item\label{Item::FuncSpaceRev::IncludHolder} For $0\leq s_1\leq s_2\leq 1$, $m\in \N$, $\HXNorm{f}{W}{m}{s_1}[M]\leq 3\HXNorm{f}{W}{m}{s_2}[M]$.
\item\label{Item::FuncSpaceRev::CmBoundsHms} $\HXNorm{f}{W}{m}{1}[M]\leq \CXjNorm{f}{W}{m+1}[M]$.
\item\label{Item::FuncSpaceRev::HmsBoundsZygs} For $s\in (0,1]$, $m\in \N$, $\ZygXNorm{f}{W}{s+m}[M]\leq 5\HXNorm{f}{W}{m}{s}[M]$.
\item\label{Item::FuncSpaceRev::IncludeZygmund} For $0<s_1\leq s_2<\infty$, $\ZygXNorm{f}{W}{s_1}[M]\leq 15 \ZygXNorm{f}{W}{s_2}[M]$.
\item\label{Item::FuncSpaceRev::IncludeSets} If $U\subseteq M$ is an open set, then $\HXNorm{f}{W}{m}{s}[U]\leq \HXNorm{f}{W}{m}{s}[M]$ and $\ZygXNorm{f}{W}{s}[U]\leq \ZygXNorm{f}{W}{s}[M]$.
\item\label{Item::FuncSpaceRev::ComegainsA} $\ComegaSpace{r}[B_{\R^n}(r)]\subseteq \ASpace{n}{r}$ and $\ANorm{f}{n}{r}\leq \ComegaNorm{f}{r}[B_{\R^n}(r)]$.
\item\label{Item::FuncSpaceRev::sAinComega} $\ASpace{n}{r}\subseteq \ComegaSpace{r/2}[B_{\R^n}(r/2)]$ and $\ComegaNorm{f}{r/2}[B^n(r/2)]\leq \ANorm{f}{n}{r}$.
\item\label{Item::FuncSpaceRev::Properites::ContainRA} Suppose $W=W_1,\ldots, W_N$ satisfies $\sC(x_0,r,M)$.  Then, $\CXomegaSpace{W}{r}[M]\subseteq \AXSpace{W}{x_0}{r}$
and $\AXNorm{f}{W}{x_0}{r}\leq \CXomegaNorm{f}{W}{r}[M]$.
\item\label{Item::FuncSpaceRev::DerivCXomega} For any $s\in (0,r)$, $W_j: \CXomegaSpace{W}{r}[M]\rightarrow \CXomegaSpace{W}{s}[M]$.  In particular, $W_j:\CXjSpace{W}{\omega}[M]\rightarrow \CXjSpace{W}{\omega}[M]$.
\item\label{Item::FuncSpaceRev::Properties::Deriv} For any $s\in (1,\infty]\cup\{\omega\}$, $W_j : \ZygXSpace{W}{s}[M]\rightarrow \ZygXSpace{W}{s-1}[M]$.
\end{enumerate}
\end{lemma}
\begin{proof}
\Cref{Item::FuncSpaceRev::IncludHolder}, \cref{Item::FuncSpaceRev::CmBoundsHms}, \cref{Item::FuncSpaceRev::HmsBoundsZygs}, \cref{Item::FuncSpaceRev::IncludeZygmund}, and \cref{Item::FuncSpaceRev::IncludeSets}
are contained in \cite[\SSCompareFunctionSpaces]{StovallStreetI}.
\Cref{Item::FuncSpaceRev::ComegainsA}, \cref{Item::FuncSpaceRev::sAinComega}, \cref{Item::FuncSpaceRev::Properites::ContainRA}, \cref{Item::FuncSpaceRev::DerivCXomega} are contained in \cite{StovallStreetIII}.
For $s\in (1,\infty]$, \cref{Item::FuncSpaceRev::Properties::Deriv} follows immediately from the definitions.  For $s=\omega$, \cref{Item::FuncSpaceRev::Properties::Deriv} follows from \cref{Item::FuncSpaceRev::DerivCXomega}; where
we are using the convention $\omega=\omega-1$.
\end{proof}

\begin{rmk}\label{Rmk::FuncSpaceRev::ZygAndHolderEquiv}
Let $\Omega\subseteq \R^n$ be an open set.  In analogy with \cref{Lemma::FuncSpaceRev::Properties} \cref{Item::FuncSpaceRev::HmsBoundsZygs}, for $m\in \N$, $s\in [0,1]$ with $m+s>0$,
we have $\HSpace{m}{s}[\Omega]\subseteq \ZygSpace{m+s}[\Omega]$.  If $\Omega$ is a bounded Lipschitz domain and $s\in (0,1)$, then we have the reverse containment as well
$\ZygSpace{m+s}[\Omega]\subseteq \HSpace{m}{s}[\Omega]$ (see \cite[Theorem 1.118 (i)]{TriebelTheoryOfFunctionSpacesIII}).  Because of this, one might hope for the reverse
inequality to the one in \cref{Lemma::FuncSpaceRev::Properties} \cref{Item::FuncSpaceRev::HmsBoundsZygs} for $s\in (0,1)$.  One can obtain such an estimate, but it requires
additional hypotheses on the vector fields.  This is discussed in \cite{StovallStreetII}.
\end{rmk}

\begin{prop}\label{Prop::FuncSpaceRev::Algebra}
The spaces $\HXSpace{W}{m}{s}[M]$, $\ZygXSpace{W}{s}[M]$, $\HSpace{m}{s}[\Omega]$, $\ZygSpace{s}[\Omega]$, $\CXomegaSpace{W}{r}[M]$, $\AXSpace{W}{x_0}{r}$, $\ComegaSpace{r}[M]$, and $\ASpace{n}{r}$
are algebras.  In fact, if $\BanachAlgebra$ denotes any one of these spaces, then
\begin{equation*}
\Norm{fg}[\BanachAlgebra]\leq C_{\BanachAlgebra} \Norm{f}[\BanachAlgebra]\Norm{g}[\BanachAlgebra].
\end{equation*}
When $\BanachAlgebra\in \{\CXomegaSpace{W}{r}[M], \AXSpace{W}{x_0}{r}, \ComegaSpace{r}[M], \ASpace{n}{r}\}$, we may take $C_{\BanachAlgebra}=1$; i.e., these spaces are Banach algebras\footnote{This remains true for the analogous spaces taking values in a Banach algebra.}.
When $\BanachAlgebra\in \{\HXSpace{W}{m}{s}[M], \ZygXSpace{W}{s}[M], \HSpace{m}{s}[\Omega], \ZygSpace{s}[\Omega]\}$, these spaces have multiplicative inverses for functions
which are bounded away from zero: if $f\in \BanachAlgebra$ with $\inf_x |f(x)|\geq c_0>0$, then $f(x)^{-1}=\frac{1}{f(x)}\in \BanachAlgebra$.
Furthermore, $\Norm{f(x)^{-1}}[\BanachAlgebra]\leq C$ where $C$ can be chosen to depend only on $\BanachAlgebra$, $c_0$, and an upper bound for $\Norm{f}[\BanachAlgebra]$.
\end{prop}
\begin{proof}
The proofs for $\HXSpace{W}{m}{s}[M]$ and $\HSpace{m}{s}[\Omega]$ are straightforward and standard, and we leave the proofs to the reader.
The results for $\ZygXSpace{W}{s}[M]$ and $\ZygSpace{s}[\Omega]$ are in \cite[\SSZygIsAlgebra]{StovallStreetI}.
The results for $\CXomegaSpace{W}{r}[M]$, $\AXSpace{W}{x_0}{r}$, $\ComegaSpace{r}[M]$, and $\ASpace{n}{r}$ are in \cite{StovallStreetIII}.
\end{proof}

\begin{rmk}\label{Rmk::FuncSpaceRev::InverseMatrix}
For $s\in (0,\infty]\cup\{\omega\}$,
suppose $A\in \ZygSpace{s}[\Omega][\M^{k\times k}]$
is such that $\inf_{t\in \Omega} |\det A(t)|>0$.  Then it follows that $A(\cdot)^{-1}\in \ZygSpace{s}[\Omega][\M^{k\times k}]$; where we write $A(\cdot)^{-1}$ for the function
$t\mapsto A(t)^{-1}$.
Indeed, for $s\in (0,\infty]$, this follows from \cref{Prop::FuncSpaceRev::Algebra} using the cofactor representation of $A(\cdot)^{-1}$.  For $s=\omega$, this is standard.
When $s\in (0,\infty)$, $\ZygNorm{A(\cdot)^{-1}}{s}[\Omega]$ can be bounded in terms of $s$, $k$, $n$, a lower bound for $\inf_{t\in \Omega} |\det A(t)|>0$, and an upper bound for
$\ZygNorm{A}{s}[\Omega]$.
\end{rmk}

\begin{lemma}\label{Lemma::FuncSpaceRev::Composition}
Let $D_1,D_2>0$, $s_1>0$, $s_2\geq s_1$, $s_2>1$, $f\in \ZygSpace{s_1}[B_{\R^n}(D_1)]$, $g\in \ZygSpace{s_2}[B_{\R^m}(D_2)][\R^n]$
with $g(B_{\R^m}(D_2))\subseteq B_{\R^n}(D_1)$.  Then, $f\circ g\in \ZygSpace{s_1}[B_{\R^m}(D_2)]$ and
$\ZygNorm{f\circ g}{s_1}[B_{\R^m}(D_2)]\leq C \ZygNorm{f}{s_1}[B_{\R^n}(D_1)]$, where $C$ can be chosen to depend only on $s_1$, $s_2$, $D_1$, $D_2$,
$m$, $n$, and an upper bound for $\ZygNorm{g}{s_2}[B_{\R^m}(D_2)]$.
\end{lemma}
\begin{proof}This is proved in \cite{StovallStreetII}.\end{proof}

\begin{lemma}\label{Lemma::FuncSpaceRev::ComposeAnal}
Let $\eta_1,\eta_2>0$, $n_1,n_2\in \N$, and let $\BanachSpace$ be a Banach space.
Suppose $f\in \ASpace{n_1}{\eta_1}[\BanachSpace]$, $g\in \ASpace{n_2}{\eta_2}[\R^{n_1}]$ with
$\ANorm{g}{n_2}{\eta_2}[\R^{n_1}]\leq \eta_1$.  Then, $f\circ g\in \ASpace{n_2}{\eta_2}[\BanachSpace]$
with $\ANorm{f\circ g}{n_2}{\eta_2}\leq \ANorm{f}{n_1}{\eta_1}$.
\end{lemma}
\begin{proof}
This is immediate from the definitions.
\end{proof}

\begin{lemma}\label{Lemma::FuncSpaceRev::DerivOfAnal}
Fix $0<\eta_2<\eta_1$, and suppose $f\in \ASpace{n}{\eta_1}[\BanachSpace]$, where $\BanachSpace$ is a Banach space.
Then, for each $j=1,\ldots, n$, $\diff{t_j} f(t)\in \ASpace{n}{\eta_2}[\BanachSpace]$ and
$\ANorm{\diff{t_j} f}{n}{\eta_2} \leq C \ANorm{f}{n}{\eta_1}$, where $C$ can be chosen to depend only on $\eta_1$ and $\eta_2$.
\end{lemma}
\begin{proof}
Without loss of generality, we prove the result for $j=1$.  We let $e_1$ denote the first standard basis element:  $e_1=(1,0,\ldots, 0)\in \N^n$.
Suppose $f(t) = \sum c_{\alpha} \frac{t^{\alpha}}{\alpha!}$.  Then, $\diff{t_1} f(t) = \sum_{\alpha_1>0} c_\alpha \frac{t^{\alpha-e_1}}{(\alpha-e_1)!}$.
Hence,
\begin{equation*}
\BANorm{\diff{t_1} f}{n}{\eta_2} = \sum_{\alpha_1>0} \frac{|c_\alpha|}{(\alpha-e_1)!} \eta_2^{|\alpha-e_1|}
=\sum_{\alpha} \frac{|c_\alpha|}{\alpha!} \eta_1^{|\alpha|} \mleft(\frac{\eta_2}{\eta_1}\mright)^{|\alpha|} \frac{\alpha_1}{\eta_1}
\leq \mleft( \sup_{\alpha}\mleft(\frac{\eta_2}{\eta_1}\mright)^{|\alpha|} \frac{\alpha_1}{\eta_1} \mright) \ANorm{f}{n}{\eta_1},
\end{equation*}
completing the proof.
\end{proof}

\begin{prop}\label{Prop::FuncSpaceRev::CompEuclid}
Let $Y_1,\ldots, Y_N$ be $C^1$ vector fields on an open ball $B\subseteq \R^n$.  Suppose $Y_1,\ldots, Y_N$ span the tangent space at every point
in the sense that for $1\leq j\leq n$,
\begin{equation*}
\diff{t_j} =\sum_{k=1}^N b_j^k Y_k,\quad b_j^k\in \CSpace{B}.
\end{equation*}
Fix $s\in (0,\infty]\cup\{\omega\}$ and suppose $Y_k\in \ZygSpace{s-1}[B][\R^n]$, $b_j^k\in \ZygSpace{s-1}[B]$, $\forall j,k$.
Then, $\ZygSpace{s}[B]=\ZygXSpace{Y}{s}[B]$.
Here we use the convention that for $s\in (-1,0]$, $\ZygSpace{s}[B]:=\HSpace{0}{(s+1)/2}[B]$.
\end{prop}
\begin{proof}
The case $s\in (0,\infty]$ is contained in \cite[Proposition 8.12]{StovallStreetI}, while the case $s=\omega$ is discussed in \cite{StovallStreetIII}.  The case $s=\omega$
is part of a more general result due to Nelson \cite[Theorem 2]{NelsonAnalyticVectors}.
\cite{StovallStreetI,StovallStreetIII} also contain quantitative versions of this result.
\end{proof} 
%
%
%
\section{Proofs of Corollaries}

	\subsection{Optimal Smoothness}\label{Section::CorProof::Smoothness}
In this section, we prove \cref{Thm::QualE::LocalThm,Thm::QualE::GlobalThm}, and describe how \cref{Thm::ResSmooth::Real::Local,Thm::ResSmooth::Real::Global,Thm::QualComplex::LocalThm,Thm::QualComplex::GlobalThm} are consequences of \cref{Thm::QualE::LocalThm,Thm::QualE::GlobalThm}.

\begin{proof}[Proof of \cref{Thm::QualE::LocalThm}]
\Cref{Item::QualE::Local::Diffeo}$\Rightarrow$\cref{Item::QualE::Local::Basis}:  Suppose the conditions of \cref{Item::QualE::Local::Diffeo} hold and without loss of generality
we may assume $0\in U$ and $\Phi(0)=x_0$; reorder the vector fields as in \cref{Item::QualE::Local::Basis}.
Because $\dim \XVS_{x_0}=r$ and $\dim \LVS_{x_0}=n+r$, we have
\begin{equation*}
\Span_{\R}\{\Phi^{*}X_1(0,0),\ldots, \Phi^{*}X_r(0,0)\}=\Span_{\R}\left\{\diff{t_1},\ldots, \diff{t_r}\right\},
\end{equation*}
\begin{equation*}
\Span_{\C}\{\Phi^{*}X_1(0,0),\ldots, \Phi^{*}X_r(0,0), \Phi^{*}L_1(0,0),\ldots, \Phi^{*}L_n(0,0)\}=\Span_{\C}\left\{
\diff{t_1},\ldots, \diff{t_r},\diff{\zb[1]},\ldots, \diff{\zb[n]}
\right\}.
\end{equation*}
Writing $X_{K_0}$ for the column vector of vector fields $[X_{1},\ldots, X_r]^{\transpose}$ and $L_{J_0}$ for the column vector $[L_{1},\ldots, L_n]^{\transpose}$
and using the hypotheses of \cref{Item::QualE::Local::Diffeo}, we may write
\begin{equation*}
\begin{bmatrix}
\Phi^{*} X_{K_0}\\
\Phi^{*} L_{J_0}
\end{bmatrix}
= B
\begin{bmatrix}
\diff{t}\\
\diff{\zb}
\end{bmatrix},
\end{equation*}
where $B\in \ZygSpace{s+1}[U][\M^{(n+r)\times (n+r)}]$ is such that $B(0,0)$ is invertible.
Letting $U_0\subseteq U$ be a sufficiently small open ball centered at $(0,0)$, we have that $|\det B(t,z)|$ is bounded away from $0$
on $U_0$.
Thus, on $U_0$, $B$ is invertible and $B(\cdot)^{-1}\in \ZygSpace{s+1}[U_0][\M^{(n+r)\times (n+r)}]$ (see \cref{Rmk::FuncSpaceRev::InverseMatrix});
and we have
\begin{equation}\label{Eqn::QualE::1imp2::matrix}
B^{-1}\begin{bmatrix}
\Phi^{*} X_{K_0}\\
\Phi^{*} L_{J_0}
\end{bmatrix}
=
\begin{bmatrix}
\diff{t}\\
\diff{\zb}
\end{bmatrix}.
\end{equation}
Thus, for $x\in \Phi(U_0)$,
$$\dim \LVS_x\geq \dim \Span_{\C} \{ X_1(x),\ldots, X_r(x), L_1(x),\ldots, L_n(x)\} =\dim \Span_{\C}\mleft\{ \diff{t_1},\ldots, \diff{t_r}, \diff{\zb[1]},\ldots, \diff{\zb[n]}\mright\}=n+r.$$
The hypothesis \cref{Item::QualE::Local::Diffeo} implies for $x\in \Phi(U)=V$,
$$\dim \LVS_x  = \dim \Span_{\C} \{ X_1(x),\ldots, X_q(x), L_1(x),\ldots, L_m(x)\}\leq \dim \Span_{\C}\mleft\{ \diff{t_1},\ldots, \diff{t_r}, \diff{\zb[1]},\ldots, \diff{\zb[n]}\mright\}=n+r.$$
This shows that the map $x\mapsto \dim \LVS_x$, $\Phi(U_0)\rightarrow \N$ is the constant function $n+r$.

Since
$\Phi^{*}\Zh_j\in \Span_{\C}\{\diff{t_1},\ldots, \diff{t_r}, \diff{\zb[1]},\ldots, \diff{\zb[n]}\}$ we can think of $\Phi^{*}\Zh_j$ as a function taking values in $\C^{n+r}$.  We have $\Phi^{*} \Zh_j\in \ZygSpace{s+1}[U][\C^{n+r}]$, and therefore
$[\Phi^{*} \Zh_j, \Phi^{*} \Zh_k]\in \ZygSpace{s}[U][\C^{n+r}]$ and it follows from \cref{Eqn::QualE::1imp2::matrix} and \cref{Prop::FuncSpaceRev::Algebra}
that
\begin{equation}\label{Eqn::EqualE::1imp2::1}
[\Phi^{*}\Zh_j, \Phi^{*} \Zh_k] = \sum_{l=1}^{n+r} \ct_{j,k}^{1, l} \Phi^{*}\Zh_l, \quad \ct_{j,k}^{1,l}\in \ZygSpace{s}[U_0].
\end{equation}
Similarly, since $[\Phi^{*} \Zh_j, \Phi^{*} \Zhb[k]]\in \ZygSpace{s}[U][\C^{2n+r}]$, we have
\begin{equation}\label{Eqn::EqualE::1imp2::2}
[\Phi^{*} \Zh_j, \Phi^{*} \Zhb[k]]= \sum_{l=1}^{n+r} \ct_{j,k}^{2,l} \Phi^{*} \Zh_l + \sum_{l=1}^{n+r} \ct_{j,k}^{3,l} \Phi^{*} \Zhb[l], \quad \ct_{j,k}^{2,l},\ct_{j,k}^{3,l}\in \ZygSpace{s}[U_0].
\end{equation}
Furthermore, since $\Phi^{*}Y_j(t,z)\in \Span_{\C}\left\{\diff{t_1},\ldots, \diff{t_r}, \diff{\zb[1]},\ldots, \diff{\zb[n]}\right\}$ and $\Phi^{*}Y_j\in \ZygSpace{s+1}[U][\C^{r+n}]$,
\cref{Eqn::QualE::1imp2::matrix} and \cref{Prop::FuncSpaceRev::Algebra} imply
\begin{equation}\label{Eqn::EqualE::1imp2::3}
\Phi^{*}Y_j = \sum_{l=1}^{n+r} \bt_j^l \Phi^{*}\Zh_l, \quad \bt_j^l\in \ZygSpace{s+1}[U_0].
\end{equation}
\Cref{Prop::FuncSpaceRev::CompEuclid}, combined with \cref{Eqn::QualE::1imp2::matrix}, shows
$\ct_{j,k}^{a,l}\in \ZygSpace{s}[U_0]=\ZygXSpace{\Phi^{*}X, \Phi^{*}L}{s}[U_0]$ and $\bt_j^l \in \ZygSpace{s+1}[U_0]=\ZygXSpace{\Phi^{*}X, \Phi^{*} L}{s+1}[U_0]$, $\forall j,k,l,a$.
Let $\ch_{j,k}^{a,l}:=\ct_{j,k}^{a,l}\circ \Phi^{-1}$, $b_j^l:=\bt_j^l\circ \Phi^{-1}$, and $V_0:=\Phi(U_0)$.  \Cref{Prop::FuncMan::DiffeoInv} shows
$\ch_{j,k}^{a,l}\in \ZygXSpace{X,L}{s}[V_0]$ and $b_j^l\in \ZygXSpace{X,L}{s+1}[V_0]$.
Pushing forward \cref{Eqn::EqualE::1imp2::1}, \cref{Eqn::EqualE::1imp2::2}, and \cref{Eqn::EqualE::1imp2::3} via $\Phi$
gives
\begin{equation*}
[\Zh_j, \Zh_k] =\sum_{l=1}^{n+r} \ch_{j,k}^{1,l}\Zh_l, \quad [Z_j, \Zhb[k]] = \sum_{l=1}^{n+r} \ch_{j,k}^{2,l} \Zh_l + \sum_{l=1}^{n+r} \ch_{j,k}^{3,l} \Zhb[l],
\quad Y_j = \sum_{l=1}^{n+r} b_j^l \Zh_l.
\end{equation*}
Along with the above remarks on $\ch_{j,k}^{a,l}$ and $b_j^l$, this completes the proof of \cref{Item::QualE::Local::Basis} with $V$ replaced by $V_0$.

\Cref{Item::QualE::Local::Basis}$\Rightarrow$\cref{Item::QualE::Local::Commute}:
Suppose \cref{Item::QualE::Local::Basis} holds.  First, we wish to show that
\begin{equation}\label{Eqn::EqualE::2imp3::1}
[Z_j,Z_k]=\sum_{l=1}^{m+q} c_{j,k}^{1,l} Z_l, \quad c_{j,k}^{1,l}\in \ZygXSpace{X,L}{s}[V].
\end{equation}
$Z_j$ and $Z_k$ are each either of the form $\Zh_l$ or $Y_l$ for some $l$ (where $\Zh_l$ and $Y_l$ are as in \cref{Item::QualE::Local::Basis}).
When $Z_j$ and $Z_k$ are both of the form $\Zh_l$ for some $l$, \cref{Eqn::EqualE::2imp3::1} is contained in \cref{Item::QualE::Local::Basis}.
We address the case when $Z_j=Y_{l_1}$, $Z_k=Y_{l_2}$ for some $l_1,l_2$.  The remaining case (when $Z_j=\Zh_{l_1}$ and $Z_k=Y_{l_2}$)
is similar, and we leave it to the reader.
We have,
\begin{equation*}
[Z_j,Z_k]=[Y_{l_1},Y_{l_2}] =\left[ \sum_{l_3} b_{l_1}^{l_3} \Zh_{l_3}, \sum_{l_4} b_{l_2}^{l_4} \Zh_{l_4} \right]
=\sum_{l_3,l_4} b_{l_1}^{l_3} b_{l_2}^{l_4} [\Zh_{l_3}, \Zh_{l_4}] + \sum_{l_3,l_4} b_{l_1}^{l_3} (\Zh_{l_3} b_{l_2}^{l_4}) \Zh_{l_4} - \sum_{l_3,l_4} b_{l_2}^{l_4} (\Zh_{l_4} b_{l_1}^{l_3}) \Zh_{l_3}.
\end{equation*}
Using \cref{Lemma::FuncSpaceRev::Properties} \cref{Item::FuncSpaceRev::Properties::Deriv} and \cref{Prop::FuncSpaceRev::Algebra}, we have
$b_{l_1}^{l_3} (\Zh_{l_3} b_{l_2}^{l_4}), b_{l_2}^{l_4} (\Zh_{l_4} b_{l_1}^{l_3})\in \ZygXSpace{X,L}{s}[V]$.
Also, we have
\begin{equation*}
\sum_{l_3,l_4} b_{l_1}^{l_3} b_{l_2}^{l_4} [\Zh_{l_3}, \Zh_{l_4}]  = \sum_{l_3,l_4}\sum_{l_5} b_{l_1}^{l_3} b_{l_2}^{l_4} \ch_{l_3,l_4}^{1,l_5}\Zh_{l_5},
\end{equation*}
and by \cref{Prop::FuncSpaceRev::Algebra}, $b_{l_1}^{l_3} b_{l_2}^{l_4} \ch_{l_3,l_4}^{1,l_5}\in \ZygXSpace{X,L}{s}[V]$.  Combining the above remarks, we have
\begin{equation*}
[Z_j,Z_k]=\sum_{l=1}^{n+r} c_{j,k}^{1,l} \Zh_l, \quad c_{j,k}^{1,l}\in \ZygXSpace{X,L}{s}[V].
\end{equation*}
Since each $\Zh_l$ is of the form $Z_{l'}$ for some $l'$, \cref{Eqn::EqualE::2imp3::1} follows.  A similar proof shows
\begin{equation*}
[Z_j, \Zb[k]]=\sum_{l} c_{j,k}^{2,l} Z_l + \sum_{l} c_{j,k}^{3,l} \Zb[l], \quad c_{j,k}^{2,l}, c_{j,k}^{3,l}\in \ZygXSpace{X,L}{s}[V],
\end{equation*}
and we leave the details to the reader.  This completes the proof of \cref{Item::QualE::Local::Commute}.

\Cref{Item::QualE::Local::Commute}$\Rightarrow$\cref{Item::QualE::Local::Diffeo}:  This is a consequence of \cref{Thm::Results::MainThm}; and we include a few remarks
on this.  First, a choice of $\eta,\delta_0>0$ as in the hypotheses of \cref{Thm::Results::MainThm} always exist; see \cref{Lemma::MoreAssume::ExistEtaDelta0}.
A choice of $J_0$, $K_0$, and $\zeta>0$ as in the hypotheses also always exist; see \cref{Rmk::AppendWedge::AboutJ0K0}.  We take $\xi>0$ so small
$B_{X,L}(x_0,\xi)\subseteq V$.

First we address the case $s\in (1,\infty]$.  In this case,
pick $s_0\in (1,s]\setminus \{\infty\}$ (the choice of $s_0$ does not matter).  We have, directly from the definitions
$$c_{j,l}^{a,l}\in \ZygXSpace{X,L}{s}[V]\subseteq \ZygXSpace{X,L}{s}[B_{X,L}(x_0,\xi)]\subseteq \ZygXSpace{X_{K_0},L_{J_0}}{s}[B_{X_{K_0},L_{J_0}}(x_0,\xi)]\subseteq \ZygXSpace{X_{K_0},L_{J_0}}{s_0}[B_{X_{K_0},L_{J_0}}(x_0,\xi)].
$$
Thus, all of the hypotheses of \cref{Thm::Results::MainThm} hold for this choice of $s_0$.  The map guaranteed by \cref{Thm::Results::MainThm} satisfies the conclusions
of \cref{Item::QualE::Local::Diffeo} and this completes the proof in the case $s\in (1,\infty]$.

When $s=\omega$, we wish to apply \cref{Thm::Results::MainThm} in the case $s_0=\omega$.  There is a slight discrepancy between the hypotheses of
\cref{Thm::Results::MainThm} and \cref{Item::QualE::Local::Commute}.  Namely, we are currently assuming $c_{j,k}^{a,l}\in \CXomegaSpace{X,L}{r_0}[V]$
for some $r_0>0$,
while \cref{Thm::Results::MainThm} assumes $c_{j,k}^{a,l}\in \AXSpace{X_{K_0},L_{J_0}}{x_0}{\eta}$ and $c_{j,k}^{a,l}$ is continuous near $x_0$.
However, $c_{j,k}^{a,l}\in \CXomegaSpace{X,L}{r_0}[V]$ clearly implies $c_{j,k}^{a,l}$ is continuous near $x_0$, and
using \cref{Lemma::FuncSpaceRev::Properties} \cref{Item::FuncSpaceRev::Properites::ContainRA} we have
$\CXomegaSpace{X,L}{\eta}\subseteq \CXomegaSpace{X_{K_0},L_{J_0}}{\eta}\subseteq \AXSpace{X_{K_0},L_{J_0}}{x_0}{\eta}$, so by shrinking $\eta$
so that $\eta\leq r_0$, these hypotheses follow.  With these remarks, \cref{Thm::Results::MainThm} applies to yield the coordinate chart $\Phi$
as in that theorem, which satisfies all the conclusions of \cref{Item::QualE::Local::Diffeo}.  This completes the proof.
\end{proof}

Before we prove \cref{Thm::QualE::GlobalThm}, we require two lemmas.
\begin{lemma}\label{Lemma::QualEPf::RecogSmooth}
Fix $s\in (0,\infty]\cup\{\omega\}$ and suppose $M_1$ and $M_2$ are $\ZygSpace{s+2}$ manifolds.  Let $Z_1,\ldots, Z_N$ be complex $\ZygSpace{s+1}$ vector fields on $M_1$
such that $Z_1,\ldots Z_N, \Zb[1],\ldots, \Zb[N]$ span the complexified tangent space to $M_1$ at every point.  Let $\Psi:M_1\rightarrow M_2$ be a $C^2$ diffeomorphism
such that $\Psi_{*} Z_j$ is a $\ZygSpace{s+1}$ vector field, $\forall 1\leq j\leq N$.  Then, $\Psi$ is a $\ZygSpace{s+2}$ diffeomorphism.
\end{lemma}
\begin{proof}
By taking real an imaginary parts, it suffices to prove the result in the case $Z_1,\ldots, Z_N$ are real and span the tangent space at every point.
In the case $s\in (0,\infty]$, this is proved in \cite{StovallStreetII}.  In the case $s=\omega$, this is proved in \cite{StovallStreetIII}.
\end{proof}

\begin{lemma}\label{Lemma::QualE::RealVFsInSpan}
$\mleft(T_{(t_0,z_0)}(\R^r\times \C^n)\mright) \cap \Span_{\C}\mleft\{ \diff{t_1},\ldots, \diff{t_r}, \diff{\zb[1]},\ldots, \diff{\zb[n]}\mright\}= \Span_{\R}\mleft\{ \diff{t_1},\ldots, \diff{t_r}\mright\}$.
\end{lemma}
\begin{proof}
This is immediate.
\end{proof}

\begin{proof}[Proof of \cref{Thm::QualE::GlobalThm}]
\Cref{Item::QualE::EMfld}$\Rightarrow$\cref{Item::QualE::ThreeEquiv}:  The inverses of the coordinate charts from the atlas given in \cref{Item::QualE::EMfld} satisfy
the conditions in \cref{Thm::QualE::LocalThm} \cref{Item::QualE::Local::Diffeo} (this uses \cref{Lemma::QualE::RealVFsInSpan}); and so \cref{Item::QualE::ThreeEquiv} follows.

\Cref{Item::QualE::ThreeEquiv}$\Rightarrow$\cref{Item::QualE::EMfld}:  Assume that \cref{Item::QualE::ThreeEquiv} holds.
Using the characterization in \cref{Thm::QualE::LocalThm} \cref{Item::QualE::Local::Commute}, we have that $x\mapsto\dim \LVS_x$, $M\rightarrow \N$ is locally constant,
and since $M$ is connected, $x\mapsto \dim\LVS_x$, $M\rightarrow \N$ is constant.  By the discussion in \cref{Section::CommentsAssump} we also have $x\mapsto \dim \XVS_x$, $M\rightarrow \N$
is constant.  Set $r:=\dim \XVS_x$ and $n+r:=\dim \LVS_x$ (so that $n$ and $r$ do not depend on $x$, by the above discussion).
Now, we use the characterization given in \cref{Thm::QualE::LocalThm} \cref{Item::QualE::Local::Diffeo}.
Thus, for each $x\in M$, there is a neighborhood $V_x\subseteq M$ of $x$ and a $C^2$ diffeomorphism $\Phi_x:U_x\rightarrow V_x$, where $U_x\subseteq \R^r\times \C^n$ is open,
such that $\forall (t,z)\in U_x$,  $1\leq k\leq q$,
$1\leq j\leq m$,
$$\Phi_x^{*}X_k(t,z)\in \Span_{\R}\left\{\diff{t_1},\ldots, \diff{t_r}\right\},\quad \Phi_x^{*}L_j(t,z)\in \Span_{\C}\left\{\diff{t_1},\ldots, \diff{t_r},\diff{\zb[1]},\ldots, \diff{\zb[n]}\right\},$$
and $\Phi_x^{*}X_k\in \ZygSpace{s+1}[U_x][\R^r]$, $\Phi_x^{*}L_j\in \ZygSpace{s+1}[U_x][\C^{r+n}]$.  Our desired atlas is $\{ (\Phi_x^{-1}, V_x) : x\in M\}$--once we show this is a $\ZygSpace{s+2}$ E-atlas,
\cref{Item::QualE::EMfld} will follow.
For $x,y\in M$, set $\Psi_{x,y}:=\Phi_y^{-1}\circ\Phi_x:\Phi_x^{-1}(V_y\cap V_x)\rightarrow U_y$; we wish to show that $\Psi_{x,y}$ is a $\ZygSpacemap{s+2}$ E-map.
Note that
\begin{equation}\label{Eqn::QualEPf::dPsi}
d\Psi_{x,y}(t,z) (\Phi_x^{*} X_k)(t,z) = (\Phi_y^{*} X_k)(\Psi_{x,y}(t,z)), \quad d\Psi_{x,y}(t,z) (\Phi_x^{*} L_j)(t,z) = (\Phi_y^{*} L_j)(\Psi_{x,y}(t,z)), \quad \forall j,k.
\end{equation}
In other words,
\begin{equation}\label{Eqn::QualEPf::PsiStar}
(\Psi_{x,y})_{*} \Phi_x^{*} X_k = \Phi_y^{*} X_k, \quad (\Psi_{x,y})_{*} \Phi_x^{*} L_j = \Phi_y^{*} L_j, \quad \forall j,k.
\end{equation}

Since $\dim \LVS_y=n+r$, $\forall y\in M$, we have $\forall (t,z)\in U_x$,
\begin{equation}\label{Eqn::QualEPf::SpanPullback}
	\Span_{\C} \{\Phi_x^{*}X_1(t,z),\ldots, \Phi_x^{*}X_q(t,z),\Phi_x^{*}L_1(t,z),\ldots ,\Phi_x^{*}L_m(t,z)  \}= \Span_{\C}\left\{\diff{t_1},\ldots, \diff{t_r},\diff{\zb[1]},\ldots, \diff{\zb[n]}\right\}.
\end{equation}

Combining \cref{Eqn::QualEPf::SpanPullback} and \cref{Eqn::QualEPf::dPsi} shows that $\Psi_{x,y}$ is an E-map.
\Cref{Eqn::QualEPf::SpanPullback} implies
$$\Phi_x^{*}X_1(t,z),\ldots, \Phi_x^{*}X_q(t,z),\Phi_x^{*}L_1(t,z),\ldots,\Phi_x^{*}L_m(t,z),\overline{\Phi_x^{*}L_1}(t,z),\ldots,\overline{\Phi_x^{*}L_m}(t,z)$$
span the complexified tangent space
at every point of $U_x$.  Since these vector fields are also $\ZygSpace{s+1}$ by hypothesis,
 \cref{Eqn::QualEPf::PsiStar} and \cref{Lemma::QualEPf::RecogSmooth} show that $\Psi_{x,y}$ is $\ZygSpacemap{s+2}$.
This completes the proof of \cref{Item::QualE::EMfld}.

\Cref{Item::QualE::Commute}$\Rightarrow$\cref{Item::QualE::ThreeEquiv}:  This is obvious, and holds for $s\in (0,\infty]\cup\{\omega\}$.

\Cref{Item::QualE::EMfld}$\Rightarrow$\cref{Item::QualE::Commute}, for $s\in (0,\infty]$: Assuming that
\cref{Item::QualE::EMfld} holds (where $M$ is an E-manifold of dimension $(r,n)$), a simple partition of unity argument shows that we may write
$[Z_j,Z_k]=\sum_{l=1}^{m+q} c_{j,k}^{1,l} Z_l$ and $[Z_j,\Zb[k]]=\sum_{l=1}^{m+q} c_{j,k}^{2,l} Z_l + \sum_{l=1}^{m+q} c_{j,k}^{3,l} \Zb[l]$,
where $c_{j,k}^{a,l}:M\rightarrow \C$ and $c_{j,k}^{a,l}$ are locally in $\ZygSpace{s}$.
We wish to show $\forall x_0\in M$, $\exists V\subseteq M$ open with $x_0\in V$ and $c_{j,k}^{a,l}\big|_{V} \in \ZygXSpace{X,L}{s}[V]$.
Fix $x_0\in M$ and let $W\subseteq M$ be a neighborhood of $x_0$ such that there is a $\ZygSpace{s+2}$ diffeomorphism $\Phi:B_{\R^r\times \C^n}(1)\rightarrow W$ with
$\Phi(0)=x_0$. 
Let $Y_1,\ldots, Y_{q+2m}$ denote the list $\Phi^{*}X_1,\ldots, \Phi^{*}X_r, 2\Phi^{*}\Real(L_1),\ldots, 2\Phi^{*}\Real(L_m), 2\Phi^{*}\Imag(L_1),\ldots, 2\Phi^{*}\Imag(L_m)$.
$Y_1,\ldots, Y_{q+2m}$ are $\ZygSpace{s+1}$ vector fields on $B_{\R^r\times \C^n}(1)$ and span the tangent space at every point.  We conclude
$Y_1,\ldots, Y_{q+2m}$ satisfy all the hypotheses of \cref{Prop::FuncSpaceRev::CompEuclid} with $B:=B_{\R^r\times \C^n}(1/2)$.
Thus, by \cref{Prop::FuncSpaceRev::CompEuclid}, $c_{j,k}^{a,l}\circ\Phi \in \ZygSpace{s}[B]=\ZygXSpace{Y}{s}[B]$.  \Cref{Prop::FuncMan::DiffeoInv} shows
$c_{j,k}^{a,l}\in \ZygXSpace{X,L}{s}[\Phi(B)]$, completing the proof with $V=\Phi(B)$.



Finally, we turn to the uniqueness claimed in the theorem; that under the equivalent hypotheses \cref{Item::QualE::EMfld} and \cref{Item::QualE::ThreeEquiv},
the E-manifold structure given in \cref{Item::QualE::EMfld} is unique.  Indeed, suppose there are two such structures on $M$.
Under these conditions, the identity map $M\rightarrow M$ is $\ZygSpacemap{s+2}$ by \cref{Lemma::QualEPf::RecogSmooth} (here we have applied
\cref{Lemma::QualEPf::RecogSmooth} with the vector fields $X_1,\ldots, X_q, L_1,\ldots, L_m$).
That the identity map is a $\ZygSpacemap{s+2}$ E-map follows from \cref{Lemma::EMfld::RecongnizeEMap}.
It follows that the identity map is a $\ZygSpace{s+2}$ E-diffeomorphism, as claimed.
\end{proof}

\begin{proof}[Proof of \cref{Thm::ResSmooth::Real::Local,Thm::ResSmooth::Real::Global}]
In the setting of \cref{Thm::ResSmooth::Real::Local,Thm::ResSmooth::Real::Global}, because $W_1,\ldots, W_N$ span the tangent space at every point, we have
$\dim \Span_{\R}\{W_1(x),\ldots, W_N(x)\} = \dim M=n$, $\forall x$; in particular, the map $x\mapsto \dim \Span_{\R}\{W_1(x),\ldots, W_N(x)\}$ is constant.
With this in mind, \cref{Thm::ResSmooth::Real::Local,Thm::ResSmooth::Real::Global} are immediate consequences of the case $m=0$ of \cref{Thm::QualE::LocalThm,Thm::QualE::GlobalThm}.
\end{proof}

\begin{proof}[Proof of \cref{Thm::QualComplex::LocalThm,Thm::QualComplex::GlobalThm}]
In the setting of \cref{Thm::QualComplex::LocalThm,Thm::QualComplex::GlobalThm}, we have $\dim \LVS_\zeta=n$, $\forall \zeta\in M$.  Thus, the map $\zeta\mapsto \dim \LVS_\zeta$ is constant.
Also, in the context of \cref{Thm::QualComplex::GlobalThm}, E-maps are holomorphic (and E-diffeomorphisms are biholomorphisms);  this is because complex manifolds embed into E-manifolds
as a \textit{full} sub-category (see \cref{Rmk::EMfld::FullSubcategory}).  With these remarks in hand, \cref{Thm::QualComplex::LocalThm,Thm::QualComplex::GlobalThm} are immediate consequences of the case $q=0$
of \cref{Thm::QualE::LocalThm,Thm::QualE::GlobalThm}.
\end{proof}

	\subsection{Sub-E geometry}
In this section, we prove \cref{Thm::SubEGeom}.  In light of \cref{Rmk::SubE::GeneralsizesSubHR}, \cref{Thm::Results::NSW} is a special case
of \cref{Thm::SubEGeom}.  \Cref{Thm::ResSubH} is also a special case of \cref{Thm::SubEGeom}:
\begin{proof}[Proof of \cref{Thm::ResSubH}]
In light of \cref{Rmk::SubE::GeneralsizesSubHR}, the hypotheses of \cref{Thm::ResSubH} imply the hypotheses of \cref{Thm::SubEGeom}.
The main issue in seeing \cref{Thm::ResSubH} as a special case of \cref{Thm::SubEGeom} is that the definitions of $\rho_H$ in the two theorems are not obviously the same.
However, if $M$ is a complex manifold and $f(t,z):B_{\R\times \C}(1/2)\rightarrow M$ is an E-map, then $f$ must be constant in $t$
and is therefore a holomorphic map $B_{\C}(1/2)\rightarrow M$.
Indeed, $df(t,z)\diff{t}$ is both a $T^{0,1}_{f(t,z)}$ tangent vector and a real tangent vector, and we conclude $df(t,z)\diff{t}\equiv 0$.
Using this, it is easy to see that the definition of $\rho_H$ in  \cref{Thm::SubEGeom} is the same as the definition of $\rho_H$
in \cref{Thm::ResSubH} when $M$ is a complex manifold.  
\end{proof}
The rest of this section is devoted to the proof of \cref{Thm::SubEGeom}.

\begin{lemma}\label{Lemma::PfSubE::rhoFCont}
$\lim_{y\rightarrow x} \rho_F(x,y)=0$, where the limit is taken
in the usual topology on $M$--recall, $M$ is a manifold and therefore comes equipped with a topology which we are referring to as the ``usual topology.''
\end{lemma}
\begin{proof}
Fix $\epsilon>0$; we wish to find a neighborhood $N\subseteq M$ of $x$
such that $\forall y\in N$, $\rho_F(x,y)<\epsilon$.
Reorder $W_1,\ldots, W_{2m+q}$ so that $W_1(x),\ldots, W_{2n+r}(x)$
form a basis for $T_x M$ and set
$$\Psi(t_1,\ldots, t_{2n+r}):=e^{t_1 W_1+\cdots+t_{2n+r}W_{2n+r}}x.$$
Since $\diff{t_j}\big|_{t=0} \Psi(t) = W_j(x)$ it follows from
the inverse function theorem that there exists an open
neighborhood $U$ of $0\in \R^{2n+r}$ such that
$\Psi(U)$ is open and $\Psi:U\rightarrow \Psi(U)$ is a $C^\infty$
diffeomorphism.
Set $0<c\leq (32(2n+r))^{-1/2}$ and let
$B:=\{t=(t_1,\ldots, t_{2n+r}) : |t_j|<c\epsilon^{d_j}\}$;
take $c$ so small that $B\subseteq U$ and set $N=\Psi(B)$.  $N$ is clearly open since $\Psi$ is
a diffeomorphism.  Thus, it remains to show
$N\subseteq B_F(x,\epsilon)$.
Take $y\in N$, so that there exists $t\in B$ with $y=\Psi(t)$.
Define $f:B_{\R}(1/2)\rightarrow M$ by
\begin{equation*}
f(s):=e^{4s(t_1W_1+\cdots+t_{2n+r}W_{2n+r})}x,
\end{equation*}
so that $f\in C^\infty$, $f(0)=x$, $f(1/4)=y$, and
\begin{equation*}
f'(s) = \sum_{j=1}^{2n+r} 4t_j W_j(f(s)) =\sum_{j=1}^{2n+r} 4\frac{t_j}{\epsilon^{d_j}} \epsilon^{d_j} W_j(f(s)).
\end{equation*}
Since
$$\sum_{j=1}^{2n+r} \left(4\frac{t_j}{\epsilon^{d_j}}\right)^2 \leq \sum_{j=1}^{2n+r} \frac{1}{2(2n+r)}\leq \frac{1}{2}<1,$$
it follows that $\rho_F(x,y)<\epsilon$, completing the proof.
\end{proof}

\begin{lemma}\label{Lemma::PfSubE::TopsAreTheSame}
The metric topology induced by $\rho_F$ is the same as the usual topology on $M$.
\end{lemma}
\begin{proof}
\Cref{Lemma::PfSubE::rhoFCont} shows that the usual topology on $M$ is finer than the metric topology induced $\rho_F$.
That the metric topology induced by $\rho_F$ is finer than the usual topology is a straightforward application of the Phragm\'en-Lindel\"of Theorem; and we leave the details to the reader.\footnote{Another way to prove \cref{Lemma::PfSubE::TopsAreTheSame} is as follows.  We see
below in the proof of  \cref{Thm::SubEGeom} \cref{Item::SubE::EasyMetrics} that $\rho_S\leq \rho_F$--and the proof of this inequality does not use \cref{Lemma::PfSubE::TopsAreTheSame}.
Thus the metric topology induced by $\rho_F$ is finer than the metric topology induced by $\rho_S$.  That the metric topology induced by $\rho_S$ is finer than the usual topology
follows from \cite[\SSFinerTopology]{StovallStreetI}.  Alternatively, one can easily adapt the proof of \cite[\SSFinerTopology]{StovallStreetI} to directly prove that the metric topology induced
by $\rho_F$ is finer than the usual topology.}
\end{proof}

\begin{proof}[Proof of  \cref{Thm::SubEGeom} \cref{Item::SubE::EasyMetrics}]
We begin by showing $\rho_F\leq \rho_S$.  Suppose $\rho_S(x,y)<\delta$.  Then, there exists $\gamma:[0,1]\rightarrow M$, $\gamma(0)=x$, $\gamma(1)=y$,
$\gamma'(t)=\sum a_j(t) \delta^{d_j} W_j(\gamma(t))$, $\Norm{\sum|a_j|^2}[L^\infty([0,1])]<1$.
For $\sigma>0$, let $\gamma_\sigma:[0,1]\rightarrow M$ be functions such that $\gamma_\sigma\big|_{(0,1)}\in C^\infty$,
$\gamma_\sigma\xrightarrow{\sigma\rightarrow 0} \gamma$ in $\CSpace{[0,1]}$, and 
$\gamma_\sigma'(t)=\sum b_j^{\sigma}(t) (\delta+\sigma)^{d_j} W_j(\gamma_\sigma(t))$ with $\BNorm{\sum|b_j^{\sigma}|^2}[L^\infty]<1$--this can be achieved by simple argument using mollifiers and the fact
that $W_1,\ldots, W_{2m+q}$ are smooth and span the tangent space at every point.
Set $x_\sigma:=\gamma_{\sigma}(\sigma)$, $y_\sigma:=\gamma_{\sigma}(1-\sigma)$, so that $\lim_{\sigma\downarrow 0} x_\sigma= x$ and $\lim_{\sigma\downarrow 0} y_\sigma= y$.
Using the function $f_\sigma:B_{\R}(1/2)\rightarrow M$ given by $f_{\sigma}(t):=\gamma_{\sigma}(t+1/2)$, it follows from the definition of $\rho_F$
that $\rho_F(x_\sigma,y_\sigma)<\delta+\sigma$.
Thus, we have
\begin{equation*}
\rho_F(x,y) \leq \rho_F(x,x_\sigma)+\rho_F(x_\sigma, y_\sigma)+\rho_F(y_\sigma,y) < \delta+\sigma + \rho_F(x,x_\sigma)+\rho_F(y_\sigma,y)\xrightarrow{\sigma\rightarrow 0} \delta,
\end{equation*}
where in the last step we have used \cref{Lemma::PfSubE::rhoFCont}.  We conclude $\rho_F(x,y)\leq \rho_S(x,y)$.

Next, we show $\rho_S\leq \rho_F$.  Suppose $\rho_F(x,y)<\delta$ and let $f_1,\ldots, f_K, \delta_1,\ldots, \delta_K$ be as in the definition of $\rho_F$.
For $w_1,w_2\in f_j(B_{\R}(1/2))$, we will show $\rho_S(w_1,w_2)<\delta_j$.  Notice, this will complete the proof since we may find $\xi_1,\ldots, \xi_{L+1}$ with
$\xi_j,\xi_{j+1}\in f_j(B_{\R}(1/2))$, $x=\xi_1$, $y=\xi_{L+1}$, and so using the triangle
inequality for $\rho_S$, we have
\begin{equation*}
\rho_S(x,y) \leq \sum_{j=1}^L \rho_S(\xi_j, \xi_{j+1})\leq \sum_{j=1}^L \rho_F(\xi_j, \xi_{j+1})<\sum_{j=1}^L \delta_j\leq \delta,
\end{equation*}
which will prove $\rho_S(x,y)\leq \rho_F(x,y)$.

Given $w_1,w_2\in f_j(B_{\R}(1/2))$, we have $w_1=f_j(t_1)$, $w_2=f_j(t_2)$ for some $t_1,t_2\in B_{\R}(1/2)$.
Set $\gamma(r):=f_j((1-r)t_1+rt_2)$.  Then,
\begin{equation*}
\gamma'(r) = f_j'((1-r)t_1+rt_2)(t_1-t_2)=\sum_{l=1}^{m+q} (t_1-t_2)s_j^l(t) \delta_j^{d_l} W_l(f_j(t)),
\end{equation*}
with $\Norm{\sum_{l} |s_j^l|^2}[L^\infty]<1$.  Since $|t_1-t_2|<1$, it follows from the definition of $\rho_S$ that $\rho_S(w_1,w_2)<\delta_j$, completing the proof
of $\rho_S\leq \rho_F$.

Finally, we show $\rho_F\leq \rho_H$.  Suppose $\rho_H(x,y)<\delta$.  Take $\delta_1,\ldots, \delta_K$ and $f_1,\ldots, f_K$ as in the definition of $\rho_H$.
We will show that if $w_1,w_2\in f_j(B_{\C\times \R}(1/2))$, then $\rho_F(w_1,w_2)<\delta_j$.  The result will then follow from the triangle inequality,
just as in the proof of $\rho_S\leq \rho_F$.

Let $w_1=f_j(\xi_1)$ and $w_2=f_j(\xi_2)$ with $\xi_1,\xi_2\in B_{\R\times \C}(1/2)$.  Fix $\epsilon>0$ small (depending on $\xi_1,\xi_2$)
and set
$\eta(r):=(\frac{1}{2}-(1+\epsilon)r) \xi_1 + (\frac{1}{2}+(1+\epsilon)r) \xi_2.$
Note (if $\epsilon>0$ is small enough), $\eta:B_{\R}(1/2)\rightarrow B_{\R\times \C}(1/2)$.  Set $g(r):=f_j(\eta(r))$.
Let $\xi_3=(1+\epsilon)(\xi_2-\xi_1)$, and we henceforth think of $\xi_3$ as an element of $B_{\R^3}(1)$, by identifying
$\R\times \C$ with $\R^3$.
We have
\begin{equation*}
g'(r) = df_j(\eta(r)) \eta'(r) =df_j(\eta(r)) \xi_3.
\end{equation*}

Let $\Sh_j(t,x_1,x_2)$ be the matrix from \cref{Rmk::SubH::Shj}.  We have
\begin{equation*}
g'(r) = \sum_{l=1}^{2m+q} (\Sh_j(t,x_1,x_2) \xi_3)_l \delta_j^{d_l} W_l(g(r)),
\end{equation*}
where $ (\Sh_j(t,z) \xi_3)_l$ denotes the $l$-th component of the vector $\Sh_j(t,z) \xi_3$.  Since $|\xi_3|<1$ and using \cref{Eqn::SubH::NormShj}, we have
\begin{equation*}
\BNorm{ \sum_l \left|(\Sh_j(\cdot) \xi_3)_l \right|^2}[L^\infty]<1.
\end{equation*}
Since $g(-1/(2(1+\epsilon)))=w_1$ and $g(1/(2(1+\epsilon))=w_2$, it follows that $\rho_F(w_1,w_2)<\delta_j$, as desired.
\end{proof}

\begin{proof}[Completion of the proof of \cref{Thm::SubEGeom}]
We will prove the theorem by applying \cref{Thm::Results::MainThm,Thm::Results::Desnity::MainResult,Cor::Results::Desnity::MainCor} to $\delta^{\beta}X,\delta^{\beta}L$,
as the base point $x_0$ ranges over $\Compact$ and as $\delta$ ranges over $(0,1]$
(where $\delta^{\beta}X$ and $\delta^{\beta}L$ are defined in \cref{Section::Res::SubE}).
Thus, our first goal is to show that the hypotheses of these results are satisfied uniformly for $x_0\in \Compact$ and $\delta\in (0,1]$; so that any type of admissible constant in those results can be chosen independently of $x_0\in \Compact$ and $\delta\in (0,1]$.
For notational simplicity, we turn to calling
the base point $x$ instead of $x_0$.

For $\delta\in (0,1]$, we multiply both sides of \cref{Eqn::SubE::AssumeNSW} by $\delta^{\beta_j+\beta_k}$ to see
\begin{equation*}
[\delta^{\beta_j} Z_j, \delta^{\beta_k} Z_k] = \sum_{\beta_l\leq \beta_j+\beta_k} (\delta^{\beta_j+\beta_k-\beta_l} c_{j,k}^{1,l}) \delta^{\beta_l} Z_l,
\end{equation*}
\begin{equation*}
[\delta^{\beta_j} Z_j, \delta^{\beta_k} \Zb[k]] = \sum_{\beta_l\leq \beta_j+\beta_k} (\delta^{\beta_j+\beta_k-\beta_l} c_{j,k}^{2,l}) \delta^{\beta_l} Z_l +\sum_{\beta_l\leq \beta_j+\beta_k} (\delta^{\beta_j+\beta_k-\beta_l} c_{j,k}^{3,l}) \delta^{\beta_l} \Zb[l].
\end{equation*}
Setting $Z_j^{\delta}:=\delta^{\beta_j} Z_j$ and
$$c_{j,k}^{a,l,\delta}:=\begin{cases}
\delta^{\beta_j+\beta_k-\beta_l} c_{j,k}^{a,l} & \text{if }\beta_l\leq \beta_j+\beta_k\\
0&\text{otherwise,}
\end{cases}
$$
we have
\begin{equation*}
[Z_j^{\delta},Z_k^{\delta}]=\sum_{l} c_{j,k}^{1,\delta} Z_l^{\delta},\quad
[Z_j^{\delta},\Zbdelta{\delta}[k]]=\sum_l c_{j,k}^{2,\delta} Z_l^{\delta} + \sum_l c_{j,k}^{3,\delta} \Zbdelta{\delta}[l].
\end{equation*}
With this notation, $\delta^\beta X,\delta^{\beta}L$ is the same as the list $Z_1^{\delta},\ldots, Z_{m+q}^{\delta}$.

For $\delta\in(0,1]$, $c_{j,k}^{a,l,\delta}\in C^\infty$ and $Z_l^{\delta}\in C^\infty$, \textit{uniformly in} $\delta$.  Thus if $\Omega\Subset M$ is a relatively compact
open set with $\Compact\subseteq \Omega$, we have, directly from the definitions,
\begin{equation*}
\CXjNorm{c_{j,k}^{a,l,\delta}}{\delta^{\beta}X,\delta^{\beta}L}{p}[\Omega]\lesssim 1, \quad \forall j,k,l,a, \quad \forall p\in \N,
\end{equation*}
where the implicit constant may depend on $p$, but does not depend on $\delta\in (0,1]$.
It follows from \cref{Lemma::FuncSpaceRev::Properties} \cref{Item::FuncSpaceRev::CmBoundsHms} and \cref{Item::FuncSpaceRev::HmsBoundsZygs}
that
\begin{equation}\label{Eqn::SubEGeom::BoundZygNormGlobal}
\ZygXNorm{c_{j,k}^{a,l,\delta}}{\delta^{\beta} X, \delta^{\beta}L}{s}[\Omega]\lesssim 1, \quad \forall j,k,l,a,\quad s>0,\delta\in (0,1],
\end{equation}
where the implicit constant may depend on $s$, but does not depend on $\delta\in (0,1]$.
We take $\xi\in (0,1]$ so small $B_{X,L}(x,\xi)\subseteq \Omega$, $\forall x\in \Compact$; as a consequence, $B_{\delta^{\beta}X,\delta^{\beta}L}(x,\xi)\subseteq B_{X,L}(x,\xi)\subseteq \Omega$, $\forall x\in \Compact$, $\delta\in (0,1]$.  By \cref{Lemma::FuncSpaceRev::Properties} \cref{Item::FuncSpaceRev::IncludeSets} and \cref{Eqn::SubEGeom::BoundZygNormGlobal} we have
\begin{equation*}
\BZygXNorm{c_{j,k}^{a,l,\delta}}{\delta^{\beta} X, \delta^{\beta}L}{s}[B_{\delta^{\beta}X,\delta^{\beta}L}(x,\xi)]\lesssim 1, \quad \forall j,k,l,a,\quad s>0,\delta\in (0,1],x\in \Compact,
\end{equation*}
where the implicit constant does not depend on $\delta\in (0,1]$ or $x\in \Compact$.
We also have $\Lie{Z_j^{\delta}} \nu = f_j^{\delta} \nu$ where $f_j^{\delta}\in C^\infty$ uniformly for $\delta\in (0,1]$ (this follows directly from the definitions and the fact that
$\nu$ is a strictly positive, $C^\infty$ density).  Similar to the above discussion, we have
\begin{equation*}
\ZygXNorm{f_j^{\delta}}{\delta^{\beta} X, \delta^{\beta}L}{s}[B_{\delta^{\beta}X,\delta^{\beta}L}(x,\xi)]\lesssim 1, \quad \forall j,\quad s>0,\delta\in (0,1],x\in \Compact.
\end{equation*}

The existence of $\eta>0$ and $\delta_0>0$ (independent of $x\in \Compact$ and $\delta\in (0,1]$) as in the hypotheses of \cref{Thm::Results::MainThm} (when applied to $\delta^\beta X, \delta^{\beta} L$ at the base point $x$) follows
from \cref{Lemma::MoreAssume::ExistEtaDelta0}; indeed
\cref{Lemma::MoreAssume::ExistEtaDelta0} directly gives the existence of these constants for $x\in \Compact$ when $\delta=1$ and it is immediate from the definitions of $\eta$ and $\delta_0$ that the same constants
may be used $\forall \delta\in (0,1]$.  The existence of $J_0=J_0(x,\delta)\in \sI(r,q)$, $K_0=K_0(x,\delta)\in \sI(n,m)$, and $\zeta\in (0,1]$  (independent
of $x\in \Compact$, $\delta\in (0,1]$) as in \cref{Thm::Results::MainThm} (when applied to $\delta^\beta X, \delta^{\beta} L$ at the base point $x$)  follows from the hypothesis \cref{Eqn::ResSubE::ExistsMaximal}.

Thus, \cref{Thm::Results::MainThm,Thm::Results::Desnity::MainResult,Cor::Results::Desnity::MainCor} apply (with, e.g., $s_0=3/2$--the choice of $s_0\in (1,\infty)$ is irrelevant for what follows), uniformly for $x\in \Compact$, $\delta\in (0,1]$.
In particular, any positive $\Zygad{s}$-admissible  constant from those results (for any $s>0$)
can be chosen independent of $x\in \Compact$, $\delta\in (0,1]$ (and is therefore $\approx 1$ in the sense of this theorem); and similarly
for any other kind of admissible constant.
We let $\xi_2\approx 1$ ($0<\xi_2\leq \xi\leq 1$) and $K\approx 1$ be the constants of the same name from   \cref{Thm::Results::MainThm}, and let $\Phi_{x,\delta}:B_{\R^r\times \C^n}(1)\rightarrow B_{\delta^{\beta}X,\delta^{\beta}L}(x,\xi)$
be the map guaranteed by \cref{Thm::Results::MainThm} when applied to $\delta^{\beta}X,\delta^{\beta}L$ at the base point $x\in \Compact$.

We turn to proving \cref{Item::SubEGeom::ExistEpsilon}.  By \cref{Thm::Results::MainThm} \cref{Item::Results::MainThm::xi1xi2} we have
\begin{equation*}
B_{\delta^{\beta}X,\delta^{\beta}L}(x,\xi_2)\subseteq \Phi_{x,\delta}(B_{\R^{r}\times \C^n}(1)) \subseteq B_{\delta^{\beta}X,\delta^{\beta}L}(x,\xi)\subseteq B_{\delta^{\beta}X,\delta^{\beta}L}(x,1)=B_S(x,\delta).
\end{equation*}
We set $\epsilon=\xi_2$, and the proof of \cref{Item::SubEGeom::ExistEpsilon} will be complete once we show
\begin{equation}\label{Eqn::SubEPf::SInBall}
B_S(x,\xi_2\delta)\subseteq B_{\delta^{\beta}X,\delta^{\beta}L}(x,\xi_2).
\end{equation}
Take $y\in B_S(x,\xi_2\delta)$.  Thus, $\exists \gamma:[0,1]\rightarrow M$, $\gamma(0)=x$, $\gamma(1)=y$, $\gamma'(t)=\sum a_j(t) \xi_2^{d_j} \delta^{d_j} W_j(\gamma(t))$,
with $\Norm{\sum |a_j|^2}[L^\infty([0,1])]<1$.  Hence,
$$\gamma'(t) = \sum_j (a_j(t) \xi_2^{d_j-1}) \xi_2\delta^{d_j} W_j(\gamma(t)),\quad \BNorm{\sum \left|a_j\xi_2^{d_j-1}\right|^2}[L^\infty]<1.$$
It follows that $y=\gamma(1)\in B_{\delta^{\beta}X,\delta^{\beta}L}(x,\xi_2)$, completing the proof of \cref{Item::SubEGeom::ExistEpsilon}.

\Cref{Item::SubEThm::ZhIsE} follows from \cref{Thm::Results::MainThm} \cref{Item::Results::MainThem::IsEMap}.
\Cref{Item::SubEThm::ZsSpan} follows from \cref{Thm::Results::MainThm} \cref{Item::ResultsMainThm::AFormula}
using the fact that if $\AMatrix$ is as in that result,
$\Norm{\AMatrix(t,z)}[\M^{(n+r)\times (n+r)}]\leq \frac{1}{4}$, $\forall t,z$ by \cref{Thm::Results::MainThm} \cref{Item::ResultsMainThm::ABound}
and therefore $I+\AMatrix(t,z)$ is invertible with $\Norm{(I+\AMatrix(t,z))^{-1}}[\M^{(n+r)\times (n+r)}]\leq \frac{4}{3}$, $\forall t,z$.

Since $\CjNorm{\cdot}{k}\leq \ZygNorm{\cdot}{k+1}$, $\forall k\in \N$, by definition,
\cref{Item::SubEThm::ZsSmooth} follows from \cref{Thm::Results::MainThm} \cref{Item::ResultsMainThm::PullbacksSmooth}.
Similarly, \cref{Item::SubEThm::Esth} follows from \cref{Thm::Results::Desnity::MainResult} \cref{Item::Results::Density::hConst} and \cref{Item::Results::Density::hSmooth}.

\Cref{Item::SubEThm::PhiDiffeo} follows from \cref{Thm::Results::MainThm} \cref{Item::ResultsMainThm::PhiOpen} and \cref{Item::ResultsMainThem::PhiDiffeo};
except that \cref{Item::ResultsMainThem::PhiDiffeo} only guarantees $\Phi_{x,\delta}$ is a $C^2$ diffeomorphism.  That $\Phi_{x,\delta}$ is $C^\infty$
follows by combining \cref{Item::SubEThm::ZsSmooth} and \cref{Lemma::QualEPf::RecogSmooth}.

Next, we prove \cref{Item::SubEThm::ImageInH}.
Let $y\in \Phi_{x,\delta}(B_{\R^r\times \C^n}(1))$.  We will show $y\in B_H(x,R\delta)$ for some $R\approx 1$ to be chosen later.
By \cref{Item::SubEThm::ZsSpan} and \cref{Item::SubEThm::ZsSmooth} we may write
\begin{equation*}
\diff{t_k} = \sum_{l=1}^{m+q} a_{k,x,\delta}^{l}\Zh_{l}^{x,\delta}, \quad \diff{\zb[j]}=\sum_{l=1}^{m+q} b_{j,x,\delta}^l \Zh_l^{x,\delta},
\end{equation*}
where,
\begin{equation*}
\CjNorm{a_{k,x,\delta}^l}{p}[B_{\R^r\times \C^n}(1)],\CjNorm{b_{j,x,\delta}^l}{p}[B_{\R^r\times \C^n}(1)]\lesssim 1, \quad \forall p\in \N.
\end{equation*}
We have
\begin{equation*}
d\Phi_{x,\delta}(t,z) \diff{t_k} =\sum_{l=1}^{m+q} a_{k,x,\delta}^l(t,z) Z_l^{\delta}(\Phi_{x,\delta}(t,z)),
\quad
d\Phi_{x,\delta}(t,z) \diff{\zb[j]} =\sum_{l=1}^{m+q} b_{j,x,\delta}^l(t,z) Z_l^{\delta}(\Phi_{x,\delta}(t,z)).
\end{equation*}
Let $y=\Phi_{x,\delta}(t_0,z_0)$ for some $(t_0,z_0)\in B_{\R^r\times \C^n}(1)$.
Define
$$f(s,w):=\Phi_{x,\delta}\left(2s\frac{t_0}{|t_0|}, 2w\frac{z_0}{|z_0|}\right)$$
so that $f:B_{\R\times \C}(1/2)\rightarrow M$, $f(0,0)=x$ and $y\in f(B_{\R\times \C}(1/2))$.  We have
\begin{equation*}
df(s,w)\diff{s} = \sum_{l=1}^{m+q} \at_l(s,w) Z_l^{\delta}(f(s,w)),\quad  df(s,w)\diff{\wb} = \sum_{l=1}^{m+q} \bt_l(s,w) Z_l^{\delta}(f(s,w)),
\end{equation*}
where
\begin{equation*}
\at_l(s,w) := \sum_{k=1}^r a_{k,x,\delta}^l\left( 2s\frac{t_0}{|t_0|}, 2w\frac{z_0}{|z_0|}\right) \mleft(2\frac{t_0}{|t_0|}\mright)_k,
\end{equation*}
where $\mleft(2\frac{t_0}{|t_0|}\mright)_k$ denotes the $k$th component of $2\frac{t_0}{|t_0|}$; and $\bt_l$ is defined similarly.
In particular
$$\Norm{\at_l}[L^\infty(B_{\R\times \C}(1/2)], \Norm{\bt_l}[L^\infty(B_{\R\times \C}(1/2))]\lesssim 1.$$
For $R\geq 1$ set $\at_l^R:=\at_l/(R^{\beta_l})$, so that we have
\begin{equation*}
df(s,w)\diff{s} = \sum_{l=1}^{m+q} \at_l^R(s,w) Z_l^{R\delta}(f(s,w)),\quad  df(s,w)\diff{\wb} = \sum_{l=1}^{m+q} \bt_l^R(s,w) Z_l^{R\delta}(f(s,w)).
\end{equation*}
By taking $R$ to be a sufficiently large admissible constant, we see that $f$ satisfies the hypotheses of the definition of $\rho_H$ with $K=1$
(i.e., we are using $f_1=f$ and $\delta_1=R\delta$).  This proves $y\in f(B_{\R\times \C}(1/2))\subseteq B_H(x,R\delta)$,
completing the proof of \cref{Item::SubEThm::ImageInH}.

We turn to \cref{Item::SubEThm::HIsSmallerThanS}.  Because $\Compact$ is compact with respect to the usual topology on $M$, $\rho_F$ induces the usual topology on $M$ (\cref{Lemma::PfSubE::TopsAreTheSame}), and  $\rho_F=\rho_S$,
it follows from \cref{Lemma::PfSubE::TopsAreTheSame} that $\Compact$ is compact with respect to the metric topology induced by $\rho_S$.
A simple compactness argument shows that to prove \cref{Item::SubEThm::HIsSmallerThanS}, it suffices to show that there exists $\epsilon'>0$
such that if $\rho_S(x,y)<\epsilon'$, $x,y\in \Compact$, then $\rho_H(x,y)\lesssim \rho_S(x,y)$.  We take $\epsilon'=\epsilon$, where $\epsilon>0$
is from \cref{Item::SubEGeom::ExistEpsilon}.  If $\rho_S(x,y)<\epsilon\delta$ (for some $\delta\in (0,1]$), we have (by \cref{Item::SubEGeom::ExistEpsilon} and \cref{Item::SubEThm::ImageInH})
$y\in B_S(x,\epsilon\delta)\subseteq \Phi_{x,\delta}(B_{\R^r\times \C^n}(1))\subseteq B_H(x,R\delta)$.  Hence $\rho_H(x,y)\leq R\delta$.  We conclude
that if $\rho_S(x,y)<\epsilon$ with $x,y\in \Compact$, then $\rho_H(x,y)\leq \frac{R}{\epsilon} \rho_S(x,y)$.  This completes the proof of \cref{Item::SubEThm::HIsSmallerThanS}.

Next we prove \cref{Item::SubEThm::EstimateVols}.  \Cref{Cor::Results::Desnity::MainCor} shows
\begin{equation}\label{Eqn::SubEPf::GivenVolEst}
\nu(B_{\delta^{\beta} X, \delta^{\beta}L}(x,\xi_2))\approx \Lambda(x,\delta)\approx \Lambda(x,\epsilon\delta),
\end{equation}
where in the second $\approx$, we have used the formula for $\Lambda$ and the fact that $\epsilon\approx 1$.
Using this, \cref{Eqn::SubEPf::SInBall}, and the fact that we chose $\epsilon=\xi_2$, we have
\begin{equation}\label{Eqn::SubEPf::UpperVol}
\nu(B_S(x,\epsilon\delta))\leq \nu(B_{\delta^{\beta}X,\delta^{\beta}L}(x,\xi_2))\lesssim \Lambda(x,\epsilon\delta).
\end{equation}
Conversely, again using \cref{Eqn::SubEPf::GivenVolEst}, we have
\begin{equation}\label{Eqn::SubEPf::LowerVol}
\Lambda(x,\delta)\lesssim \nu(B_{\delta^{\beta}X,\delta^{\beta}L}(x,\xi_2))\leq \nu(B_{\delta^{\beta}X,\delta^{\beta}L}(x,1))=\nu(B_S(x,\delta)).
\end{equation}
Since \cref{Eqn::SubEPf::UpperVol} and \cref{Eqn::SubEPf::LowerVol} hold $\forall \delta\in (0,1]$, it follows that
$\nu(B_S(x,\delta))\approx \Lambda(x,\delta)$, $\forall \delta\in (0,\epsilon]$.
By \cref{Item::SubE::EasyMetrics} we have (for $\delta\in (0,\epsilon]$),
\begin{equation}\label{Eqn::SubEPf::UpperVol::H}
\nu(B_H(x,\delta)) \leq \nu(B_S(x,\delta))\approx \Lambda(x,\delta).
\end{equation}
By \cref{Item::SubEThm::ImageInH} and \cref{Item::SubEGeom::ExistEpsilon}, we have (for $\delta\in (0,1]$)
\begin{equation}\label{Eqn::SubEPf::LowerVol::H}
\Lambda(x,R\delta)\approx \Lambda(x,\epsilon\delta)\approx \nu(B_S(x,\epsilon\delta))\leq \nu(B_H(x,R\delta)),
\end{equation}
where in the first $\approx$, we have used $R,\epsilon\approx 1$ and the formula for $\Lambda$.
Combining \cref{Eqn::SubEPf::UpperVol::H} and \cref{Eqn::SubEPf::LowerVol::H}, we have for $\delta\in (0,\min\{\epsilon,1/R\}]$,
\begin{equation*}
\nu(B_H(x,\delta))\approx \Lambda(x,\delta).
\end{equation*}
This completes the proof of \cref{Item::SubEThm::EstimateVols}.
\Cref{Item::SubEThm::Doubling} is a consequence of \cref{Item::SubEThm::EstimateVols} and the formula for $\Lambda$.
\end{proof}

\section{Nirenberg's Theorem for Elliptic Structures}\label{Section::Nirenbeg}
In this section, we present the main technical result from \cite{StreetNirenberg}.
This can be seen as a sharp (in terms of regularity) version
of Nirenberg's theorem that formally integrable elliptic structures are integrable \cite{NirenbergAComplexFrobeniusTheorem}.
Here, unlike the setting of \cref{Thm::Results::MainThm}, we assume the vector fields already have the desired regularity, and that we have good estimates on the coefficients in a
given coordinate system.  The goal is to pick a new coordinate system in which the vector fields are spanned by $\diff{t_1},\ldots, \diff{t_r}, \diff{\zb[1]},\ldots, \diff{\zb[n]}$,
while maintaining the regularity of the vector fields.

Fix $s_0\in (0,\infty)\cup \{\omega\}$ and let
$X_1,\ldots, X_r, L_1,\ldots, L_n$ be complex vector fields
on $B_{\R^r\times \C^n}(1)$ with:
\begin{itemize}
\item If $s_0\in (0,\infty)$, $X_k, L_j\in \ZygSpace{s_0+1}[B_{\R^r\times \C^n}(1)][\C^{r+2n}]$.
\item If $s_0=\omega$, $X_k, L_j\in \ASpace{r+2n}{1}[\C^{r+2n}]$.
\end{itemize}
We suppose:
\begin{itemize}
\item $X_k(0)=\diff{t_k}$, $L_j(0)=\diff{\zb[j]}$.
\item $\forall \zeta\in B_{\R^r\times \C^n}(1)$,
$[X_{k_1},X_{k_2}](\zeta), [X_k, L_j](\zeta), [L_{j_1}, L_{j_2}](\zeta)
\in \Span_{\C}\mleft\{X_1(\zeta),\ldots, X_r(\zeta), L_1(\zeta),\ldots, L_n(\zeta)\mright\}$.
\end{itemize}

Under these hypotheses, Nirenberg's theorem\footnote{Originally, Nirenberg considered only the case of $C^\infty$ vector fields and worked in the case when $X_1,\ldots, X_r$ were real.} implies
that there exists a map $\Phi_1:B_{\R^r\times \C^n}(1)\rightarrow B_{\R^r\times \C^n}(1)$, with $\Phi_1(0)=0$, $\Phi_1$ is a diffeomorphism
onto its image (which is an open neighborhood of $0\in B_{\R^r\times \C^n}(1)$), and such that
$\Phi_1^{*}X_k (u,w), \Phi_1^{*}L_j(u,w)\in \Span_{\C}\mleft\{\diff{u_1},\ldots, \diff{u_r}, \diff{\wb[1]},\ldots, \diff{\wb[r]}\mright\}$, $\forall (u,w)$ (here
we are giving the domain space $\R^r\times \C^n$ coordinates $(u,w)$).
In \cite{StreetNirenberg} this is improved to
 a quantitative version which gives $\Phi_1$ the optimal regularity (namely,
when $s_0\in (0,\infty)$, $\Phi_1$ is in $\ZygSpace{s_0+2}$,
and when $s_0=\omega$, $\Phi_1$ is real analytic).
Unlike the results in the rest of this paper, the results in this section are not quantitatively diffeomorphically invariant:  the estimates depend on the particular coordinate system we are using
(the standard coordinate system on $\R^r\times \C^n$).


\begin{defn}\label{Defn::Nirenberg::Admissible}
If $s_0\in (0,\infty)$, for $s\geq s_0$ if we say $C$ is an
$\Zygad{s}$-admissible constant, it means that we assume
$X_k, L_j\in \ZygSpace{s+1}[B_{\R^r\times \C^n}(1)][\C^{r+2n}]$, $\forall j,k$.  $C$ can then be chosen to depend only on
$n$, $r$, $s$, $s_0$, and upper bounds for
$\ZygNorm{X_k}{s+1}[B_{\R^r\times \C^n}(1)]$ and
$\ZygNorm{L_j}{s+1}[B_{\R^r\times \C^n}(1)]$, $1\leq k\leq r$, $1\leq j\leq n$.
For $s\leq s_0$, we define $\Zygad{s}$-admissible constants
to be $\Zygad{s_0}$-admissible constants.
\end{defn}

\begin{rmk}\label{Rmk::Nirenberg::AdmissibleOkay}
In Definition \ref{Defn::Nirenberg::Admissible} we have defined admissible constants differently than they were defined in Definitions \ref{Defn::Results::0Admiss}, \ref{Defn::Results::ZygAdmiss}, and \ref{Defn::Results::omegaAdmiss}.  This reuse of notation is justified when we turn to the proof of the main theorem (\cref{Thm::Results::MainThm}).
Indeed, when we apply \cref{Thm::Nirenberg::MainThm} in the proof \cref{Thm::Results::MainThm}, we apply it to a choice of vector fields in such a way that constants
which are admissible in the sense of \cref{Thm::Nirenberg::MainThm} are admissible in the sense of \cref{Thm::Results::MainThm}.  Thus, \textit{in the particular application}
of \cref{Thm::Nirenberg::MainThm} used to prove \cref{Thm::Results::MainThm}, the definitions of admissible constants do coincide.
\end{rmk}

\begin{defn}
If $s_0=\omega$, we say $C$ is an $\Zygad{\omega}$-admissible constant
if $C$ can be chosen to depend only on $n$, $r$, and upper bounds
for $\ANorm{X_k}{2n+r}{1}$, $\ANorm{L_j}{2n+r}{1}$, $1\leq k\leq r$, $1\leq j\leq n$.
\end{defn}

\begin{thm}\label{Thm::Nirenberg::MainThm}
There exists an $\Zygad{s_0}$-admissible constant $K_1\geq 1$
and a map $\Phi_1:B_{\R^r\times \C^n}(1)\rightarrow B_{\R^r\times \C^n}(1)$ such that
\begin{enumerate}[(i)]
\item\label{Item::Nirenberg::PhiRegularity}
\begin{itemize}
\item If $s_0\in (0,\infty)$, $\Phi_1\in \ZygSpace{s_0+2}[B_{\R^r\times \C^n}(1)][\R^r\times \C^n]$ and
    $\ZygNorm{\Phi_1}{s+2}[B_{\R^r\times \C^n}(1)]\lesssim_{\Zygad{s}} 1$, $\forall s>0$.
\item If $s_0=\omega$, $\Phi_1\in \ASpace{2n+r}{2}[\R^r\times \C^n]$
and $\ANorm{\Phi_1}{2n+r}{2}\leq 1$.
\end{itemize}
\item\label{Item::Nirenberg::Phiof0} $\Phi_1(0)=0$ and $d_{(t,x)} \Phi_1(0) = K_1^{-1} I_{(r+2n)\times (r+2n)}$.  See \cref{Section::Notation} for the notation $d_{(t,x)}$.

\item\label{Item::Nirenberg::JacobianConst} $\forall \zeta\in B_{\R^r\times\C^n}(1)$, $\det d_{(t,x)} \Phi_1 (\zeta)\approx_{\Zygad{s_0}} 1$.

\item\label{Item::Nirenberg::Phi4Open} $\Phi_1(B_{\R^r\times \C^n}(1))\subseteq B_{\R^r\times \C^n}(1)$ is an open set and $\Phi_1:B_{\R^r\times \C^n}(1)\rightarrow \Phi_1(B_{\R^r\times \C^n}(1))$ is a diffeomorphism\footnote{By diffeomorphism we mean that $\Phi_1:B_{\R^r\times \C^n}(1)\rightarrow \Phi_1(B_{\R^r\times \C^n}(1))$ is a bijection and $d\Phi_1$ is everywhere nonsingular.}.
\item\label{Item::Nirenberg::AMatrix}
\begin{equation*}
\begin{bmatrix}
\diff{u} \\ \diff{\wb}
\end{bmatrix}
= K_1^{-1}(I+\AMatrix)
\begin{bmatrix}
\Phi_1^{*} X \\ \Phi_1^{*} L
\end{bmatrix},
\end{equation*}
where $\AMatrix:B_{\R^r\times\C^n}(1)\rightarrow \M^{(n+r)\times (n+r)}(\C)$, $\AMatrix(0)=0$ and
\begin{itemize}
    \item If $s_0\in (0,\infty)$, $\ZygNorm{\AMatrix}{s+1}[B_{\R^r\times\C^n}(1)][\M^{(n+r)\times (n+r)}]\lesssim_{\Zygad{s}} 1$, $\forall s>0$ and $$\ZygNorm{\AMatrix}{s_0+1}[B_{\R^r\times\C^n}(1)][\M^{(n+r)\times (n+r)}]\leq \frac{1}{4}.$$
    \item If $s_0=\omega$, $\ANorm{\AMatrix}{2n+r}{1}[\M^{(n+r)\times (n+r)}]\leq \frac{1}{4}$.
\end{itemize}
In either case, note that this implies $(I+\AMatrix)$ is an invertible matrix on $B_{\R^r\times\C^n}(1)$.
\item\label{Item::Nirenberg::MainThm::ZEst} Suppose $Z$ is another complex vector field on $B_{\R^r\times\C^n}(1)$.  Then,
    \begin{itemize}
        \item If $s_0\in (0,\infty)$,
        $\ZygNorm{\Phi_1^{*} Z}{s+1}[B_{\R^r\times\C^n}(1)]\lesssim_{\Zygad{s}} \ZygNorm{Z}{s+1}[B_{\R^r\times\C^n}(1)]$, $\forall s>0$.
        \item If $s_0=\omega$,
        $\ANorm{\Phi_1^{*} Z}{2n+r}{1}\lesssim_{\Zygad{\omega}} \ANorm{Z}{2n+r}{1}$.
    \end{itemize}
\end{enumerate}
\end{thm}
\begin{proof}
This is \cite[Theorem 7.3]{StreetNirenberg}.
\end{proof}

\section{The Real Case}\label{Section::RealCase}
The case when $m=0$ of \cref{Thm::Results::MainThm} (i.e.,
when there are no complex vector fields), was the
subject of the series
\cite{StovallStreetI,StovallStreetII,StovallStreetIII}.
In this section, we present a simplified version of this for use
in proving \cref{Thm::Results::MainThm}.

Let $W_1,\ldots, W_Q$ be $C^1$ real vector fields on a
$C^2$ manifold $\fM$.  Fix $x_0\in \fM$ and let
$N:=\dim \Span_{\R}\{W_1(x_0),\ldots, W_Q(x_0)\}$.
Fix $\xi,\zeta\in (0,1]$.  We assume that
on $B_W(x_0,\xi)$, the $W_j$ satisfy
\begin{equation*}
[W_j, W_k]=\sum_{l=1}^Q c_{j,k}^l W_l, \quad c_{j,k}^l\in \CSpace{B_W(x_0,\xi)},
\end{equation*}
where $B_W(x_0,\xi)$ is given the metric topology induced by
the corresponding sub-Riemannian metric \cref{Eqn::FuncMan::Defnrho}.
Under the above hypotheses, $B_{W}(x_0,\xi)$ is a $C^2$, injectively immersed submanifold of $\fM$ of dimension $N$ and $ T_xB_{W}(x_0,\xi)=\Span_{\R}\{W_1(x),\ldots, W_Q(x)\}$, $\forall x\in B_{W}(x_0,\xi)$  (see \cref{Prop::AppendImmerse}).
Henceforth we view $W_1,\ldots, W_Q$ as $C^1$ vector fields on $B_{W}(x_0,\xi)$.

Let $P_0\in \sI(N,Q)$ be such that $\bigwedge W_{P_0}(x_0)\ne 0$
and moreover
\begin{equation*}
\max_{P\in \sI(N,Q)} \mleft|\frac{\bigwedge W_{P}(x_0)}{\bigwedge W_{P_0}(x_0)}\mright|\leq \zeta^{-1}.
\end{equation*}
Without loss of generality, reorder the vector fields
so that $P_0=(1,\ldots, N)$.

We take $\eta>0$ and $\delta_0>0$ as in  \cref{Thm::Results::MainThm}; i.e.,
\begin{itemize}
\item Fix $\eta>0$ so that $W_{P_0}$ satisfies $\sC(x_0,\eta,\fM)$.
\item Fix $\delta_0>0$ such that $\forall \delta\in (0,\delta_0]$, the following holds.  If $z\in B_{W_{P_0}}(x_0,\xi)$
is such that $W_{P_0}$ satisfies $\sC(z, \delta, B_{W_{P_0}}(x_0,\xi))$ and if
$t\in B_{\R^{2n+r}}(\delta)$ is such that $e^{t_1 W_1+\cdots+t_{2n+r} W_{N}}z=z$
and if $W_1(z),\ldots, W_{N}(z)$ are linearly independent, then $t=0$.
\end{itemize}

\begin{defn}
We say $C$ is a $0$-admissible constant if $C$ can be chosen
to depend only on upper bounds for
$Q$, $\zeta^{-1}$, $\xi^{-1}$, and $\CNorm{c_{j,k}^l}{B_{W_{P_0}}(x_0,\xi)}$, $1\leq j,k,l\leq Q$.
\end{defn}

Fix $s_0\in (1,\infty)\cup\{\omega\}$.

\begin{defn}
Suppose $s_0\in (1,\infty)$.  For $s\in [s_0,\infty)$ if we
say $C$ is an $\Zygad{s}$-admissible constant it means that
we assume $c_{j,k}^l\in \ZygXSpace{W_{P_0}}{s}[B_{W_{P_0}}(x_0,\xi)]$.
$C$ is allowed to depend only on $s$, $s_0$, and upper bounds
for $\zeta^{-1}$, $\xi^{-1}$, $\eta^{-1}$, $\delta_0^{-1}$, $Q$,
and $\ZygXNorm{c_{j,k}^l}{W_{P_0}}{s}[B_{W_{P_0}}(x_0,\xi)]$, $1\leq j,k,l\leq Q$.  For $s\in (0,s_0)$, we define $\Zygad{s}$-admissible
constants to be $\Zygad{s_0}$-admissible constants.
\end{defn}

\begin{defn}
Suppose $s_0=\omega$.  If we say $C$ is an $\Zygad{s_0}$-admissible
constant it means that we assume $c_{j,k}^l\in \AXSpace{W_{P_0}}{x_0}{\eta}$.  $C$ is allowed to depend only
on anything a $0$-admissible constant may depend on, as well
as upper bounds for $\eta^{-1}$, $\delta_0^{-1}$,
and $\AXNorm{c_{j,k}^l}{W_{P_0}}{x_0}{\eta}$, $1\leq j,k,l\leq Q$.
\end{defn}

\begin{prop}\label{Prop::MainRealProp}
There exists a $0$-admissible constant $\chi\in (0,\xi]$ such that
\begin{enumerate}[label=(\roman*),series=oldtheoremenumeration]
\item\label{Item::MainRealProp1} $\forall y\in B_{W_{P_0}}(x_0,\chi)$, $\bigwedge W_{P_0}(y)\ne 0$.
\item\label{Item::MainRealProp2} $\forall y\in B_{W_{P_0}}(x_0,\chi)$,
\begin{equation*}
    \max_{P\in \sI(N,Q)} \mleft|\frac{\bigwedge W_P(y)}{\bigwedge W_{P_0}(y)}\mright|\approx_0 1.
\end{equation*}
\item\label{Item::MainRealProp3} $\forall \chi'\in (0,\chi]$, $B_{W_{P_0}}(x_0,\chi')$ is an open
subset of $B_W(x_0,\chi)$, and is therefore a submanifold.
\end{enumerate}
For the remainder of the proposition, we assume:
\begin{itemize}
\item If $s_0\in (1,\infty)$, we assume $c_{j,k}^l\in\ZygXSpace{W_{P_0}}{s_0}[B_{W_{P_0}}(x_0,\xi)]$.
\item If $s_0=\omega$, we assume $c_{j,k}^l\in \AXSpace{W_{P_0}}{x_0}{\eta}$.
\end{itemize}
There exists a $C^2$ map $\Phi_0:B_{\R^N}(1)\rightarrow B_{W_{P_0}}(x_0,\chi)$ such that:
\begin{enumerate}[resume*=oldtheoremenumeration]
\item\label{Item::MainRealProp::Open} $\Phi_0(B_{\R^N}(1))$ is an open subset of $B_{W_{P_0}}(x_0,\chi)$ and is therefore a submanifold.
\item\label{Item::MainRealProp::Phiof0} $\Phi_0(0)=x_0$.
\item\label{Item::MainRealProp::Diffeo} $\Phi_0:B_{\R^N}(1)\rightarrow \Phi_0(B_{\R^N}(1))$ is a $C^2$
diffeomorphism.
\item\label{Item::MainRealProp::SmoothnessWj} \begin{itemize}
    \item If $s_0\in (1,\infty)$, $\ZygNorm{\Phi_0^{*} W_j}{s+1}[B_{\R^N}(1)][\R^N]\lesssim_{\Zygad{s}} 1$, $\forall s>0$, $1\leq j\leq Q$.
    \item If $s_0=\omega$, $\ANorm{\Phi_0^{*} W_j}{N}{1}[\R^N]\lesssim_{\Zygad{\omega}} 1$, $1\leq j\leq Q$.
\end{itemize}
\item\label{Item::MainRealProp::AMatrix} There exists an $\Zygad{s_0}$-admissible constant $K_0\geq 1$ such that
    \begin{equation*}
    \Phi_0^{*} W_{P_0} = K_0 (I+\AMatrix_0) \diff{t},
    \end{equation*}
    where $\AMatrix_0:B_{\R^N}(1)\rightarrow \M^{N\times N}(\R)$,
    $\AMatrix_0(0)=0$, $\sup_{t\in B_{\R^N}(1)}\Norm{\AMatrix_0(t)}[\M^{N\times N}]\leq \frac{1}{2}$,
    and:
    \begin{itemize}
        \item If $s_0\in (1,\infty)$, $\ZygNorm{\AMatrix_0}{s}[B_{\R^N}(1)][\M^{N\times N}]\lesssim_{\Zygad{s}} 1$, $\forall s>0$.
        \item If $s_0=\omega$, $\ANorm{\AMatrix_0}{N}{1}[\M^{N\times N}]\leq \frac{1}{2}.$
    \end{itemize}
\end{enumerate}
\end{prop}


    \subsection{Densities}
We take the same setting as \cref{Prop::MainRealProp}, and let
$\chi\in (0,\xi]$ be as in that proposition.  Let $\nu$ be a real $C^1$
density on $B_{W_{P_0}}(x_0,\chi)$ and suppose
for $1\leq j\leq N$ (recall, we are assuming $P_0=(1,\ldots,N)$),
\begin{equation*}
\Lie{W_j} \nu = f_j \nu, \quad f_j\in \CSpace{B_{W_{P_0}}(x_0,\chi)}.
\end{equation*}

\begin{defn}
If we say $C$ is a $\Zygsonu$-admissible constant, it means that $C$ is a $\Zygad{s_0}$-admissible constant, which is also allowed
to depend on upper bounds for $\CNorm{f_j}{B_{W_{P_0}}(x_0,\chi)}$,
$1\leq j\leq N$.  This definition holds in both cases:  $s_0\in (1,\infty)$ and $s_0=\omega$.
\end{defn}

\begin{defn}
If $s_0\in (1,\infty)$, for 
$s>0$
if we say $C$
is an $\Zygad{s;\nu}$-admissible constant it means that
$f_j\in \ZygXSpace{W_{P_0}}{s}[B_{W_{P_0}}(x_0,\chi)]$.
$C$ is then allowed to depend on anything an $\Zygad{s}$-admissible
constant may depend on, and is allowed to depend on
upper bounds for $\ZygXNorm{f_j}{W_{P_0}}{s}[B_{W_{P_0}}(x_0,\chi)]$,
$1\leq j\leq N$.
For $s\leq 0$, we define $\Zygad{s;\nu}$-admissible constants to be $\Zygsonu$-admissible constants.
\end{defn}

If $s_0=\omega$, we fix some number $r_0>0$.

\begin{defn}
If $s_0=\omega$ and if we say $C$ is a $\Zygad{\omega;\nu}$-admissible
constant, it means that we assume $f_j\in \AXSpace{W_{P_0}}{x_0}{r_0}$, 
$1\leq j\leq N$.
$C$ is then allowed to depend on anything a $\Zygad{\omega}$-admissible
constant may depend on, and is allowed to depend on upper bounds for
$r_0^{-1}$ and
$\AXNorm{f_j}{W_{P_0}}{x_0}{r_0}$, $1\leq j\leq N$.
\end{defn}

\begin{prop}\label{Prop::RealDensities}
Define $h_0\in \CjSpace{1}[B_{\R^N}(1)]$ by
$\Phi_0^{*} \nu = h_0\LebDensity$.  Then,
\begin{enumerate}[(a)]
\item\label{Item::RealDensities::h0Const} $h_0(t)\approx_{\Zygsonu} \nu(W_1,\ldots, W_N)(x_0)$,
$\forall t\in B_{\R^N}(1)$.\footnote{Recall, we are assuming without loss
of generality that $P_0=(1,\ldots, N)$.}
In particular, $h_0(t)$ always has the same sign, and is either
never zero or always zero.
\item\label{Item::RealDensities::hSmooth} 
\begin{itemize}
    \item If $s_0\in (1,\infty)$, for $s>0$,
        \begin{equation*}
        \ZygNorm{h_0}{s}[B_{\R^N}(1)]\lesssim_{\Zygad{s-1;\nu}} |\nu(W_1,\ldots, W_N)(x_0)|.
        \end{equation*}
    \item If $s_0=\omega$, 
        \begin{equation*}
        \ANorm{h_0}{N}{\min\{1,r_0\}} \lesssim_{\Zygad{\omega;\nu}} |\nu(W_1,\ldots, W_N)(x_0)|.
        \end{equation*}
\end{itemize}
\end{enumerate}
\end{prop}

    \subsection{Proofs}
In this section, we discuss the proofs of
\cref{Prop::MainRealProp} and \cref{Prop::RealDensities}.
When $s_0\in (1,\infty)$, \cref{Prop::MainRealProp} and \cref{Prop::RealDensities} follow directly from the main
results in \cite{StovallStreetII}, and so we focus on the case
$s_0=\omega$.

The main results of \cite{StovallStreetIII} are very similar
to \cref{Prop::MainRealProp} and \cref{Prop::RealDensities}
when $s_0=\omega$.
\Cref{Item::MainRealProp1}, \cref{Item::MainRealProp2},
and \cref{Item::MainRealProp3} of \cref{Prop::MainRealProp} are
directly contained in \cite{StovallStreetIII}.
The main result of \cite{StovallStreetIII} shows that there
exists an $\Zygad{\omega}$-admissible constant
$\etah\in (0,1]$ and a map
\begin{equation*}
    \Phih:B_{\R^N}(\etah)\rightarrow B_{W_{P_0}}(x_0,\xi)
\end{equation*}
such that
\begin{itemize}
\item $\Phih(B_{\R^N}(\etah))$ is an open subset of $B_{W_{P_0}}(x_0,\chi)$ and is therefore a submanifold of $B_W(x_0,\xi)$.
\item $\Phih:B_{\R^N}(\etah)\rightarrow \Phih(B_{\R^N}(\etah))$
is a $C^2$ diffeomorphism, and $\Phih(0)=x_0$.
\item $\ANorm{\Phih^{*} W_j}{N}{\etah}[\R^N]\lesssim_{\Zygad{\omega}} 1$, $1\leq j\leq Q$.
    \item \begin{equation*}
        \Phih^{*} W_{P_0} = (I+\AMatrixh)\diff{t},
    \end{equation*}
    where $\AMatrixh:B_{\R^n}(\etah)\rightarrow \M^{N\times N}(\R)$,
    $\AMatrixh(0)=0$, and $\ANorm{\AMatrixh}{N}{\etah}\leq \frac{1}{2}$.
\end{itemize}

Define $\Psi:B_{\R^N}(1)\rightarrow B_{\R^N}(\etah)$ by
$\Psi(t) = \etah t$, and set $\Phi_0:=\Phih\circ \Psi$.
The remainder of  \cref{Prop::MainRealProp} follows
from the above properties of $\Phih$,
with $K_0:=\etah^{-1}$ and $\AMatrix_0:=\AMatrixh\circ \Psi$.

Now let $\nu$ be a real $C^1$ density on $B_{W_{P_0}}(x_0,\chi)$
as in \cref{Prop::RealDensities}.  The main result on
densities in \cite{StovallStreetIII} shows that if
$\hh\in \CjSpace{1}[B_{\R^n}(\etah)]$ is defined by
$\Phih^{*} \nu = \hh \LebDensity$, then
\begin{itemize}
\item $\hh(t) \approx_{\Zygomeganu} \nu(W_1,\ldots, W_N)(x_0)$,
$\forall t\in B_{\R^n}(\etah)$.
\item $\hh\in \ASpace{N}{\min\{\etah,r_0\}}$ and
$\ANorm{\hh}{N}{\min\{\etah,r_0\}}\lesssim_{\Zygad{\omega;\nu}} 1$.
\end{itemize}
Note that $h_0=(\hh\circ \Psi) \det d\Psi = \etah^{N} \hh\circ \Psi$.
Since $\etah\approx_{\Zygad{\omega}} 1$, \cref{Prop::RealDensities}
follows from the above estimates on $\hh$. 

\section{Proofs of the Main Results}
In this section, we prove \cref{Thm::Results::MainThm},
\cref{Thm::Results::Desnity::MainResult}, and \cref{Cor::Results::Desnity::MainCor}.

Note that, by the definitions,
$B_{W_{P_0}}(x_0, \xi)=B_{X_{K_0},L_{J_0}}(x_0,\xi)$,
$\ZygXSpace{W_{P_0}}{s}[U]=\ZygXSpace{X_{K_0},L_{J_0}}{s}[U]$,
and $\AXSpace{W_{P_0}}{r}{\eta}=\AXSpace{X_{K_0},L_{J_0}}{r}{\eta}$
(with equality of norms).
It follows from the hypotheses that we may write
(for $1\leq j,k\leq 2m+q$)
\begin{equation*}
    [W_j, W_k]=\sum_{l=1}^{2m+q} \ct_{j,k}^l W_l, \quad \ct_{j,k}^l\in \CSpace{B_{W_{P_0}}(x_0,\xi)},
\end{equation*}
    $\CNorm{\ct_{j,k}^l}{B_{W_{P_0}}(x_0,\xi)}\lesssim_{0} 1$,
and
\begin{itemize}
    \item If $s_0\in (1,\infty)$, $\ZygXNorm{\ct_{j,k}^l}{W_{P_0}}{s}[B_{W_{P_0}}(x_0,\chi)]\lesssim_{\Zygad{s}} 1$, $\forall s>0$.
    \item If $s_0=\omega$, $\AXNorm{\ct_{j,k}^l}{W_{P_0}}{x_0}{\eta}\lesssim_{\Zygad{\omega}} 1$.
\end{itemize}

Recall,
\begin{equation}\label{Eqn::ProofsMain::RecallWP0}
W_{P_0}=W_1,\ldots, W_{2n+r}=X_1,\ldots, X_r, 2\Real(L_1),\ldots, 2\Real(L_n),2\Imag(L_1),\ldots, 2\Imag(L_n).
\end{equation}
Combining \cref{Eqn::MainRes::ChooseZeta} with \cref{Prop::AppendWedge::EquivQuot} \cref{Item::AppendWedge::EquivQuot1}, we see
\begin{equation}\label{Eqn::MainRes::RealZeta}
\max_{P\in \sI(2n+r,2m+q)} \mleft|\frac{\bigwedge W_P(x_0)}{\bigwedge W_{P_0}(x_0)}\mright|\leq \mleft( 2\zeta^{-1} \sqrt{2n+r} \mright)^{2n+r}\lesssim_0 1.
\end{equation}
In light of these remarks, and the definition of $\eta$ and $\delta_0$, \cref{Prop::MainRealProp} applies
to the vector fields $W_1,\ldots, W_{2m+q}$ (with $N=2n+r$) and any constant which is $*$-admissible in the sense of \cref{Prop::MainRealProp}
is $*$-admissible in the sense of this section (where $*$ is any symbol).

We take the $0$-admissible constant $\chi\in (0,\xi]$ from \cref{Prop::MainRealProp}.  By \cref{Prop::MainRealProp} \cref{Item::MainRealProp1} and \cref{Item::MainRealProp2},
$\forall y\in B_{W_{P_0}}(x_0,\chi)$, $\bigwedge W_{P_0}(y)\ne 0$ and
\begin{equation}\label{Eqn::ProofsMain::ConcludeWQuotient}
    \max_{P\in \sI(N,Q)} \mleft|\frac{\bigwedge W_P(y)}{\bigwedge W_{P_0}(y)}\mright|\approx_0 1.
\end{equation}
By hypothesis, $\dim \LVS_y=\dim\LVS_{x_0} = 2n+r$ and $\dim \XVS_y=\dim\XVS_{x_0}=r$, $\forall y\in B_{W_{P_0}}(x_0,\chi)\subseteq B_{X,L}(x_0,\xi)$.
Combining this with \cref{Eqn::ProofsMain::ConcludeWQuotient}, \cref{Prop::AppendWedge::EquivQuot} \cref{Item::AppendWedge::EquivQuot2} implies $\forall y\in B_{W_{P_0}}(x_0,\chi)= B_{X_{K_0}, L_{J_0}}(x_0,\chi)$,
\begin{equation*}
	\left(\bigwedge X_{K_0}(y)\right)\bigwedge \left(\bigwedge L_{J_0}(y)\right)\ne 0, 
\end{equation*}
and moreover
\begin{equation}\label{Eqn::PfsMain::TmpQuotitentXsLs}
	\max_{\substack{J\in \sI(n_1,m), K\in \sI(r_1,q) \\ n_1+r_1=n+r}} \left| \frac{ \left(\bigwedge X_{K}(y)\right)\bigwedge \left(\bigwedge L_{J}(y)\right)}{\left(\bigwedge X_{K_0}(y)\right)\bigwedge \left(\bigwedge L_{J_0}(y)\right)} \right|\lesssim_0 1.
\end{equation}
Since the left hand side of \cref{Eqn::PfsMain::TmpQuotitentXsLs} is $\geq 1$, it follows that the left hand side of \cref{Eqn::PfsMain::TmpQuotitentXsLs} is $\approx_0 1$.
\Cref{Thm::Results::MainThm} \cref{Item::Results::MainThm::1} and \cref{Item::Results::MainThm::2} follow.
\Cref{Thm::Results::MainThm} \cref{Item::Results::MainThm::3} follows from
\cref{Prop::MainRealProp} \cref{Item::MainRealProp3}.
Since $\dim \LVS_{x}=\dim\LVS_{x_0}=2n+r$, $\forall x\in B_{X,L}(x_0,\xi)$, \cref{Thm::Results::MainThm} \cref{Item::Results::MainThm::1}
implies that $X_1(x),\ldots, X_r(x),L_1(x),\ldots, L_n(x)$ form a basis for $\LVS_{x}$, $\forall x\in B_{X_{K_0},L_{J_0}}(x_0,\chi)$.
In particular, for $ x\in B_{X_{K_0},L_{J_0}}(x_0,\chi)$, $1\leq k, k_1, k_2\leq q$, $1\leq j, j_1,j_2\leq m$,
\begin{equation}\label{Eqn::ProofsMain::Invol1}
X_k(x), L_j(x), [X_{k_1},X_{k_2}](x), [L_{j_1},L_{j_2}](x),[X_k,L_j](x)\in \LVS_{x}=\Span_{\C}\{X_1(x),\ldots, X_r(x), L_1(x),\ldots, L_n(x) \}.
\end{equation}

Let $\Phi_0:B_{\R^{r+2n}}(1)\rightarrow B_{X_{K_0},L_{J_0}}(x_0,\chi)$ be the map from  \cref{Prop::MainRealProp}.
\begin{itemize}
\item If $s_0\in (1,\infty)$,  \cref{Prop::MainRealProp} \cref{Item::MainRealProp::SmoothnessWj} gives $\ZygNorm{\Phi_0^{*} W_j}{s+1}[B_{\R^{2n+r}}(1)]\lesssim_{\Zygad{s}} 1$, $1\leq  j\leq 2m+q$,
and therefore
\begin{equation}\label{Eqn::ProofsMain::Phi0Reg::Finite}
\ZygNorm{\Phi_0^{*} X_k}{s+1}[B_{\R^N}(1)], \ZygNorm{\Phi_0^{*} L_j}{s+1}[B_{\R^{2n+r}}(1)]\lesssim_{\Zygad{s}} 1,\quad 1\leq j\leq m, 1\leq k\leq q.
\end{equation}
\item If $s_0=\omega$, \cref{Prop::MainRealProp} \cref{Item::MainRealProp::SmoothnessWj} gives $\ANorm{\Phi_0^{*} W_j}{2n+r}{1}\lesssim_{\Zygad{\omega}} 1$, $1\leq  j\leq 2m+q$,
and therefore
\begin{equation}\label{Eqn::ProofsMain::Phi0Reg::Omega}
\ANorm{\Phi_0^{*} X_k}{2n+r}{1}, \ANorm{\Phi_0^{*} L_j}{2n+r}{1}\lesssim_{\Zygad{\omega}} 1, \quad 1\leq j\leq m, 1\leq k\leq q.
\end{equation}
\end{itemize}

We identify $\R^{r+2n}\cong \R^r\times \C^n$, via the map $(t_1,\ldots, t_r,x_1,\ldots, x_{2n})\mapsto (t_1,\ldots, t_r, x_1+ix_{n+1},\ldots, x_n+ix_{2n})$.
Let $K_2\geq 1$ be the $\Zygad{s_0}$-admissible constant called $K_0$ in \cref{Prop::MainRealProp}.  By \cref{Prop::MainRealProp} \cref{Item::MainRealProp::AMatrix}
(and since $P_0=(1,\ldots, 2n+r)$), we have
\begin{equation*}
\Phi_0^{*} K_2^{-1}W_j(0) =
\begin{cases}
\diff{t_j} & 1\leq j\leq r,\\
\diff{x_{j-r}} & r+1\leq j\leq 2n.
\end{cases}
\end{equation*}
Using this and \cref{Eqn::ProofsMain::RecallWP0} shows,
for $1\leq k\leq r$, $1\leq j\leq n$,
\begin{equation*}
\Phi_0^{*} K_2^{-1} X_k(0) = \diff{t_k}, \quad \Phi_0^{*}K_2^{-1}L_j(0) = \diff{\zb[j]}.
\end{equation*}
Pulling \cref{Eqn::ProofsMain::Invol1} back via $\Phi_0$ (and multiplying by $K_2^{-2}$), we have for $1\leq k,k_1,k_2\leq r$, $1\leq j,j_1,j_2\leq n$, $\zeta\in B_{\R^r\times \C^n}(1)$,
\begin{equation*}
\begin{split}
&\mleft[\Phi_0^{*} K_2^{-1} X_{k_1}, \Phi_0^{*} K_2^{-1} X_{k_2}\mright](\zeta), \mleft[ \Phi_0^{*} K_2^{-1} L_{j_1}, \Phi_0^{*} K_2^{-1} L_{j_2}\mright](\zeta), \mleft[ \Phi_0^{*} K_2^{-1} X_k, \Phi_0^{*} K_2^{-1} L_j\mright](\zeta)
\\&\in \Span_{\C}\mleft\{ \Phi_0^{*} K_2^{-1} X_1(\zeta),\ldots \Phi_0^{*} K_2^{-1} X_r(\zeta), \Phi_0^{*} K_2^{-1} L_1(\zeta), \ldots, \Phi_0^{*} K_2^{-1} L_n(\zeta)  \mright\}.
\end{split}
\end{equation*}

The above remarks show that \cref{Thm::Nirenberg::MainThm} applies to the vector fields
$$\Phi_0^{*} K_2^{-1} X_1,\ldots, \Phi_0^{*} K_2^{-1} X_r, \Phi_0^{*} K_2^{-1} L_1,\ldots, \Phi_0^{*} K_2^{-1} L_n,$$
and any constant which is $\Zygad{s}$-admissible in the sense of \cref{Thm::Nirenberg::MainThm} is $\Zygad{s}$-admissible in the sense of this section.
We let $K_1\geq 1$ be the $\Zygad{s_0}$-admissible constant from \cref{Thm::Nirenberg::MainThm},
and $\Phi_1:B_{\R^r\times \C^n}(1)\rightarrow B_{\R^r\times \C^n}(1)$ and $\AMatrix:B_{\R^r\times \C^n}(1)\rightarrow \M^{(r+n)\times (r+n)}(\C)$ be
as in \cref{Thm::Nirenberg::MainThm}.  Set $K=K_2 K_1$ and $\Phi = \Phi_0\circ \Phi_1$.
Note that $\Phi^{*} = \Phi_1^{*} \Phi_0^{*}$.
\Cref{Thm::Results::MainThm} \cref{Item::ResultsMainThm::PhiOpen} follows from \cref{Thm::Nirenberg::MainThm} \cref{Item::Nirenberg::Phi4Open}
and \cref{Prop::MainRealProp} \cref{Item::MainRealProp::Open,Item::MainRealProp::Diffeo}.
\Cref{Thm::Results::MainThm} \cref{Item::ResultsMainThem::PhiDiffeo} follows from \cref{Thm::Nirenberg::MainThm} \cref{Item::Nirenberg::PhiRegularity,Item::Nirenberg::Phi4Open}
and \cref{Prop::MainRealProp} \cref{Item::MainRealProp::Diffeo}.
\Cref{Thm::Results::MainThm} \cref{Item::Results::MainThm::Phiof0} follows from \cref{Thm::Nirenberg::MainThm} \cref{Item::Nirenberg::Phiof0} and   \cref{Prop::MainRealProp} \cref{Item::MainRealProp::Phiof0}.
\Cref{Thm::Results::MainThm} \cref{Item::Results::MainThm::AMatrix0,Item::ResultsMainThm::ABound} follow from \cref{Thm::Nirenberg::MainThm} \cref{Item::Nirenberg::AMatrix}.

Using \cref{Thm::Nirenberg::MainThm} \cref{Item::Nirenberg::AMatrix} we have
\begin{equation*}
\begin{bmatrix}
\diff{u} \\ \diff{\wb}
\end{bmatrix}
= K_1^{-1}(I+\AMatrix)
\begin{bmatrix}
\Phi_1^{*}  \Phi_0^{*} K_2^{-1} X_{K_0} \\ \Phi_1^{*} \Phi_0^{*} K_2^{-1} L_{J_0}
\end{bmatrix}
= K^{-1}(I+\AMatrix)
\begin{bmatrix}
\Phi^{*}  X_{K_0} \\ \Phi^{*}  L_{J_0}
\end{bmatrix}.
\end{equation*}
\Cref{Thm::Results::MainThm} \cref{Item::ResultsMainThm::AFormula} follows.

Because $X_1(x),\ldots,X_r(x),L_1(x),\ldots, L_n(x)$ forms a basis for $\LVS_x$, $\forall x\in B_{X_{K_0},L_{J_0}}(x_0,\chi)$,
$$\Phi^{*} X_1(\zeta),\ldots, \Phi^{*} X_r(\zeta), \Phi^{*}  L_1(\zeta),\ldots, \Phi^{*} L_n(\zeta)$$
forms a basis for $(\Phi^{*} \LVS)_{\zeta}$, $\forall \zeta\in B_{\R^r\times \C^n}(1)$.
\Cref{Thm::Results::MainThm} \cref{Item::ResultsMainThm::ABound} (which we have already shown) implies that
$$\sup_{\zeta\in B_{\R^r\times \C^n}(1)} \Norm{\AMatrix(\zeta)}[\M^{(r+n)\times (r+n)}]\leq \frac{1}{4}.$$
In particular, the matrix $I+\AMatrix(\zeta)$ is invertible, $\forall \zeta\in B_{\R^r\times \C^n}(1)$.
Hence, \cref{Thm::Results::MainThm} \cref{Item::ResultsMainThm::AFormula} (which we have already proved) implies, $\forall  \zeta\in B_{\R^r\times \C^n}(1)$,
\begin{equation*}
(\Phi^{*} \LVS)_{\zeta}=\Span_{\C}\mleft\{ \Phi^{*} X_1(\zeta), \ldots, \Phi^{*} X_r(\zeta), \Phi^{*} L_1(\zeta), \ldots, \Phi^{*} L_n(\zeta)\mright\} = \Span_{\C}\mleft\{\diff{t_1},\ldots, \diff{t_r}, \diff{\zb[1]},\ldots, \diff{\zb[n]}\mright\}.
\end{equation*}
Since $X_k(x), L_j(x)\in \LVS_x$, $\forall x$, $1\leq k\leq q$, $1\leq j\leq m$, it follows that for $1\leq k\leq q$, $1\leq j\leq m$ and $\forall \zeta\in B_{\R^r\times \C^n}(1)$ we have
\begin{equation*}
\Phi^{*} X_k(\zeta), \Phi^{*} L_j(\zeta) \in \Span_{\C}\mleft\{\diff{t_1},\ldots, \diff{t_r}, \diff{\zb[1]},\ldots, \diff{\zb[n]}\mright\}.
\end{equation*}
Because $\Phi^{*} X_k$ is a real vector field, we conclude for $1\leq k\leq q$,
\begin{equation*}
\Phi^{*} X_k(\zeta)\in \Span_{\R}\mleft\{\diff{t_1},\ldots, \diff{t_r}\mright\}, \quad \forall \zeta\in B_{\R^r\times \C^n}(1).
\end{equation*}
\Cref{Thm::Results::MainThm} \cref{Item::Results::MainThem::IsEMap} follows.
Since $\Phi^{*} = \Phi_1^{*} \Phi_0^{*}$,
\Cref{Thm::Results::MainThm} \cref{Item::ResultsMainThm::PullbacksSmooth} follows by combining
\cref{Thm::Nirenberg::MainThm} \cref{Item::Nirenberg::MainThm::ZEst}, \cref{Eqn::ProofsMain::Phi0Reg::Finite}, and \cref{Eqn::ProofsMain::Phi0Reg::Omega}.

All that remains of \cref{Thm::Results::MainThm} is \cref{Item::Results::MainThm::xi1xi2}.  We already have, by the range of $\Phi_0$, that
$\Phi(B_{\R^r\times \C^n}(1))\subseteq B_{X_{K_0},L_{J_0}}(x_0,\chi)\subseteq B_{X,L}(x_0,\xi)$ and the final two containments in \cref{Item::Results::MainThm::xi1xi2} follow.
Let $\xi_1\in (0,\xi]$ be a constant to be chosen later, and suppose $y\in B_{X_{K_0},L_{J_0}}(x_0,\xi_1)=B_{W_{P_0}}(x_0,\xi_1)$.
Thus, there exists $\gamma:[0,1]\rightarrow B_{W_{P_0}}(x_0,\xi_1)$ with $\gamma(0)=x_0$, $\gamma(1)=y$,
$\gamma'(t) = \sum_{j=1}^{2n+r} b_j(t) \xi_1 W_j(\gamma(t))$, $\Norm{\sum |b_j(t)|^2}[L^\infty]<1$.
Define
\begin{equation*}
t_0:=\sup\mleft\{  t\in [0,1] : \gamma(t')\in \Phi(B_{\R^r\times \C^n}(1/2)), \forall 0\leq t'\leq t \mright\}.
\end{equation*}
We want to show that by taking $\xi_1>0$ to be a sufficiently small $\Zygad{s_0}$-admissible constant, we have $t_0=1$
and $\gamma(1)\in \Phi(B_{\R^r\times \C^n}(1/2))$.  Note that  $t_0\geq 0$, since $\gamma(0)=x_0=\Phi(0)$.

Suppose $t_0<1$.  Then $|\Phi^{-1}(\gamma(t_0))| = 1/2$.  Using that $\CNorm{\Phi^{*} W_j}{B_{\R^r\times \C^n}(1)}[\R^{r+2n}]\lesssim_{\Zygad{s_0}} 1$ (by \cref{Thm::Results::MainThm} \cref{Item::ResultsMainThm::PullbacksSmooth} and the definition of the $W_j$),
and $\Phi(0)=x_0$ (by \cref{Thm::Results::MainThm} \cref{Item::Results::MainThm::Phiof0}) and therefore
$\Phi^{-1}(\gamma(0))=\Phi^{-1}(x_0)=0$, we have
\begin{equation*}
1/2 = |\Phi^{-1}(\gamma(t_0))| = \mleft| \int_0^{t_0} \frac{d}{dt} \Phi^{-1}\circ \gamma(t) \: dt \mright| = \mleft| \int_0^{t_0} \sum_{j=1}^{2n+r} b_j(t) \xi_1 (\Phi^{*} W_j)(\Phi^{-1} \circ\gamma(t))\: dt \mright|\lesssim_{\Zygad{s_0}} \xi_1.
\end{equation*}
This a contradiction if $\xi_1$ is a sufficiently small $\Zygad{s_0}$-admissible constant, which proves the second containment in \cref{Thm::Results::MainThm} \cref{Item::Results::MainThm::xi1xi2}.
The existence of $\xi_2>0$ as in \cref{Thm::Results::MainThm} \cref{Item::Results::MainThm::xi1xi2} follows from \cite[\SSExistXiTwo]{StovallStreetI}.
This completes the proof of \cref{Thm::Results::MainThm}.

Now let $\nu$ be a density as in \cref{Section::MainResult::Densities}.
\Cref{Prop::RealDensities} applies to $\nu$, and any constant which is $\Zygsonu$ or $\Zygad{s;\nu}$-admissible in the sense of that proposition is $\Zygsonu$ or $\Zygad{s;\nu}$-admissible, respectively, in the sense of this section.
Let $h_0$ be as in \cref{Prop::RealDensities} so that $\Phi_0^{*} \nu = h_0 \LebDensity$.  Thus,
$$h\LebDensity = \Phi^{*} \nu = \Phi_1^{*} h_0 \LebDensity = (h_0\circ \Phi_1) \det d\Phi_1 \LebDensity.$$
We conclude $h=(h_0\circ \Phi_1) \det d\Phi_1$.
\Cref{Prop::RealDensities} \cref{Item::RealDensities::h0Const}
combined with
\cref{Thm::Nirenberg::MainThm} \cref{Item::Nirenberg::JacobianConst}
yields \cref{Thm::Results::Desnity::MainResult} \cref{Item::Results::Density::hConst}.

Combining \cref{Prop::RealDensities} \cref{Item::RealDensities::hSmooth} with \cref{Thm::Nirenberg::MainThm} \cref{Item::Nirenberg::PhiRegularity}
(and using \cref{Lemma::FuncSpaceRev::Composition,Lemma::FuncSpaceRev::ComposeAnal}) shows:
\begin{itemize}
\item If $s_0\in (1,\infty)$, for $s>0$,
\begin{equation*}
	\ZygNorm{h_0\circ \Phi_1}{s}[B_{\R^r\times \C^n}(1)]\lesssim_{\Zygad{s-1;\nu}}  \left| \nu(X_1,\ldots, X_r, 2\Real(L_1),\ldots, 2\Real(L_n), 2\Imag(L_1),\ldots, 2\Imag(L_n))(x_0)\right|.
\end{equation*}
\item If $s_0=\omega$,
\begin{equation*}
	\ANorm{h_0\circ \Phi_1}{2n+r}{\min\{1,r_0\}}\lesssim_{\Zygad{\omega;\nu}}   \left| \nu(X_1,\ldots, X_r, 2\Real(L_1),\ldots, 2\Real(L_n), 2\Imag(L_1),\ldots, 2\Imag(L_n))(x_0)\right|.
\end{equation*}
\end{itemize}
Also by \cref{Thm::Nirenberg::MainThm} \cref{Item::Nirenberg::PhiRegularity} (and using \cref{Prop::FuncSpaceRev::Algebra,Lemma::FuncSpaceRev::DerivOfAnal}) we have:
\begin{itemize}
\item If $s_0\in (1,\infty)$, for $s>0$,
$
	\ZygNorm{\det d \Phi_1}{s}[B_{\R^r\times \C^n}(1)]\lesssim_{\Zygad{s-1}} 1.
$
\item If $s_0=\omega$,
$
	\ANorm{\det d \Phi_1}{2n+r}{1}\lesssim_{\Zygad{\omega}}   1.
$
\end{itemize}
Combining the above estimates and using \cref{Prop::FuncSpaceRev::Algebra} yields
\cref{Thm::Results::Desnity::MainResult} \cref{Item::Results::Density::hSmooth}.

Finally, we turn to \cref{Cor::Results::Desnity::MainCor}.
To prove this, we introduce a corollary of  \cref{Thm::Results::MainThm}.

\begin{cor}\label{Cor::Containments}
Let $\Phi$, $\xi_1$, and $\xi_2$ be as in \cref{Thm::Results::MainThm}.  Then, there exist $\Zygad{s_0}$-admissible constants $0<\xi_4\leq \xi_3\leq \xi_2$
and a map $\Phih:B_{\R^r\times \C^n}(1)\rightarrow B_{X_{K_0},L_{J_0}}(x_0,\xi_2)$, which satisfies all the same estimates as $\Phi$,
so that
\begin{equation*}
\begin{split}
&B_{X,L}(x_0,\xi_4) \subseteq B_{X_{K_0},L_{J_0}}(x_0,\xi_3)\subseteq \Phih(B_{\R^r\times \C^n}(1))\subseteq B_{X_{K_0},L_{J_0}}(x_0,\xi_2)\subseteq B_{X,L}(x_0,\xi_2)
\\& \subseteq B_{X_{K_0},L_{J_0}}(x_0,\xi_1)\subseteq \Phi(B_{\R^r\times \C^n}(1)) \subseteq B_{X_{K_0},L_{J_0}}(x_0,\chi)\subseteq B_{X_{K_0},L_{J_0}}(x_0,\xi).
\end{split}
\end{equation*}
\end{cor}
\begin{proof}
After applying \cref{Thm::Results::MainThm} to obtain $\Phi$, $\xi_1$, and $\xi_2$, we apply \cref{Thm::Results::MainThm} again with $\xi$ replaced by $\xi_2$
to obtain $\xi_3$, $\xi_4$, and $\Phih$ as above.
\end{proof}

\begin{proof}[Proof of  \cref{Cor::Results::Desnity::MainCor}]
Using \cref{Thm::Results::Desnity::MainResult} \cref{Item::Results::Density::hConst}, we have
\begin{equation*}
\begin{split}
&\nu(\Phi(B_{\R^r\times \C^n}(1))) = \int_{\Phi(B_{\R^r\times \C^n}(1))} \nu = \int_{B_{\R^r\times \C^n}(1)} \Phi^{*} \nu = \int_{B_{\R^r\times \C^n}(1)} h(t,x) \: dt\: dx
\\&\approx_{\Zygsonu}
 \nu(X_1,\ldots, X_r, 2\Real(L_1),\ldots, 2\Real(L_n), 2\Imag(L_1),\ldots, 2\Imag(L_n))(x_0),
\end{split}
\end{equation*}
with the same result with $\Phi$ replaced by $\Phih$, where $\Phih$ is as in \cref{Cor::Containments}.  Since
\begin{equation*}
\Phih(B_{\R^r\times \C^n}(1))\subseteq B_{X_{K_0},L_{J_0}}(x_0,\xi_2)\subseteq B_{X,L}(x_0,\xi_2) \subseteq \Phi(B_{\R^r\times \C^n}(1)),
\end{equation*}
and since $h(t,x)$ always has the same sign (\cref{Thm::Results::Desnity::MainResult} \cref{Item::Results::Density::hConst}), \cref{Eqn::Results::Density::MainEqual} follows.

We turn to \cref{Eqn::Results::Density::ToShow1}.  It follows from the definitions that
\begin{equation*}
\begin{split}
	&\left|\nu(X_1,\ldots, X_r, 2\Real(L_1),\ldots, 2\Real(L_n), 2\Imag(L_1),\ldots, 2\Imag(L_n))(x_0)\right|
	\\&\leq
	\max_{\substack{K\in \sI(r,q), J\in \sI(n,m)}} \left|\nu( X_K, 2\Real(L)_J, 2\Imag(L)_J)(x_0)\right|
	\\&\leq
	\max_{P\in \sI(2n+r,2m+q)} \left|\nu(W_P)(x_0)\right|.
\end{split}
\end{equation*}
Thus, with \cref{Eqn::Results::Density::MainEqual} in hand, to prove \cref{Eqn::Results::Density::ToShow1} it suffices to show
\begin{equation*}
\max_{P\in \sI(2n+r,2m+q)} \left|\nu(W_P)(x_0)\right| \lesssim_{0} \left|\nu(X_1,\ldots, X_r, 2\Real(L_1),\ldots, 2\Real(L_n), 2\Imag(L_1),\ldots, 2\Imag(L_n))(x_0)\right|;
\end{equation*}
i.e., we wish to show
\begin{equation}\label{Eqn::Results::Density::ToShow2}
\max_{P\in \sI(2n+r,2m+q)} \left|\nu(W_P)(x_0)\right| \lesssim_{0} \left|\nu(W_{P_0})(x_0)\right|.
\end{equation}
Since $W_{P_0}(x_0)$ forms a basis for the tangent space $T_{x_0} B_{X,L}(x_0,\xi)$, if the right hand side is $0$, the left hand side must be zero as well.
If the right hand side is nonzero, it follows from \cref{Lemma::AppendWedge::FormulasForQuot}
that
\begin{equation*}
\max_{P\in \sI(2n+r,2m+q)} \frac{\left|\nu(W_P)(x_0)\right|}{\left|\nu(W_{P_0})(x_0)\right|} = \max_{P\in \sI(2n+r, 2m+q)} \mleft|\frac{\bigwedge W_P(x_0)}{\bigwedge W_{P_0}(x_0)}\mright|\lesssim_0 1,
\end{equation*}
where the final inequality follows from \cref{Eqn::MainRes::RealZeta}.  \Cref{Eqn::Results::Density::ToShow2} follows, which completes the proof.
\end{proof}

\begin{rmk}
The most important special case of \cref{Thm::Results::MainThm} is the case when $r=0$.  In that case, we can always pick $J_0$ so that \cref{Eqn::MainRes::ChooseZeta} holds
with $\zeta=1$.  However, even in this case, because of \cref{Eqn::ProofsMain::RecallWP0}, we require  \cref{Prop::MainRealProp} in the general case $\zeta\in (0,1]$.
Thus, even for the reader only interested in \cref{Thm::Results::MainThm} in the case $\zeta=1$, it is important that we at least have \cref{Prop::MainRealProp} for general $\zeta\in (0,1]$.
In any case, having \cref{Thm::Results::MainThm} for general $\zeta\in (0,1]$ gives additional, convenient flexibility in applications, even when $r=0$.
\end{rmk}

\section{H\"older Spaces}\label{Section::Holder}
Let $\Omega \subset \R^n$ be a bounded, Lipschitz domain.  It follows immediately from the definitions that for $\hm\in \N$, $s\in [0,1]$ with $\hm+s>0$,
we have the containment $\HSpace{\hm}{s}[\Omega]\subseteq \ZygSpace{\hm+s}[\Omega]$.  For $\hm\in \N$, $s\in (0,1)$ we also have the reverse
containment $\ZygSpace{\hm+s}[\Omega]\subseteq \HSpace{\hm}{s}[\Omega]$; this follows easily from \cite[Theorem 1.118 (i)]{TriebelTheoryOfFunctionSpacesIII}.

When we move to the corresponding spaces with respect to $C^1$ real vector fields $W_1,\ldots, W_N$ on a $C^2$ manifold $M$,
we have similar results.  For any $\hm\in \N$, $s\in [0,1]$, $\hm+s>0$, we have $\HXSpace{W}{\hm}{s}[M]\subseteq \ZygXSpace{W}{\hm+s}[M]$;
this follows from \cref{Lemma::FuncSpaceRev::Properties}.  The reverse containment for $\hm\in \N$, $s\in (0,1)$ requires more hypotheses on the vector fields.
This is described in \cite{StovallStreetII}.

In a similar vein, we can create H\"older versions of \cref{Thm::QualE::LocalThm} and \cref{Thm::QualE::GlobalThm}.  We present these here.

Let $X_1,\ldots, X_q$ be real $C^1$ vector fields on a connected $C^2$ manifold $M$ and let $L_1,\ldots, L_m$ be complex $C^1$ vector fields on $M$.
For $x\in M$ set
\begin{equation}\label{Eqn::Holder::QualE::DefineLVS}
\LVS_x:=\Span_{\C}\{X_1(x),\ldots, X_q(x), L_1(x),\ldots, L_m(x)\},\quad \XVS_x:=\Span_{\C}\{X_1(x),\ldots, X_q(x)\}.
\end{equation}
We assume:
\begin{itemize}
\item $\LVS_x+\LVSb[x]=\C T_xM$, $\forall x\in M$.
\item $\XVS_x=\LVS_x\cap \LVSb[x]$, $\forall x\in M$.
\end{itemize}

\begin{cor}[The Local Result]\label{Cor::Holder::LocalThm}
Fix $x_0\in M$, $\hm\in \N$, $\hm\geq 1$, $s\in (0,1)$, and set $r:=\dim \XVS_{x_0}$ and $n+r:=\dim \LVS_{x_0}$.  The following three conditions are equivalent:
\begin{enumerate}[(i)]
\item\label{Item::Holder::Local::Diffeo} There exists an open neighborhood $V\subseteq M$ of $x_0$ and a $C^2$ diffeomorphism $\Phi:U\rightarrow V$, where $U\subseteq \R^r\times \C^n$ is open,
such that $\forall(t,z)\in U$, $1\leq k\leq q$,
$1\leq j\leq m$,
$$\Phi^{*}X_k(t,z)\in \Span_{\R}\left\{\diff{t_1},\ldots, \diff{t_r}\right\},\quad \Phi^{*}L_j(t,z)\in \Span_{\C}\left\{\diff{t_1},\ldots, \diff{t_r},\diff{\zb[1]},\ldots, \diff{\zb[n]}\right\},$$
and $\Phi^{*}X_k\in \HSpace{\hm+1}{s}[U][\R^r]$, $\Phi^{*}L_j\in \HSpace{\hm+1}{s}[U][\C^{r+n}]$.

\item\label{Item::Holder::Local::Basis} Reorder $X_1,\ldots, X_q$ so that $X_1(x_0),\ldots, X_r(x_0)$ are linearly independent, and reorder $L_1,\ldots, L_m$ so that $L_1(x_0),\ldots, L_n(x_0),X_1(x_0),\ldots, X_r(x_0)$
are linearly independent.  Let $\Zh_1,\ldots, \Zh_{n+r}$ denote the list $X_1,\ldots, X_r,L_1,\ldots, L_n$, and let $Y_1,\ldots, Y_{m+q-(r+n)}$ denote the list $X_{r+1},\ldots, X_q,L_{n+1},\ldots, L_m$.
There exists an open neighborhood $V\subseteq M$ of $x_0$ such that:
\begin{itemize}
\item $[\Zh_j,\Zh_k]=\sum_{l=1}^{n+r} \ch_{j,k}^{1,l} \Zh_l$, and $[\Zh_j,\Zhb[k]]=\sum_{l=1}^{n+r} \ch_{j,k}^{2,l} \Zh_l + \sum_{l=1}^{n+r} \ch_{j,k}^{3,l} \Zhb[l]$, where
$\ch_{j,k}^{a,l}\in \HXSpace{X,L}{\hm}{s}[V]$, $1\leq j,k,l\leq n+r$, $1\leq a\leq 3$.
\item $Y_j=\sum_{l=1}^{n+r} b_j^l \Zh_l$, where $b_j^l\in \HXSpace{X,L}{\hm+1}{s}[V]$, $1\leq j\leq m+q-(r+n)$, $1\leq l\leq n+r$.
\end{itemize}
Furthermore, the map $x\mapsto \dim \LVS_x$, $V\rightarrow \N$ is constant in $x$.

\item\label{Item::Holder::Local::Commute} Let $Z_1,\ldots, Z_{m+q}$ denote the list $X_1,\ldots, X_q, L_1,\ldots, L_m$.  There exists a neighborhood $V\subseteq M$ of $x_0$ such that
$[Z_j,Z_k]=\sum_{l=1}^{m+q} c_{j,k}^{1,l} Z_l$ and $[Z_j, \Zb[k]]=\sum_{l=1}^{m+q} c_{j,k}^{2,l} Z_l  + \sum_{l=1}^{m+q} c_{j,k}^{3,l} \Zb[l]$, where
$c_{j,k}^{a,l}\in \HXSpace{X,L}{\hm}{s}[V]$, $1\leq a\leq 3$, $1\leq j,k,l\leq m+q$.
Furthermore, the map $x\mapsto \dim \LVS_x$, $V\rightarrow \N$ is constant in $x$.

\end{enumerate}
\end{cor}
\begin{proof}
\Cref{Item::Holder::Local::Diffeo}$\Rightarrow$\cref{Item::Holder::Local::Basis}$\Rightarrow$\cref{Item::Holder::Local::Commute} has a nearly identical proof to the corresponding
parts of \cref{Thm::QualE::LocalThm}, and we leave the details to the reader.  Assume \cref{Item::Holder::Local::Commute} holds.
Then, since $\HXSpace{X,L}{\hm}{s}[V]\subseteq \ZygXSpace{X,L}{\hm+s}[V]$ (by \cref{Lemma::FuncSpaceRev::Properties} \cref{Item::FuncSpaceRev::HmsBoundsZygs})
we have that \cref{Thm::QualE::LocalThm} \cref{Item::QualE::Local::Commute} holds (with $s$ replaced by $\hm+s$).  Therefore \cref{Thm::QualE::LocalThm} \cref{Item::QualE::Local::Diffeo} holds
(again, with $s$ replaced by $\hm+s$); we may shrink $U$ in \cref{Thm::QualE::LocalThm} \cref{Item::QualE::Local::Diffeo} so that it is a Euclidean ball.
This establishes all of \cref{Item::Holder::Local::Diffeo}, except that it shows $\Phi^{*}X_k\in \ZygSpace{\hm+s+1}[U][\R^r]$, $\Phi^{*}L_j\in \ZygSpace{\hm+s+1}[U][\C^{r+n}]$ instead of
$\Phi^{*}X_k\in \HSpace{\hm+1}{s}[U][\R^r]$, $\Phi^{*}L_j\in \HSpace{\hm+1}{s}[U][\C^{r+n}]$.  However, since $U$ is a ball and $s\in (0,1)$ (this is the only place we use $s\ne 0,1$),
it follows from \cite[Theorem 1.118 (i)]{TriebelTheoryOfFunctionSpacesIII}
that $\ZygSpace{\hm+s+1}[U][\R^r]= \HSpace{\hm+1}{s}[U][\R^r]$.  This establishes \cref{Item::Holder::Local::Commute}$\Rightarrow$\cref{Item::Holder::Local::Diffeo} and completes the proof.
\end{proof}

\begin{rmk}
The only place where $\hm\geq 1$, $s\ne 0,1$ was used in the proof of \cref{Cor::Holder::LocalThm} was the implication \cref{Item::Holder::Local::Commute}$\Rightarrow$\cref{Item::Holder::Local::Diffeo}.  The implications \cref{Item::Holder::Local::Diffeo}$\Rightarrow$\cref{Item::Holder::Local::Basis}$\Rightarrow$\cref{Item::Holder::Local::Commute}
hold for $\hm\in \N$ and $s\in [0,1]$ with the same proof.
\end{rmk}

\begin{cor}[The Global Result]
For $\hm\in \N$,  $\hm\geq 1$, $s\in (0,1)$ the following three conditions are equivalent:
\begin{enumerate}[label=(\roman*)]
\item\label{Item::Holder::QualE::EMfld}
There exists a $\HSpace{\hm+2}{s}$ E-manifold structure on $M$, compatible with its $C^2$ structure, such that $X_1,\ldots, X_q,L_1,\ldots, L_m$ are  $\HSpace{\hm+1}{s}$
vector fields on $M$ and $\LVS$ (as defined in \cref{Eqn::Holder::QualE::DefineLVS})  is the
associated elliptic structure (see \cref{Defn::EMfld::sL}).

\item
For each $x_0\in M$, any of the three equivalent conditions from \cref{Cor::Holder::LocalThm} hold for this choice of $x_0$.

\item
Let $Z_1,\ldots, Z_{m+q}$ denote the list $X_1,\ldots, X_q,L_1,\ldots, L_m$.  Then,
$[Z_j,Z_k]=\sum_{l=1}^{m+q} c_{j,k}^{1,l} Z_l$ and $[Z_j,\Zb[k]]=\sum_{l=1}^{m+q} c_{j,k}^{2,l} Z_l + \sum_{l=1}^{m+q} c_{j,k}^{3,l} \Zb[l]$,
where $\forall x\in M$, there exists an open neighborhood $V\subseteq M$ of $x$ such that
$c_{j,k}^{a,l}\big|_V\in \HXSpace{X,L}{\hm}{s}[V]$, $1\leq a \leq 3$, $1\leq j,k,l\leq m+q$.
Furthermore, the map $x\mapsto \dim \LVS_x$, $M\rightarrow \N$ is constant.
\end{enumerate}
Furthermore, under these conditions, the $\HSpace{\hm+2}{s}$ E-manifold structure in \cref{Item::Holder::QualE::EMfld} is unique, in the sense that if $M$ has another
$\HSpace{g+2}{s}$ E-manifold structure satisfying the conclusions of \cref{Item::Holder::QualE::EMfld}, then the identity map $M\rightarrow M$ is
a $\HSpace{\hm+2}{s}$ E-diffeomorphism
between these two E-manifold structures.
\end{cor}
\begin{proof}
With \cref{Cor::Holder::LocalThm} in hand, the proof is nearly identical to the proof of \cref{Thm::QualE::GlobalThm}, and we leave the details to the interested reader.
\end{proof}

When $s\in \{0,1\}$, the use of Zygmund spaces (as in  \cref{Thm::QualE::LocalThm} and \cref{Thm::QualE::GlobalThm}) is essential.  Indeed, the above results do not hold with $s\in \{0,1\}$, at least in the special case when $n\geq 1$, $r=0$.
This follows from the next lemma.

\begin{lemma}\label{Lemma::Holder::ZygmundRequired}
Fix $n\geq 1$, $g\in \N$.  There exists an open neighborhood $V'\subseteq \C^n$ of $0$ and complex vector fields $L_1,\ldots, L_n\in \CjSpace{g+1}[V'][\C^{2n}]$ such that
\begin{enumerate}[(i)]
	\item For every $\zeta\in V$, $L_1(\zeta),\ldots, L_n(\zeta), \Lb[1](\zeta),\ldots, \Lb[n](\zeta)$ form a basis for $\C T_{\zeta} V'$.
	\item $[L_j, L_k] = \sum_{l=1}^m c_{j,k}^{1,l} L_l$ and  $[L_j, \Lb[k]] =\sum_{l=1}^m c_{j,k}^{2,l} L_l + \sum_{l=1}^m c_{j,k}^{3,l} \Lb[l]$,
where $c_{j,k}^{a,l}\in \CXjSpace{L}{g}[V']$, $1\leq a\leq 3$, $1\leq j,k,l\leq n$.
	\item\label{Item::Holder::ExistencePhi} There does not exist a $C^2$ diffeomorphism $\Phi:U\rightarrow V$, where $V\subseteq V'$ is an open neighborhood of $0$ and $U\subseteq \C^n$ is open
	such that $\Phi^{*}L_1,\ldots, \Phi^{*}L_n\in \HSpace{g}{1}[U]$ and $\forall \zeta\in U$,
	$$\Phi^{*} L_1(\zeta),\ldots, \Phi^{*}L_n(\zeta)\in \Span_{\C}\mleft\{ \diff{\zb[1]},\ldots, \diff{\zb[n]}\mright\}.$$
\end{enumerate}
\end{lemma}
\begin{proof}
The idea of the proof is that our results (e.g., \cref{Thm::QualComplex::LocalThm}) imply the sharp regularity of the classical Newlander-Nirenberg theorem,
and it is a result of Liding Yao \cite{YaoACounterExample} that sharp results for the classical Newlander-Nirenberg theorem require Zygmund spaces.
Indeed, he exhibits a set complex vector fields $L_1,\ldots, L_n\in \CjSpace{g+1}[V'][\C^{2n}]$, defined on an open neighborhood $V'$ of the origin in $\C^n$
such that
\begin{itemize}
	\item For every $\zeta\in V$, $L_1(\zeta),\ldots, L_n(\zeta), \Lb[1](\zeta),\ldots, \Lb[n](\zeta)$ form a basis for $\C T_{\zeta} V'$.
	\item There does not exist a $\HSpace{g+1}{1}$ diffeomorphism $\Phi:U\rightarrow V$, where $V\subseteq V'$ is an open neighborhood of $0$ and $U\subseteq \C^n$ is open
	such that $\forall \zeta\in U$,
	$$\Phi^{*} L_1(\zeta),\ldots, \Phi^{*}L_n(\zeta)\in \Span_{\C}\mleft\{ \diff{\zb[1]},\ldots, \diff{\zb[n]}\mright\}.$$
\end{itemize}

To see why this choice of $L_1,\ldots, L_n$ satisfies the conclusion of them lemma, note that
\begin{equation*}
	[L_j, L_k] = \sum_{l=1}^m c_{j,k}^{1,l} L_l, \quad [L_j, \Lb[k]] =\sum_{l=1}^m c_{j,k}^{2,l} L_l + \sum_{l=1}^m c_{j,k}^{3,l} \Lb[l],
\quad  c_{j,k}^{a,l}\in \CjSpace{g}[V'], 1\leq a\leq 3, 1\leq j,k,l\leq n.
\end{equation*}
Since $L_1,\ldots, L_n\in \CjSpace{g+1}[V'][\C^{2n}]$, it follows immediately from the definitions
that $\CjSpace{g}[V']\subseteq \CXjSpace{L}{g}[V']$.  Now suppose, for contradiction, $\Phi:U\rightarrow V$ is a $C^2$ diffeomorphism as in \cref{Item::Holder::ExistencePhi}.
Then, the obvious of analog of \cref{Lemma::QualEPf::RecogSmooth} for H\"older spaces shows $\Phi\in \HSpace{g+1}{1}[U]$, contradicting the choice
of $L_1,\ldots, L_n$.
\end{proof}

\appendix
\section{Immersed Submanifolds}
Let $W_1,\ldots, W_N$ be real $C^1$ vector fields on a $C^2$ manifold $\fM$.  For $x,y\in \fM$, define $\rho(x,y)$ as in \cref{Eqn::FuncMan::Defnrho}.
Fix $x_0\in \fM$ and let $Z:=\{y\in \fM : \rho(x_0,y)<\infty\}$.  $\rho$ is a metric on $Z$, and we give $Z$ the topology induced by $\rho$ (this is
finer\footnote{See \cite[\SSFinerTopology]{StovallStreetI} for a proof that this topology is finer than the subspace topology.}
than the topology as a subspace of $\fM$, and may be strictly finer).  Let $M\subseteq Z$ be a connected open subset of $Z$
containing $x_0$.  We give $M$ the topology of a subspace of $Z$.

\begin{prop}\label{Prop::AppendImmerse}
Suppose $[W_j,W_k]=\sum_{l=1}^N c_{j,k}^l W_l$, where $c_{j,k}^l:M\rightarrow \R$ are locally bounded.  Then, there is a $C^2$ manifold structure on $M$ (compatible with its topology)
such that:
\begin{itemize}
\item The inclusion $M\hookrightarrow \fM$ is a $C^2$ injective immersion.
\item $W_1,\ldots, W_N$ are $C^1$ vector fields tangent to $M$.
\item $W_1,\ldots, W_N$ span the tangent space at every point of $M$.
\end{itemize}
Furthermore, this $C^2$ structure is unique in the sense that if $M$ is given another $C^2$ structure (compatible with its topology) such that the inclusion
map $M\hookrightarrow \fM$ is a $C^2$ injective immersion, then the identity map $M\rightarrow M$ is a $C^2$ diffeomorphism between these two $C^2$ structures
on $M$.
\end{prop}

\Cref{Prop::AppendImmerse} is standard; see \cite[\SSProofInjectiveImmersion]{StovallStreetI} for a proof. 

\section{Linear Algebra}

	\subsection{Real and Complex Vector Spaces}
Let $\VVS$ be a real vector space and let $\VVS^{\C}=\VVS\otimes_{\R} \C$ be its complexification.  We consider $\VVS\hookrightarrow \VVS^{\C}$ as a real subspace
by identifying $v$ with $v\otimes 1$.  There are natural maps:
\begin{equation*}
\Real:\VVS^{\C}\rightarrow \VVS, \quad \Imag:\VVS^{\C}\rightarrow \VVS,\quad \text{complex conjugation}:\VVS^{\C}\rightarrow \VVS^{\C},
\end{equation*}
defined as follows.  Every $v\in \VVS^{\C}$ can be written uniquely as $v=v_1\otimes 1+ v_2\otimes i$, with $v_1,v_2\in V$.  Then,
$\Real(v):=v_1$, $\Imag(v):=v_2$, and $\overline{v}:=v_1\otimes 1 - v_2\otimes i$.

\begin{lemma}\label{Lemma::AppendCR::dimFormula}
Let $\LVS\subseteq \VVS^{\C}$ be a finite dimensional complex subspace.  Then,
$\dim (\LVS+\overline{\LVS}) + \dim (\LVS\bigcap \overline{\LVS}) = 2\dim(\LVS).$
\end{lemma}
\begin{proof}
It is a standard fact that $\dim (\LVS+\overline{\LVS}) + \dim (\LVS\bigcap \overline{\LVS}) = \dim(\LVS)+\dim(\overline{\LVS})$.
Using that $w\mapsto \overline{w}$, $\LVS\rightarrow \overline{\LVS}$ is an anti-linear isomorphism, the result follows.
\end{proof}

\begin{lemma}\label{Lemma::AppendVS::CRBasis}
Let $\LVS\subseteq \VVS^{\C}$ be a finite dimensional subspace.  Let $x_1,\ldots, x_r\in \LVS\bigcap \overline{\LVS}\bigcap V$ be a basis
for $\LVS\bigcap \overline{\LVS}$ and let $l_1,\ldots, l_n\in \LVS$.  The following are equivalent:
\begin{enumerate}[(i)]
\item\label{Item::AppendCR::RealBasis} $x_1,\ldots, x_r, \Real(l_1),\ldots, \Real(l_n), \Imag(l_1),\ldots, \Imag(l_n)$ is a basis for $\LVS+\overline{\LVS}$.
\item\label{Item::AppendCR::ComplexBasis} $x_1,\ldots, x_r, l_1,\ldots, l_n$ is a basis for $\LVS$.
\end{enumerate}
\end{lemma}
\begin{proof}
Clearly $r=\dim(\LVS\cap \overline{\LVS})$.

\cref{Item::AppendCR::RealBasis}$\Rightarrow$\cref{Item::AppendCR::ComplexBasis}:  Suppose \cref{Item::AppendCR::RealBasis} holds.  Then $\dim(\LVS+\overline{\LVS})=2n+r$.
\Cref{Lemma::AppendCR::dimFormula} implies $\dim(\LVS)=n+r$.  Thus, once we show $x_1,\ldots, x_r,l_1,\ldots, l_n$ are linearly independent, they will form a basis.
Suppose
\begin{equation}\label{Eqn::AppendCR::ToShowLinIndep}
\sum_{k=1}^r a_k x_k + \sum_{j=1}^n (b_j+ic_j)l_j = 0,
\end{equation}
with $a_k\in \C$, $b_j,c_j\in \R$.  We wish to show $a_k=b_j=c_j=0$, $\forall j,k$.
Applying $\Real$ to \cref{Eqn::AppendCR::ToShowLinIndep}, we see
\begin{equation*}
\sum_{k=1}^n \Real(a_k) x_k+ \sum_{j=1}^n b_j \Real(l_j)-\sum_{j=1}^n c_j \Imag(l_j)=0.
\end{equation*}
Since $x_1,\ldots, x_r, \Real(l_1),\ldots, \Real(l_n),\Imag(l_1),\ldots, \Imag(l_n)$ are linearly independent by hypothesis, we see
$\Real(a_k)=b_j=c_j=0$, $\forall j,k$.  Plugging this into \cref{Eqn::AppendCR::ToShowLinIndep} we have
\begin{equation*}
\sum_{k=1}^r i \Imag(a_k) x_k =0.
\end{equation*}
Since $x_1,\ldots, x_r$ are linearly independent, we see $\Imag(a_k)=0$, $\forall k$.  Thus, $a_k=b_j=c_j=0$, $\forall j,k$ and \cref{Item::AppendCR::ComplexBasis} follows.

\cref{Item::AppendCR::ComplexBasis}$\Rightarrow$\cref{Item::AppendCR::RealBasis}:  Suppose $x_1,\ldots, x_r, l_1,\ldots, l_n$ form a basis for $\LVS$.  Then,
$\dim(\LVS)=n+r$ and \cref{Lemma::AppendCR::dimFormula} shows $\dim(\LVS+\overline{\LVS})=2n+r$.
Thus, once we show $x_1,\ldots, x_r, \Real(l_1),\ldots, \Real(l_n), \Imag(l_1),\ldots, \Imag(l_n)$ span $\LVS+\overline{\LVS}$ it will follow that they are a basis.
But it is immediate to verify that $\Real(\LVS)$ spans $\LVS+\overline{\LVS}$, thus since $\Real(x_j)=x_j$, $\Real(ix_j)=0$, and $\Real(-il_j)=\Imag(l_j)$, it follows that
$x_1,\ldots, x_r, \Real(l_1),\ldots, \Real(l_n), \Imag(l_1),\ldots, \Imag(l_n)$ span $\LVS+\overline{\LVS}$, which completes the proof.
\end{proof}

\begin{lemma}\label{Lemma::AppendCR::FromRealFormulaToC}
Let $\LVS\subseteq \VVS^{\C}$ be a finite dimensional complex subspace.  Suppose $x_1,\ldots, x_r\in \LVS\bigcap \overline{\LVS}\bigcap \VVS$ is a basis
for $\LVS\bigcap \overline{\LVS}$ and extend this to a basis $x_1,\ldots, x_r,l_1,\ldots, l_n\in \LVS$.  Suppose $z\in \LVS$ and
\begin{equation*}
\Real(z)=\sum_{k=1}^r a_k x_k + \sum_{j=1}^n b_j\Real(l_j)+\sum_{j=1}^n c_j\Imag(l_j), \quad
\Imag(z)=\sum_{k=1}^r d_k x_k + \sum_{j=1}^n e_j \Real(l_j)+\sum_{j=1}^n f_j\Imag(l_j),
\end{equation*}
with $a_k,b_j,c_j,d_k,e_j,f_j\in \R$.  Then,
\begin{equation*}
z=\sum_{k=1}^r (a_k+id_k) x_k +\sum_{j=1}^n (b_j-ic_j) l_j.
\end{equation*}
\end{lemma}
\begin{proof}
Set $z_0=\sum_{k=1}^r (a_k+id_k) x_k +\sum_{j=1}^n (b_j-ic_j) l_j$; we wish to show $z=z_0$.
Clearly $\Real(z-z_0)=0$.
We have
$$\Imag(z-z_0)=\sum_{j=1}^n (e_j+c_j) \Real(l_j)+ \sum_{j=1}^n (f_j-b_j)\Imag(l_j)\in \Span_{\C}\{\Real(l_1),\ldots, \Real(l_n),\Imag(l_1),\ldots, \Imag(l_n)\}.$$
However, since $\Real(z-z_0)=0$,
\begin{equation*}
\Imag(z-z_0)=\frac{1}{i}(z-z_0) = -\frac{1}{i}(\overline{z}-\overline{z_0})\in \LVS\bigcap\overline{\LVS}=\Span_{\C}\{x_1,\ldots, x_r\}.
\end{equation*}
Thus,
\begin{equation*}
\Imag(z-z_0)\in \Span_{\C}\{\Real(l_1),\ldots, \Real(l_n),\Imag(l_1),\ldots, \Imag(l_n)\}\bigcap \Span_{\C}\{x_1,\ldots, x_r\}.
\end{equation*}
Since $x_1,\ldots, x_r,\Real(l_1),\ldots, \Real(l_n),\Imag(l_1),\ldots, \Imag(l_n)$ are linearly independent (by \cref{Lemma::AppendVS::CRBasis}), it follows
that $\Imag(z-z_0)=0$, which completes the proof.
\end{proof} 
	
	\subsection{Wedge Products}\label{Append::LA::Wedge}
Let $\ZVS$ be a one dimensional vector space over a field $\Field$ (we will always be using $\Field=\C\text{ or } \R$).
For $z_1,z_2\in \ZVS$, $z_1\ne 0$ we set
\begin{equation*}
\frac{z_2}{z_1}:=\frac{\lambda(z_2)}{\lambda(z_1)}\in \Field,
\end{equation*}
where $\lambda:\ZVS\rightarrow \Field$ is any non-zero linear functional.  It is easy to see that $\frac{z_2}{z_1}$ is independent of the choice of $\lambda$.

Let $\WVS$ be an $N$-dimensional vector space over $\Field$, so that $\bigwedge^N \WVS$ is a one-dimensional vector space over $\Field$.
Let $w_1,\ldots, w_N\in \WVS$ be a basis for $\WVS$ and let $w_1',\ldots, w_N'\in \WVS$.  Using the above definition, it makes sense to consider
\begin{equation*}
\frac{w_1'\wedge w_2'\wedge \cdots \wedge w_N'}{w_1\wedge w_2\wedge \cdots \wedge w_N}.
\end{equation*}

\begin{lemma}\label{Lemma::AppendWedge::FormulasForQuot}
In the above setting, the following three quantities are equal:
\begin{enumerate}[(i)]
\item\label{Item::AppendWedge::ComputeQ1} $\frac{w_1'\wedge w_2'\wedge \cdots \wedge w_N'}{w_1\wedge w_2\wedge \cdots \wedge w_N}$.
\item\label{Item::AppendWedge::ComputeQ2} $\det(B)$, where $B$ is the linear transformation defined by $Bw_j=w_j'$.
\item\label{Item::AppendWedge::ComputeQ3} $\det(C)$, where $C$ is the $N\times N$ matrix with components $c_{j}^k$, where $w_j' = \sum c_j^k w_k$.
\end{enumerate}
\end{lemma}
\begin{proof}
Clearly \cref{Item::AppendWedge::ComputeQ2} and \cref{Item::AppendWedge::ComputeQ3} are equal.
To see that \cref{Item::AppendWedge::ComputeQ1} and \cref{Item::AppendWedge::ComputeQ2} are equal, let $B$ be as in \cref{Item::AppendWedge::ComputeQ2}.  Then, we have
\begin{equation*}
\frac{w_1'\wedge w_2'\wedge \cdots \wedge w_N'}{w_1\wedge w_2\wedge \cdots \wedge w_N}
=\frac{(Bw_1)\wedge (Bw_2)\wedge \cdots \wedge (B w_N)}{w_1\wedge w_2\wedge \cdots \wedge w_N}
=\frac{\det(B)(w_1\wedge w_2\wedge \cdots \wedge w_N)}{w_1\wedge w_2\wedge \cdots \wedge w_N}
=\det(B),
\end{equation*}
completing the proof.
\end{proof}

Let $\VVS$ be a real vector space and let $\VVS^{\C}$ be its complexification.  Let $\LVS\subseteq \VVS^{\C}$ be a finite dimensional subspace
and let $\XVS:=\LVS\bigcap \overline{\LVS}$; note that $\XVS=\overline{\XVS}$.
Set $r=\dim(\XVS)$ and $n+r=\dim(\LVS)$.
Set $\WVS:=(\LVS+\overline{\LVS})\bigcap \VVS=\Span_{\R}\{\Real(l) : l\in \LVS\}\subseteq \VVS$ (so that $\WVS$ is a real vector space).
By \cref{Lemma::AppendCR::dimFormula}, $\dim(\WVS)=2n+r$.

Fix $x_1,\ldots, x_q\in \XVS\bigcap \VVS$ and $l_1,\ldots, l_m\in \LVS$ such that $\XVS=\Span_{\C}\{x_1,\ldots, x_q\}$ and $\LVS=\Span_{\C}\{x_1,\ldots, x_q, l_1,\ldots, l_m\}$.
For $K=(k_1,\ldots, k_{r_1})\in \sI(r_1,q)$ (where $\sI(r_1,q)=\{1,\ldots, q\}^{r_1}$; see \cref{Eqn::MainRes::DefnsI}), set
$\bigwedge X_K:=x_{k_1}\wedge x_{k_2}\wedge \cdots\wedge x_{k_{r_1}}$.
For $J=(j_1,\ldots, j_{n_1})\in \sI(n_1,m)$ set
\begin{equation}\label{Eqn::AppedWedge::DefineWedges}
\begin{split}
\bigwedge L_J := l_{j_1}\wedge l_{j_2}\wedge \cdots\wedge l_{j_{n_1}}, \quad &\bigwedge 2\Real(L)_J := 2\Real(l_{j_1})\wedge 2\Real(l_{j_2})\wedge \cdots \wedge 2\Real(l_{j_{n_1}}),
\\\bigwedge 2\Imag(L)_J := &2\Imag(l_{j_1})\wedge 2\Imag(l_{j_2})\wedge \cdots \wedge 2\Imag(l_{j_{n_1}}).
\end{split}
\end{equation}
Let $w_1,\ldots, w_{2m+q}$ denote the list $x_1,\ldots, x_q, 2\Real(l_1),\ldots, 2\Real(l_m),2\Imag(l_1),\ldots, 2\Imag(l_m)$, so that
$\WVS=\Span_{\R}\{w_1,\ldots, w_{2m+q}\}$.
For $P=(p_1,\ldots, p_{2n+r})\in \sI(2n+r,2m+q)$, we set $\bigwedge W_P:=w_{p_1}\wedge w_{p_2}\wedge \cdots \wedge w_{p_{2n+r}}$.

\begin{prop}\label{Prop::AppendWedge::EquivQuot}
Fix $\zeta\in (0,1]$, $J_0\in \sI(n,m)$, $K_0\in \sI(r,q)$.
\begin{enumerate}[(i)]
\item\label{Item::AppendWedge::EquivQuot1} Suppose $\left(\bigwedge X_{K_0}\right)\bigwedge \left(\bigwedge L_{J_0}\right)\ne 0$ and moreover,
\begin{equation*}
\max_{\substack{J\in \sI(n_1,m), K\in \sI(r_1,q)\\n_1+r_1=n+r}}\left| \frac{  \left(\bigwedge X_K\right)\bigwedge\left(\bigwedge L_J\right)   }{ \left(\bigwedge X_{K_0}\right)\bigwedge\left(\bigwedge L_{J_0}\right)  }  \right|
\leq \zeta^{-1}.
\end{equation*}
Then, $\left(\bigwedge X_{K_0}\right)\bigwedge \left(\bigwedge 2\Real(L)_{J_0}\right)\bigwedge \left(\bigwedge 2\Imag(L)_{J_0}\right)\ne 0$ and moreover,
\begin{equation}\label{Eqn::AppendWedge::EquivQuot1}
\max_{P\in \sI(2n+r,2m+q)} \left|  \frac{ \bigwedge W_P     }{   \left(\bigwedge X_{K_0}\right)\bigwedge \left(\bigwedge 2\Real(L)_{J_0}\right)\bigwedge \left(\bigwedge 2\Imag(L)_{J_0}\right)  }  \right|\leq \left(2\zeta^{-1}\sqrt{2n+r}\right)^{2n+r}.
\end{equation}
\item\label{Item::AppendWedge::EquivQuot2} Conversely, suppose $\left(\bigwedge X_{K_0}\right)\bigwedge \left(\bigwedge 2\Real(L)_{J_0}\right)\bigwedge \left(\bigwedge 2\Imag(L)_{J_0}\right)\ne 0$ and moreover,
\begin{equation*}
\max_{P\in \sI(2n+r,2m+q)} \left|  \frac{ \bigwedge W_P     }{   \left(\bigwedge X_{K_0}\right)\bigwedge \left(\bigwedge 2\Real(L)_{J_0}\right)\bigwedge \left(\bigwedge 2\Imag(L)_{J_0}\right)  }  \right|\leq \zeta^{-1}.
\end{equation*}
Then, $\left(\bigwedge X_{K_0}\right)\bigwedge \left(\bigwedge L_{J_0}\right)\ne 0$ and moreover,
\begin{equation}\label{Eqn::AppendWedge::EquivQuot2}
\max_{\substack{J\in \sI(n_1,m), K\in \sI(r_1,q)\\n_1+r_1=n+r}}\left| \frac{  \left(\bigwedge X_K\right)\bigwedge\left(\bigwedge L_J\right)   }{ \left(\bigwedge X_{K_0}\right)\bigwedge\left(\bigwedge L_{J_0}\right)  }  \right|
\leq \left(4\zeta^{-1}\sqrt{n+r}\right)^{n+r}.
\end{equation}
\end{enumerate}
\end{prop}

\begin{rmk}\label{Rmk::AppendWedge::AboutJ0K0}
A choice of $K_0$, $J_0$, and $\zeta$ as in \cref{Item::AppendWedge::EquivQuot1} or \cref{Item::AppendWedge::EquivQuot2} always exist:  take $K_0=(k_1,\ldots, k_r)$ and $J_0=(j_1,\ldots, j_n)$ so that
$x_{k_1},\ldots, x_{k_r},l_{j_1},\ldots, l_{j_{n}}$ form a basis for $\LVS$.  With this choice, the conditions for  \cref{Item::AppendWedge::EquivQuot1} and \cref{Item::AppendWedge::EquivQuot2} then hold for some $\zeta\in (0,1]$.
If $\XVS\bigcap \Span_{\C}\{l_1,\ldots,l_m\}=\{0\}$, one may pick $J_0$ and $K_0$ so that the conditions of \cref{Item::AppendWedge::EquivQuot1} hold with $\zeta=1$.
This occurs in the two most important special cases:  $r=0$ or $m=0$.
\end{rmk}

\begin{rmk}
The estimates \cref{Eqn::AppendWedge::EquivQuot1,Eqn::AppendWedge::EquivQuot2} are not optimal; however, we do not know the optimal estimates, and so content ourselves with proving the simplest
estimates which are sufficient for our purposes.
\end{rmk}

\begin{proof}
Suppose $K_0$, $J_0$, and $\zeta$ are as in \cref{Item::AppendWedge::EquivQuot1}; let $K_0=(k_1,\ldots, k_r)$, $J_0=(j_1,\ldots, j_n)$.
Since $\dim \LVS=n+r$ and since $x_{k_1},\ldots, x_{k_r},l_{j_1},\ldots, l_{j_n}$ are linearly independent, it follows that
$x_{k_1},\ldots, x_{k_r},l_{j_1},\ldots, l_{j_n}$ are a basis for $\LVS$.  By \cref{Lemma::AppendVS::CRBasis},
$x_{k_1},\ldots, x_{k_r}, 2\Real(l_{j_1}),\ldots, 2\Real(l_{j_n}),2\Imag(l_{j_1}),\ldots, 2\Imag(l_{j_n})$ are a basis for $\WVS$,
and therefore $\left(\bigwedge X_{K_0}\right)\bigwedge \left(\bigwedge 2\Real(L)_{J_0}\right)\bigwedge \left(\bigwedge 2\Imag(L)_{J_0}\right)\ne 0$.

Let $P=(p_1,\ldots, p_{2n+r})\in \sI(2n+r,2m+q)$.  We claim, for $t=1,\ldots, 2n+r$,
\begin{equation}\label{Eqn::AppendWedge::ToShowSumWpt}
w_{p_t} = \sum_{\alpha=1}^r a_t^\alpha x_{k_{\alpha}} + \sum_{\beta=1}^n b_t^\beta 2\Real(l_{j_\beta}) + \sum_{\beta=1}^n c_t^\beta 2\Imag(l_{j_{\beta}}), \quad \frac{1}{2}|a_t^\alpha|, |b_t^\beta|, |c_t^\beta|\leq \zeta^{-1},\forall t,\alpha,\beta.
\end{equation}
By its definition $w_{p_t}=2\Real(z)$, where $z\in \{\frac{1}{2}x_1,\ldots, \frac{1}{2}x_q, l_1,\ldots,l_m, -il_1,\ldots, -il_m\}$.  Using Cramer's rule, we have
\begin{equation*}
\begin{split}
z &= \sum_{\alpha=1}^r \frac{ x_{k_1}\wedge \cdots \wedge x_{k_{\alpha-1}}\wedge z\wedge x_{k_{\alpha+1}}\wedge \cdots \wedge x_{k_r}\wedge l_{j_1}\wedge \cdots\wedge l_{j_n} }{x_{k_1}\wedge \cdots \wedge x_{k_r}\wedge l_{j_1}\wedge \cdots\wedge l_{j_n}} x_{k_\alpha}
\\&\quad+\sum_{\beta=1}^n \frac{x_{k_1}\wedge \cdots \wedge x_{k_r}\wedge l_{j_1}\wedge \cdots\wedge l_{j_{\beta-1}}\wedge z\wedge l_{j_{\beta+1}}\wedge \cdots\wedge l_{j_n}}{x_{k_1}\wedge \cdots \wedge x_{k_r}\wedge l_{j_1}\wedge \cdots\wedge l_{j_n}}l_{j_\beta}
\\&=:\sum_{\alpha=1}^r d_{\alpha}x_{k_{\alpha}} + \sum_{\beta=1}^n e_\beta  l_{j_{\beta}},
\end{split}
\end{equation*}
where $|d_{\alpha}|,|e_\beta|\leq \zeta^{-1}$ by hypothesis.
Thus,
\begin{equation*}
\begin{split}
w_{p_t}&=2\Real(z) = z+\overline{z} = \sum_{\alpha=1}^r (d_\alpha+\overline{d_\alpha})x_{k_{\alpha}} +\sum_{\beta=1}^n (e_\beta l_{j_{\beta}}) +(\overline{e_\beta} \overline{l_{j_{\beta}}})
\\&=\sum_{\alpha=1}^r 2\Real(d_\alpha) x_{k_{\alpha}} + \sum_{\beta=1}^n \Real(e_\beta) 2\Real(l_{j_\beta}) + \sum_{\beta=1}^n -\Imag(e_{\beta}) 2\Imag(l_{j_{\beta}}).
\end{split}
\end{equation*}
\cref{Eqn::AppendWedge::ToShowSumWpt} follows.

Using \cref{Eqn::AppendWedge::ToShowSumWpt}, \cref{Lemma::AppendWedge::FormulasForQuot} shows
\begin{equation*}
 \frac{ \bigwedge W_P     }{   \left(\bigwedge X_{K_0}\right)\bigwedge \left(\bigwedge 2\Real(L)_{J_0}\right)\bigwedge \left(\bigwedge 2\Imag(L)_{J_0}\right)  }
\end{equation*}
is equal to the determinant of a $(2n+r)\times (2n+r)$ matrix, all of whose components are bounded by $2\zeta^{-1}$.
\Cref{Eqn::AppendWedge::EquivQuot1} now follows from Hadamard's inequality.

Suppose $K_0$, $J_0$, and $\zeta$ are as in \cref{Item::AppendWedge::EquivQuot2}; let $K_0=(k_1,\ldots, k_r)$, $J_0=(j_1,\ldots, j_n)$.
Since $$x_{k_1},\ldots, x_{k_r}, 2\Real(l_{j_1}),\ldots, 2\Real(l_{j_n}),2\Imag(l_{j_1}),\ldots, 2\Imag(l_{j_n})$$ are linearly independent,
and since $\dim \WVS=2n+r$, it follows that they are a basis for $\WVS$.  By \cref{Lemma::AppendVS::CRBasis}, $x_{k_1},\ldots, x_{r},l_{j_1},\ldots, l_{j_n}$
is a basis for $\LVS$, and therefore  $\left(\bigwedge X_{K_0}\right)\bigwedge \left(\bigwedge L_{J_0}\right)\ne 0$.

Let $J\in \sI(n_1,m), K\in \sI(r_1,q)$ with $n_1+r_1=n+r$.
$\left(\bigwedge X_K\right)\bigwedge \left(\bigwedge L_J\right)=z_1\wedge z_2\wedge \cdots \wedge z_{n+r}$,
where each $z_t$ is of the form $x_k$ or $l_j$ for some $j$ or $k$.  We claim
\begin{equation}\label{Eqn::AppendWedge::ToShowzt}
z_t = \sum_{\alpha=1}^r g_t^{\alpha} x_{k_{\alpha}} + \sum_{\beta=1}^n h_t^{\beta} l_{j_\beta}, \quad |g_t^{\alpha}|,|h_t^{\beta}|\leq 4\zeta^{-1},\forall t,\alpha,\beta.
\end{equation}
Indeed, suppose $z_t=l_j$ for some $j$.  Then,
\begin{equation*}
\begin{split}
&2\Real(z_t)
=\sum_{\alpha=1}^r a_t^{\alpha}x_{k_{\alpha}} + \sum_{\beta=1}^n b_t^{\beta} 2\Real(l_{j_\beta}) + \sum_{\beta=1}^n c_t^{\beta} 2\Imag(l_{j_\beta}),
\end{split}
\end{equation*}
where, by Cramer's rule,
\begin{equation*}
 a_t^{\alpha}=\frac{ \bigwedge W_{P_{t,\alpha}}     }{   \left(\bigwedge X_{K_0}\right)\bigwedge \left(\bigwedge 2\Real(L)_{J_0}\right)\bigwedge \left(\bigwedge 2\Imag(L)_{J_0}\right)  },
\end{equation*}
and $\bigwedge W_{P_{t,\alpha}}$ is defined by replacing $x_{k_{\alpha}}$ with $2\Real(z_t)$ in  $\left(\bigwedge X_{K_0}\right)\bigwedge \left(\bigwedge 2\Real(L)_{J_0}\right)\bigwedge \left(\bigwedge 2\Imag(L)_{J_0}\right) $, and therefore $|a_t^{\alpha}|\leq \zeta^{-1}$, by hypothesis.  Similarly, $|b_t^\beta|, |c_t^{\beta}|\leq \zeta^{-1}$.
Similarly,
\begin{equation*}
2\Imag{z_t} = \sum_{\alpha=1}^r d_t^{\alpha} x_{k_{\alpha}} + \sum_{\beta=1}^n e_t^{\beta} 2\Real(l_{j_\beta}) + \sum_{\beta=1}^n f_t^{\beta} 2\Real(l_{j_\beta}), \quad |d_t^{\alpha}|,|e_t^{\beta}|,|f_t^{\beta}|\leq \zeta^{-1}, \forall t,\alpha,\beta.
\end{equation*}
\Cref{Eqn::AppendWedge::ToShowzt} now follows from \cref{Lemma::AppendCR::FromRealFormulaToC} (with, in fact, $4\zeta^{-1}$ replaced by $2\zeta^{-1}$).  A similar proof
works when $z_t=x_k$ for some $k$, yielding \cref{Eqn::AppendWedge::ToShowzt}.

Using \cref{Eqn::AppendWedge::ToShowzt},  \cref{Lemma::AppendWedge::FormulasForQuot} shows
$$\frac{  \left(\bigwedge X_K\right)\bigwedge\left(\bigwedge L_J\right)   }{ \left(\bigwedge X_{K_0}\right)\bigwedge\left(\bigwedge L_{J_0}\right)  } $$
is equal to the determinant of an $(n+r)\times (n+r)$ matrix, all of whose components are bounded by $4\zeta^{-1}$.
\Cref{Eqn::AppendWedge::EquivQuot2} now follows from Hadamard's inequality.
\end{proof}

%
%
%


\bibliographystyle{amsalpha}

\bibliography{subh}

\center{\it{University of Wisconsin-Madison, Department of Mathematics, 480 Lincoln Dr., Madison, WI, 53706}}

\center{\it{street@math.wisc.edu}}

\center{MSC 2010:  32M25 (Primary), 53C17 and 32W50 (Secondary)}

\end{document}